\documentclass[10pt,a4paper]{article}

\usepackage[tt=false]{libertine}

\usepackage[utf8]{inputenc}
\usepackage{answers}
\usepackage{setspace}

\usepackage[nottoc]{tocbibind}       
\usepackage{graphicx}
\usepackage{enumitem}
\usepackage{multicol}
\usepackage{csquotes}
\usepackage[margin=2cm]{geometry}
\usepackage{tikz-cd}
\usepackage{amsmath,amsthm}
\usepackage{newtxmath}

\usepackage{pst-node}
\usepackage{mathtools}
\usepackage{diagbox}
\usepackage{quiver}
\usepackage{mathpartir}
\usepackage{tcolorbox}
\usepackage{abraces}
\usepackage{relsize}

\usepackage{xcolor}
\definecolor{newpurple}{RGB}{50, 0, 90}

\usepackage{hyperref}
\hypersetup{
	citecolor=newpurple,
	linkcolor=newpurple,
	colorlinks=true,
}

\usepackage[backend=biber,style=alphabetic]{biblatex}
\theoremstyle{definition}
\newtheorem{theorem}{Theorem}[section]
\newtheorem{lemma}[theorem]{Lemma}
\newtheorem{definition}[theorem]{Definition}
\newtheorem{example}[theorem]{Example}
\newtheorem*{definition*}{Definition}
\newtheorem*{proposition*}{Proposition}
\newtheorem{proposition}[theorem]{Proposition}

\newtheorem{remark}[theorem]{Remark}
\newtheorem{notation}[theorem]{Notation}
\newtheorem*{notation*}{Notation}
\newtheorem{convention}[theorem]{Convention}
\newtheorem*{theorem*}{Theorem}
\newtheorem*{lemma*}{Lemma}
\newtheorem*{convention*}{Convention}
\newtheorem*{example*}{Example}

\newtheorem{corollary}[theorem]{Corollary}
\newtheorem{construction}[theorem]{Construction}

\numberwithin{equation}{section}

\newcommand{\Id}{\text{Id}}

\newcommand{\Ar}{\text{Ar}}
\newcommand{\iiAr}{\mathbf{Ar}}
\newcommand{\PSh}{\text{PSh}}

\newcommand{\Hom}{\text{Hom}}
\newcommand{\iiHom}{\textbf{Hom}}

\newcommand{\Set}{\text{Set}}

\newcommand{\Ob}{\text{Ob}}
\newcommand{\iiOb}{\mathbf{Ob}}

\newcommand{\Fun}{\text{Fun}}

\newcommand{\Cat}{\text{Cat}}

\newcommand{\cod}{\text{cod}}

\newcommand{\ct}{\text{ct}}

\newcommand{\GAT}{\text{GAT}}

\newcommand{\Perm}{\text{Perm}}

\newcommand{\doverline}[1]{\overline{\overline{#1}}}

\newcommand{\Cont}{\text{Cont}}
\newcommand{\iiCont}{\mathbf{Cont}}

\newcommand{\Precont}{\text{Precont}}
\newcommand{\iiPrecont}{\mathbf{Precont}}

\newcommand{\cont}{\text{cont}}
\newcommand{\iicont}{\textbf{cont}}

\newcommand{\Mod}{\text{Mod}}

\newcommand{\gap}{\text{gap}}

\newcommand{\Att}{\text{Att}}
\newcommand{\iiAtt}{\mathbf{Att}}

\newcommand{\att}{\text{att}}
\newcommand{\iiatt}{\textbf{att}}

\renewcommand{\lim}{\underleftarrow{\text{lim }}}

\newcommand{\tp}{\;\mathsf{sort}}

\newcommand{\Ty}{\text{Ty}}

\newcommand{\Fam}{\text{Fam}}
\newcommand{\fm}{\text{fam}}

\newcommand{\bbA}{\mathbb A}
\newcommand{\bbB}{\mathbb B}
\newcommand{\bbC}{\mathbb C}
\newcommand{\bbD}{\mathbb D}

\newcommand{\bbN}{\mathbb N}
\newcommand{\bbO}{\mathbb O}

\newcommand{\Cell}{\text{cd}}

\newcommand{\tle}{\vartriangleleft}
\newcommand{\btle}{\blacktriangleleft}

\newcommand{\LFPos}{\text{LFPos}}

\newcommand{\Sieve}{\text{Sieve}}

\newcommand{\cub}{\text{cub}}

\def\circrightarrow{\mathrel{{
			\setbox0\hbox{$\longrightarrow$}
			\rlap{\hbox to \wd0{\hss$\circ$\hss}}\box0
}}}

\def\bulletrightarrow{\mathrel{{
			\setbox0\hbox{$\longrightarrow$}
			\rlap{\hbox to \wd0{\hss$\bullet$\hss}}\box0
}}}

\newcommand{\squa}[8]{\begin{tikzcd}[ampersand replacement=\&]
		{#1} \& {#2} \\
		{#3} \& {#4}
		\arrow["{#7}"', from=1-1, to=2-1]
		\arrow["{#6}"', from=2-1, to=2-2]
		\arrow["{#8}", from=1-2, to=2-2]
		\arrow["{#5}", from=1-1, to=1-2]
	\end{tikzcd}
}

\newcommand{\widesqua}[8]{\begin{tikzcd}[ampersand replacement=\&,column sep=2.5cm]
		#1 \arrow[swap]{d}{#7} \arrow[]{r}{#5} \& #2 \arrow[]{d}{#8}\\
		#3 \arrow[swap]{r}{#6} \& #4
	\end{tikzcd}
}

\newcommand{\dsqua}[8]{\begin{tikzcd}[ampersand replacement=\&]
		#1 \arrow[swap,->>]{d}{#7} \arrow[]{r}{#5} \& #2 \arrow[->>]{d}{#8}\\
		#3 \arrow[swap]{r}{#6} \& #4
	\end{tikzcd}
}

\newcommand{\widedsqua}[8]{\begin{tikzcd}[ampersand replacement=\&,column sep=2.5cm]
		#1 \arrow[swap,->>]{d}{#7} \arrow[]{r}{#5} \& #2 \arrow[->>]{d}{#8}\\
		#3 \arrow[swap]{r}{#6} \& #4
	\end{tikzcd}
}

\makeindex                              

\begin{document}

\title{A monoidal category of dependently sorted algebraic theories\\
II: \emph{categorical aspects}}

\author{Daniel Almeida\thanks{email: ddeal056@uottawa.ca}}

\date{}

\maketitle

\begin{abstract}
This is the second of a pair of papers where we construct and investigate a closed monoidal structure on the category of generalized algebraic theories (in the sense of Cartmell).

Having presented the tensor product of theories in a syntactic way, we now study the same structure from the perspective of contextual categories. We define the exponential $\mathcal A^\mathcal B$ between two contextual categories $\mathcal A$, $\mathcal B$, and show how this yields, as a particular case, a cotensor $\mathcal A^B$ by a small category $B$. We also introduce a concept of multimorphism $(\mathcal A_1, ..., \mathcal A_n) \rightarrow \mathcal B$ for contextual categories $\mathcal A_i$, $\mathcal B$, and describe a bijective correspondence between bimorphisms $(\mathcal A, \mathcal B) \rightarrow \mathcal C$ and morphisms $\mathcal A \rightarrow \mathcal C^\mathcal B$. We give an abstract proof that there exists a contextual category $\mathcal A \otimes \mathcal B$ such that bimorphisms $(\mathcal A, \mathcal B) \rightarrow \mathcal C$ are in natural bijection with morphisms $\mathcal A \otimes \mathcal B \rightarrow \mathcal C$.

We extend $\otimes:\Cont \times \Cont \rightarrow \Cont$ into a closed symmetric monoidal structure and give a description of certain pushout-tensor maps that, in particular, allows us to prove that the tensor product of theories from \cite{Alm25} is functorial and presents the one constructed here.
\end{abstract}

\setcounter{tocdepth}{2}
\tableofcontents

\section{Introduction}
Generalized algebraic theories (\textsc{gat}s for short), introduced by J. Cartmell in the 1970s (\cite{Car78}, \cite{Car86}), expand the framework of classical universal algebra, focused on the study of equational theories, by allowing the introduction of dependent sorts. Similarly to how equational theories were translated into categorical language in Lawvere's 1963 thesis (\cite{Law04}), Cartmell modelled generalized algebraic theories using a class of structures called contextual categories. The main result of \cite{Car78} is the proof that, writing $\GAT$ (resp. $\Cont$) for the category of generalized algebraic theories and equivalence classes of interpretations (resp. of contextual categories and contextual functors), the functor
$$
\mathcal C(-):\GAT \rightarrow \Cont,
$$
associating with each theory its syntactic category is an equivalence of categories.

In \cite{Alm25}, we defined the tensor product product $\bbA \otimes \bbB$ of two generalized algebraic theories $\bbA$ and $\bbB$, extending Freyd's tensor product (or Kronecker product) of algebraic theories (\cite{Fre66}), and in analogy with Kelly's tensor product of finitely (co)complete categories (\cite{Kel82}). This was done in an algorithmic and purely syntactic way: the axioms of $\bbA \otimes \bbB$ are exhibited mechanically from those of $\bbA$ and $\bbB$, a feature that is helpful in performing explicit calculations and comparing our construction with previously known ones. As a downside, we had very little access to information allowing us to work with that construction in a categorical way. In particular, our techniques were not sufficient to prove that $\otimes$ defines a (symmetric) monoidal structure on the category $\GAT$ of generalized algebraic theories and equivalence classes of interpretations: we made limited progress towards turning $\otimes$ into a functor $\GAT \times \GAT \rightarrow \GAT$, and we sketched the construction of isomorphisms $\bbA \otimes (\bbB \otimes \bbC) \cong (\bbA \otimes \bbB) \otimes \bbC$. One of the main obstructions was the lack of a presentation-independent formulation of the tensor product of \textsc{gat}s, in the sense that a \textsc{gat} $\bbA$ can be thought of as a presentation of the structure exhibited by the syntactic (contextual) category $\mathcal C(\bbA)$ --- see \cite{Car86}, \S13, 14. Alternatively, we had not been able to describe what exactly $\otimes$ amounts to once we move across the equivalence $\mathcal C(-):\GAT \simeq \Cont$, where $\Cont$ denotes the category of contextual categories and contextual functors.

The present text aims at filling these gaps in the study of the tensor product and, simultaneously, exploring other closely related structures, such as multimorphisms and exponentials of contextual categories. In fact, aside from the characterization of the tensor product of \textsc{gat}s in terms of contextual categories, most of this article does not use the syntax of \textsc{gat}s and can be read independently of \cite{Alm25}.

\vspace{0.5em}

For contextual categories $\mathcal A$, $\mathcal B$, $\mathcal C$, a bimorphism $(\mathcal A,\mathcal B) \rightarrow \mathcal C$\footnote{In fact, in the definition of a bimorphism $(\mathcal A,\mathcal B) \rightarrow \mathcal C$ we allow $\mathcal A$, $\mathcal B$ to be, more generally, precontextual categories; see the discussion below.} from $(\mathcal A,\mathcal B)$ to $\mathcal C$ will be a special kind of functor $\mathcal A \times \mathcal B \rightarrow \mathcal C$ where $\mathcal A \times \mathcal B$ denotes the cartesian product of the underlying categories.\footnote{Which differs, generally, from cartesian product of $\mathcal A$ and $\mathcal B$ in $\Cont$.} The definition is designed, in particular, to capture the main features of the canonical functor $\mathcal C(\bbA) \times \mathcal C(\bbB) \rightarrow \mathcal C(\bbA \otimes \mathcal B)$ from \cite{Alm25}, \S 6; and indeed, the latter is verified to be a bimorphism in Proposition \ref{prop: comparison functor is bimorphism}. But more generally, we would like a bimorphism $(\mathcal C(\bbA), \mathcal C(\bbB)) \rightarrow \mathcal E$ to decompose as
$$
\mathcal C(\bbA) \times \mathcal C(\bbB) \overset{\otimes_{\bbA,\bbB}}{\longrightarrow} \mathcal C(\bbA \otimes \bbB) \longrightarrow \mathcal E
$$
for a unique contextual functor $\mathcal C(\bbA \otimes \bbB) \rightarrow \mathcal E$. In other words, we want $\otimes_{\bbA,\bbB}$ to be a \emph{universal} bimorphism. This means that the axioms for a bimorphism should capture, even if indirectly, the same kind of interaction between $\bbA$ and $\bbB$ as $\bbA \otimes \bbB$ does (including the apparent asymmetry imposed by the choice of lexicographic order; see \cite{Alm25}, \S8).

This universal property of $\mathcal C(\bbA \otimes \bbB)$, which then gives $\bbA \otimes \bbB \in \GAT \simeq \Cont$ a universal property, will be crucial in filling some of the gaps from \cite{Alm25}. In particular, it immediately implies that the tensor product defines a functor $\otimes:\GAT \times \GAT \rightarrow \GAT$, while from syntactic approach it was not clear even whether $\bbA \cong \bbA'$ and $\bbB \cong \bbB'$ yield $\bbA \otimes \bbB \cong \bbA' \otimes \bbB'$. This also means that we can talk about $\mathcal A \otimes \mathcal B$ (as a structure defined up to canonical isomorphism) for contextual categories $\mathcal A$, $\mathcal B$.

The definition of a bimorphism that we adopt is inspired by the ``cofibration part" of the concept of a Quillen bifunctor from categorical homotopy theory: after passing to opposite categories\footnote{A clearer analogy is achieved by working with ``co-contextual categories".}, we require that certain colimits are preserved in each variable, and that a pushout-product operation sends display maps to display maps. However, ensuring that $\otimes_{\bbA,\bbB}$ has a $1$-categorical universal property requires asking bimorphisms to be strictly compatible with the non-categorical structure present in contextual categories: lengths of objects, display maps, and distinguished squares.\footnote{Recall that the latter two are not closed under isomorphism in the categories of arrows and of commutative squares, resp., of the contextual category.} Our list of axioms is discussed in \S2.3.

In fact, bimorphisms $(\mathcal C(\bbA), \mathcal C(\bbB)) \rightarrow \mathcal E$ can be described in a recursive way once we present $\bbA$, resp. $\bbB$, as the increasing union of an ordinal sequence of theories where at level $\nu+1$ we add as an axiom a judgment that is well-formed at level $\nu$. This will be a key point in proving that $\otimes_{\bbA,\bbB}$ is a universal bimorphism. To have a slightly more precise idea of how this will work, consider theory extensions $\bbA \rightarrow \bbA'$ and $\bbB \rightarrow \mathbb B'$, each adding a single axiom, say $J$ and $J'$, respectively. Then we have a commutative diagram\footnote{Note that this does not require functoriality of the tensor product, which we currently have no access to. We only use that, by definition, $\bbA \otimes \bbB'$ is an extension of $\bbA \otimes \bbB$, and so on.}
\[
\squa{\bbA \otimes \bbB}{\bbA \otimes \bbB'}{\bbA' \otimes \bbB}{\bbA' \otimes \bbB'}{}{}{}{}
\]
and also a theory $(\bbA \otimes \bbB') \cup (\bbA' \otimes \bbB)$ -- which is a pushout $(\bbA \otimes \bbB') \sqcup_{\bbA \otimes \bbB} (\bbA' \otimes \bbB)$ in $\GAT$ -- obtained by taking the union of the two sets of axioms. Then, by construction, either $\bbA' \otimes \bbB'$ is $(\bbA \otimes \bbB') \cup (\bbA' \otimes \bbB)$ itself, or it is obtained from the latter by adding an axiom $J \odot J'$, in which case we have a pushout diagram
\[
\squa{\mathcal K}{\mathcal C((\bbA \otimes \bbB') \cup (\bbA' \otimes \bbB))}{\mathcal K'}{\mathcal C(\bbA' \otimes \bbB')}{\chi}{}{\kappa}{}
\]
where $\kappa$ is a ``free-standing axiom": it only depends on what kind $J \odot J'$ is of (sort, term, sort equality, or term equality) and on the length of its context. For example, if $J \odot J'$ is a sort judgment whose context has length $n$, then $\kappa$ will be the canonical embedding $\mathcal O_n \rightarrow \mathcal O_{n+1}$ of the contextual category freely generated by a length-$n$ object into the one freely generated by a length-$(n+1)$ object.\footnote{See the discussion below on precontextual categories.} On the other hand, $\chi$ will be obtained by expressing $J \odot J'$ in terms of objects and morphisms in the images of
$$
\mathcal C(\bbA) \times \mathcal C(\bbB') \overset{\otimes_{\bbA,\bbB'}}{\longrightarrow} \mathcal C(\bbA \otimes \bbB') \quad \text{and} \quad \mathcal C(\bbA') \times \mathcal C(\bbB) \overset{\otimes_{\bbA',\bbB}}{\longrightarrow} \mathcal C(\bbA' \otimes \bbB).
$$
Now, the problem of classifying contextual functors $\mathcal C(\bbA' \otimes \bbB') \rightarrow \mathcal E$ splits into two:
\begin{enumerate}[label=(\roman*)]
	\item classifying morphisms $\mathcal C((\bbA \otimes \bbB') \cup (\bbA' \otimes \bbB)) \rightarrow \mathcal E$, hence commutative diagrams
	\[
	\squa{\mathcal C(\bbA \otimes \bbB)}{\mathcal C(\bbA \otimes \bbB')}{\mathcal C(\bbA' \otimes \bbB)}{\mathcal E;}{}{}{}{}
	\]
	
	\item classifying, for a given morphism $F:\mathcal C((\bbA \otimes \bbB') \cup (\bbA' \otimes \bbB)) \rightarrow \mathcal E$, extensions of $F \circ \chi$ along $\kappa$.
\end{enumerate}

If we assume that $\otimes_{\bbA,\bbB}$, $\otimes_{\bbA,\bbB'}$ and $\otimes_{\bbA',\bbB}$ are universal bimorphisms, then diagrams as in (i) correspond precisely to pairs consisting of bimorphisms $G:\mathcal C(\bbA) \times \mathcal C(\bbB') \rightarrow \mathcal E$ and $H:\mathcal C(\bbA') \times \mathcal C(\bbB) \rightarrow \mathcal E$ such that
\[
\squa{\mathcal C(\bbA) \times \mathcal C(\bbB)}{\mathcal C(\bbA) \times \mathcal C(\bbB')}{\mathcal C(\bbA') \times \mathcal C(\bbB)}{\mathcal E}{}{H}{}{G}
\]
commutes. These determine a class of functors $(\mathcal C(\bbA) \times \mathcal C(\bbB')) \sqcup_{\mathcal C(\bbA) \times \mathcal C(\bbB)} (\mathcal C(\bbA') \times \mathcal C(\bbB)) \rightarrow \mathcal E$, namely, those that give bimorphisms after restriction to the two components. This is the point where we are able to establish the desired bijective correspondence between morphisms $\mathcal C(\bbA' \otimes \bbB') \rightarrow \mathcal E$ and bimorphisms $\mathcal C(\bbA') \times \mathcal C(\bbB') \rightarrow \mathcal E$: consider a morphism $F:\mathcal C(\bbA \otimes \bbB') \sqcup_{\mathcal C(\bbA \otimes \bbB)} \mathcal C(\bbA' \otimes \bbB) \rightarrow \mathcal E$, and let $F':(\mathcal C(\bbA) \times \mathcal C(\bbB')) \sqcup_{\mathcal C(\bbA) \times \mathcal C(\bbB)} (\mathcal C(\bbA') \times \mathcal C(\bbB)) \rightarrow \mathcal E$ be the corresponding ``componentwise bimorphism"; then we can explicitly classify
\begin{itemize}
	\item extensions of $F \circ \chi$ along $\kappa:\mathcal K \rightarrow \mathcal K'$ (which is straightforward once we know how to describe $\chi$), and
	
	\item extensions of $F'$ to a bimorphism $\mathcal C(\bbA') \times \mathcal C(\bbB') \rightarrow \mathcal E$,
\end{itemize}
and construct a bijection between these data.

\vspace{0.5em}

Generally, suppose given \textsc{gat}s $\bbA$ and $\bbB$. Consider a sequence of theories $(\bbA_\mu)_{\mu \le \alpha}$, where $\alpha$ is an ordinal, such that
\begin{itemize}
	\item $\bbA_\alpha = \bbA$;
	
	\item for $\mu < \alpha$, the theory $\bbA_{\mu+1}$ is obtained from $\bbA_\mu$ by adding a single axiom;
	
	\item for a limit ordinal $\mu \le \alpha$, we have $\bbA_\mu = \bigcup_{\mu' < \mu}\bbA_{\mu'}$ (in particular, $\bbA_0 = \varnothing$).
\end{itemize}
Similarly, present $\bbB$ via a sequence $(\bbB_\nu)_{\nu < \beta}$. Then we can obtain by recursion on the well-founded poset $\{\mu+1\} \times \{\nu+1\}$, using the aforementioned construction as the recursion step, a bijection between the set of contextual functors $\mathcal C(\bbA \otimes \bbB) \rightarrow \mathcal E$ and that of bimorphisms $\mathcal C(\bbA) \times \mathcal C(\bbB) \rightarrow \mathcal E$.

\vspace{0.5em}

We will also introduce (Definition \ref{def: exponential}) an exponentiation operation $(\mathcal A,\mathcal C) \mapsto \mathcal C^\mathcal A$ for which we have isomorphisms
$$
\Hom_\Cont(\mathcal A \otimes \mathcal B,\mathcal C) \cong \Hom_\Cont(\mathcal B,\mathcal C^\mathcal A)
$$
natural in $\mathcal A$, $\mathcal B \in \Cont^{\text{op}}$ and $\mathcal C \in \Cont$. Note that, by the Yoneda lemma, knowledge of the exponentiation functor characterizes the tensor product up to canonical isomorphism. On the other hand, it turns out that $\mathcal C^\mathcal A$ can be described explicitly and independently of the tensor product, and even of the concept of a bimorphism: it will be presented by a certain category with attributes $D_{\text{att}}(\mathcal A,\mathcal C)$ whose dependent types encode a relative notion of morphism $\mathcal A \rightarrow \mathcal C$ (which we call an \emph{indexed sort}). This construction takes place naturally in the framework of categories with attributes, hence in that of contextual categories, and it is not clear to me how to express it via the syntax of \textsc{gat}s in a simple way.

When $\mathcal A = \mathcal C(D)$ for a locally finite direct category $D$, we expect $\mathcal C^\mathcal A$ to be isomorphic to (the contextual category presented by) the \textsc{cwa} $\mathcal C^{D^{\text{op}}}$ in the sense of \cite{KapLum21}. While our exponent $\mathcal A$ can be any contextual category, our construction doesn't generalize theirs since they work with a base \textsc{cwa} $\mathcal C$ that is not necessarily contextual. Also, working with $\mathcal C(D)$ instead of directly with $D$ simplifies the presentation at an abstract level at the cost of hiding much of the interesting combinatorics of $D$ itself.\footnote{This is related to the fact, mentioned above, that we have no canonical embedding $D^{\text{op}} \rightarrow \mathcal C(D)$ up to equality of functors. We only do up to canonical isomorphism, which is why using $D$ directly requires choosing some additional structure on it, such as a suitable linear order on the arrows of each slice $D/x$. See \cite{KapLum21}, Def. 3.17 and the remark that follows.}

Although the definition of $\mathcal C^\mathcal B$ is not as direct as that of a bimorphism $(\mathcal A, \mathcal B) \rightarrow \mathcal C$, we believe it can be illuminating to understand how the former, seen as a structure of independent interest, leads to the latter. The connection between the tensor product and the exponentiation is useful in justifying the definition of the tensor product as an appropriate one, while parts of the definition of a bimorphism might seem, at first glance, arbitrary. For instance, the asymmetry between the roles of $\mathcal A$ and $\mathcal B$ in the definition of a bimorphism $(\mathcal A,\mathcal B) \rightarrow \mathcal C$ becomes evident when the latter is viewed as a contextual functor $\mathcal B \rightarrow \mathcal C^\mathcal A$. Because of that, we chose to first introduce $\mathcal C^\mathcal A$, and then define bimorphisms $(\mathcal A,\mathcal B) \rightarrow \mathcal C$ in such a way that these are in one-to-one correspondence with morphisms $\mathcal B \rightarrow \mathcal C^\mathcal A$ (Proposition \ref{prop: characterization of maps to exponential, sym}).

We also note that the full subcategory of $\mathcal C^\mathcal A$ spanned by its length-$1$ objects --- denoted by $\iiOb_1(\mathcal C^\mathcal A)$ in \S\ref{sec: compatibility with the Cat-enrichment} --- is isomorphic to the category of contextual functors $\mathcal A \rightarrow \mathcal C$ and natural transformations between them.

\subsubsection*{Poset-shaped cellular diagrams, and multimorphisms}

Bimorphisms are, as they are defined, difficult to work with. Note that although a contextual category $\mathcal A \otimes \mathcal B$ is characterized, up to isomorphism, by the statement that we have a universal bimorphism $\otimes_{\mathcal A,\mathcal B}:\mathcal A \times \mathcal B \rightarrow \mathcal A \otimes \mathcal B$, part of the structure of $\mathcal A \otimes \mathcal B$ is only captured implicitly by $\otimes_{\mathcal A,\mathcal B}$. There are a few instances where this will be relevant in the text:
\begin{enumerate}[label=(\arabic*)]
	\item We want to recover the symmetry isomorphism $\bbA \otimes \bbB \cong \bbB \otimes \bbA$ from \cite{Alm25}, \S8 in terms of a correspondence between bimorphisms $(\mathcal C(\bbA),\mathcal C(\bbB)) \rightarrow \mathcal E$ and bimorphisms $(\mathcal C(\bbB),\mathcal C(\bbA)) \rightarrow \mathcal E$. However, this correspondence cannot be expressed purely in terms of the images of $\otimes_{\mathcal A,\mathcal B}$ and $\otimes_{\mathcal B,\mathcal A}$. Indeed, for contexts $\textbf{X}$, $\textbf{Y}$ in $\bbA$, $\bbB$, the isomorphism $\bbA \otimes \bbB \cong \bbB \otimes \bbA$ sends $[\textbf{X} \otimes \textbf{Y}]$ to $[\textbf{Y} \otimes \textbf{X}]$ up to isomorphism, but in general not strictly. Accordingly, if $\mathcal A \times \mathcal B \rightarrow \mathcal E$ is a bimorphism from $(\mathcal A,\mathcal B)$ to $\mathcal E$, then its composite with the canonical isomorphism of categories $\mathcal B \times \mathcal A \cong \mathcal A \times \mathcal B$ is not a bimorphism from $(\mathcal B,\mathcal A)$ to $\mathcal E$. Thus we need another strategy for transforming bimorphisms $(\mathcal A,\mathcal B) \rightarrow \mathcal E$ into bimorphisms $(\mathcal B,\mathcal A) \rightarrow \mathcal E$.
	
	\item Similarly, constructing an isomorphism $(\mathcal A \otimes \mathcal B) \otimes \mathcal C \cong \mathcal A \otimes (\mathcal B \otimes \mathcal C)$ is equivalent to obtaining a bijection, natural in $\mathcal E$, between bimorphisms $(\mathcal A \otimes \mathcal B,\mathcal C) \rightarrow \mathcal E$ and bimorphisms $(\mathcal A,\mathcal B \otimes \mathcal C) \rightarrow \mathcal E$. But these are certain functors of the forms $(\mathcal A \otimes \mathcal B) \times \mathcal C \rightarrow \mathcal E$ and $\mathcal A \times (\mathcal B \otimes \mathcal C) \rightarrow \mathcal E$, respectively, and so cannot be compared directly. To deal with that, the idea is to show that the two kinds of functors correspond, by restriction along $\otimes_{\mathcal A,\mathcal B} \times \Id_C:(\mathcal A \times \mathcal B) \times \mathcal C \rightarrow (\mathcal A \otimes \mathcal B) \times \mathcal C$ and $\Id_A \times \otimes_{\mathcal B,\mathcal C}:\mathcal A \times (\mathcal B \times \mathcal C) \rightarrow \mathcal A \times (\mathcal B \otimes \mathcal C)$, resp., to the same class of functors $\mathcal A \times \mathcal B \times \mathcal C \rightarrow \mathcal E$.
	
	\item We want to be able to characterize bimorphisms $(\mathcal A,\mathcal B) \rightarrow \mathcal E$ in terms of suitable generating data of $\mathcal A$ and $\mathcal B$. More precisely, if $\mathcal A'$ (resp. $\mathcal B'$) is a \emph{precontextual category}, as discussed below, that is completed into the contextual category $\mathcal A$ (resp. $\mathcal B$) via a reflection morphism $\mathcal A' \rightarrow \mathcal A$ (resp. $\mathcal B' \rightarrow \mathcal B$), then we expect to have a bijection between bimorphisms $(\mathcal A',\mathcal B') \rightarrow \mathcal E$ and bimorphisms $(\mathcal A,\mathcal B) \rightarrow \mathcal E$. But the former are often much easier to describe, and this will be fundamental in the recursive proof that $\otimes_{\bbA,\bbB}:\mathcal C(\bbA) \times \mathcal C(\bbB) \rightarrow \mathcal C(\bbA \otimes \mathcal B)$ is a universal bimorphism. The point is that if $\bbD \rightarrow \bbD'$ is an extension by a single axiom, it can be hard to explicitly describe $\mathcal C(\bbD')$ in terms of $\mathcal C(\bbD)$. But it will be straightforward to factorize $\mathcal C(\bbD) \rightarrow \mathcal C(\bbD')$ as $\mathcal C(\bbD) \rightarrow \mathcal D' \rightarrow \mathcal C(\bbD')$ where $\mathcal D'$ is a precontextual category and $\mathcal D' \rightarrow \mathcal C(\bbD')$ is a reflection morphism.
\end{enumerate}

Accomplishing these goals will be possible by working with what we will call ``cellular diagrams" whose shape is a finite poset $(P,\le)$ endowed with a linear refinement of its order, that is, a linear order $\tle$ on $P$ such that $a \tle b$ whenever $a \le b$. We expect our construction to be obtainable by specializing the Reedy diagrams from \cite{KapLum21} to the case where the domain inverse category is a poset and the codomain \textsc{cwa} is a contextual category,\footnote{See \cite{KapLum21}, Remark 7.3 for instances and variants of their construction that had been previously considered in the literature.} but we don't establish a formal correspondence.

As discussed in \cite{Sub21}, following Makkai's use of simple categories in the study of first-order logic with dependent sorts (\cite{Mak95}), there is close connection between locally finite direct categories\footnote{A category $D$ is \emph{direct} if it has no infinite sequence $\cdots \rightarrow \bullet \rightarrow \bullet \rightarrow \bullet$ of non-identity morphisms, and it is \emph{locally finite} if the slice $D/x$ is finite for every object $x$. A category is locally finite and direct precisely when its opposite category is simple in the sense of \cite{Mak95}, or inverse (which includes a finiteness assumption) in the sense of \cite{KapLum21}.} and dependent type signatures. A lfdc $D$ defines a contextual category $\mathcal C(D)$ whose underlying category is, essentially, opposite to the category $\PSh_{fp}(D)$ of finitely presentable presheaves on $D$.\footnote{Although the contextual structure on $\mathcal C(D)$, which involves operations on the \emph{set} of objects, cannot be transferred to one on $\PSh_{fp}(D)$.} In turn, $\mathcal C(D)$ can be presented by a \textsc{gat} $\bbD$ having one sort axiom for each object of $D$; conversely, every \textsc{gat} that only has sort axioms arises essentially uniquely, by this process, from a locally finite direct category. Restricting to the case where $D$ is a finite poset, we obtain those theories that only have finitely many sort axioms, and such that the context of each axiom lists any sort symbol at most once.

The underlying idea, implicit in Proposition \ref{prop: representability cellular diagrams}, is to try to characterize contextual functors $\mathcal C(D) \rightarrow \mathcal E$ for a given contextual category $\mathcal E$ in terms of their restrictions $D^{\text{op}} \rightarrow \mathcal E$ along a functor $I:D^{\text{op}} \rightarrow \mathcal C(D)$ such that the composite $D^{\text{op}} \overset{I}{\rightarrow} \mathcal C(D) \simeq \PSh_{fp}(D)^{\text{op}} \hookrightarrow \PSh(D)^{\text{op}}$ is isomorphic to the Yoneda embedding. Since $\mathscr Y:D \rightarrow \PSh(D)$ is full-and-faithful and objects of $D$ have no non-identity automorphisms, choosing $I$ amounts to specifying for each $a \in D$ an object of $\mathcal C(D)$ isomorphic (such an isomorphism being unique) to $\mathscr Y(a)$ in $\PSh(D)$. Once $I$ has been chosen, we would like to characterize when a functor $D \rightarrow \mathcal E$ is ``cellular" in the sense that it is the restriction along $I$ of a contextual functor $\mathcal C(D) \rightarrow \mathcal E$. Doing that when $D$ is a finite poset turns out to involve simpler combinatorial data and to be sufficient for the intended applications in the present text, which is why we restrict our attention to this case.

More precisely, for a finite poset $P$, one way to choose $I$ is to linearly order the elements of $P$, say as $a_1$, ..., $a_n$, so that the ``boundary" $\partial a_i = \{x \in P \mid x < a_i\}$ is contained in $\{a_1, ..., a_{i-1}\}$ for $1 \le i \le n$. Equivalently, we consider a linear refinement $\tle$ of the given partial order. Restricting $\tle$ defines a linear order on each $a_i^\le = \{x \in P \mid x \le a_i\}$, so the latter now corresponds to an object of $\mathcal C(D)$, say $c_i$. Sending $a_i$ to $c_i$ uniquely extends to a functor $I:P \rightarrow \mathcal C(P)$ isomorphic to the Yoneda embedding. Note that this is not the only way of choosing $I$ -- at this point, there need not be any relation between the linear orders on the different $a_i^\le$. The reason why we use this strategy is that it gives us methods for
\begin{itemize}
	\item determining when a functor $P \rightarrow \mathcal E$ is cellular with respect to $\tle$ once we know that its restriction $\{a_1, ..., a_{n-1}\} \rightarrow \mathcal E$ is cellular with respect to the restriction of $\tle$;
	
	\item given linear refinements $\tle$, $\tle'$ of $P$, modifying in a canonical way a functor $F:P \rightarrow \mathcal E$ cellular with respect to $\tle$ to one, say $F'$, cellular with respect to $\tle'$ and endowed with an isomorphism $F' \cong F$.
\end{itemize}

In a contextual category $\mathcal A$, the display maps arrange the set of objects of length $\ge 1$ into a tree $T(\mathcal A)$, and $T(\mathcal A) \hookrightarrow \mathcal A$ extends canonically to a contextual functor $\mathcal C(T(\mathcal A)) \rightarrow \mathcal A$. Thus we expect to have, for contextual categories $\mathcal A_1$, ..., $\mathcal A_k$, a cellular diagram
$$
T(\mathcal A_1) \times \cdots \times T(\mathcal A_k) \overset{I}{\longrightarrow} \mathcal C(T(\mathcal A_1) \times \cdots \times T(\mathcal A_k)) \cong \mathcal C(T(\mathcal A_1)) \otimes \cdots \otimes \mathcal C(T(\mathcal A_k)) \longrightarrow \mathcal A_1 \otimes \cdots \otimes \mathcal A_k
$$
where the embedding $I$ is to be chosen in accordance with the above discussion. Since $T(\mathcal A_1) \times \cdots \times T(\mathcal A_k)$ may not be finite, we don't choose a linear refinement of $T(\mathcal A_1) \times \cdots \times T(\mathcal A_k)$: it suffices to do that locally, that is, to refine $a_1^\le \times \cdots \times a_k^\le$ for each $a_1 \in T(A_1)$, ..., $a_k \in T(A_k)$.

This is why having a way of explicitly describing cellular diagrams out of posets will be useful in extending the concept of bimorphism to that of $k$-ary morphism for $k \ge 1$; for example, bimorphisms can be recovered by taking $k = 2$ and characterizing cellular maps when $T(\mathcal A_1) \times T(\mathcal A_2)$ is (locally) equipped with the lexicographic order. Our toolkit on poset-shaped cellular diagrams will help in addressing (1) above due to the possibility of switching between different linear refinements of a poset (the ``reverse" lexicographic order on $T(\mathcal A_1) \times T(\mathcal A_2)$ can be used to express the symmetry of the tensor product), as well as (2) due to the associativity of the lexicographic product.

\subsubsection*{Precontextual categories}

Many of our constructions use what we will call \emph{precontextual categories}. They will be categories equipped with a length function on objects, a chosen arrow (display map) out of each object of length $\ge 1$, and a class of commutative squares (distinguished squares); but, unlike for contextual categories, we impose a minimal list of axioms for this extra structure. This will be a convenient framework for obtaining contextual categories that have a desired universal property, which is sometimes difficult using \textsc{gat}s.

In fact, (small) precontextual categories will be organized into a locally presentable category $\Precont$ having $\Cont$ as a reflective subcategory such that the inclusion functor is $\omega$-accessible, i.e. it preserves filtered colimits. For example, it will be straightforward to describe the precontextual category $\mathcal A_{m,n}^{\text{pre}}$ freely generated by the following data: a length-$m$ object $a$, a length-$n$ object $b$, and and a morphism $f:a \rightarrow b$. Applying the reflection functor $L:\Precont \rightarrow \Cont$ to it yields a contextual category $\mathcal A_{m,n}$ which is, again, free on the above data. In the syntax of \textsc{gat}s, specifying $f$ would lead us to introduce $n$ many term symbols whose sorts are defined recursively. This kind of procedure can become unwieldy if the generating data involve several morphisms and relations involving them. Note that while in $\mathcal A_{m,n}^{\text{pre}}$ the arrow $f$ is a primitive structure, after passing to the contextual category $\mathcal A_{m,n}$ we gain access to a canonical decomposition of $f$ as a sequence of sections of display maps.

\subsection*{Organization of the text}

The text is structured as follows:

\begin{itemize}
	\item In Section 2, we recall the definitions of a contextual category, of a category with attributes (\textsc{cwa}), and of the categories of those --- $\Cont$ and $\Att$. We also describe, following \cite{KapLum18}, the coreflection functor $\Att \rightarrow \Cont$, which will play a major role in the text. We also discuss precontextual categories and the category of these, $\Precont$. The exposition is set up so that contextual categories are defined as precontextual categories satisfying further conditions, and $\Cont$ is the corresponding full subcategory of $\Precont$. We also check that $\Cont \hookrightarrow \Precont$ is an $\omega$-accessible right adjoint functor.
	
	\item In Section 3, we define exponentiation and bimorphisms of (pre)contextual categories, and discuss some instances of these constructions.
	
	Let $\iiCont$ be the strict $2$-category obtained from $\Cont$ by taking the natural transformations between contextual functors as $2$-cells; we verify in \S\ref{subsec: cat-power} that $\iiCont$ has $\Cat$-powers and that these are examples of exponentials as previously introduced. We also provide an alternative description, based on a notion of family-valued model relative to a set-valued functor, of $\Fam^\mathcal A$ where $\mathcal A \in \Cont$ and $\Fam$ is the contextual category of iterated families of sets from \cite{Car86}.
	
	We verify in \S\ref{subsec: syntactic example} that for \textsc{gat}s $\bbA$ and $\bbB$, the functor $\otimes_{\bbA,\bbB}:\mathcal C(\bbA) \times \mathcal C(\bbB) \rightarrow \mathcal C(\bbA \otimes \bbB)$ is a bimorphism. This example is central to connecting the present text with \cite{Alm25}, and it can be useful for familiarizing oneself with the definition of a bimorphism, especially the aspects that involve the strictness of the data in a contextual category.
	
	\item In Section 4, we define and study what we will call (poset-shaped) cellular diagrams. These will be functors $P^{\text{op}} \rightarrow \mathcal C$, where $P$ is a suitably structured locally finite poset and $\mathcal C$ is a contextual category, that can be used to assign display maps (resp. objects) in $\mathcal C$ to elements (resp. sieves) of $P$ with objects in a particular way.
	
	\item In Section 5, we define multimorphisms $(\mathcal A_1,...,\mathcal A_n) \rightarrow \mathcal B$ from a tuple of precontextual categories to a contextual category, and prove several results about those.
	
	In \S\ref{subsec: transporting multimorphisms} we describe a process for modifying a multimorphism out of $(\mathcal A_1, ..., \mathcal A_n)$ to one out of $(\mathcal A_{\sigma 1}, ..., \mathcal A_{\sigma n})$ for a permutation $\sigma \in S_n$, and introduce certain concepts related to this construction --- shuffling diagrams, permutative morphisms --- that will be useful later. In \S\ref{subsec: comparison with bimorphisms as introduced previously} we check that multimorphisms $(\mathcal A_1,\mathcal A_2) \rightarrow \mathcal B$ coincide with the previously-introduced bimorphisms. In \S\ref{subsec: multimorphisms and exponentials} we construct (see Theorem \ref{th: isomorphism of multihoms}) a (natural) isomorphism $\Hom(\mathcal A_1, \mathcal A_2,..., \mathcal A_n;\mathcal C) \cong \Hom(\mathcal A_2, ..., \mathcal A_n;\mathcal C^{\mathcal A_1})$ where $\mathcal A_1$, ..., $\mathcal A_n$ are precontextual categories and $\mathcal C$ is a contextual category. The main results of \S6 and \S7 are proved by combining Theorem \ref{th: isomorphism of multihoms} with properties of the structures from \S\ref{subsec: transporting multimorphisms}. In \S\ref{subsec: multimorphisms and exponentials via precont}, we prove that exponentials and sets of multimorphisms are invariant under replacing any precontextual category among the arguments by its image under the reflection functor $\Precont \rightarrow \Cont$.
	
	\item In Section 6, multimorphisms whose source is a tuple of contextual categories are proved to be closed under (multi)composition, and we use this to obtain a multicategory $\mathscr{Cont}$ of contextual categories and multimorphisms between them.
	
	\item In Section 7, we realize $\Cont$ as a (closed) symmetric monoidal category. We start by proving the existence of binary tensor products, that is, of a universal bimorphism out of any given pair of contextual categories (Construction \ref{constr: binary tensor products}). Our approach is based on the functor $L:\Precont \rightarrow \Cont$ and on the representability of functors of poset-shaped cellular diagrams (Proposition \ref{prop: representability cellular diagrams}), and it does not rely on the tensor product of \textsc{gat}s defined in \cite{Alm25}. We then use this to prove the existence of $n$-ary tensor products for all $n \ge 1$ (Proposition \ref{prop: n-ary tensor product}). After that, we lift $\otimes:\Cont \times \Cont \rightarrow \Cont$ to a symmetric monoidal structure by constructing natural isomorphisms $\alpha$, $\lambda$, $\rho$, $\beta$ and proving that they satisfy the required coherence identities.
	
	\item In Section 8 we describe the pushout-tensor maps associated with pairs of contextual functors $(F,G)$ of a special kind: each of $F$, $G$ encodes in a universal way one of the kinds of axioms for generalized algebraic theories.
	
	\item In Section 9, we use the description from \S8 to prove that the functor $\otimes_{\bbA,\bbB}:\mathcal C(\bbA) \times \mathcal C(\bbB) \rightarrow \mathcal C(\bbA \otimes \bbB)$ from \cite{Alm25} is a universal bimorphism, which allows us to conclude that the monoidal structure on $\GAT$ obtained by transferring the one on $\Cont$ along $\Cont \simeq \GAT$ is compatible with the tensor product from \cite{Alm25}.
	
	\item In Section 10, which can be seen as a continuation of \S\ref{subsec: cat-power}, we prove that the $\Cat$-enrichment and the symmetric monoidal structure on $\Cont$ are, in a certain sense, compatible.
\end{itemize}

\subsection*{Acknowledgement}

I am extremely grateful to Simon Henry, my doctoral advisor, who introduced me to this topic, helped me navigate it through countless and insightful discussions, and provided valuable feedback on this text.
\section{A brief overview of (pre)contextual categories and categories with attributes}

In this section, we will recall the definitions of a contextual category and of a category with attributes, fix some notation and terminology, and, along the way, introduce what we will call precontextual categories.

\vspace{0.5em}

\subsection{(Pre)contextual categories}
\label{subsec: (pre)contextual categories}

Contextual categories were introduced in J. Cartmell's doctoral thesis \cite{Car78} (see also \cite{Car86}, \S 14) in order to capture the aspects of the syntax of generalized algebraic theories that are essential to formulate their semantics. For this reason, as remarked by Cartmell, contextual categories can be viewed as providing an \emph{algebraic semantics} for generalized algebraic theories.

A generalized algebraic theory $\bbA$ has a \emph{syntactic category}, $\mathcal C(\bbA)$, whose objects (resp. morphisms) are equivalence classes, with respect to derivable equality, of contexts (resp. of realizations between contexts\footnote{In \cite{Alm25} we referred to these also as \emph{morphisms}; so in the terminology of that article, given contexts $\textbf{X}$, $\textbf{Y}$ in $\bbA$, a morphism $[\textbf{X}] \rightarrow [\textbf{Y}]$ in $\mathcal C(\bbA)$ is an equivalence class of morphisms $\textbf{X} \rightarrow \textbf{Y}$.}) in $\bbA$. The category $\mathcal C(\bbA)$ inherits extra structure from $\bbA$: each object has a \emph{length}, defined as the length of any context in the given equivalence class; each object of length $\ge 1$ has a distinguished projection map out of it --- a \emph{display map} --- that corresponds to removing from a non-empty context its last variable; display maps have specified and strictly functorial pullbacks along arbitrary morphisms, corresponding to substitution of sorts and terms along arbitrary realizations. A \emph{contextual category} is then defined as a category equipped with the above kinds of additional structure.

Morphisms of contextual categories $\mathcal C(\bbA) \rightarrow \mathcal C(\bbB)$ --- which, by definition, are the functors that preserve all the additional structure strictly, including the specified length of objects --- correspond bijectively to morphisms $\bbA \rightarrow \bbB$. Also, every contextual category is isomorphic to $\mathcal C(\bbA)$ for some \textsc{gat} $\bbA$. It follows that the syntactic category construction defines an equivalence $\GAT \simeq \Cont$ between the category of \textsc{gat}s and that of contextual categories; this is one of the central results from \cite{Car78}. We refer the reader to Cartmell's work for the details of this correspondence.

\vspace{0.5em}

As mentioned in the introduction, precontextual categories will be used to present contextual categories having certain universal properties of interest; informally, they can be thought of as a variant of contextual categories in which substitution is only partially (and possibly ambiguously) specified. As we will see, the category $\Precont$ of precontextual categories is an $\omega$-orthogonality class (see \cite{AdaRos94}) in $\Cont$, from which will follow that the full-and-faithful inclusion $\Cont \subset \Precont$ is an accessible right adjoint functor.

\begin{convention}
\label{convention: general notation}
We work within ZFC set theory with, additionally, an uncountable Grothendieck universe $\mathscr U$. We say that a set $X$ is \emph{small} if $X \in \mathscr U$. Specifically when discussing the family-valued semantics of contextual categories, we will suppose given another universe, $\mathscr U^+$, such that $\mathscr U \in \mathscr U^+$.

For us, a category $C$ will have, by definition, a set $\Ob(C)$ and a set $\Ar(C)$ of morphisms. We say that it is \emph{small} if $\Ob(C)$ and $\Ar(C)$ are both small. When this causes no ambiguity, we abbreviate $a \in \Ob(C)$ as $a \in C$. We write $\iiAr(C)$ for the category of arrows of $C$.

We write $B^A$ for the category of functors $A \rightarrow B$ and natural transformations between them. The set of functors $\Ob(B^A)$ will be denoted by $\Fun(A,B)$.

\vspace{0.5em}

For $n \ge 0$, we write $[n]$ for the ordered set $\{0 \le 1 \le \cdots \le n\}$, as well as for the corresponding category.

We denote by $\Set$ the category having as objects the small sets, and as morphisms the functions between those.
\end{convention}

\begin{definition}
\label{def: precontextual category}
A \emph{precontextual category} is a tuple $\mathcal C = (C,\ell,\partial,\textbf{p},\textbf{Q})$ consisting of:
\begin{enumerate}[label=(\roman*)]
	\item A category $C$, which we call the \emph{underlying category} of $\mathcal C$.
	
	\item A function $\ell:\Ob(C) \rightarrow \bbN$. We refer to $\ell(a)$ as the \emph{length} of $a$.
	
	\item A function $\partial:\Ob(C) \rightarrow \Ob(C)$ such that $\ell(\partial a) = \ell(a) - 1$ if $\ell(a) \ge 1$, and $\partial(a) = a$ if $\ell(a) = 0$.
	
	\item For each $a \in C$ of length $\ge 1$, a morphism $\textbf{p}_a:a \rightarrow \partial a$.
	
	A finite composite of arrows of this form will be referred to as a \emph{display map}. Note that for all $a \in C$ and $n \le \ell(a)$, there exists a unique display map of the form $a \rightarrow b$ with $\ell(b) = n$.
	
	We use the notation $a \twoheadrightarrow b$ to indicate that a morphism is a display map. This will be convenient since, in this case, we can identify the arrow without naming it: by the above remark, for fixed $a$, $b$, there exists at most one display map from $a$ to $b$.
	
	The \emph{length} of $p:a \twoheadrightarrow b$, for which we write $\ell(p)$ by abuse of notation, is defined as $\ell(a) - \ell(b)$.
	
	\item A set $\textbf{Q}$ of commutative squares in $C$ of the form
	\[
	\dsqua{a}{b}{\partial a}{\partial b.}{f'}{f}{\textbf{p}_a}{\textbf{p}_b}
	\]
	A \emph{distinguished square} is a finite vertical composite of squares in $\textbf{Q}$. In particular, every distinguished square is of the form
	\[
	\tag{\texttt{*}}
	\dsqua{a}{b}{a'}{b'.}{f'}{f}{}{}
	\]
	where $a \twoheadrightarrow a'$ and $b \twoheadrightarrow b'$ are display maps of the same length $n \ge 0$. We refer to $n$ as the \emph{length} of ($\texttt{*}$). Note that the length-$1$ distinguished squares are precisely the elements of $Q$.
\end{enumerate}
We say that $\mathcal C$ is a \emph{contextual category}\footnote{The concept was introduced in \cite{Car78}. More recently, contextual categories also became known as \emph{C-systems} following Voevodsky's work on the subject; see \cite{Voe16} for an explanation for the suggested change in terminology.} if, in addition, it has the following properties:
\begin{enumerate}[label=(\roman*)]
\setcounter{enumi}{5}
	\item It has a unique object of length $0$, and it is a terminal object.
	
	We call it the \emph{distinguished terminal object} and denote it by $1_\mathcal C$.
	
	\item Every element of $\textbf{Q}$ is a pullback square.
	
	This implies, by the pasting law for pullbacks, that every distinguished square is a pullback square.
	
	\item $\textbf{Q}$ forms a category (whose objects are the length-$1$ display maps) under horizontal composition:
	\[
	\dsqua{a}{a}{\partial a}{\partial a}{id_a}{id_a}{p_\textbf{a}}{p_\textbf{a}}
	\]
	is in $\textbf{Q}$ for all $a$ of length $\ge 1$, and if the left and right squares in
	\[\begin{tikzcd}
		a & b & c \\
		{\partial a} & {\partial b} & {\partial c}
		\arrow["{f'}", from=1-1, to=1-2]
		\arrow["{\textbf{p}_a}"', from=1-1, to=2-1]
		\arrow["{g'}", from=1-2, to=1-3]
		\arrow["{\textbf{p}_b}", from=1-2, to=2-2]
		\arrow["{\textbf{p}_c}", from=1-3, to=2-3]
		\arrow["f"', from=2-1, to=2-2]
		\arrow["g"', from=2-2, to=2-3]
	\end{tikzcd}\]
	are in $\textbf{Q}$, then so is
	\[
	\dsqua{a}{c}{\partial a}{\partial c.}{g' \circ f'}{g \circ f}{\textbf{p}_a}{\textbf{p}_c}
	\]
	Note that, as a consequence, the set of all distinguished squares is closed under (vertical and) horizontal composition.
	
	\item For every diagram of the form
	\[\begin{tikzcd}[ampersand replacement=\&]
		\& a \\
		b \& {\partial a,}
		\arrow["{\textbf{p}_a}", two heads, from=1-2, to=2-2]
		\arrow["f"', from=2-1, to=2-2]
	\end{tikzcd}\]
	there exists a unique element of $\textbf{Q}$ of the form
	\[
	\dsqua{b'}{a}{b}{\partial a.}{f'}{f}{\textbf{p}_{b'}}{\textbf{p}_a}
	\]
\end{enumerate}

As long as this causes no ambiguity, when dealing with two or more (pre)contextual categories we will abuse notation and write $\ell$, $\partial$, $\textbf{p}$, $\textbf{Q}$ for the corresponding structures in all of them. The underlying category of $\mathcal C$ (i.e. $C$ in the above definition) will be denoted by $|\mathcal C|$.

\vspace{0.5em}

For precontextual categories $\mathcal C$ and $\mathcal D$, say with underlying categories $C$ and $D$, respectively, we define a \emph{morphism} or \emph{contextual functor} from $\mathcal C$ to $\mathcal D$ as a functor $F:C \rightarrow D$ that (strictly) preserves all the additional structure: it commutes with $\ell$, $\partial$ and $\textbf{p}$, and the induced function on commutative squares sends elements of $\textbf{Q}$ to elements of $\textbf{Q}$. We compose morphisms as usual for functors.

\vspace{0.5em}

We say that a (pre)contextual category is \emph{small} if its underlying category is small.

\vspace{0.5em}

We denote by $\Precont$ the category whose objects are the small precontextual categories and whose morphisms are the ones defined above. We let $\Cont$ be the full subcategory of $\Precont$ whose objects are the small contextual categories. It is immediate that $\Cont$ is isomorphic to the category of contextual categories considered in \cite{Car78}.

By taking natural transformations as $2$-cells, we have strict $2$-categories $\iiPrecont$ and $\iiCont$.\footnote{See Remark \ref{rem: comparison with 2-cat of comprehension categories} for a discussion of this definition.}
\end{definition}

\begin{remark}
It is possible to prove, in a standard but laborious way, that $\Precont$ is the category of $\Set$-valued models of a finite-limit sketch, hence a locally finitely presentable category (see \cite{AdaRos94}). The idea is to use a sketch whose underlying graph has, as a first step, a vertex $o_n$ for each $n \ge 0$, to be interpreted as the set of length-$n$ objects, and a vertex $a_{m,n}$ for each $m$, $n \ge 0$, to be interpreted as the set of arrows from a length-$m$ object to a length-$n$ one. One then encodes the remaining structure separately for each tuple of lengths. For example, $\partial$ is encoded by an edge $o_n \rightarrow o_{n-1}$ for each $n \ge 1$, co/domains by edges $a_{m,n} \rightarrow o_m$ and $a_{m,n} \rightarrow o_n$, and composition by edges $a_{m,n} \times_{o_n} a_{n,p} \rightarrow a_{m,p}$, where $a_{m,n} \times_{o_n} a_{n,p}$ is a new vertex to be interpreted as a pullback in the standard way.
\end{remark}

\begin{definition}
Recall that in a category $A$, an object $a$ is said to be \emph{orthogonal} to a morphism $i:b \rightarrow b'$ if for every $g:b \rightarrow a$, there exists a unique $g':b' \rightarrow a$ such that $g' \circ i = g$.

For a set $M$ of morphisms in $A$, we write $M^\perp$ for the full subcategory of $A$ spanned by all objects that are orthogonal to every element of $M$.

An \emph{orthogonality class} in $A$ is a full subcategory of the form $M^\perp$ for some set of morphisms $M$. We say that it is a \emph{$\kappa$-orthogonality class} for a regular cardinal $\kappa$ if $M$ can be taken as a set of morphisms whose domain and codomain are $\kappa$-presentable.
\end{definition}

\begin{proposition}[\cite{AdaRos94}, Theorem 1.39]
Let $A$ be a locally $\kappa$-presentable category. If $B \subset A$ is a $\kappa$-orthogonality class, then $B$ is locally $\kappa$-presentable, and the inclusion functor $B \rightarrow A$ is $\kappa$-accessible (that is, it preserves $\kappa$-filtered colimits) and has a left adjoint.
\end{proposition}

\begin{proposition}
\label{prop: Cont is orthogonality class in Precont}
The full subcategory $\Cont \subset \Precont$ is an $\omega$-orthogonality class.
\end{proposition}

We defer its proof to the end of this section.

\begin{notation}
\label{not: reflection functor}
We let $L:\Precont \rightarrow \Cont$ be a left adjoint of the inclusion $\Cont \hookrightarrow \Precont$.
\end{notation}

It is often very difficult to describe explicitly the effect of the reflection functor $L:\Cont \rightarrow \Precont$ on a given precontextual category; so throughout the text, this process usually allows us to conclude that a functor $\Cont \rightarrow \Set$ is representable without having a full description of a representing object --- that is, we won't be able to characterize maps \emph{to} a representing object. This applies, for example, to the construction of the tensor product of two contextual categories; the syntactic approach from \cite{Alm25} addresses this shortcoming by employing a different method and language, and so can be seen as complementary to the one from the present text.

\begin{remark}
We will prove in Corollary \ref{cor: category of contextual functors, reflection morphism} that $L$ lifts to a strict $2$-adjoint of $\iiCont \hookrightarrow \iiPrecont$.\footnote{This uses in an essential way (which is not the case for the $1$-categorical adjunction) that a precontextual category $\mathcal C$ has an underlying tree of display maps rather than just a length function $\Ob(|\mathcal C|) \rightarrow \bbN$ and a set arrows specified as display maps.}
\end{remark}

\subsubsection{Colimits of precontextual categories, an example, and the proof of Proposition \ref{prop: Cont is orthogonality class in Precont}}

\begin{remark}[Colimits of precontextual categories]
\label{rem: colimits in precont}
Consider a diagram $(\mathcal A_i)_{i \in I}$ in $\Precont$ where $I$ is a small category. Let $(\varphi_i:|A_i| \rightarrow C)$ be a colimit cocone of the underlying diagram in $\Cat$. Since $\Ob:\Cat \rightarrow \Set$ preserves colimits, objects of $C$ correspond bijectively to equivalence classes of the smallest equivalence relation on $\coprod_{i \in I}\Ob(\mathcal A_i)$ with $(i,a) \sim (j,b)$ whenever there exists $f:i \rightarrow j$ such that the induced morphism $A_i \rightarrow \mathcal A_j$ sends $a$ to $b$.
But note that
\begin{itemize}
	\item as contextual functors preserve length, if $(i,a) \sim (j,b)$, then $\ell_{\mathcal A_i}(a) = \ell_{\mathcal A_j}(b)$;
	
	\item as contextual functors preserve display maps, if $(i,a) \sim (j,b)$ with $\ell_{\mathcal A_i}(a) \ge 1$, then $\varphi_i(\textbf{p}_a) = \varphi_j(\textbf{p}_b)$.
\end{itemize}
This implies that we have a precontextual category $\mathcal C$ with underlying category $C$ and such that
\begin{itemize}
	\item for $i \in I$ and $a \in |\mathcal A_i|$, we have $\ell_{\mathcal B}(\varphi_i (a)) = \ell_{\mathcal A_i}(a)$; if $\ell_{\mathcal A_i}(a) \ge 1$, then $\textbf{p}_{\varphi_i (a)} = \varphi_i(\textbf{p}_a)$;
	
	\item distinguished squares in $\mathcal C$ are the images of distinguished squares under the functors $\varphi_i$.
\end{itemize}
Given a cone $(K_i:\mathcal A_i \rightarrow \mathcal D)_{i \in I}$ in $\Precont$, there exists a unique functor $K':|C| \rightarrow |\mathcal D|$ such that $K' \circ \varphi_i = K_i$ for all $i$; it is immediate from the definition of $\mathcal C$ that $F$ defines a morphism $\mathcal C \rightarrow \mathcal D$. We conclude that $(\varphi_i:\mathcal A_i \rightarrow \mathcal C)$ is a colimit cocone in $\Precont$.
\end{remark}

\begin{example}
\label{example: On pre}
For $n \ge 0$, we let $\mathcal O_n^{\text{pre}}$ be the precontextual category whose underlying category is a linearly ordered set with $n+1$ elements, which we write as
$$
o_n \rightarrow o_{n-1} \rightarrow \cdots \rightarrow o_1 \rightarrow o_0,
$$
and where: $\ell(o_i) = i$; $\partial o_i = o_{i-1}$ with $\textbf{p}_{o_i}$ the unique arrow $o_i \rightarrow o_{i-1}$; and $\textbf{Q}$ is empty. We denote $L(\mathcal O_n^{\text{pre}})$ by $\mathcal O_n$.

Note that $\mathcal O_n^{\text{pre}}$ (resp. $\mathcal O_n$) represents the functor $\Precont \rightarrow \Set$ (resp. $\Cont \rightarrow \Set$) that sends (pre)contextual category to its set of length-$n$ objecs. The same functor is represented by the syntactic category of the \textsc{gat} $\bbO_n$ given by
\begin{align*}
	& \vdash O_1 \tp\\
	x_1:O_1 & \vdash O_2(x_1) \tp\\
	& \vdots\\
	x_1:O_1, ..., x_{n-1}:O_{n-1}(x_1, ..., x_{n-2}) & \vdash O_n(x_1, ..., x_{n-1}) \tp,
\end{align*}
so the Yoneda lemma yields an isomorphism $\bbO_n \cong \mathcal C(\bbO_n)$.
\end{example}

In what follows, we consider certain precontextual categories obtained by adding further structure to some coproduct such as $\mathcal O_m^{\text{pre}} \sqcup \mathcal O_n^{\text{pre}} \sqcup \mathcal O_p^{\text{pre}}$. We then write $o_i$, $o'_i$, $o''_i$ for the objects of $\mathcal O_m^{\text{pre}}$, $\mathcal O_n^{\text{pre}}$, $\mathcal O_p^{\text{pre}}$, respectively, and analogously in similar cases.

\begin{proof}[Proof of Proposition \ref{prop: Cont is orthogonality class in Precont}]
	We have $\Cont = M^\perp$ for the set of morphisms $M$ in $\Precont$ consisting of:
	\begin{enumerate}[label=(\arabic*)]
		\item The inclusion $\varnothing \rightarrow \mathcal O_0^{\text{pre}}$.
		
		Note that a precontextual category $\mathcal C$ orthogonal to this map if and only if it has a single length $0$ object.
		
		\item For each $n \ge 0$, the inclusion $\mathcal O_n^{\text{pre}} \sqcup \mathcal O_0^{\text{pre}} \rightarrow \mathcal A$ where $\mathcal A$ is obtained from $\mathcal O_n^{\text{pre}} \sqcup \mathcal O_0^{\text{pre}}$ by freely adding a morphism $o_n \rightarrow o'_0$ from the length $n$-object of $\mathcal O_n^{\text{pre}}$ to the length-$0$ object of $\mathcal O_0^{\text{pre}}$.
		
		A precontextual category $\mathcal C$ is orthogonal to this map precisely when all of its length-$0$ objects are terminal.
		
		\item For each $m$, $n \ge 1$ and $p \ge 0$, the inclusion $\mathcal A \rightarrow \mathcal B$ where:
		\begin{itemize}
			\item $\mathcal A$ is obtained from $\mathcal O_m^{\text{pre}} \sqcup \mathcal O_n^{\text{pre}} \sqcup \mathcal O_p^{\text{pre}}$ by freely adding arrows $f$, $f'$, $r$, $s$ as in
			\[\begin{tikzcd}[ampersand replacement=\&]
				{o''_p} \\
				\& {o_m} \& {o'_n} \\
				\& {o_{m-1}} \& {o'_{n-1}}
				\arrow["r", curve={height=-18pt}, from=1-1, to=2-3]
				\arrow["s"', curve={height=18pt}, from=1-1, to=3-2]
				\arrow["{f'}", from=2-2, to=2-3]
				\arrow[two heads, from=2-2, to=3-2]
				\arrow[two heads, from=2-3, to=3-3]
				\arrow["f"', from=3-2, to=3-3]
			\end{tikzcd}\]
			and imposing that this diagram commute and the lower right square be distinguished.
			
			\item $\mathcal B$ is obtained from $\mathcal A$ by freely adding an arrow $t:o''_p \rightarrow o_m$ and imposing the relations $f' \circ t = r$ and $\textbf{p}_{o_m} \circ t = s$.
		\end{itemize}
		$\mathcal C$ being orthogonal to this map means that all of its length-$1$ distinguished squares are cartesian.
		
		\item For each $m$, $n$, $p \ge 1$, the inclusion $\mathcal A \rightarrow \mathcal B$ (whose underlying functor is the identity) where:
		\begin{itemize}
			\item $\mathcal A$ is obtained from $\mathcal O_m^{\text{pre}} \sqcup \mathcal O_n^{\text{pre}} \sqcup \mathcal O_p^{\text{pre}}$ by freely adding arrows $f$, $f'$, $g$, $g'$ as in
			\[\begin{tikzcd}[ampersand replacement=\&]
				{o_m} \& {o'_n} \& {o''_p} \\
				{o_{m-1}} \& {o'_{n-1}} \& {o''_{p-1}}
				\arrow["{f'}", from=1-1, to=1-2]
				\arrow[two heads, from=1-1, to=2-1]
				\arrow["{g'}", from=1-2, to=1-3]
				\arrow[two heads, from=1-2, to=2-2]
				\arrow[two heads, from=1-3, to=2-3]
				\arrow["f"', from=2-1, to=2-2]
				\arrow["g"', from=2-2, to=2-3]
			\end{tikzcd}\]
			and imposing that the two squares commute and are distinguished.
			
			\item $\mathcal B$ is obtained from $\mathcal A$ by imposing that the square
			\[
			\dsqua{o_m}{o''_p}{o_{m-1}}{o''_{p-1}}{g' \circ f'}{g \circ f}{}{}
			\]
			be distinguished.
		\end{itemize}
		$\mathcal C$ being orthogonal to this map means that its set of length-$1$ distinguished squares are closed under horizontal composition.
		
		\item For each $n \ge 1$, the inclusion $\mathcal O_n^{\text{pre}} \rightarrow \mathcal A$ where $\mathcal A$ is obtained from $\mathcal O_n^{\text{pre}}$ by imposing that the square
		\[
		\dsqua{o_n}{o_n}{o_{n-1}}{o_{n-1}}{id}{id}{}{}
		\]
		be distinguished.
		
		$\mathcal C$ is orthogonal to this map precisely when the identity square of each length-$1$ display map is distinguished.
		
		\item For each $m$, $n \ge 1$, the inclusion $\mathcal A \rightarrow \mathcal B$ where:
		\begin{itemize}
			\item $\mathcal A$ is obtained from $\mathcal O_{m-1}^{\text{pre}} \sqcup \mathcal O_n^{\text{pre}}$ by freely adding an arrow $f:o_{m-1} \rightarrow o'_n$.
			
			\item $\mathcal B$ is obtained from $\mathcal O_m^{\text{pre}} \sqcup \mathcal O_n^{\text{pre}}$ by freely adding arrows $f$, $f'$ as in
			\[
			\dsqua{o_m}{o'_n}{o_{m-1}}{o'_{n-1}}{f'}{f}{}{}
			\]
			and imposing that this square commute and be distinguished.
		\end{itemize}
		$\mathcal C$ being orthogonal to this map means that it satisfies (ix) from Definition \ref{def: precontextual category}.
	\end{enumerate}
	It remains to verify that all precontextual categories considered above are $\omega$-presentable. Note that they are, by construction, finite colimits of the following kinds of precontextual categories:
	\begin{enumerate}[label=(\roman*)]
		\item $\mathcal O_n^{\text{pre}}$ for $n \ge 0$.
		
		\item For $m$, $n \ge 0$, the precontextual category obtained from $\mathcal O_m^{\text{pre}} \sqcup \mathcal O_n^{\text{pre}}$ by freely adding an arrow $o_m \rightarrow o'_n$.
		
		\item $\mathcal B$ from (6) above.
	\end{enumerate}
	Since finite colimits of $\omega$-presentable objects are $\omega$-presentable (\cite{AdaRos94}, Prop. 1.3), it suffices to check $\omega$-presentability of (i)-(iii). Observe that
	\begin{enumerate}[label=(\roman*)]
		\item represents the functor that sends $\mathcal A$ to its set of length-$n$ objects;
		
		\item represents the functor sending $\mathcal A$ to its set of arrows $f:a \rightarrow b$ with $\ell(a) = m$, $\ell(b)$;
		
		\item represents the functor sending $\mathcal A$ to its set of distinguished squares
		\[
		\dsqua{a}{b}{\partial a}{\partial b}{f'}{f}{\textbf{p}_a}{\textbf{p}_b}
		\]
		such that $\ell(a) = m$, $\ell(b) = n$.
	\end{enumerate}
As $\Ob$, $\Ar:\Cat \rightarrow \Set$ preserve filtered colimits, it follows from the description in Remark \ref{rem: colimits in precont} that so do the functors represented by (i)-(iii).
\end{proof}

\subsection{Categories with attributes}
\label{subsec: cwa}

Categories with attributes --- \textsc{cwa}s for short --- are a useful generalization of contextual categories. Similarly to the latter, a \textsc{cwa} has a set of types over each object, arranged contravariantly so as to encode substitution, and each type $U$ over an object $X$ has an associated ``projection" or ``display map" $X.U \twoheadrightarrow X$ thought of as adjoining to $X$ an element or free variable of type $U$. However, an object $X$ of a \textsc{cwa} might fail to fit uniquely into a sequence of display maps $X = X_n \twoheadrightarrow X_{n-1} \twoheadrightarrow \cdots \twoheadrightarrow X_0 = 1$ where $1$ is a previously chosen terminal object. That is, such a decomposition may not exist, or there may be several of them, possibly of different lengths. Also, we can have $X.U = X.V$ for two distinct types $U$, $V$ over $X$.

\begin{definition}
A \emph{category with attributes}\footnote{The definition we use, which is common in the literature, differs slightly from the one given in \cite{Car78}, where the presence of dependence sums is also required.} is a tuple $\mathcal C = (C,1,\Ty,R)$ consisting of a category $C$, a terminal object $1 \in C$, a presheaf (of \emph{types}) $\Ty:C^{\text{op}} \rightarrow \Set$\footnote{It can be useful, for example when studying the semantics of \textsc{cwa}s or contextual categories, to consider a presheaf taking values in a larger universe of sets. For simplicity, we will use the same terminology when talking about such variants.}, and a functor $R:\int \Ty \rightarrow \iiAr(C)$ such that:
\begin{itemize}
	\item The diagram
	\[
	\begin{tikzcd}
		\int \Ty \arrow[]{rr}{R} \arrow[swap]{dr}{P} & & \iiAr(C) \arrow[]{dl}{\cod} \\
		& C, &
	\end{tikzcd}
	\]
	where $P$ is the canonical projection from the category of elements of $R$ and $\cod$ is the codomain projection, commutes strictly.
	
	\item Every morphism in $\int \Ty$ is sent by $R$ to a cartesian square in $C$.
\end{itemize}
A \emph{morphism} from $(C,1,\Ty,R)$ to $(C',1',\Ty',R')$ is a pair $(F,\phi)$ consisting of a functor $F:C \rightarrow C'$ and a natural transformation $\phi:\Ty \Rightarrow \Ty' \circ F$ such that, letting $\overline{\phi}:\int \Ty \rightarrow \int \Ty'$ be the induced functor, the diagram
\[
\squa{\int \Ty}{\int \Ty'}{\iiAr(C)}{\iiAr(C')}{\overline{\phi}}{\iiAr(F)}{R}{R'}
\]
strictly commutes. Composition of morphisms is defined by $(G,\gamma) \circ (F,\phi) = (G \circ F,\; \gamma F \circ \phi)$.

We denote by $\Att$ the category of small \textsc{cwa}s (ones whose base category is small) and morphisms between them.
\end{definition}

Following standard notation (see e.g. \cite{KapLum18}, Definition 4.1), for $X \in C$ and $U \in \Ty(X)$ we write $R(X,U)$ as $p_U:X.U \rightarrow X$. This matches how \textsc{cwa}s are used to model dependent type theory: we think of $X.U$ as being obtained from $X$ by adjoining a free variable of type $U$, with $p_U$ the corresponding projection/display map. Now, the image under $R$ of an arrow $f:(Y,f^*U) \rightarrow (X,U)$ in $\int \Ty$ takes the form
\[
\squa{Y.f^*U}{X.U}{Y}{X.}{f'}{f}{p_{f^*U}}{p_U}
\]

\begin{construction}[From contextual categories to \textsc{cwa}s and back]
\label{constr: cwas and cont cats}
A contextual category $\mathcal C$ gives rise to a category with attributes $\att(\mathcal C) = (|\mathcal C|,1_\mathcal C,\Ty,R)$ where:
\begin{itemize}
	\item $\Ty:|\mathcal C|^{\text{op}} \rightarrow \Set$ sends $a \in C$ to $\{a' \in C \mid \partial a' = a\}$ and, for an arrow $f:b \rightarrow a'$, the map $R(f)$ sends $a'$ to the unique $b'$ fitting into a distinguished square
	\[
	\dsqua{b'}{a'}{b}{a.}{}{\Ty(f)}{}{}
	\]
	Hence $\int \Ty$ is canonically isomorphic to the category of length-$1$ distinguished squares from Definition \ref{def: precontextual category}(viii).
	
	\item In terms of the above isomorphism, $R$ sends each length-$1$ distinguished square to itself viewed as an arrow in $\iiAr(|\mathcal C|)$.
\end{itemize}

Conversely, a \textsc{cwa} $\mathcal C = (C,1,\Ty,R)$ gives rise to a contextual category $\cont(\mathcal C)$ defined as follows:
\begin{itemize}
	\item For $n \ge 0$, its length-$n$ objects are the sequences $((X_0,U_0), ..., (X_{n-1},U_{n-1}))$ of objects of $\int \Ty$ such that $X_0 = 1$ and $X_i = X_{i-1}.U_{i-1}$ for $1 \le i \le n-1$. (Thus $X_i = 1.U_0.\cdots.U_{i-1}$.)
	
	Morphisms and their composition are induced from $\mathcal C$ via
	$$
	\Hom_{\cont(\mathcal C)}((X_i,U_i)_{i < m},\; (Y_i,V_i)_{i < n}) := \Hom_\mathcal C(X_{m-1}.U_{m-1},\;Y_{n-1}.V_{n-1}).
	$$
	
	Note that we then have a full-and-faithful functor $U_\mathcal C:\cont(\mathcal C) \rightarrow \mathcal C$.
	
	\item For $n \ge 1$, we let $\partial((X_i,U_i)_{i < n}) = (X_i,U_i)_{i < {n-1}}$ with corresponding display map characterized by $U(\textbf{p}_{(X_i,U_i)_{i < n}}) = p_{U_{n-1}}: X_{n-1}.U_{n-1} \rightarrow X_{n-1}$. Distinguished squares are defined similarly using the functor $R$.
\end{itemize}
\end{construction}

Note that $U_\mathcal C$ is a morphism $\att(\cont(\mathcal C)) \rightarrow \mathcal C$. Moreover, the above constructions extend into functors
$$
\att:\Cont \rightarrow \Att, \qquad\quad \cont:\Att \rightarrow \Cont,
$$
and $\cont \circ \att$ is isomorphic to the identity functor. The following result can be obtained by a routine calculation:

\begin{proposition}[\cite{KapLum18}, Proposition 4.4]
\label{prop: adjunction cont cwa}
The functor $\att$ is full-and-faithful, and we have an adjunction $\att \dashv \cont$ with counit components $U_\mathcal C:\att(\cont(\mathcal C)) \rightarrow \mathcal C$. Explicitly, the map $\Hom_{\Cont}(\mathcal D,\cont(\mathcal C)) \rightarrow \Hom_{\Att}(\att(\mathcal D), \mathcal C)$ given by $F \mapsto U_\mathcal C \circ \att(F)$ is a bijection natural in $\mathcal D \in \Cont^{\text{op}}$, $\mathcal C \in \Att$.
\end{proposition}

\begin{remark}
Taking natural transformations as $2$-cells, we have a strict $2$-category $\iiAtt$.

Since $\att:\Cont \rightarrow \Att$ commutes with taking underlying categories/functors, it extends into a $2$-full-and-faithful strict $2$-functor $\iiatt:\iiCont \rightarrow \iiAtt$. Similarly, we have a strict $2$-functor $\iicont:\iiAtt \rightarrow \iiCont$.

In the notation of Proposition \ref{prop: adjunction cont cwa}, the fact that $U_\mathcal C$ is full-and-faithful implies that we actually have a $2$-natural isomorphism $\iiHom_\iiCont(\mathcal D,\cont(\mathcal C)) \cong \iiHom_\iiAtt(\att(\mathcal D),\mathcal C)$, hence a strict $2$-adjunction $\iiatt \dashv \iicont$.
\end{remark}

\begin{remark}
\label{rem: comparison with 2-cat of comprehension categories}
\cite{AhrLumNor24} describes in Def. 5 (based on \cite{CurGarHof14}) a $2$-category $\textbf{CompCat}$ of comprehension categories (in the sense of \cite{Jac93}), pseudomorphisms, and transformations. Viewing \textsc{cwa}s as comprehension categories whose underlying fibration is discrete, we obtain (essentially) a full sub-$2$-category of $\textbf{CompCat}$, which is denoted in that article by $\textbf{CwA}^{ps}$ (see Def. 37 and Prop. 38). The latter differs from the $2$-category $\iiAtt$ considered in the present paper in two ways: morphisms in $\iiAtt$ preserve all available structure strictly, rather than up to coherent isomorphism; and $2$-cells in $\iiAtt$ are arbitrary natural transformations, whereas $\textbf{CwA}^{ps}$ only keeps those that are suitably compatible with base change of types in the codomain.

As explained in \cite{AhrLumNor24}, Prop. 46, the full sub-$2$-category of $\textbf{CwA}^{ps}$ spanned by the contextual \textsc{cwa}s (those in the essential image of $\Cont \rightarrow \Att$) is actually equivalent to the $1$-category $\Cont$: between contextual categories, every pseudofunctor is isomorphic to a strict one, and there exists at most one $2$-cell between two given pseudofunctors. For $\mathcal A$, $\mathcal B \in \Cont$ and $F$, $G \in \Hom_\Cont(\mathcal A,\mathcal B)$, a natural transformation $\eta:F \Rightarrow G$ defines a morphism in $\textbf{CwA}^{ps}$ if and only if for every $a \in \Ob(\mathcal A)$ of length $\ge 1$, the square
\[
\squa{F(a)}{G(a)}{F(\partial a)}{G(\partial a)}{\eta_a}{\eta_{\partial a}}{F(\textbf{p}_a)}{G(\textbf{p}_a)}
\]
is distinguished. But this implies, as can be checked by induction on the length of $a$, that $\eta$ is an identity natural transformation.\footnote{This fact can be viewed in an alternative way using constructions and results discussed later in the paper: recalling that $[1] = \{0 < 1\}$, let $[1]^{\text{dist}}_{\text{pre}}$ be the precontextual category obtained from $[1]_{\text{pre}}$ (see Proposition \ref{prop: precontextual category from a category}) by regarding
\[
\dsqua{0}{1}{z}{z}{!}{}{}{}
\]
as a distinguished square. By Proposition \ref{prop: multimorphisms from contextual vs precontextual categories}, using that $L([1]_{\text{pre}}^{\text{dist}}) \cong \mathcal O_1 \cong L(\mathcal O_1^{\text{pre}})$, we have a bijection between bimorphisms $\mathcal A \times [1]^{\text{dist}}_{\text{pre}} \rightarrow \mathcal B$ and bimorphisms $\mathcal A \times \mathcal O_1^{\text{pre}} \rightarrow \mathcal B$. The latter correspond to contextual functors $\mathcal A \rightarrow \mathcal B$, and the former to pairs of contextual functors equipped with a natural transformation such that the naturality squares of all length-$1$ display maps are distinguished.
}
\end{remark}

\section{Exponentiation and bimorphisms}
\label{sec: exponentiation and bimorphisms}

\subsection{Exponentiation of contextual categories}
\label{subsec: exponentiation}

We will now define the exponential $\mathcal C^\mathcal A$ of two contextual categories $\mathcal A$, $\mathcal C$. In fact, we will construct a certain category with attributes $D_{\text{att}}(\mathcal A,\mathcal C)$ and then let $\mathcal C^\mathcal A = \cont(D_{\text{att}}(\mathcal A,\mathcal C))$.

\begin{definition}
\label{def: gap map, relative length-$1$ display map}
Let $\mathcal C$ be a contextual category. Suppose given a commutative square $Q:[1]^2 \rightarrow \mathcal C$ of the form
\[
\dsqua{b'}{b}{a'}{a,}{f'}{f}{p'}{p}
\]
i.e. such that $p$ and $p'$ are display maps (of any length). Let
\[\begin{tikzcd}[ampersand replacement=\&]
	{b'} \\
	\& c \& b \\
	\& {a'} \& a
	\arrow["g", dashed, from=1-1, to=2-2]
	\arrow["{f'}", curve={height=-18pt}, from=1-1, to=2-3]
	\arrow["{p'}"', curve={height=18pt}, two heads, from=1-1, to=3-2]
	\arrow[from=2-2, to=2-3]
	\arrow[two heads, from=2-2, to=3-2]
	\arrow["p", two heads, from=2-3, to=3-3]
	\arrow["f"', from=3-2, to=3-3]
\end{tikzcd}\]
be such that the bottom right square is distinguished and $g$ is the unique arrow making the diagram commute. We call $g$ the \emph{gap map} of $Q$. We say that $Q$ is a \emph{relative length-$1$ display map} if $g$ is a length-$1$ display map.
\end{definition}

\begin{remark}
\leavevmode
\begin{itemize}
	\item If $Q$ as above is a relative length-$1$ display map, then $\ell(p') = \ell(p)+1$. Also, if $f$ is an identity arrow, then $Q$ is a relative length-$1$ display map precisely when $f'$ is a length-$1$ display map.
	
	\item If $Q$, $Q':[1]^2 \rightarrow \mathcal C$ are relative length-$1$ display maps, then a natural transformation $\alpha:Q \rightarrow Q'$ induces, by functoriality of limits, a morphism $\alpha_*:g \rightarrow g'$ in $\iiAr(\mathcal C)$ between the respective gap maps.
\end{itemize}
\end{remark}

\begin{definition}
\label{def: indexed sort}
Let $\mathcal A$ be a precontextual category and $\mathcal C$ a contextual category. Given functors $F$, $G:|\mathcal A| \rightarrow |\mathcal C|$, we will say that a natural transformation $\pi:G \Rightarrow F$ is an \emph{indexed sort} if it has the following properties:\footnote{Indexed sorts play a role analogous to that of strict Reedy types in \cite{KapLum21}.}
\begin{enumerate}[label=\textbf{IS(\roman*)}]	
	\item For every length-$1$ display map $p:a \rightarrow \partial a$ in $\mathcal A$, the square
	\[
	\squa{G(a)}{F(a)}{G(\partial a)}{F(\partial a)}{\pi_{a}}{\pi_{\partial a}}{G(p)}{F(p)}
	\]
	is a relative length-$1$ display map.
	
	\item Given a length-$1$ distinguished square
	\[
	\dsqua{a}{b}{\partial a}{\partial b}{f'}{f}{\textbf{p}_a}{\textbf{p}_b}
	\]
	in $\mathcal A$, let $p:G(a) \twoheadrightarrow x$ be the gap map of
	\[
	\dsqua{G(a)}{F(a)}{G(\partial a)}{F(\partial a),}{\pi_{a}}{\pi_{\partial a}}{G(\textbf{p}_a)}{F(\textbf{p}_a)}
	\]
	and $q:G(b) \twoheadrightarrow y$ that of
	\[
	\dsqua{G(b)}{F(b)}{G(\partial b)}{F(\partial b).}{\pi_{b}}{\pi_{\partial b}}{G(\textbf{p}_b)}{F(\textbf{p}_b)}
	\]
	Then the commutative square
	\[
	\dsqua{G(a)}{G(b)}{x}{y,}{G(f')}{}{p}{q}
	\]
	where the bottom arrow is induced by $\pi$, is distinguished.
\end{enumerate}
\end{definition}

\begin{remark}
$x \rightarrow y$ is the unique dashed arrow making the following diagram commute:
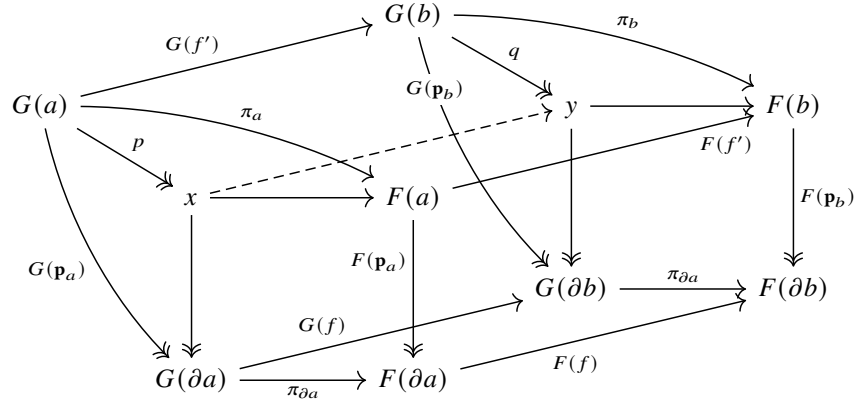
\[\begin{tikzcd}
	&&& {G(b)} \\
	{G(a)} &&&& y && {F(b)} \\
	& x && {F(a)} \\
	&&&& {G(\partial b)} && {F(\partial b)} \\
	& {G(\partial a)} && {F(\partial a)}
	\arrow["{q}", two heads, from=1-4, to=2-5]
	\arrow["{\pi_{b}}", curve={height=-12pt}, from=1-4, to=2-7]
	\arrow["{G(\textbf{p}_b)}"{description, pos=0.2}, curve={height=12pt}, two heads, from=1-4, to=4-5]
	\arrow["{G(f')}", from=2-1, to=1-4]
	\arrow["{p}", two heads, from=2-1, to=3-2]
	\arrow["{\pi_{a}}", curve={height=-12pt}, from=2-1, to=3-4]
	\arrow["{G(\textbf{p}_a)}"', curve={height=12pt}, two heads, from=2-1, to=5-2]
	\arrow[from=2-5, to=2-7]
	\arrow[two heads, from=2-5, to=4-5]
	\arrow["{F(\textbf{p}_b)}", two heads, from=2-7, to=4-7]
	\arrow[dashed, from=3-2, to=2-5]
	\arrow[from=3-2, to=3-4]
	\arrow[two heads, from=3-2, to=5-2]
	\arrow["{F(f')}"'{pos=0.8}, from=3-4, to=2-7]
	\arrow["{F(\textbf{p}_a)}"'{pos=0.3}, two heads, from=3-4, to=5-4]
	\arrow["{\pi_{\partial a}}", from=4-5, to=4-7]
	\arrow["{G(f)}"{pos=0.4}, from=5-2, to=4-5]
	\arrow["{\pi_{\partial a}}"', from=5-2, to=5-4]
	\arrow["{F(f)}"'{pos=0.3}, from=5-4, to=4-7]
\end{tikzcd}\]
\end{remark}

\begin{example}
\label{ex: indexed sort over terminal functor}
If $F:|\mathcal A| \rightarrow |\mathcal C|$ is the constant functor on $1_\mathcal C$, then $!:G \Rightarrow F$ is an indexed sort if and only if $G$ is a contextual functor from $\mathcal A$ to $\mathcal C$.
\end{example}

\begin{notation}
\label{not: D(A,C)}
Let $\mathcal A$ be a precontextual category and $\mathcal C$ a contextual category with underlying categories $A$ and $C$, respectively. We will denote by $D(\mathcal A, \mathcal C)$ the full subcategory of the functor category $C^A$ spanned by the functors that preserve display maps (we do not require that their lengths be preserved) and send every length-$0$ object of $\mathcal A$ to $1_\mathcal C$.
\end{notation}

The following will be useful later when we discuss bimorphisms of contextual categories:

\begin{definition}
\label{def: flexible morphism}
Let $\mathcal A$ be a precontextual category and $\mathcal C$ a contextual category. A \emph{flexible morphism} from $\mathcal A$ to $\mathcal C$ is a functor $F:\mathcal A \rightarrow \mathcal C$ that preserves display maps, distinguished squares, and sends every length-$0$ object to $1_\mathcal C$.
\end{definition}

\subsubsection{Exponentials as categories with attributes}
\label{subsubsec: exponentials as cwas}

We will now describe, given a precontextual category $\mathcal A$ and a contextual category $\mathcal C$, a \textsc{cwa} structure $D_{\text{att}}(\mathcal A,\mathcal C)$ on the category $D(\mathcal A, \mathcal C)$ as in Notation \ref{not: D(A,C)}. The exponential $\mathcal C^\mathcal A$ will then be defined as $\cont(D_{\text{att}}(\mathcal A,\mathcal C))$. In \S\ref{subsec: bimorphisms}, we will define bimorphisms in such a way that we have a natural bijection
$$
\{\text{bimorphisms }(\mathcal A,\mathcal B) \rightarrow \mathcal C\} \cong \Hom_\Att(\att(\mathcal B),D_{\text{att}}(\mathcal A,\mathcal C)) \cong \Hom_\Cont(\mathcal B,\mathcal C^\mathcal A).
$$

Let $A$, $C$ be the underlying categories of $\mathcal A$, $\mathcal C$, respectively.

\begin{construction}[Pullback of an indexed sort along a natural transformation]
For a functor $F \in D(\mathcal A,\mathcal C)$, let $E(F)$ be the set of indexed sorts, in the sense of Definition \ref{def: indexed sort}, with codomain $F$. We now describe how $F \mapsto E(F)$ can be made into a functor $E:D(\mathcal A, \mathcal C)^{\text{op}} \rightarrow \Set$. Suppose that we have a diagram
\[
\begin{tikzcd}
	 & G \arrow[]{d}{\pi} \\
	F' \arrow[swap]{r}{\alpha} & F
\end{tikzcd}
\]
in $D(\mathcal A, \mathcal C)$ where $\pi$ is an indexed sort. It can be completed into a diagram
\[
\tag{1}
\begin{tikzcd}
	G' \arrow[]{r}{\alpha'} \arrow[swap]{d}{\pi'} & G \arrow[]{d}{\pi} \\
	F' \arrow[swap]{r}{\alpha} & F
\end{tikzcd}
\]
in $D(\mathcal A, \mathcal C)$ in the following way: firstly, we construct a function $G':\Ob(A) \rightarrow \Ob(C)$ and families of arrows ${(\alpha'_a:G'(a) \rightarrow G(a))_{a \in A}}$, $(\pi'_a:G'(a) \rightarrow F'(a))_{a \in A}$ such that for each $a \in A$, the diagram
\[
\tag{2}
\squa{G'(a)}{G(a)}{F'(a)}{F(a)}{\alpha'_a}{\alpha_a}{\pi'_a}{\pi_a}
\]
commutes and is cartesian; then, by functoriality of pullbacks (computed in $C$), $G'$ can be uniquely extended into a functor such that $\alpha'$ and $\pi'$ are natural transformations $G' \Rightarrow G$ and $G' \Rightarrow F'$, respectively. Note that (1) will then be a pullback square in $C^A$. We construct the desired diagrams (2) recursively: for $a \in A$ of length $0$, we let (2) be the square
\[
\widesqua{1_\mathcal C}{G(a) = 1_\mathcal C}{F'(a) = 1_\mathcal C}{F(a) = 1_\mathcal C.}{id}{\alpha_a = id}{id}{\pi_a = id}
\]
Now, suppose given $a$ of length $n \ge 1$ and let $p:a \twoheadrightarrow \partial a$ be the corresponding length-$1$ display map. Having defined $G'(\partial a)$, $\alpha'_{\partial a}$ and $\pi'_{\partial a}$, consider the diagram
\[
\tag{3}
\begin{tikzcd}[ampersand replacement=\&]
	\&\& {G(a)} \\
	\&\&\& x \&\& {F(a)} \\
	y \&\& {F'(a)} \\
	\&\&\& {G(\partial a)} \&\& {F(\partial a)} \\
	{G'(\partial a)} \&\& {F'(\partial a)}
	\arrow["q"{description}, from=1-3, to=2-4]
	\arrow["{\pi_a}", curve={height=-12pt}, from=1-3, to=2-6]
	\arrow["{G(p)}"{description, pos=0.2}, curve={height=12pt}, from=1-3, to=4-4]
	\arrow[from=2-4, to=2-6]
	\arrow[from=2-4, to=4-4]
	\arrow["{F(p)}", from=2-6, to=4-6]
	\arrow[from=3-1, to=3-3]
	\arrow[from=3-1, to=5-1]
	\arrow["{\alpha_a}"'{pos=0.8}, from=3-3, to=2-6]
	\arrow["{F(p)}"'{pos=0.3}, from=3-3, to=5-3]
	\arrow["{\pi_{\partial a}}", from=4-4, to=4-6]
	\arrow["{\alpha'_{\partial a}}"{pos=0.4}, from=5-1, to=4-4]
	\arrow["{\pi'_{\partial a}}"', from=5-1, to=5-3]
	\arrow["{\alpha_{\partial a}}"'{pos=0.3}, from=5-3, to=4-6]
\end{tikzcd}\]
where the squares with vertices $x$ and $y$ are distinguished, and $q$ is the gap map of
\[
\dsqua{G(a)}{F(a)}{G(\partial a)}{F(\partial a)}{\pi_a}{\pi_{\partial a}}{G(p)}{F(p)}
\]
(so $q$ is a length-$1$ display map). By functoriality of pullbacks we obtain an arrow $r:y \rightarrow x$, and by taking the distinguished pullback of $q$ along $r$ we obtain a commutative diagram
\[\begin{tikzcd}
	&&& {G(a)} \\
	z &&&& x && {F(a)} \\
	& y && {F'(a)} \\
	&&&& {G(\partial a)} && {F(\partial a)} \\
	& {G'(\partial a)} && {F'(\partial a)}
	\arrow["q"{description}, two heads, from=1-4, to=2-5]
	\arrow["{\pi_a}", curve={height=-12pt}, from=1-4, to=2-7]
	\arrow["{G(p)}"{description, pos=0.2}, curve={height=12pt}, two heads, from=1-4, to=4-5]
	\arrow["{r'}"{description}, from=2-1, to=1-4]
	\arrow["{q'}"{description}, two heads, from=2-1, to=3-2]
	\arrow[from=2-5, to=2-7]
	\arrow[two heads, from=2-5, to=4-5]
	\arrow["{F(p)}", two heads, from=2-7, to=4-7]
	\arrow["r"{description}, from=3-2, to=2-5]
	\arrow[from=3-2, to=3-4]
	\arrow[two heads, from=3-2, to=5-2]
	\arrow["{\alpha_a}"'{pos=0.8}, from=3-4, to=2-7]
	\arrow["{F(p)}"'{pos=0.3}, two heads, from=3-4, to=5-4]
	\arrow["{\pi_{\partial a}}", from=4-5, to=4-7]
	\arrow["{\alpha'_{\partial a}}"{pos=0.4}, from=5-2, to=4-5]
	\arrow["{\pi'_{\partial a}}"', from=5-2, to=5-4]
	\arrow["{\alpha_{\partial a}}"'{pos=0.3}, from=5-4, to=4-7]
\end{tikzcd}\]
where $q$, $q'$, $r$ and $r'$ form a distinguished square. Let $s$ be the composite $z \overset{q'}{\rightarrow} y \overset{}{\rightarrow} F'(a)$ in the above diagram. The square
\[
\tag{4}
\squa{z}{G(a)}{F'(a)}{F(a)}{r'}{\alpha_a}{s}{\pi_a}
\]
is cartesian. To check this, by the pasting lemma for pullbacks it suffices that
\[
\squa{y}{x}{F'(a)}{F(a)}{r}{\alpha_a}{}{}
\]
be cartesian; this in turn follows from the front, back and bottom faces in diagram (3) being pullback squares. We then let $G'(a)$, $\alpha'_a$ and $\pi'_a$ be $z$, $r'$ and $s$, respectively -- in other words, the desired square (2) is taken to be (4).
\end{construction}

It follows immediately that the natural transformation $\pi':G' \Rightarrow F'$ thus obtained satisfies condition IS(i) from Definition \ref{def: indexed sort}. To verify IS(ii), consider a length-$1$ distinguished square
\[
\dsqua{a}{b}{\partial a}{\partial b.}{f'}{f}{p}{q}
\]
Let
$$
g_1:G(a) \longrightarrow x, \qquad g_2:G(b) \longrightarrow y,
$$
$$
g_3:G'(a) \longrightarrow x', \qquad g_4:G'(b) \longrightarrow y'
$$
be the gap maps of, respectively, the squares
\[\begin{tikzcd}[ampersand replacement=\&]
	{G(a)} \& {F(a)} \&\& {G(b)} \& {F(b)} \\
	{G(\partial a)} \& {F(\partial a),} \&\& {G(\partial b)} \& {F(\partial b),}
	\arrow["{\pi_{a}}", from=1-1, to=1-2]
	\arrow["{G(p)}"', two heads, from=1-1, to=2-1]
	\arrow["{(1)}"{description}, draw=none, from=1-1, to=2-2]
	\arrow["{F(p)}", two heads, from=1-2, to=2-2]
	\arrow["{\pi_{b}}", from=1-4, to=1-5]
	\arrow["{G(q)}"', two heads, from=1-4, to=2-4]
	\arrow["{(2)}"{description}, draw=none, from=1-4, to=2-5]
	\arrow["{F(q)}", two heads, from=1-5, to=2-5]
	\arrow["{\pi_{\partial a}}"', from=2-1, to=2-2]
	\arrow["{\pi_{\partial b}}"', from=2-4, to=2-5]
\end{tikzcd}\]
\[\begin{tikzcd}[ampersand replacement=\&]
	{G'(a)} \& {F'(a)} \&\& {G'(b)} \& {F'(b)} \\
	{G'(\partial a)} \& {F'(\partial a),} \&\& {G'(\partial b)} \& {F'(\partial b).}
	\arrow["{\pi'_{a}}", from=1-1, to=1-2]
	\arrow["{G'(p)}"', two heads, from=1-1, to=2-1]
	\arrow["{(3)}"{description}, draw=none, from=1-1, to=2-2]
	\arrow["{F'(p)}", two heads, from=1-2, to=2-2]
	\arrow["{\pi'_{b}}", from=1-4, to=1-5]
	\arrow["{G'(q)}"', two heads, from=1-4, to=2-4]
	\arrow["{(4)}"{description}, draw=none, from=1-4, to=2-5]
	\arrow["{F'(q)}", two heads, from=1-5, to=2-5]
	\arrow["{\pi'_{\partial a}}"', from=2-1, to=2-2]
	\arrow["{\pi'_{\partial b}}"', from=2-4, to=2-5]
\end{tikzcd}\]
Then we obtain a commutative cube
\[\begin{tikzcd}[ampersand replacement=\&]
	\& {G(a)} \&\& {G(b)} \\
	{G'(a)} \&\& {G'(b)} \\
	\& x \&\& y \\
	{x'} \&\& {y'.}
	\arrow["{G(f')}", from=1-2, to=1-4]
	\arrow["{g_1}"'{pos=0.8}, two heads, from=1-2, to=3-2]
	\arrow["{g_2}", two heads, from=1-4, to=3-4]
	\arrow["{\alpha'_{a'}}", from=2-1, to=1-2]
	\arrow["{G'(f')}"{pos=0.8}, from=2-1, to=2-3]
	\arrow["{g_3}"', two heads, from=2-1, to=4-1]
	\arrow["{\alpha'_{b'}}"', from=2-3, to=1-4]
	\arrow["{g_4}"{pos=0.2}, two heads, from=2-3, to=4-3]
	\arrow[from=3-2, to=3-4]
	\arrow[from=4-1, to=3-2]
	\arrow[from=4-1, to=4-3]
	\arrow[from=4-3, to=3-4]
\end{tikzcd}\]
As $\pi:G \Rightarrow F$ is an indexed sort, the face with vertical maps $g_1$ and $g_2$ is a distinguished square. Moreover, it follows from the construction of $G'$, $\pi'$ and $\alpha'$ that the left and right faces are also distinguished squares. We now conclude from the pasting lemma for distinguished pullbacks that the face with vertical maps $g_3$ and $g_4$ is a distinguished square. Thus $\pi'$ is an indexed sort.

\begin{definition}
\label{def: indexed distinguished square}
A diagram of natural transformations
\[
\begin{tikzcd}
	G' \arrow[]{r}{\alpha'} \arrow[swap]{d}{\pi'} & G \arrow[]{d}{\pi} \\
	F' \arrow[swap]{r}{\alpha} & F
\end{tikzcd}
\]
obtained as above will be called an \emph{indexed distinguished square}.
\end{definition}

\begin{remark}
\label{rem: characterization indexed distinguished squares}
	It can be verified inductively that indexed distinguished squares admit the following characterization: given an indexed sort $\pi:G \rightarrow F$ and a natural transformation $\alpha:F' \rightarrow F$ where $F$, $G$, $F' \in D(\mathcal A,\mathcal C)$, there exists a unique triple $(G',\alpha',\pi')$ -- which is then constructed as above -- as in the diagram
	\[
	\squa{G'}{G}{F'}{F}{\alpha'}{\alpha}{\pi'}{\pi}
	\]
	such that
	\begin{enumerate}[label=\textbf{IDS(\roman*)}]
		\item $\pi'$ is an indexed sort.
		
		\item Let $p:a \rightarrow \partial a$ be a length-$1$ display map in $\mathcal A$. Let $g:G(a) \rightarrow x$ and $h:G'(a) \rightarrow y$ be the gap maps of
		\[\begin{tikzcd}[ampersand replacement=\&]
			{G(a)} \& {F(a)} \&\& {G'(a)} \& {F'(a)} \\
			{G(\partial a)} \& {F(\partial a),} \&\& {G'(\partial a)} \& {F'(\partial a),}
			\arrow["{\pi_a}", from=1-1, to=1-2]
			\arrow["{G(p)}"', two heads, from=1-1, to=2-1]
			\arrow["{F(p)}", two heads, from=1-2, to=2-2]
			\arrow["{\sigma_a}", from=1-4, to=1-5]
			\arrow["{G'(p)}"', two heads, from=1-4, to=2-4]
			\arrow["{F'(p)}", two heads, from=1-5, to=2-5]
			\arrow["{\pi_{\partial a}}"', from=2-1, to=2-2]
			\arrow["{\sigma_{\partial a}}"', from=2-4, to=2-5]
		\end{tikzcd}\]
		respectively. Then the canonical square
		\[
		\dsqua{G'(a)}{G(a)}{y}{x}{\alpha'_a}{}{h}{g}
		\]
		is distinguished.
	\end{enumerate}
\end{remark}

It follows from this description than for a commutative diagram
\[\begin{tikzcd}[ampersand replacement=\&]
	{G''} \& {G'} \& G \\
	{F''} \& {F'} \& F
	\arrow["{\beta'}", from=1-1, to=1-2]
	\arrow["{\pi''}"', from=1-1, to=2-1]
	\arrow["{\alpha'}", from=1-2, to=1-3]
	\arrow["{\pi'}", from=1-2, to=2-2]
	\arrow["\pi", from=1-3, to=2-3]
	\arrow["\beta"', from=2-1, to=2-2]
	\arrow["\alpha"', from=2-2, to=2-3]
\end{tikzcd}\]
in $D(\mathcal A, \mathcal C)$, if the left and right squares are indexed distinguished squares, then so is the composite square.

This allows us to extend, as desired, the assignment $F \mapsto E(F)$ for $F \in D(\mathcal A,\mathcal C)$ to a functor $E:D(\mathcal A, \mathcal C)^{\text{op}} \rightarrow \Set$. Its category of elements $\int E$ is then canonically isomorphic to the category whose objects are the indexed sorts and whose morphisms are the indexed distinguished squares, with composition performed horizontally as indicated above.

We then have a category with attributes, which will be denoted by $D_{\text{att}}(\mathcal A, \mathcal C)$, defined by the diagram
\[
\begin{tikzcd}
	\int E \arrow[]{rr}{} \arrow[swap]{dr}{P} & & \iiAr(D(\mathcal A, \mathcal C)) \arrow[]{dl}{\cod} \\
	 & D(\mathcal A, \mathcal C) &
\end{tikzcd}
\]
where: $P:\int E \rightarrow D(\mathcal A,\mathcal C)$ is the canonical projection from the category of elements; the functor $\int E \rightarrow \iiAr(D(\mathcal A, \mathcal C))$ maps an object, resp. arrow, to its corresponding indexed sort, resp. indexed distinguished square; and the constant functor with value $1_\mathcal C$ is taken as the distinguished terminal object.

\subsubsection{Exponential contextual categories}

\begin{definition}
\label{def: exponential}
The \emph{exponential} between a precontextual category $\mathcal A$ and a contextual category $\mathcal C$, which will be denoted by $\mathcal C^\mathcal A$, is defined as the contextual category obtained by applying the coreflection functor $\cont:\Att \rightarrow \Cont$ to the category with attributes $D_{\text{att}}(\mathcal A,\mathcal C)$ constructed above.
\end{definition}

By construction, length-$n$ objects of $\mathcal C^\mathcal A$ correspond bijectively to diagrams
$$
F_n \rightarrow F_{n-1} \rightarrow \cdots \rightarrow F_1 \rightarrow F_0
$$
in the functor category $C^A$ such that $F_0$ is constant on $1_\mathcal C$ and each $F_i \rightarrow F_{i-1}$ is an indexed sort. Under this translation, morphisms in $\mathcal C^\mathcal A$ from $(F_i)_{i \le n}$ to $(G_i)_{i \le m}$ correspond to natural transformations $F_n \rightarrow G_m$.

\begin{notation}
\label{not: heart}
The above observation defines a ``top level" functor $|\mathcal C^\mathcal A| \rightarrow C^A$. We will denote its action on objects and morphisms by the subscript $\varheartsuit$:
$$
(F_i)_{i \le n} \mapsto F_\varheartsuit := F_n, \qquad (\eta:(F_i)_{i \le n} \rightarrow (G_i)_{i \le m}) \mapsto \eta_\varheartsuit := (\eta:F_n \rightarrow G_m).
$$

This functor is precisely the composite
\[
\mathcal C^\mathcal A = \cont(D_{\text{att}}(\mathcal A,\mathcal C)) \xrightarrow{U_{D_{\text{att}}(\mathcal A,\mathcal C)}} D_{\text{att}}(\mathcal A,\mathcal C) \hookrightarrow C^A
\]
where $U_{D_{\text{att}}(\mathcal A,\mathcal C)}$ is as in Construction \ref{constr: cwas and cont cats}.
\end{notation}

\begin{remark}
\label{rem: category of contextual functors as subcategory of exponential}
It follows from Example \ref{ex: indexed sort over terminal functor} that the full subcategory of $\mathcal C^\mathcal A$ spanned by its length-$1$ objects is isomorphic to the category of contextual functors $\iiHom_\iiPrecont(\mathcal A,\mathcal C)$.
\end{remark}

\subsection{$\Cat$-powers of contextual categories}
\label{subsec: cat-power}

Recall that the category of (small) contextual categories, $\Cont$, underlies a strict $2$-category $\iiCont$ in which the $2$-cells are the natural transformations between contextual functors. From now on, we drop the subscript and write just $\iiHom(\mathcal A,\mathcal B)$ for the category of morphisms from $\mathcal A$ to $\mathcal B$.

We will now construct, for a contextual category $\mathcal C$ and a category $A$, a contextual category $\mathcal C^A$ with isomorphisms
$$
\iiHom(\mathcal B,\mathcal C^A) \cong \iiHom(\mathcal B,\mathcal C)^A
$$
natural in $\mathcal B \in \Cont^{\text{op}}$. After giving a direct description of $\mathcal C^A$, we will remark how it can be constructed as an exponential, in the sense of Definition \ref{def: exponential}, between $\mathcal C$ and a precontextual category.

Writing $C$ for the underlying category of $\mathcal C$, we will equip the functor category $C^A$ with a structure of \textsc{cwa}. For each functor $F:A \rightarrow C$, let $E(F)$ be the set of natural transformations $\pi:G \rightarrow F$ such that $\pi_a$ is a length-$1$ display map for each $a \in A$. The assignment $F \mapsto E(F)$ extends into a functor $E:(C^A)^{\text{op}} \rightarrow \Set$ in the following way: given $\pi:G \rightarrow F$ in $E(F)$ and a natural transformation $\alpha:F' \rightarrow F$, we let $E(\alpha)(\pi)$ be the natural transformation $\pi':G' \rightarrow F'$ as in the unique pullback square
\[
\squa{G'}{G}{F'}{F}{\alpha'}{\alpha}{\pi'}{\pi}
\]
whose diagram of $a$-components is a distinguished square in $\mathcal C$ for all $a \in A$. Since distinguished pullbacks are strictly functorial, $E$ preserves composition and identities.

\begin{definition}
\label{def: cat-power}
Denote by $\Fun_\att(A,\mathcal C)$ the category with attributes given by the above construction with distinguished terminal object the functor $A \rightarrow C$ constant on $1_\mathcal C$. We define $\mathcal C^A$ as the contextual category $\cont(\Fun_\att(A,\mathcal C))$.
\end{definition}

\begin{remark}
\label{rem: explicit description of power}
Using the definition of $\cont:\Att \rightarrow \Cont$, we can explicitly describe $\mathcal C^A$ as follows:
\begin{enumerate}[label=(\roman*)]
	\item For $n \ge 0$, its length-$n$ objects are sequences $(F_0,...,F_n)$ of functors $A \rightarrow C$ such that
	\begin{itemize}
		\item $F_0$ is the constant functor with value $1_\mathcal C$, and
		
		\item for $1 \le i \le n$, we have that $\partial(F_i(a)) = F_{i-1}(a)$ and the family of display maps $(F_i(a) \twoheadrightarrow F_{i-1}(a))_{a \in A}$ is a natural transformation $F_n \rightarrow F_{n-1}$.
	\end{itemize}
	
	Alternatively, length-$n$ objects of $\cont(\Fun_{\text{att}}(A,\mathcal C))$ correspond to morphisms $A \rightarrow \iiHom(\mathcal O_n^{\text{pre}},\mathcal C)$.
	
	\item Morphisms from $(F_0, ..., F_m)$ to $(G_0, ..., G_n)$ are the natural transformations from $F_m$ to $G_n$.
	
	\item For $n \ge 1$, the predecessor of an object $(F_0,...,F_n)$ is $(F_0, ..., F_{n-1})$, and the corresponding length-$1$ display map is the natural transformation $F_n \rightarrow F_{n-1}$ whose components are length-$1$ display maps in $\mathcal C$.
	
	Given a display map $\pi:(F_0, ..., F_m) \rightarrow (F_0, ..., F_{m-1})$ and a morphism $\alpha:(G_0, ..., G_n) \rightarrow (F_0, ..., F_{m-1})$, the corresponding distinguished square is
	\[
	\squa{(G_0,...,G_n,G')}{(F_0,...,F_m)}{(G_0,...,G_n)}{(F_0,...,F_{m-1})}{\alpha'}{\alpha}{\pi'}{\pi}
	\]
	where $G'$, $\pi'$ and $\alpha'$ are as in the pullback square
	\[
	\squa{G'}{F_m}{G_n}{F_{m-1}}{\alpha'}{\alpha}{\pi'}{\pi}
	\]
	given by the structure of $\Fun_{\text{att}}(A,\mathcal C)$, i.e. the latter diagram is componentwise a distinguished square in $\mathcal C$.
\end{enumerate}
\end{remark}

\begin{proposition}
Given $A \in \Cat$ and $\mathcal C \in \Cont$, there exist isomorphisms of categories
$$
\iiHom(\mathcal B, \mathcal C^A) \cong \iiHom(\mathcal B,\mathcal C)^A
$$
natural in $\mathcal B \in \Cont$.
\end{proposition}

\begin{proof}
The strict $2$-adjunction between $\cont$ and $\att$ yields an isomorphism
$$
\iiHom(\mathcal B,\mathcal C^A) = \iiHom(\mathcal B,\cont(\Fun_{\text{att}}(A,\mathcal C))) \cong \iiHom_{\iiAtt}(\att(\mathcal B), \Fun_{\text{att}}(A,\mathcal C))
$$
natural in $\mathcal B \in \Cont$. Let $B$, $C$ be the underlying categories of $\mathcal B$, $\mathcal C$. Morphisms from $\att(\mathcal B)$ to $\Fun_{\text{att}}(A,\mathcal C)$ are functors $F:\mathcal B \rightarrow C^A$ such that (i) for each length-$1$ display map $p$ in $\mathcal B$, $F(p):A \rightarrow \mathcal C$ is componentwise a length-$1$ display map; and (ii) each distinguished square in $\mathcal B$ is sent to a commutative square in $C^A$ whose $a$-component is a distinguished square in $\mathcal C$ for all $a \in A$. We conclude that the canonical isomorphism $(C^A)^B \cong (C^B)^A$ restricts to an isomorphism between $\iiHom_{\iiAtt}(\att(\mathcal B),\Fun_{\text{att}}(A,\mathcal C))$ and the full subcategory of $(C^B)^A$ spanned by those functors $H$ such that for each $a \in A$, $H(a):\mathcal B \rightarrow \mathcal C$ is a contextual functor. In other words, we have obtained an isomorphism $\iiHom_{\iiAtt}(\att(\mathcal B),\Fun_{\text{att}}(A,\mathcal C)) \cong \iiHom(\mathcal B,\mathcal C)^A$.
\end{proof}

Now, we note that $\mathcal C^A$ can be obtained as a certain exponential between a precontextual category and a contextual category.

\begin{proposition}
\label{prop: precontextual category from a category}
For a category $A$, let $A_{\text{pre}}$ be the following precontextual category:
\begin{itemize}
	\item Its underlying category is the category $A_+$ obtained by freely adjoining a terminal object $z$ to $A$.
	
	\item The length function is given by $l(z) = 0$ and $l(a) = 1$ for $a \in A$. The display maps are the terminal arrows $a \rightarrow z$ for $a \in A$.
	
	\item There are no distinguished squares.
\end{itemize}
There exists an isomorphism of contextual categories $\mathcal C^{A_{\text{pre}}} \cong \mathcal C^A$.
\end{proposition}

\begin{proof}
It suffices to check that there exists an isomorphism of categories with attributes $D_{\text{att}}(A_{\text{pre}},\mathcal C) \cong \Fun_{\text{att}}(A,\mathcal C)$. Let $C$ be the underlying category of $\mathcal C$. A functor $F:A_+ \rightarrow \mathcal C$ belongs to $D(A_{\text{pre}}, \mathcal C)$ if and only if $F(z) = 1_\mathcal C$ (as this implies that display maps are preserved). Since $z$ was freely adjoined, restricting to $A$ defines an isomorphism of categories $D(A_{\text{pre}}, \mathcal C) \cong C^A$.

Now, note that a natural transformation $\pi:G \rightarrow F$ in $D(A_{\text{pre}},\mathcal C)$ is an indexed sort if and only if for every $a \in A$, the square
\[
\widesqua{G(a)}{F(a)}{G(z) = 1_\mathcal C}{F(z) = 1_\mathcal C}{\pi_a}{\alpha_z = \; id}{}{}
\]
is a relative length-$1$ display map, which in turn is equivalent to $\pi_a$ being a length-$1$ display map. This means precisely that $\pi|_A:F|_A \rightarrow G|_A$ is a display map with respect to $\Fun_{\text{att}}(A,\mathcal C)$. Finally, consider a diagram
\[
\tag{\texttt{*}}
\squa{G'}{G}{F'}{F}{\alpha'}{\alpha}{\pi'}{\pi}
\]
where $\pi$, $\pi'$ are indexed sorts. By Remark \ref{rem: characterization indexed distinguished squares}, it is an indexed distinguished square if and only if for each $a \in A$ the square of $a$-components is a distinguished square in $\mathcal C$. This is equivalent to the diagram of functors $A \rightarrow C$ obtained by restriction of (\texttt{*}) being distinguished with respect to $\Fun_{\text{att}}(A,\mathcal C)$. We conclude that restriction along $A \hookrightarrow A_+$ defines an isomorphism $D_{\text{att}}(A_{\text{pre}},\mathcal C) \cong \Fun_{\text{att}}(A,\mathcal C)$.
\end{proof}

\begin{remark}
The contextual category associated with $A_{\text{pre}}$ is isomorphic to the one presented by the generalized algebraic theory (in fact, multisorted algebraic theory) given by the following data:
\begin{itemize}
	\item for each object $a \in A$, a sort symbol $\underline{a}$ introduced by $ \vdash \underline{a} \tp$;
	
	\item for each arrow $f:a \rightarrow b$ in $A$, a term symbol $\underline{f}$ introduced by $x:\underline{a} \vdash \underline{f}(x):\underline{b}$;
	
	\item for each object $a$, an axiom $x:\underline{a} \vdash \underline{id_a}(x) \equiv x: \underline{a}$, and for each diagram $a \overset{f}{\rightarrow} b \overset{g}{\rightarrow} c$, an axiom
	$$
	x:\underline{a} \vdash \underline{g \circ f}(x) \equiv \underline{g}(\underline{f}(x)): \underline{c}.
	$$
\end{itemize}
\end{remark}

As an application of $\iiCont$ being $\Cat$-powered, we obtain:

\begin{remark}
It is not difficult, at this point, to construct an isomorphism $\iiHom_\iiCont(L(\mathcal A),\mathcal C) \cong \iiHom_\iiPrecont(\mathcal A,\mathcal C)$ for $\mathcal A \in \Precont$ and $\mathcal C \in \Cont$. However, we will derive it in \S\ref{subsec: multimorphisms and exponentials via precont} as a corollary of $\mathcal C^{L(\mathcal A)} \cong \mathcal C^\mathcal A$.
\end{remark}

\subsection{The contextual category of family-valued models}

A standard notion of model of a contextual category $\mathcal A$ is a morphism
$$
\mathcal A \longrightarrow \Fam
$$
where $\Fam$ is the contextual category, defined in \cite{Car78}, of iterated families of (small) sets: a length-$1$ object is a small set $X$, a length-$2$ object over $X$ is a family of small sets $(Y_x)_{x \in X}$, a length-$3$ object is a family $(Z_y)_{y \in \coprod_{x \in X} Y_x}$, and so on. Precisely, it can be described as $\cont(\Fam_{\text{att}})$ for the (large) category with attributes $\Fam_{\text{att}} = (\Set,\{*\},\fm,\Sigma)$ where:
\begin{itemize}
	\item $\fm(X)$ is the (large) set of all of families of small sets $(Y_x)_{x \in X}$. In other words, $\fm(X) = \mathscr U^X$ where $\mathscr U$ is the universe from Convention \ref{convention: general notation}. On morphisms, $\fm$ acts by precomposition.
	
	\item $\Sigma$ sends a family $(Y_x)_{x \in X}$ to the projection $\pi_Y:\coprod_{x \in X}Y_x \rightarrow X$.
\end{itemize}

Note that the contextual category $\Fam^\mathcal A = \cont(D_{\text{att}}(\mathcal A,\Fam))$ has the family-valued models $\mathcal A \rightarrow \Fam$ as its length-$1$ objects. However, the default description of higher dependencies in $\Fam^\mathcal A$, given by the display maps in $D_{\text{att}}(\mathcal A,\Fam)$, is convoluted as it involves functors $\mathcal A \rightarrow \Fam$ rather than $|\mathcal A| \rightarrow \Set$; see Definition \ref{def: exponential}. In order to clarify the structure of $\Fam^\mathcal A$, we will now show how to present it, up to isomorphism, from a \textsc{cwa} whose base category is $\Set^{|\mathcal A|}$. The idea is that for a functor $M:\mathcal A \rightarrow \Fam$ in $D(\mathcal A,\Fam)$ --- that is, it preserves display maps (but not necessarily their length) and the distinguished terminal object ---, the set $E(M)$ of relative length-$1$ display maps over $M:\mathcal A \rightarrow \Fam$ only depends, up to isomorphism, on the composite $\mathcal A \overset{M}{\rightarrow} \Fam \overset{U}{\rightarrow} \Set$.

More precisely, we will consider a notion of family-valued model relative to a given ``base" functor $B:|\mathcal A| \rightarrow \Set$; the usual models $\mathcal A \rightarrow \Fam$ are recovered by taking $B$ as the terminal functor. For $M$ as above, $E(M)$ will be isomorphic to the set of models relative to $U \circ M$.

\subsubsection{Relative models}

\begin{notation}
	Let $\mathscr F$ be the category whose objects are the families of small sets $Y:X \rightarrow \mathscr U$ (which, as usual, we also write as $(Y_x)_{x \in X}$), and whose morphisms $(Y_x)_{x \in X} \rightarrow (Y'_x)_{x \in X'}$ are the functions $\amalg Y \rightarrow \amalg Y'$. By construction, we have a full-and-faithful (and essentially surjective) functor $\amalg:\mathscr F \rightarrow \Set$. As above, we let $\pi_Y: \amalg Y \rightarrow X$ be the projection map.
\end{notation}

\begin{definition}
	\label{def: relative model}
	Consider a functor $B:|\mathcal A| \rightarrow \Set$. A \emph{(family-valued) model of $\mathcal A$ relative to $B$} is a functor $F:|\mathcal A| \rightarrow \mathscr F$ equipped with a natural transformation $\varphi:\amalg \circ F \Rightarrow B$ such that:
	\begin{enumerate}[label=\textbf{RM(\roman*)}]
		\item $F(1_\mathcal A)$ is the constant family $\{*\}:B(1_\mathcal A) \rightarrow \mathscr U$, and $\amalg_{x \in B(1_\mathcal A)} \{*\}  = \amalg F(1_\mathcal A) \overset{\varphi_{1_\mathcal A}}{\rightarrow} B(1_\mathcal A)$ is the standard projection.
		
		\item If $a \in \mathcal A$ has length $\ge 1$, then $F(a)$ is a family of the form\footnote{In this definition, which keeps track of set-theoretic data, we fix a pullback functor $\iiAr(\Set) \times_{\Set} \iiAr(\Set) \rightarrow \Set$. For definiteness, we can use the standard one, whose action on objects is $(f:Y \rightarrow X,\; g:Y \rightarrow Z) \mapsto \{(y,z) \in Y \times Z \mid f(y) = g(z)\}$.} $\big(\amalg F(\partial a)\big) \times_{B(\partial a)} B(a) \rightarrow \mathscr U$, and $F(\textbf{p}_a)$ and $\varphi_a$ are the composites indicated in the diagram
		\[\begin{tikzcd}[ampersand replacement=\&]
			{\amalg F(a)} \\
			\& {\big(\amalg F(\partial a)\big) \times_{B(\partial a)} B(a)} \& {B(a)} \\
			\& {\amalg F(\partial a)} \& {B(\partial a).}
			\arrow["{\pi_{F(a)}}"{description}, from=1-1, to=2-2]
			\arrow["{\varphi_a}", curve={height=-18pt}, from=1-1, to=2-3,dashed]
			\arrow["{F(\textbf{p}_a)}"',curve={height=18pt}, from=1-1, to=3-2,dashed]
			\arrow[from=2-2, to=2-3]
			\arrow[from=2-2, to=3-2]
			\arrow["\lrcorner"{anchor=center, pos=0.125}, draw=none, from=2-2, to=3-3]
			\arrow[from=2-3, to=3-3]
			\arrow["{\varphi_{\partial a}}"', from=3-2, to=3-3]
		\end{tikzcd}\]
		
		\item For a distinguished square
		\[
		\dsqua{a}{b}{\partial a}{\partial b}{f'}{f}{\textbf{p}_a}{\textbf{p}_b}
		\]
		in $\mathcal A$, the diagram of sets
		$$
		\begin{tikzcd}
			\big(\amalg F(\partial a)\big) \times_{B(\partial a)} B(a) \arrow[]{rr}{F(f) \times_{B(f)} B(f')} \arrow[swap]{dr}{F(a)} & & \big(\amalg F(\partial b)\big) \times_{B(\partial b)} B(b) \arrow[]{dl}{F(b)} \\
			& \mathscr U &
		\end{tikzcd}
		$$
		commutes.
	\end{enumerate}
	The latter condition implies that the square
	\[
	\widesqua{\amalg F(a)}{\amalg F(b)}{\big(\amalg F(\partial a)\big) \times_{B(\partial a)} B(a)}{\big(\amalg F(\partial b)\big) \times_{B(\partial b)} B(b)}{F(f')}{F(f) \times_{B(f)} B(f')}{\pi_{F(a)}}{\pi_{F(b)}}
	\]
	is cartesian, but it also encodes the set-theoretic property of $F(a)$ being a \emph{strict} pullback, as a family, of $F(b)$ along $F(f) \times_{B(f)} B(f')$.
\end{definition}

Write $\Mod_B(\mathcal A)$ for the (large) set of models of $\mathcal A$ relative to $B$. We will now show how to turn the assignment $B \mapsto \Mod_B(\mathcal A)$ into a functor $(\Set^{|\mathcal A|})^{\text{op}} \rightarrow \Set$.

\begin{construction}
	\label{constr: distinguished base change of relative models}
	Consider functors $B^1$, $B^2:|\mathcal A| \rightarrow \Set$, a natural transformation $\beta:B^2 \rightarrow B^1$, and a relative model
	$$
	(F^1:|\mathcal A| \rightarrow \mathscr F, \; \varphi^1:\amalg \circ F^1 \Rightarrow B^1).
	$$
	Let us describe a functor $F^2:|\mathcal A| \rightarrow \mathscr F$ and natural transformations $\gamma$, $\varphi^2$ fitting into a pullback square
	\[
	\squa{\coprod \circ F^2}{\coprod \circ F^1}{B^2}{B^1.}{\gamma}{\beta}{\varphi^2}{\varphi^1}
	\]
	We start by constructing, recursively on the tree of display maps of $\mathcal A$, a function $F^2_\Ob:\Ob(\mathcal A) \rightarrow \Ob(\mathscr F)$ and a pullback square of the form
	\[
	\squa{\coprod F^2_\Ob(a)}{\coprod F^1(a)}{B^2(a)}{B^1(a).}{\gamma_a}{\beta_a}{\varphi^2_a}{\varphi^1_a}
	\]
	\begin{itemize}
		\item We let $F^2_\Ob(1_\mathcal A)$ be the family $(\{*\})_{x \in B^2(1_\mathcal A)}$, with $\varphi^2_{1_\mathcal A}: \amalg_{x \in B^2(1_\mathcal A)}\{*\} \rightarrow B^2(1_\mathcal A)$ the standard projection and $\gamma_{1_\mathcal A}:\amalg_{x \in B^2(1_\mathcal A)}\{*\} \rightarrow \amalg F^1(1_\mathcal A) = \amalg_{x \in B^1(1_\mathcal A)}\{*\}$ the map induced by $B^2(1_\mathcal A) \rightarrow B^1(1_\mathcal A)$.
		
		\item For $a \in \mathcal A$ of length $\ge 1$, suppose that $F^2_\Ob(\partial a)$, $\varphi^2_{\partial a}$ and $\gamma^2_{\partial a}$ have been constructed. Then we let $F^2_\Ob(a)$ be the composite
		$$
		(\amalg F^2_\Ob(\partial a)) \times_{B^2(\partial a)} B^2(a) \overset{\gamma_{\partial a} \times_{\beta_{\partial a}} \beta_a}{\longrightarrow} (\amalg F^1(\partial a)) \times_{B^1(\partial a)} B^1(a) \overset{F^1(a)}{\longrightarrow} \mathscr U,
		$$
		and $\gamma_a$ the top arrow in the canonical cartesian square
		\[\begin{tikzcd}[ampersand replacement=\&]
			{\amalg F^2_\Ob(a)} \&\& {\amalg F^1(a)} \\
			{(\amalg F^2_\Ob(\partial a)) \times_{B^2(\partial a)} B^2(a)} \&\& {(\amalg F^1(\partial a)) \times_{B^1(\partial a)} B^1(a).}
			\arrow[from=1-1, to=1-3]
			\arrow["{\pi_{F^2_\Ob(a)}}"',from=1-1, to=2-1]
			\arrow["{\pi_{F^1(a)}}",from=1-3, to=2-3]
			\arrow["{\gamma_{\partial a} \times_{\beta_{\partial a}} \beta_a}"', from=2-1, to=2-3]
		\end{tikzcd}\]
		Now, note that in the commutative diagram
		\[\begin{tikzcd}[ampersand replacement=\&]
			{(\amalg F^2_\Ob(\partial a)) \times_{B^2(\partial a)} B^2(a)} \&\& {(\amalg F^1(\partial a)) \times_{B^1(\partial a)} B^1(a)} \\
			{B^2(a)} \&\& {B^1(a)} \\
			\& {\amalg F^2_\Ob(\partial a)} \&\& {\amalg F_1(\partial a)} \\
			\& {B^2(\partial a)} \&\& {B^1(\partial a),}
			\arrow[""{name=0, anchor=center, inner sep=0}, "{\gamma_{\partial a} \times_{\beta_{\partial a}} \beta_a}"{description}, from=1-1, to=1-3]
			\arrow[from=1-1, to=2-1]
			\arrow[""{name=1, anchor=center, inner sep=0}, from=1-1, to=3-2]
			\arrow[from=1-3, to=2-3]
			\arrow[""{name=2, anchor=center, inner sep=0}, from=1-3, to=3-4]
			\arrow[""{name=3, anchor=center, inner sep=0}, "{\beta_a}"{description}, from=2-1, to=2-3]
			\arrow[""{name=4, anchor=center, inner sep=0}, "{B^2(\textbf{p}_a)}"{description}, from=2-1, to=4-2]
			\arrow[""{name=5, anchor=center, inner sep=0}, "{B^1(\textbf{p}_a)}"{description, pos=0.3}, from=2-3, to=4-4]
			\arrow[""{name=6, anchor=center, inner sep=0}, "{\gamma_{\partial a}}"{description}, from=3-2, to=3-4]
			\arrow[from=3-2, to=4-2]
			\arrow[from=3-4, to=4-4]
			\arrow[""{name=7, anchor=center, inner sep=0}, "{\beta_{\partial a}}"{description}, from=4-2, to=4-4]
			\arrow["{\text{(I)}}"{description}, draw=none, from=0, to=3]
			\arrow["{\text{(III)}}"{description}, draw=none, from=1, to=4]
			\arrow["{\text{(II)}}"{description}, draw=none, from=2, to=5]
			\arrow["{\text{(IV)}}"{description}, draw=none, from=6, to=7]
		\end{tikzcd}\]
		the squares (II), (III) (by definition) and (IV) (by the induction hypothesis) are cartesian, hence so is (I) by the pasting lemma for pullbacks. It follows that the outer composite square in
		\[\begin{tikzcd}[ampersand replacement=\&]
			{\amalg F^2_\Ob(a)} \&\&\& {\amalg F^1(a)} \\
			{(\amalg F^2_\Ob(\partial a)) \times_{B^2(\partial a)} B^2(a)} \&\&\& {(\amalg F^1(\partial a)) \times_{B^1(\partial a)} B^1(a)} \\
			{B^2(a)} \&\&\& {B^1(a)}
			\arrow["{\gamma_a}", from=1-1, to=1-4]
			\arrow["{\pi_{F^2_\Ob(a)}}"', from=1-1, to=2-1]
			\arrow["{\pi_{F^1(a)}}", from=1-4, to=2-4]
			\arrow["{\gamma_{\partial a} \times_{\beta_{\partial a}} \beta_a}"{description}, from=2-1, to=2-4]
			\arrow[from=2-1, to=3-1]
			\arrow[from=2-4, to=3-4]
			\arrow["{\beta_a}"', from=3-1, to=3-4]
		\end{tikzcd}\]
		is cartesian, and we let $\varphi^2_a:\amalg F^2_\Ob(a) \rightarrow B^2(a)$ be the left vertical composite.
		
		\item Having constructed $F^2_\Ob$, $\gamma$ and $\varphi^2$, observe that $F^2_\Ob$ extends uniquely, by functoriality of pullbacks, into a functor $F^2:|\mathcal A| \rightarrow \mathscr F$ such that the families $\gamma$ and $\varphi^2$ define natural transformations $\amalg \circ F^2 \Rightarrow \amalg \circ F^1$ and $\amalg \circ F^2 \Rightarrow B^2$, respectively.
	\end{itemize}
\end{construction}

\begin{proposition}
	The pair $(F^2,\varphi^2)$ constructed above is a model of $\mathcal A$ relative to $B^2$.
\end{proposition}

\begin{proof}
	Conditions RM(i) and RM(ii) from Definition \ref{def: relative model} are satisfied by construction. For RM(iii), note that for a distinguished square
	\[
	\dsqua{a}{b}{\partial a}{\partial b,}{f'}{f}{\textbf{p}_a}{\textbf{p}_b}
	\]
	the diagram
	\[\begin{tikzcd}[ampersand replacement=\&]
		{\big(\amalg F^2(\partial a)\big) \times_{B^2(\partial a)} B^2(a)} \&\& {\big(\amalg F^2(\partial b)\big) \times_{B^2(\partial b)} B^2(b)} \\
		{\big(\amalg F^1(\partial a)\big) \times_{B^1(\partial a)} B^1(a)} \&\& {\big(\amalg F^1(\partial b)\big) \times_{B^1(\partial b)} B^1(b)} \\
		\& {\mathscr U}
		\arrow["{F^2(f) \times_{B^2(f)}B^2(f')}", from=1-1, to=1-3]
		\arrow["{\gamma_{\partial a} \times_{\beta_{\partial a}} \beta_a}"', from=1-1, to=2-1]
		\arrow["{\gamma_{\partial b} \times_{\beta_{\partial b}} \beta_b}", from=1-3, to=2-3]
		\arrow["{F^1(f) \times_{B^1(f)}B^1(f')}", from=2-1, to=2-3]
		\arrow["{F^1(a)}"', from=2-1, to=3-2]
		\arrow["{F^1(b)}", from=2-3, to=3-2]
	\end{tikzcd}\]
	commutes, and the composites $\big(\amalg F^2(\partial a)\big) \times_{B^2(\partial a)} B^2(a) \rightarrow \mathscr U$ and $\big(\amalg F^2(\partial b)\big) \times_{B^2(\partial b)} B^2(b) \rightarrow \mathscr U$ are, by definition, $F^2(a)$ and $F^2(b)$, respectively.
\end{proof}

\begin{definition}
	\label{def: cwa of relative models}
	We will say that the cartesian square
	\[
	\squa{\coprod \circ F^2}{\coprod \circ F^1}{B^2}{B^1}{\gamma}{\beta}{\varphi^2}{\varphi^1}
	\]
	realizes the relative model $(F^2,\varphi^2)$ as the \emph{distinguished base change} of $(F^1,\varphi^1)$ along $\beta$.
	
	\vspace{0.5em}
	
	We leave it as an exercise to verify from Construction \ref{constr: distinguished base change of relative models} that the distinguished base change operation is strictly functorial. This yields a (large) category with attributes $\Mod_{rel}(\mathcal A)$ with base category $\Set^{|\mathcal A|}$, presheaf of types $\Mod_\bullet(\mathcal A):(\Set^{|\mathcal A|})^{\text{op}} \rightarrow \Set^+$, where $\Set^+$ is the category of sets $X \in \mathscr U^+$ (as in Convention \ref{convention: general notation}) and functions between them.

\end{definition}

\subsubsection{Characterizing $\Fam^\mathcal A$ via relative models}

It can be proved that if a functor $M:\mathcal A \rightarrow \Fam$ belongs to $D(\mathcal A,\Fam)$, then we have a canonical bijection between $\Mod_{U \circ M}(\mathcal A)$ and the set $E(M)$ of relative length-$1$ display maps with codomain $M$. Let us sketch how to associate with a relative length-$1$ display map $\varphi: F \Rightarrow M$ a functor $F':\mathcal A \rightarrow \mathscr F$, a natural transformation $\varphi':\amalg \circ F' \Rightarrow U \circ M$, and a natural isomorphism $\omega:\amalg \circ F' \Rightarrow U \circ F$ recursively on the tree of display maps of $\mathcal A$:
\begin{itemize}
	\item Noting that $UF(1_\mathcal A) = UM(1_\mathcal A) = U(1_\Fam) = \{*\}$, we (must) let $F'(1_\mathcal A) = (\{*\})_{x \in \{*\}}$, with $\varphi'_{1_\mathcal A}$, $\omega_{1_\mathcal A}:\amalg_{x \in \{*\}} \{*\} \rightarrow \{*\}$ the unique map.
	
	\item For $a \in \mathcal A$ of length $\ge 1$, suppose that $F'(\partial a)$, $\varphi'_{\partial a}$ and $\omega_{\partial a}$ have been constructed. Then from the relative length-$1$ display map
	\[
	\dsqua{F(a)}{M(a)}{F(\partial a)}{M(\partial a)}{\varphi_a}{\varphi_{\partial a}}{F(\textbf{p}_a)}{M(\textbf{p}_a)}
	\]
	in $\Fam$ we obtain a length-$1$ display map $F(a) \twoheadrightarrow F(\partial a) \times_{M(\partial a)} M(a)$, which in turn corresponds to a family $X:U\big( F(\partial a) \times_{M(\partial a)} M(a)\big) \rightarrow \mathscr U$. We let $F'(a)$ be the composite
	$$
	\amalg F'(\partial a) \times_{UM(\partial a)} UM(a) \overset{\omega_{\partial a} \times id}{\cong} UF(\partial a) \times_{UM(\partial a)} UM(a) \cong U\big(F(\partial a) \times_{M(\partial a)} M(a)\big) \overset{X}{\longrightarrow} \mathscr U.
	$$
	This yields, in particular, a canonical isomorphism $\coprod F'(a) \cong UF(a)$, which we take as $\omega_a$. Also, we let $\varphi'_a:\coprod F'(a) \rightarrow UM(a)$ be $U(\varphi_a) \circ \omega_a$.
	
	\item For $a$, $b \in \mathcal A$, we let $F'$ act on morphisms $a \rightarrow b$ by the composite
	$$
	\Hom_\mathcal A(a,b) \overset{UF}{\rightarrow} \Hom_\Set(UF(a),UF(b)) \overset{\blacklozenge}{\cong} \Hom_\Set(\amalg F'(a),\amalg F'(b)) = \Hom_{\mathscr F}(F'(a),F'(b))
	$$
	where $\blacklozenge$ is given by $f \mapsto \omega_{b^{-1}} \circ f \circ \omega_a$.
\end{itemize}

It is not difficult to verify functoriality of $F'$, naturality of $\omega$ and $\varphi'$, and that $F'$ satisfies the last condition from Definition \ref{def: relative model}.

A similar construction, also using an auxiliary natural isomorphism $\omega$, allows us to obtain a relative length-$1$ display map with codomain $M$ from a model relative to $U \circ M$; this construction is set up so that it is inverse to the previous one. This yields a morphism of \textsc{cwa}s
$$
K:D_{\text{att}}(\mathcal A,\Fam) \longrightarrow \Mod_{rel}(\mathcal A)
$$
whose functor between base categories is $U \circ - :D(\mathcal A,\Fam) \rightarrow \Set^{|\mathcal A|}$ and whose natural transformation $E \Rightarrow \Mod_{U \circ -}(\mathcal A)$ has as components the bijections constructed above. We then get a contextual functor
$$
\Fam^\mathcal A = \cont(D_{\text{att}}(\mathcal A,\Fam)) \overset{\cont(K)}{\longrightarrow} \cont(\Mod_{rel}(\mathcal A)),
$$
which is full-and-faithful since $U \circ -$ is full-and-faithful, and bijective on objects since $E \Rightarrow \Mod_{U \circ -}(\mathcal A)$ is an isomorphism; hence $\cont(K)$ is an isomorphism.

\subsection{Bimorphisms of contextual categories}
\label{subsec: bimorphisms}

Let $\mathcal A$, $\mathcal B$ and $\mathcal C$ be contextual categories with underlying categories $A$, $B$ and $C$, respectively. A contextual functor
$$
F:\mathcal B \longrightarrow \mathcal C^\mathcal A = \cont(D_{\text{att}}(\mathcal A,\mathcal C))
$$
corresponds, via the adjunction $\att \dashv \cont$, to a morphism $\att(\mathcal B) \rightarrow D_{\text{att}}(\mathcal A,\mathcal C)$. By taking the base categories of these \textsc{cwa}s, the latter morphism yields a functor $B \rightarrow C^A$, hence a functor $A \times B \rightarrow C$. Our next goal will be to give a more explicit characterization of the functors $A \times B \rightarrow C$ that arise in this way. Such functors will be the \emph{bimorphisms} from $(\mathcal A,\mathcal B)$ to $\mathcal C$. In fact, our characterization will make sense when $\mathcal A$, $\mathcal B$ are precontextual categories and $\mathcal C$ is a contextual category, so we introduce bimorphisms in this more general setting. This will be useful later when doing explicit calculations.

\begin{definition}
\label{def: bimorphism}
Let $\mathcal A$, $\mathcal B$ be precontextual categories and $\mathcal C$ a contextual category with underlying categories $A$, $B$, $C$, respectively. A functor $H:A \times B \rightarrow C$ is said to be a \emph{bimorphism} from $(\mathcal A,\mathcal B)$ to $\mathcal C$ if it satisfies the following conditions:
\begin{enumerate}[label=\textbf{Bim(\roman*)}]
	\item For each $b \in B$, $H(-,b):A \rightarrow C$ belongs to $D(\mathcal A,\mathcal C)$, and it is the constant functor on $1_\mathcal C$ if $b$ has length $0$.
	
	\item Suppose given length-$1$ display maps $p:a \rightarrow \partial a$ in $\mathcal A$ and $q:b \rightarrow \partial b$ in $\mathcal B$. Then the diagram
	\[
	\tag{\texttt{*}}
	\dsqua{H(a,b)}{H(a,\partial b)}{H(\partial a,b)}{H(\partial a,\partial b)}{H(id,q)}{H(id,q)}{H(p,id)}{H(p,id)}
	\]
	is a relative length-$1$ display map.
	
	In what follows, we denote by $\gap_H(a,b)$ the gap map of (\texttt{*}) (recall Definition \ref{def: gap map, relative length-$1$ display map}).
	
	\item Consider a length-$1$ display map $q:b \rightarrow \partial b$ in $\mathcal B$ and a length-$1$ distinguished square
	\[
	\dsqua{a_1}{a_2}{\partial a_1}{\partial a_2}{f'}{f}{p_1}{p_2}
	\]
	in $\mathcal A$. The commutative diagram
	\[
	\dsqua{H(a_1,b)}{H(a_2,b)}{x_1}{x_2,}{H(f',id)}{}{\gap_H(a_1,b)}{\gap_H(a_2,b)}
	\]
	where the bottom arrow is induced by $H$ and functoriality of pullbacks, is a distinguished square in $\mathcal C$.
	
	\item Consider a length-$1$ display map $p:a \rightarrow \partial a$ in $\mathcal A$ and a length-$1$ distinguished square
	\[
	\dsqua{b_1}{b_2}{\partial b_1}{\partial b_2}{f'}{f}{q_1}{q_2}
	\]
	in $\mathcal B$. The commutative diagram
	\[
	\dsqua{H(a,b_1)}{H(a,b_2)}{x_1}{x_2,}{H(id,f')}{}{\gap_H(a,b_1)}{\gap_H(a,b_2)}
	\]
	where the bottom arrow is induced by $H$ and functoriality of pullbacks, is a distinguished square in $\mathcal C$.
\end{enumerate}
\end{definition}

\begin{proposition}
\label{prop: characterization of maps to exponential, sym}
In the notation of Definition \ref{def: bimorphism}, a functor $H:A \times B \rightarrow C$ is a bimorphism from $(\mathcal A,\mathcal B)$ to $\mathcal C$ if and only if there exists a contextual functor $F:\mathcal B \rightarrow \mathcal C^\mathcal A$ such that $H$ equals the image under $(C^A)^B \cong C^{A \times B}$ of the composite (recall Notation \ref{not: heart})
$$
B \overset{F}{\longrightarrow} |\mathcal C^\mathcal A| \overset{(-)_\varheartsuit}{\longrightarrow} C^A.
$$
This establishes a bijection, natural in $\mathcal A$, $\mathcal B \in \Precont^{\text{op}}$ and $\mathcal C \in \Cont$, between contextual functors $\mathcal B \rightarrow \mathcal C^\mathcal A$ and bimorphisms $(\mathcal A,\mathcal B) \rightarrow \mathcal C$.
\end{proposition}

\begin{proof}
Firstly, note that for a morphism $F:\mathcal B \rightarrow \mathcal C^\mathcal A$, the functor $B \overset{F}{\rightarrow} |\mathcal C^\mathcal A| \overset{(-)_\varheartsuit}{\rightarrow} C^A$ factors through $D(\mathcal A,\mathcal C) \subset C^A$; also, the image of a length-$0$ object $b \in \mathcal B$ under $B \rightarrow D(\mathcal A,\mathcal C)$ is $F(b)_\varheartsuit = (1_{\mathcal C^\mathcal A})_\varheartsuit$, which in turn is the functor $\ct_{1_\mathcal C}$ constant on $1_\mathcal C$. So Bim(i) holds.

From the adjunction $\att \dashv \cont$ and the definition of $\mathcal C^\mathcal A$, morphisms $\mathcal B \rightarrow \mathcal C^\mathcal A$ correspond, via the above construction, to morphisms of categories with attributes $\att(\mathcal B) \rightarrow D_{\text{att}}(\mathcal A,\mathcal C)$. Thus we must prove that for a functor $F:B \rightarrow D(\mathcal A,\mathcal C)$ such that $F(b) = \ct_{1_\mathcal C}$ whenever $\ell(b) = 0$, the following are equivalent:
\begin{itemize}
	\item it is a morphism $\att(\mathcal B) \rightarrow D_{\text{att}}(\mathcal A,\mathcal C)$;
	
	\item its adjunct $H:A \times B \rightarrow C$ satisfies Bim(ii)-(iv).
\end{itemize}

Now, observe that Bim(ii) holds for $H$ if and only if condition IS(i) from Definition \ref{def: indexed sort} holds for $F(q):F(b) \rightarrow F(\partial b)$ for every length-$1$ display map $q:b \rightarrow \partial b$ in $\mathcal B$. Similarly, Bim(iii) holds for $H$ if and only if IS(ii) holds for $F(q)$ for every length-$1$ display map $q$ in $\mathcal B$.

Also, applying $F$ to a length-$1$ distinguished square
\[
\tag{\texttt{*}}
\dsqua{b_1}{b_2}{\partial b_1}{\partial b_2}{f'}{f}{q_1}{q_2}
\]
yields
\[
\dsqua{H(-,b_1)}{H(-,b_2)}{H(-,\partial b_2)}{H(-,\partial b_1).}{H(-,f')}{H(-,f)}{H(-,q_1)}{H(-,q_2)}
\]
By Remark \ref{rem: characterization indexed distinguished squares}, the latter is an indexed distinguished square (in other words, a square that is distinguished in $D_{\text{att}}(\mathcal A,\mathcal C)$) if and only if $H$ satisfies Bim(iv) for the square (\texttt{*}) and every length-$1$ display map $p:a \rightarrow \partial a$ in $\mathcal A$.
\end{proof}

\begin{proposition}
\label{prop: characterization of bimorphisms}
In the notation of Definition \ref{def: bimorphism}, a functor $H:A \times B \rightarrow C$ is a bimorphism from $(\mathcal A,\mathcal B)$ to $\mathcal C$ if and only if it satisfies
\begin{enumerate}[label=\textbf{Bim(\roman*')}]
	\item for each $b \in B$, $H(-,b):A \rightarrow C$ is a flexible morphism from $\mathcal A$ to $\mathcal C$, and it is constant on $1_\mathcal C$ if $b$ has length $0$,
\end{enumerate}
and conditions Bim(ii) and Bim(iv) from the statement of Proposition \ref{prop: characterization of maps to exponential, sym}.
\end{proposition}

\begin{proof}
It suffices to prove that if $H$ satisfies Bim(i), Bim(ii) and Bim(iv) from Proposition \ref{prop: characterization of maps to exponential, sym}, then it satisfies Bim(iii) if and only if for each $b \in B$, $H(-,b)$ preserves distinguished squares.

Firstly, suppose that $H$ satisfies Bim(iii). Let us check by induction on $n \ge 0$ that for every $b \in \mathcal B$ of length $n$, $H(-,b)$ preserves distinguished squares. For $n = 0$ the claim holds as if $\ell(b) = 0$, then $H(-,b) = \ct_{1_\mathcal C}$. Given $n \ge 1$, assume that it holds for $0$, ..., $n-1$. Suppose that $q:b \rightarrow \partial b$ is a length-$1$ display map in $\mathcal B$ with $\ell(b) = n$, and
\[
\tag{\texttt{*}}
\dsqua{a_1}{a_2}{\partial a_1}{\partial a_2}{f'}{f}{p_1}{p_2}
\]
is a length-$1$ distinguished square in $\mathcal A$. Consider the diagram
\[\begin{tikzcd}[ampersand replacement=\&]
	\&\&\& {H(a_2,b)} \&\& {H(a_2,\partial b)} \\
	{H(a_1,b)} \&\& {H(a_1,\partial b)} \\
	\&\&\& {H(\partial a_2,b)} \&\& {H(\partial a_2,\partial b)} \\
	{H(\partial a_1,b)} \&\& {H(\partial a_1,\partial b)}
	\arrow["{H(id,q)}", from=1-4, to=1-6]
	\arrow["{H(p_2,id)}"{description, pos=0.6}, two heads, from=1-4, to=3-4]
	\arrow["{(2)}"{description}, draw=none, from=1-4, to=3-6]
	\arrow["{H(p_2,id)}"{description}, two heads, from=1-6, to=3-6]
	\arrow["{H(f',id)}"{description}, from=2-1, to=1-4]
	\arrow["{H(id,q)}"', from=2-1, to=2-3]
	\arrow["{H(p_1,id)}"{description}, two heads, from=2-1, to=4-1]
	\arrow["{(1)}"{description}, draw=none, from=2-1, to=4-3]
	\arrow["{H(f',id)}"'{pos=0.8}, from=2-3, to=1-6]
	\arrow["{H(p_1,id)}"{description, pos=0.3}, two heads, from=2-3, to=4-3]
	\arrow["{H(id,q)}", from=3-4, to=3-6]
	\arrow["{H(f,id)}"{description, pos=0.4}, from=4-1, to=3-4]
	\arrow["{H(id,q)}"', from=4-1, to=4-3]
	\arrow["{H(f,id)}"'{pos=0.3}, from=4-3, to=3-6]
\end{tikzcd}\]
By Bim(ii), faces (1) and (2) are relative length-$1$ display maps. Then from the arrows $\gap(a_1,b):H(a_1,b) \rightarrow x_1$ and $\gap(a_2,b):H(a_2,b) \rightarrow x_2$ we obtain a commutative diagram
\[\begin{tikzcd}[ampersand replacement=\&]
	\&\&\& {H(a_2,b)} \\
	{H(a_1,b)} \&\&\&\& {x_2} \&\& {H(a_2,\partial b)} \\
	\& {x_1} \&\& {H(a_1,\partial b)} \\
	\&\&\&\& {H(\partial a_2,b)} \&\& {H(\partial a_2,\partial b)} \\
	\& {H(\partial a_1,b)} \&\& {H(\partial a_1,\partial b)}
	\arrow["{\gap(a_2,b)}"{description}, two heads, from=1-4, to=2-5]
	\arrow["{H(id,q)}", curve={height=-12pt}, from=1-4, to=2-7]
	\arrow["{H(p_2,id)}"{description, pos=0.2}, curve={height=12pt}, two heads, from=1-4, to=4-5]
	\arrow[""{name=0, anchor=center, inner sep=0}, "{H(f',id)}"{description}, from=2-1, to=1-4]
	\arrow["{\gap(a_1,b)}"{description}, two heads, from=2-1, to=3-2]
	\arrow["{H(id,q)}"', curve={height=-12pt}, from=2-1, to=3-4]
	\arrow["{H(p_1,id)}"{description}, curve={height=12pt}, two heads, from=2-1, to=5-2]
	\arrow[from=2-5, to=2-7]
	\arrow[two heads, from=2-5, to=4-5]
	\arrow["{(2')}"{description}, draw=none, from=2-5, to=4-7]
	\arrow["{H(p_2,id)}"{description}, two heads, from=2-7, to=4-7]
	\arrow[""{name=1, anchor=center, inner sep=0}, from=3-2, to=2-5]
	\arrow[from=3-2, to=3-4]
	\arrow[two heads, from=3-2, to=5-2]
	\arrow["{(1')}"{description}, draw=none, from=3-2, to=5-4]
	\arrow["{H(f',id)}"'{pos=0.8}, from=3-4, to=2-7]
	\arrow["{H(p_1,id)}"{description, pos=0.3}, two heads, from=3-4, to=5-4]
	\arrow["{H(id,q)}", from=4-5, to=4-7]
	\arrow["{H(f,id)}"{description, pos=0.4}, from=5-2, to=4-5]
	\arrow["{H(id,q)}"', from=5-2, to=5-4]
	\arrow["{H(f,id)}"'{pos=0.3}, from=5-4, to=4-7]
	\arrow["{(3)}"{description}, draw=none, from=0, to=1]
\end{tikzcd}\]
where (1') and (2') are distinguished. By the induction hypothesis, the face with vertices $H(a_1,\partial b)$, $H(a_2,\partial b)$, $H(\partial a_1,\partial b)$, $H(\partial a_2,\partial b)$ is a distinguished square, hence by the pasting lemma for distinguished pullbacks, the square containing $x_1$, $x_2$, $H(\partial a_1, b)$, $H(\partial a_2,b)$ is also distinguished. By Bim(iii), square (3) is distinguished, so again by the pasting lemma, so is the square
\[
\tag{\texttt{**}}
\dsqua{H(a_1,b)}{H(a_2,b)}{H(\partial a_1,b)}{H(\partial a_1,b).}{H(f',id)}{H(f,id)}{H(p_1,id)}{H(p_2,id)}
\]
This concludes the induction step.

Conversely, suppose that $H(-,b)$ preserves distinguished squares for all $b \in \mathcal B$. To verify condition Bim(iii), suppose given a length-$1$ display map $q:b \rightarrow \partial b$ in $\mathcal B$ and a length-$1$ distinguished square in $\mathcal A$ of the form (\texttt{*}). Then Bim(ii) allows us to construct diagrams of the same form as the ones above. An analogous application of the pasting lemma -- now using that $H(-,b)$ preserves distinguished pullbacks -- shows that the square with vertices $x_1$, $x_2$, $H(\partial a_1,b)$, $H(\partial a_2,b)$ is distinguished. On the other hand, as $H(-,b')$ preserves distinguished pullbacks, the square of the form (\texttt{**}) above is distinguished. It now follows from the pasting lemma that so is (3). Thus Bim(iii) holds, as required.
\end{proof}

\subsection{A syntactic example}
\label{subsec: syntactic example}

We will now present one of the key results towards comparing the monoidal structure from the present text with the tensor product from \cite{Alm25}.

\begin{proposition}
\label{prop: comparison functor is bimorphism}
Let $\bbA$ and $\bbB$ be generalized algebraic theories. The functor
$$
\otimes_{\bbA,\bbB}:\mathcal C(\bbA) \times \mathcal C(\bbB) \longrightarrow \mathcal C(\bbA \otimes \bbB)
$$
from \cite{Alm25} is a bimorphism from $(\mathcal C(\bbA),\mathcal C(\bbB))$ to $\mathcal C(\bbA \otimes \bbB)$.
\end{proposition}

\begin{proof}
Following Proposition \ref{prop: characterization of bimorphisms}, let us verify that the functor $\otimes_{\bbA,\bbB}$, which will be denoted simply by $\otimes$, satisfies conditions Bim(i'), Bim(ii) and Bim(iv).

\vspace{0.5em}

For Bim(i'), consider a context $\textbf{Y} = (y_1:Y_1, ..., y_n:Y_n)$ in $\bbB$. Let us prove that $- \otimes [\textbf{Y}]:\mathcal C(\bbA) \rightarrow \mathcal C(\bbA \otimes \bbB)$ is a flexible morphism. Firstly, note that it preserves the distinguished terminal object:
$$
1_{\mathcal C(\bbA)} \otimes [\textbf{Y}] = [\varnothing] \otimes [\textbf{Y}] = [\varnothing \otimes \textbf{Y}] = [\varnothing] = 1_{\mathcal C(\bbA \otimes \bbB)}.
$$	
Now, suppose given a length-$1$ distinguished square in $\mathcal C(\bbA)$, say corresponding to a substitution diagram of context morphisms
	\[
	\squa{\textbf{X}'}{\textbf{X}}{\partial \textbf{X}'}{\partial \textbf{X}}{\textbf{f}'}{\textbf{f}}{\textbf{p}'}{\textbf{p}}
	\]
	where
	$$
	\textbf{X} = (x_1:X_1, ..., x_{m+1}:X_{m+1}), \qquad \textbf{X}' = (x'_1:X'_1, ...,x'_n:X'_n, x'_{m'+1}:X_{m+1}[\textbf{f}]),
	$$
	$$
	\textbf{f} = (f_1, ..., f_m), \quad \textbf{f}' = (f_1, ..., f_m,x'_{m'+1}),\quad \textbf{p} = (x_1, ..., x_m), \quad \textbf{p}' = (x'_1, ..., x'_{m'}).
	$$
	It is mapped by $- \otimes [\textbf{Y}]$ to
	\[
	\tag{$\texttt{*}$}
	\squa{\textbf{X}' \otimes \textbf{Y}}{\textbf{X}\otimes \textbf{Y}}{\partial \textbf{X}'\otimes \textbf{Y}}{\partial \textbf{X}\otimes \textbf{Y},}{\textbf{f}'\otimes \textbf{Y}}{\textbf{f}\otimes \textbf{Y}}{\textbf{p}'\otimes \textbf{Y}}{\textbf{p}\otimes \textbf{Y}}
	\]
	which is given explicitly by
	\begin{align*}
		\partial\textbf{X} \otimes \textbf{Y} & = (x_iy_j:X_i \otimes Y_j)_{\substack{i \le m \\ j \le n}}\\
		\partial\textbf{X}' \otimes \textbf{Y} & = (x'_iy_j:X'_i \otimes Y_j)_{\substack{i \le m' \\ j \le n}}\\
		\textbf{X} \otimes \textbf{Y} & = (\partial \textbf{X} \otimes \textbf{Y}) * (x_{m+1}y_1:X_{m+1} \otimes Y_1, ..., x_{m+1}y_n:X_{m+1} \otimes Y_n)\\
		\textbf{X}' \otimes \textbf{Y} &= (\partial \textbf{X}' \otimes \textbf{Y}) * (x'_{m'+1}y_1:X_{m+1}[\textbf{f}] \otimes Y_1, ..., x'_{m'+1}y_n:X_{m+1}[\textbf{f}] \otimes Y_n)\\
		\textbf{f} \otimes \textbf{Y} &= (f_i \otimes y_j)_{\substack{i \le m \\ j \le n}}\\
		\textbf{f}' \otimes \textbf{Y} &= (\textbf{f} \otimes \textbf{Y}) * (x'_{m'+1}y_1, ..., x'_{m'+1}y_n)\\
		\textbf{p} \otimes \textbf{Y} &= (x_iy_j)_{\substack{i \le m \\ j \le n}}\\
		\textbf{p}' \otimes \textbf{Y} &= (x'_iy_j)_{\substack{i \le m' \\ j \le n}}
	\end{align*}
	where, by abuse of notation, a family indexed by a set of the form $\{1, ..., a\} \times \{1, ..., b\}$ indicates the one indexed by $\{1, ..., ab\}$ induced by the lexicographic order on the former set, and $*$ denotes concatenation of finite sequences. It is immediate from this description that $\textbf{p} \otimes \textbf{Y}$ and $\textbf{p}' \otimes \textbf{Y}$ are (length-$n$) display maps. Moreover, the two-sided substitution property discussed in \cite{Alm25} implies that for $j = 1$, ..., $n$, the sort equality
	$$
	X_{m+1}[\textbf{f}] \otimes Y_j \equiv (X_{m+1} \otimes Y_j)[\textbf{f} \otimes \textbf{Y}]
	$$
	is derivable in context $(\partial \textbf{X}' \otimes \textbf{Y}) * (x'_{m'+1}y_1:X_{m+1}[\textbf{f}] \otimes Y_1, ..., x'_{m'+1}y_{j-1}:X_{m+1}[\textbf{f}] \otimes Y_{j-1})$. It follows that $\textbf{X}' \otimes \textbf{Y}$ is provably equal to
	$$
	(\partial \textbf{X}' \otimes \textbf{Y}) * (x'_{m+1}y_1:(X_{m+1} \otimes Y_1)[\textbf{f} \otimes \textbf{Y}], ..., x'_{m+1}y_n:(X_{m+1} \otimes Y_n)[\textbf{f} \otimes \textbf{Y}]).
	$$
	From this and the above description of $\textbf{f}' \otimes \textbf{Y}$ we conclude that the square in $\mathcal C(\bbA \otimes \bbB)$ associated with ($\texttt{*}$) is distinguished. This also shows, in particular, that $- \otimes [\textbf{Y}]$ preserves display maps. Therefore $- \otimes [\textbf{Y}]$ is a flexible morphism.
	
	If $\textbf{Y}$ is the empty context, it is clear that $- \otimes [\textbf{Y}]$ is constant with value $1_{\mathcal C(\bbA \otimes \bbB)}$. This concludes the proof of Bim(i').
	
	\vspace{0.5em}
	
	For Bim(ii), we must prove that for contexts $\textbf{X} = (x_1:X_1, ..., x_{m+1}:X_{m+1})$ in $\bbA$ and $\textbf{Y} = (y_1:Y_1, ..., y_{n+1}:Y_{n+1})$ in $\bbB$, the diagram of contexts
	\[
	\tag{$\texttt{**}$}
	\squa{\textbf{X} \otimes \textbf{Y}}{\textbf{X} \otimes \partial \textbf{Y}}{\partial \textbf{X} \otimes \textbf{Y}}{\partial \textbf{X} \otimes \partial \textbf{Y}}{(x_iy_j)_{\substack{i \le m+1 \\ j \le n}}}{(x_iy_j)_{\substack{i \le m \\ j \le n}}}{(x_iy_j)_{\substack{i \le m \\ j \le n+1}}}{(x_iy_j)_{\substack{i \le m \\ j \le n}}}
	\]
	corresponds to a relative length-$1$ display map in $\mathcal C(\bbA \otimes \bbB)$. Observe that substitution of $\textbf{X} \otimes \partial \textbf{Y} \rightarrow \partial \textbf{X} \otimes \partial \textbf{Y}$ along $\partial \textbf{X} \otimes \textbf{Y} \rightarrow \partial \textbf{X} \otimes \partial \textbf{Y}$ yields
	\[
	\squa{\textbf{Z}}{\textbf{X} \otimes \partial \textbf{Y}}{\partial \textbf{X} \otimes \textbf{Y}}{\partial \textbf{X} \otimes \partial \textbf{Y}}{(x_iy_j)_{\substack{i \le m+1 \\ j \le n}}}{(x_iy_j)_{\substack{i \le m \\ j \le n}}}{(x_iy_j)_{\substack{i \le m \\ j \le n+1}}}{(x_iy_j)_{\substack{i \le m \\ j \le n}}}
	\]
	where $\textbf{Z} = (\partial \textbf{X} \otimes \textbf{Y}) * (x_{m+1}y_1:Y_{m+1} \otimes Y_1, ..., x_{m+1}y_n:Y_{m+1} \otimes Y_n)$. Hence the gap map of ($\texttt{**}$) in $\mathcal C(\bbA \otimes \bbB)$ is
	$$
	[(x_iy_j)_{(i,j) \in \{1, ..., m+1\} \times \{1, ..., n+1\} \smallsetminus (m+1,n+1)}]: [\textbf{X} \otimes \textbf{Y}] \longrightarrow [\textbf{Z}],
	$$
	which is clearly a length-$1$ display map.
	
	\vspace{0.5em}
	
	Finally, let us check Bim(iv). Consider a length-$1$ display map in $\mathcal C(\bbA)$ and a length-$1$ distinguished square in $\mathcal C(\bbB)$, say corresponding, respectively, to
	$$
	\textbf{p}:\textbf{X} \rightarrow \partial \textbf{X} \qquad \text{and} \qquad
	\squa{\textbf{Y}'}{\textbf{Y}}{\partial \textbf{Y}'}{\partial \textbf{Y}}{\textbf{f}'}{\textbf{f}}{\textbf{q}'}{\textbf{q}}
	$$
	where
	$$
	\textbf{X} = (x_1:X_1, ..., x_{m+1}:X_{m+1}), \qquad \textbf{p} = (x_1, ..., x_m),
	$$
	$$
	\textbf{Y} = (y_1:Y_1, ..., y_{n+1}:Y_{n+1}), \qquad \textbf{Y}' = (y'_1:Y'_1, ..., y'_n:Y'_n,y'_{n'+1}:Y_{n+1}[\textbf{f}]),
	$$
	$$
	\textbf{f} = (f_1, ..., f_n), \quad \textbf{f}' = (f_1, ..., f_n,y'_{n'+1}), \quad \textbf{q} = (y_1, ..., y_n), \quad \textbf{q}' = (y'_1, ..., y'_{n'}).
	$$
	From $\textbf{p}$ and $\textbf{q}$ we obtain, as in the above proof of Bim(ii), the gap map
	$$
	(x_iy_j)_{(i,j) \in \{1, ..., m+1\} \times \{1, ..., n+1\} \smallsetminus (m+1,n+1)}: \textbf{X} \otimes \textbf{Y} \longrightarrow (\partial \textbf{X} \otimes \textbf{Y}) * (x_{m+1}y_j:X_{m+1} \otimes Y_j)_{j \le n} = \partial(\textbf{X} \otimes \textbf{Y}),
	$$
	and from $\textbf{p}$ and $\textbf{q}'$ we obtain the gap map
	$$
	(x_iy'_j)_{(i,j) \in \{1, ..., m+1\} \times \{1, ..., n'+1\} \smallsetminus (m+1,n'+1)}: \textbf{X} \otimes \textbf{Y}' \longrightarrow (\partial \textbf{X} \otimes \textbf{Y}') * (x_{m+1}y_j:X_{m+1} \otimes Y'_j)_{j \le n'} = \partial(\textbf{X} \otimes \textbf{Y}').
	$$
	Then, noting that $	(\partial \textbf{X} \otimes \textbf{f}') * (x_{m+1} \otimes f_1, ..., x_{m+1} \otimes f_n) = \partial(\textbf{X} \otimes \textbf{f}')$, the diagram from the statement of Bim(iv) is the equivalence class of
	$$
	\squa{\textbf{X} \otimes \textbf{Y}'}{\textbf{X} \otimes \textbf{Y}}{\partial(\textbf{X} \otimes \textbf{Y}')}{\partial(\textbf{X} \otimes \textbf{Y}).}{\textbf{X} \otimes \textbf{f}'}{\partial(\textbf{X} \otimes \textbf{f}')}{(x_iy'_j)_{(i,j) \neq (m+1,n'+1)}}{(x_iy_j)_{(i,j) \neq (m+1,n+1)}}
	$$
	Now, by the two-sided substitution property from \cite{Alm25} we can derive the sort equality
	$$
	X_{m+1} \otimes Y_{n+1}[\textbf{f}] \equiv (X_{m+1} \otimes Y_{n+1})[\partial(\textbf{X} \otimes \textbf{f}')]
	$$
	in context $\partial (\textbf{X} \otimes \textbf{Y}')$, which implies that $\textbf{X} \otimes \textbf{Y}'$ is provably equal to $\partial(\textbf{X} \otimes \textbf{Y}') * (x_{m+1}y'_{n'+1}:(X_{m+1} \otimes Y_{n+1})[\partial(\textbf{X} \otimes \textbf{f}')])$. It follows from this (also observing that the last entry of $\textbf{X} \otimes \textbf{f}'$ is $x_{m+1}y'_{n'+1}$) that the above diagram is a distinguished square in $\mathcal C(\bbA \otimes \bbC)$.
	
	\vspace{0.5em}
	
	Having verified Bim(i'), Bim(ii) and Bim(iv), we conclude that $\otimes$ is a bimorphism from $(\mathcal C(\bbA),\mathcal C(\bbB))$ to $\mathcal C(\bbA \otimes \bbB)$.
\end{proof}

\section{Poset-shaped cellular diagrams}

We will now study what we will call \emph{cellular diagrams} from a locally finite poset to a contextual category. As mentioned in the introduction, we expect these to correspond, essentially, to the Reedy diagrams from \cite{KapLum21} in the special case where the domain category is a poset and the codomain \textsc{cwa} is a contextual category.

Poset-shaped cellular diagrams $(P,\le) \rightarrow \mathcal C$ will come in two variants, each one requiring a kind of additional structure on $P$. Firstly, for finite $P$, we will introduce (Definition \ref{def: finite poset-shaped cellular diagram}) cellular diagrams $(P,\le,\tle) \rightarrow \mathcal C$ where $\tle$ is a linear refinement of $\le$. Then, for locally finite $P$, we will define cellular diagrams $(P,\le,\tle) \rightarrow \mathcal C$ where $\tle$ is a \emph{local linearization} (Definition \ref{def: locally finite poset, local linearization}): we only compare elements that have a common upper bound with respect to $\le$.\footnote{When $P$ is finite, the two concepts do not coincide but can be used interchangeably. See Remark \ref{rem: cellular diagram from locally linearized finite poset}.}

\begin{definition}
\label{def: linearization}
Let $(P,\le)$ be a finite poset. We define a \emph{linearization} of $(P,\le)$ as a linear order $\tle$ on $P$ that refines $\le$: for all $a$, $b \in P$, if $a \le b$, then $a \tle b$.
	
A \emph{linearized finite poset} is a triple $\textbf{P} = (P,\le,\tle)$ where $(P,\le)$ is a finite poset and $\tle$ is a linearization of $(P,\le)$.
\end{definition}

\begin{definition}
A \emph{sieve} in a poset $(P,\le)$ is a downwards closed subset, i.e. $S \subset P$ is a sieve if for any $a \in P$, $b \in S$, if $a \le b$, then $a \in S$. The set of sieves in $(P,\le)$ will be denoted by $\Sieve(P,\le)$ or simply by $\Sieve(P)$.

A poset morphism $F:(P,\le) \rightarrow (P',\le')$ is said to be a \emph{sieve embedding} if $F(P)$ is a sieve in $P'$ and $F$ defines a poset isomorphism $P \cong F(P)$. The latter condition means that (i) $F$ is injective, and (ii) if $F(a) \le' F(b)$, then $a \le b$.

\vspace{0.5em}

A sieve in a linearly ordered set will be referred to as an \emph{initial segment}, and a sieve embedding between linearly ordered sets is referred to as an \emph{initial segment embedding}.
	
For a linearized finite poset $(P,\le,\tle)$, the notation $\Sieve(P)$ will be used for $\Sieve(P,\le)$. On the other hand, by an initial segment of $\textbf{P}$ we mean an initial segment of $(P,\tle)$; in a similar way we talk about initial segment embeddings between linearized finite posets.

When this causes no confusion, we denote $(P,\le,\tle)$ simply by $P$.
\end{definition}

\begin{definition}
	Let $(P,\le)$ be a poset that has a top element $\top$. Its \emph{boundary}, denoted by $\partial(P,\le)$ or simply by $\partial P$, is the sieve $P \smallsetminus \{\top\}$.
	
	Given a linearization $\tle$ of $(P,\le)$, we will also write $\partial P$ for the linearized poset given by endowing $P \smallsetminus \{\top\}$ with the restriction of $\tle$. In this case, note that the inclusion $\partial P \rightarrow P$ is both a sieve embedding (with respect to $\le$) and an initial segment embedding (with respect to $\tle$). When no further description is made, $\partial P \rightarrow P$ will denote the inclusion map.
\end{definition}

\begin{notation}
	Given a poset $(P,\le)$ and $x \in P$, we let
	$$
	x^\le  \coloneqq \{y \in P \mid y \le x\}, \qquad \partial x \coloneqq \partial(x^\le) = \{y \in P \mid y < x\}.
	$$
\end{notation}

\subsection{Defining cellular diagrams}

In what follows we let $\mathcal C$ be a contextual category.

\begin{definition}
	\label{def: finite poset-shaped cellular diagram}
	Given a linearized finite poset $\textbf{P} = (P,\le,\tle)$, we define a \emph{cellular diagram} of shape $\textbf{P}$ in $\mathcal C$ as a functor
	$$
	F:(P,\le)^{\text{op}} \longrightarrow \mathcal C
	$$
	such that one can choose\footnote{Although we don't take such a choice as part of the structure, we will see in Proposition \ref{prop: properties cellular diagrams} that it is uniquely determined.} for each sieve $X \subset P$ a limit cone of the form
	$$
	\varphi^X:a^X \Longrightarrow F|_X
	$$
	(where $a^X$ is an object of $\mathcal C$, and we use the same notation for the corresponding constant diagram) in a way that the following hold:
	\begin{enumerate}[label=(\roman*)]
		\item $a^\varnothing = 1_\mathcal C$ (with $\varphi^\varnothing$ the empty cone).
		
		\item For each $x \in P$, $a^{x^\le} = F(x)$ and $\varphi^{x^\le}$ is the cone whose $y$-component is the image of the unique arrow $y \rightarrow x$ under $F$.
		
		\item Suppose given sieves $X \subset Y$ such that $Y \smallsetminus X$ is a singleton $\{y\}$ with $x \tle y$ for all $x \in X$. Then the induced map between limits
		$$
		a^Y \longrightarrow a^X
		$$
		is a length-$1$ display map.
		
		\item Consider a diagram of sieve inclusions
		\[
		\squa{X}{X'}{Y}{Y'}{}{}{}{}
		\]
		such that $Y \smallsetminus X = \{y\} = Y' \smallsetminus X'$ with $x \tle y$ for all $x \in X'$. Then the induced diagram
		\[
		\dsqua{a^{Y'}}{a^{Y}}{a^{X'}}{a^X}{}{}{}{}
		\]
		is a distinguished square.
	\end{enumerate}
	
	We will say that a family of cones as above \emph{realizes $F$ as a cellular diagram} (of shape $\textbf{P}$).
\end{definition}

\begin{proposition}
	\label{prop: properties cellular diagrams}
	Consider a linearized finite poset $\textbf{P} = (P,\le,\tle)$, a contextual category $\mathcal C$, and a functor $F:{(P,\le)^{\text{op}}} \rightarrow \mathcal C$. Then
	\begin{enumerate}[label=(\alph*)]
		\item If $F$ is a cellular diagram of shape $\textbf{P}$, then for every sieve $Y \subset P$ we have that $F|_Y:(Y,\le|_Y)^{\text{op}} \rightarrow C$ is a cellular diagram of shape $Y$ (endowed with the restrictions of $\le$ and $\tle$). In fact, for a family $(a^X,\varphi^X)_{X \in \Sieve(P)}$ realizing $F$ as a cellular diagram, $(a^X,\varphi^X)_{X \in \Sieve(Y)}$ realizes $F|_Y$ as a cellular diagram.
		
		\item There exists at most one family of cones $(a^X,\varphi^X)_{X \in \Sieve(P)}$ realizing $F$ as a cellular diagram.
		
		\item Suppose that $P \neq \varnothing$. Denoting by $\top$ the top element of $(P,\tle)$ and recalling that $\partial P = P \smallsetminus \{\top\}$, assume that $F|_{\partial P}:\partial P^{\text{op}} \rightarrow \mathcal C$ is a cellular diagram, say whose corresponding family of cones is $(a^X,\varphi^X)_{X \in \Sieve(P \smallsetminus \{\top\})}$. Then $F$ is cellular if and only if the induced map $F(\top) \rightarrow a^{\partial \top}$ is a length-$1$ display map.
		
		(Hence extending a cellular diagram of shape $P \smallsetminus \{\top\}$ to one of shape $P$ amounts to choosing a length-$1$ display map with codomain $a^{\partial \top}$.)
	\end{enumerate}
\end{proposition}

\begin{proof}
	\leavevmode
	\begin{enumerate}[label=(\alph*)]
		\item Conditions (i)-(iv) from Definition \ref{def: finite poset-shaped cellular diagram} are immediate.
		
		\item We prove by induction on $n \ge 0$ that the claim holds whenever $P$ has cardinality $n$. For $n = 0$, it holds by condition (i) in Definition \ref{def: finite poset-shaped cellular diagram}. Let $n \ge 1$ and suppose that the claim holds for $0$, ..., $n-1$. Let $F:(P,\le)^{\text{op}} \rightarrow C$ be a cellular diagram with $\sharp P = n$ and realized by families
		$$
		(a^X,\varphi^X)_{X \in \Sieve(P)},
		$$
		$$
		(b^X,\psi^X)_{X \in \Sieve(P)}.
		$$
		By (a) and the induction hypothesis, we have $a^X = b^X$ and $\varphi^X = \psi^X$ whenever $X$ does not contain the top element $\top$ of $(P,\tle)$. On the other hand, if $\top \in X$, consider the diagram of sieve inclusions
		\[
		\squa{\partial \top}{\partial X = X \smallsetminus\{\top\}}{\top^\le}{X.}{}{}{}{}
		\]
		By condition (iv) from Definition \ref{def: finite poset-shaped cellular diagram}, we have distinguished squares
		\[
		\dsqua{a^X}{a^{\top^\le} = F(\top)}{a^{\partial X}}{a^{\partial \top},}{}{}{}{}
		\]
		\[
		\dsqua{b^X}{b^{\top^\le} = F(\top)}{b^{\partial X}}{b^{\partial \top}.}{}{}{}{}
		\]
		As $a$, $b$ and $\varphi$, $\psi$ coincide for sieves in $\partial P$, the bottom arrows are equal. By the same remark and condition (ii) from Definition \ref{def: finite poset-shaped cellular diagram}, the right vertical arrows are equal. It follows that the left vertical arrows are also equal. Now, the components of $\varphi^X$ are the top arrow $a^X \longrightarrow F(\top)$ and the composites
		$$
		a^X \twoheadrightarrow a^{\partial X} \longrightarrow F(x)
		$$
		for $x \in \partial X$. By comparing this with the analogous description of $\psi^X$ we conclude that $\varphi^X = \psi^X$.
		
		\item We define a family $(b^X,\psi^X)_{X \in \Sieve(P)}$ following the same strategy used in (b). For $X$ not containing $\top$, we let $b^X = a^X$ and $\psi^X = \varphi^X$. For $X$ containing $\top$, we let $b^X$ be as in the distinguished square
		\[
		\tag{\texttt{*}}
		\dsqua{b^X}{F(\top)}{a^{\partial X}}{a^{\partial \top}}{}{}{}{}
		\]
		In this case, we let $\psi^X$ be the cone whose components are $b^X \rightarrow F(\top)$ and the composites
		$$
		b^X \twoheadrightarrow a^{\partial X} \longrightarrow F(x)
		$$
		for $x \in \partial X$ (the fact that these arrows form a cone follows from commutativity of (\texttt{*})). A routine calculation shows that $\psi^X$ is a limit cone from $b^X$ to $F$. It remains to verify conditions (ii)-(iv) from Definition \ref{def: finite poset-shaped cellular diagram}.
		
		For $x \neq \top$, (ii) follows from the corresponding condition for $F|_{\partial P}$. For $\top$, it follows from the construction of $b^{\top^\le}$ and $\psi^{\top^\le}$, which is given by pulling back along $a^{\partial \top} \overset{id}{\rightarrow} a^{\partial \top}$.
		
		For (iii), consider sieves $X \subset Y$ such that $Y \smallsetminus X = \{y\}$ with $x \tle y$ for all $x \in X$. Note that $X \subset \partial P$. If $y \neq \top$, the condition follows from the corresponding property of $F|_{\partial P}$. For $y = \top$, note that the comparison map $b^Y \rightarrow b^X = a^X$ is the left arrow in (\texttt{*}) with $X$ replaced by $Y$, hence a length-$1$ display map.

		For (iv), suppose given sieve inclusions
		\[
		\squa{X}{X'}{Y}{Y'}{}{}{}{}
		\]
		such that $Y \smallsetminus X = \{y\} = Y' \smallsetminus X'$ with $x \tle y$ for all $x \in X'$. If $y \neq \top$, the assertion follows from the corresponding one for $F|_{\partial P}$. On the other hand, for $y = \top$ we consider the diagram
		\[\begin{tikzcd}[ampersand replacement=\&]
			{a^{Y'}} \& {a^Y} \& {F(\top)} \\
			{a^{X'}} \& {a^X} \& {a^{\partial \top}.}
			\arrow[from=2-1, to=2-2]
			\arrow[from=1-1, to=1-2]
			\arrow[from=1-2, to=1-3]
			\arrow[from=2-2, to=2-3]
			\arrow[two heads, from=1-1, to=2-1]
			\arrow[two heads, from=1-2, to=2-2]
			\arrow[two heads, from=1-3, to=2-3]
		\end{tikzcd}\]
		Since the right square and the composite one are distinguished, so is the left one, as required.
	\end{enumerate}
\end{proof}

\begin{definition}
\label{def: distinguished limit}	
	
	Let $\mathcal C$ be a contextual category, $\textbf{P} = (P,\le,\tle)$ a linearized finite poset, and $F:(P,\le)^{\text{op}} \rightarrow \mathcal C$ a cellular diagram. For $(a^X,\varphi^X)_{X \in \Sieve(P)}$ the family as in Definition \ref{def: finite poset-shaped cellular diagram} and Proposition \ref{prop: properties cellular diagrams}, $a^P$ will be called the \emph{distinguished limit} of $F$. We will denote it by $F[\textbf{P}]$, or by $F[P]$ when no ambiguity is caused.
	
	For a sieve $X \subset P$, which we view as a linearized poset via the restrictions of $\le$ and $\tle$, we abbreviate $F|_X[X]$ as $F[X]$. We write $F[X]_\tle$ when dealing with more than one linearization of $(P,\le)$.

	Although we do not fix a notation for the natural transformation $\varphi^X$, we usually denote it by $\pi$ or variants (such as $\pi^{F,(P,\le,\tle)}$, or $\pi'$ when different cellular diagrams are being considered).
\end{definition}

\begin{example}
	\label{ex: cellular diagram to N}
	The terminal contextual category $\mathcal Z$ is the co-discrete category whose objects are the natural numbers (that is, there is a unique arrow $m \rightarrow n$ for each $m$, $n \ge 0$), with length function the identity $\mathbb N \rightarrow \mathbb N$, and display maps and distinguished squares given in the unique possible way (for instance, an arrow $m \rightarrow n$ is a display map if and only if $m \ge n$).
	
	Let $\textbf{P} = (P,\le,\tle)$ be a linearized finite poset. By Proposition \ref{prop: properties cellular diagrams}(c), there exists a unique cellular diagram $F$ of shape $\textbf{P}$ in $\mathcal Z$. Induction on the cardinality of $P$ shows that for each $x \in P$, $F(x)$ is the cardinality of $x^\le$; more generally, for a sieve $X \subset P$, $F[X]$ is the cardinality of $X$.
\end{example}

\begin{definition}
\label{def: strong sieve embedding, set of cellular diagrams}	
	
Given linearized finite posets $(P,\le,\tle)$, $(P',\le',\tle')$, and a sieve embedding $F:(P,\le) \rightarrow (P',\le')$, we will say that $F$ is a \emph{strong sieve embedding} from $(P,\le,\tle)$ to $(P',\le',\tle')$ if $\tle$ is the restriction of $\tle'$ along $F$. Equivalently, $F$ is monotone with respect to $\tle$ and $\tle'$.

We will denote by $\LFPos$ the category whose objects are the linearized finite posets and whose morphisms are the strong sieve embeddings.

Let $\mathcal C$ be a contextual category. For $\textbf{P} \in \LFPos$ we will write $\Cell(\textbf{P},\mathcal C)$ for the set of cellular diagrams of shape $\textbf{P}$ in $\mathcal C$ (Definition \ref{def: finite poset-shaped cellular diagram}).

Since restricting a cellular diagram along a strong sieve embedding defines a cellular diagram, we have a functor
$$
\Cell(-,\mathcal C):\LFPos^{\text{op}} \longrightarrow \Set.
$$
\end{definition}

\subsection{Some operations on cellular diagrams}

Let $\textbf{P} = (P,\le,\tle)$ be a linearized finite poset, $\mathcal C$ a contextual category, and
$$
F:(P,\le)^{\text{op}} \rightarrow \mathcal C
$$
a cellular diagram of shape $\textbf{P}$.

We will now describe a procedure that, given another linearization $\tle'$ of $(P,\le)$, yields a cellular diagram $F':{(P,\le)^{\text{op}}} \rightarrow \mathcal C$ of shape $\textbf{P}'=(P,\le,\tle')$ endowed with an isomorphism $F' \cong F$.

\begin{proposition}
	\label{prop: transporting cellular diagrams}
	In the above setting, there exists a unique pair $(F',\varphi)$ consisting of a cellular diagram $F':{(P,\le)^{\text{op}}} \rightarrow \mathcal C$ of shape $\textbf{P}' = (P,\le,\tle')$ and a natural transformation $\varphi:F' \Rightarrow F$ with the following property: for every $x \in P$, the square
	\[
	\tag{\texttt{*}}
	\dsqua{F'(x) = F'[x^\le]_{\tle'}}{F(x) = F[x^\le]_\tle}{F'[\partial x]_{\tle'}}{F[\partial x]_\tle,}{}{}{}{}
	\]
	where the horizontal arrows are induced by $\varphi$, is distinguished. Also, $\varphi$ is an isomorphism.
\end{proposition}

\begin{proof}
	Write $x_1$, ..., $x_n$ for the elements of $P$ ordered by $\tle'$, let $P_i=\{x_1, ..., x_i\}$ for $0 \le i \le n$ (which is a sieve), and let $F_i$ be the cellular diagram obtained by restricting $F$ to $(P_i,\le|_{P_i},\tle|_{P_i})$. We will prove by induction that for $0 \le i \le n$, the claim in the statement holds for $F_i$ and $\tle'|_{P_i}$.
	
	The claim holds for $i = 0$ as there is a unique cellular diagram indexed by $\varnothing$. Now, let $1 \le i \le n$ and suppose that it holds for $i-1$. Let $(G,\gamma)$ be the unique pair consisting of a cellular diagram $G$ of shape $(P_{i-1},\le|_{P_{i-1}},\tle'|_{P_{i-1}})$ and a natural transformation $\gamma:G \Rightarrow F_{i-1}$ such that the square analogous to (\texttt{*}) is distinguished for all $x \in P_{i-1}$. Consider the distinguished square
	\[
	\tag{\texttt{**}}
	\dsqua{a}{F(x_i) = F[x_i^\le]_\tle}{G[\partial x_i]_{\tle'}}{F[\partial x_i]_{\tle}}{}{}{p'}{p}
	\]
	As $p$ is a length-$1$ display map (Definition \ref{def: finite poset-shaped cellular diagram}(iii)), so is $p'$. By Proposition \ref{prop: properties cellular diagrams}(c), we can extend $G$ to a cellular diagram $G'$ of shape $(P_i,\le|_{P_i},\tle'|_{P_i})$ by setting $G'(x_i) = a$ and, for each $y \le x_i$, mapping the unique arrow $y \rightarrow x_i$ to the composite
	$$
	a \overset{p'}{\twoheadrightarrow} G[\partial x_i]_{\tle'} \longrightarrow G(y).
	$$
	It then follows that $\gamma$ extends to a natural transformation $\gamma':G' \Rightarrow F_i$ by taking $\gamma'_{x_i}$ as the top arrow in (\texttt{\texttt{**}}). It follows, in particular, that $(G',\gamma')$ is the unique pair satisfying the condition in the statement for $F_i$ and $\tle'|_{P_i}$.
	
	Hence the statement holds for $F_n = F$ and $\tle'|_{P_n} = \tle'$, as required.
	
	Let $(F',\varphi)$ be the pair so obtained. An induction using that (\texttt{*}) is a distinguished square shows that $\varphi_{x_i}$ is an isomorphism for $1 \le i \le n$. The base case $i = 1$ follows from the canonical map $F'[\varnothing]_{\tle'} \rightarrow F[\varnothing]_{\tle}$ being $id_{1_\mathcal C}$. The induction step follows from $F'[\partial x_i]_{\tle'} \rightarrow F[\partial x_i]_{\tle}$ being induced by the isomorphisms $\varphi_y$ for $y < x_i$.
\end{proof}

\begin{remark}
	\label{rem: naturality of comparison between slims}
	In the setting of Proposition \ref{prop: transporting cellular diagrams}, given sieves $X \subset Y$, the diagram
	\[
	\squa{F'[Y]_{\tle'}}{F[Y]_{\tle}}{F'[X]_{\tle'}}{F[X]_{\tle}}{}{}{}{}
	\]
	commutes, where the arrows are given by functoriality of limits, i.e. the vertical ones are obtained by restricting the indexing category and the horizontal ones are induced by $\varphi$.
\end{remark}

As a consequence we obtain the following variant of condition (iv) from Definition \ref{def: finite poset-shaped cellular diagram}:

\begin{lemma}
	\label{lem: generalized pullback preservation}
	Let $\mathcal C$ be a contextual category and $F:(P,\le)^{\text{op}} \rightarrow \mathcal C$ a cellular diagram of shape $(P,\le,\tle)$. Suppose given a diagram of sieve inclusions
	\[
	\squa{X}{X'}{Y}{Y'}{}{}{}{}
	\]
	with $Y \smallsetminus X = Y' \smallsetminus X'$. Then the induced diagram
	\[
	\tag{\texttt{*}}
	\squa{F[Y']_\tle}{F[Y]_\tle}{F[X']_\tle}{F[X]_\tle}{}{}{}{}
	\]
	is a pullback square.
\end{lemma}

\begin{proof}
	First, we consider the case where $Y \smallsetminus X = Y' \smallsetminus X'$ has a single element, say $y$. Let $\tle'$ be a linearization of $(P,\le)$ such that $x \tle' y$ for all $x \in X'$. By Remark \ref{rem: naturality of comparison between slims}, diagram (\texttt{*}) is isomorphic to
	\[
	\dsqua{F[Y']_{\tle'}}{F[Y]_{\tle'}}{F[X']_{\tle'}}{F[X]_{\tle'}}{}{}{}{}
	\]
	which is a pullback square by Definition \ref{def: finite poset-shaped cellular diagram}(iv).
	
	The general case follows inductively using the above case and the pasting law for pullbacks.
\end{proof}

This can be further generalized as follows:

\begin{proposition}
	\label{prop: generalized limit preservation}
	Let $\mathcal C$ be a contextual category and $F:(P,\le)^{\text{op}} \rightarrow \mathcal C$ a cellular diagram of shape $(P,\le,\tle)$. Let $\mathcal S$ be a subset of $\Sieve(P)$ closed under intersections. Then the family of projection maps
	$$
	(F[\cup \mathcal S] \longrightarrow F[X])_{X \in \mathcal S}
	$$
	is a limit cone over the diagram indexed by the poset $(\mathcal S,\subset)^{\text{op}}$.
\end{proposition}

\begin{proof}
	We proceed by induction on the cardinality of $\mathcal S$. The claim is trivial for $\mathcal S = \varnothing$. For $n \ge 1$, assume that the claim holds for $0$, ..., $n-1$ and suppose given $\mathcal S$ of cardinality $n$. Let $Z$ be a maximal element of $\mathcal S$. Then $\mathcal S' = \mathcal S \smallsetminus \{Z\}$ is closed under intersections and, by Lemma \ref{lem: generalized pullback preservation},
	\[
	\squa{F[\cup \mathcal S]}{F[Z]}{F[\cup \mathcal S']}{F[(\cup \mathcal S') \cap Z]}{}{}{}{}
	\]
	is a pullback square. Suppose given $a \in C$ and a cone
	$$
	(a \overset{\theta_X}{\longrightarrow} F[X])_{X \in \mathcal S}
	$$
	over $F[-]:(\mathcal S,\subset)^{\text{op}} \rightarrow C$. By the induction hypothesis, we have a commutative diagram
	\[
	\squa{a}{F[Z]}{F[\cup \mathcal S']}{F[(\cup \mathcal S') \cap Z].}{\theta_Z}{}{(\theta_X)_{X \in \mathcal S'}}{}
	\]
	Hence there exists a unique arrow $f:a \rightarrow F[\cup \mathcal S]$ such that $\theta_Z$ equals
	$$
	a \overset{f}{\longrightarrow} F[\cup \mathcal S] \longrightarrow F[Z]
	$$
	and $(\theta_X)_{\mathcal S'}$ equals
	$$
	a \overset{f}{\longrightarrow} F[\cup \mathcal S] \longrightarrow F[\cup \mathcal S'].
	$$
	This means that $f$ is the unique arrow $a \rightarrow F[\cup \mathcal S]$ defining a factorization of $(\theta_X)_{X \in \mathcal S}$ through $(F[\cup \mathcal S] \rightarrow F[X])_{X \in \mathcal S}$.
\end{proof}

\subsubsection{Pasting cellular diagrams}

We will now describe how to construct cellular diagrams of a given shape $\textbf{P} = (P,\le,\tle)$ by gluing, in a suitable sense, cellular diagrams whose shapes are sieves in $(P,\le)$. The main tool will be the following lemma, which combines Proposition \ref{prop: properties cellular diagrams}(c) with the transport procedure described in Proposition \ref{prop: transporting cellular diagrams}.

\begin{lemma}
	\label{lem: pasting, first version}
	Let $(P,\le,\tle) \in \LFPos$ and suppose that $x \in P$ is maximal with respect to $\le$. Then the diagram of strong sieve embeddings (see Definition \ref{def: strong sieve embedding, set of cellular diagrams})
	\[
	\squa{\partial x}{x^{\le}}{P \smallsetminus \{x\}}{P}{}{}{}{}
	\]
	is mapped under $\Cell(-,\mathcal C)$ to a pullback square of sets.
\end{lemma}

\begin{proof}
	Using that $x$ is maximal for $\le$, choose a linearization $\btle$ of $(P,\le)$ with respect to which $x$ is the top element. Then we obtain a commutative diagram
	\[\begin{tikzcd}[ampersand replacement=\&]
		\& {\Cell((\partial x,\le,\btle),\mathcal C)} \&\& {\Cell((x^{\le},\le,\btle),\mathcal C)} \\
		{\Cell((\partial x,\le,\tle),\mathcal C)} \&\& {\Cell((x^{\le},\le,\tle),\mathcal C)} \\
		\& {\Cell((P \smallsetminus \{x\},\le,\btle),\mathcal C)} \&\& {\Cell((P,\le,\btle),\mathcal C)} \\
		{\Cell((P \smallsetminus \{x\},\le,\tle),\mathcal C)} \&\& {\Cell((P,\le,\tle),\mathcal C)}
		\arrow[from=1-2, to=1-4]
		\arrow[from=1-2, to=3-2]
		\arrow[from=1-4, to=3-4]
		\arrow["\cong"{description}, from=2-1, to=1-2]
		\arrow[from=2-1, to=2-3]
		\arrow[from=2-1, to=4-1]
		\arrow["\cong"{description}, from=2-3, to=1-4]
		\arrow[from=2-3, to=4-3]
		\arrow[from=3-2, to=3-4]
		\arrow["\cong"{description}, from=4-1, to=3-2]
		\arrow[from=4-1, to=4-3]
		\arrow["\cong"{description}, from=4-3, to=3-4]
	\end{tikzcd}\]
	where the horizontal arrows connecting the front and back faces are the bijections from Proposition \ref{prop: transporting cellular diagrams}, and the arrows in the front and back faces are obtained by applying the functor $\Cell(-,\mathcal C)$. Now, the front face is a cartesian square if and only if the back one is cartesian. To conclude, we note that by the remark in Proposition \ref{prop: properties cellular diagrams}(c),
	\[\begin{tikzcd}[ampersand replacement=\&]
		{\Cell((\partial x,\le,\btle),\mathcal C)} \&\& {\Cell((x^{\le},\le,\btle),\mathcal C)} \\
		\\
		{\Cell((P \smallsetminus \{x\},\le,\btle),\mathcal C)} \&\& {\Cell((P,\le,\btle),\mathcal C)}
		\arrow[from=1-1, to=1-3]
		\arrow[from=1-1, to=3-1]
		\arrow[from=1-3, to=3-3]
		\arrow[from=3-1, to=3-3]
	\end{tikzcd}\]
	is a pullback square.
\end{proof}

\begin{proposition}
	\label{prop: pasting, second version}
	Let $(P,\le,\tle) \in \LFPos$ and let $A$, $B \subset P$ be sieves such that $A \cup B = P$. Then the diagram of strong sieve embeddings
	\[
	\squa{A \cap B}{A}{B}{P}{}{}{}{}
	\]
	is mapped under $\Cell(-,\mathcal C)$ to a pullback square of sets.
\end{proposition}

\begin{proof}
	For the general case, we proceed by induction on the cardinality, say $n$, of $P \smallsetminus A$. If $A = P$, the claim holds trivially. For $n \ge 1$, suppose that the claim holds for $0$, ..., $n-1$. If $A \subset B$, we have $B = P$, in which case the claim holds. Otherwise, there exists an element $z \in A \smallsetminus B$ which is maximal in $(P,\le)$. Then we have a diagram
	\[
	\tag{\texttt{*}}
	\begin{tikzcd}[ampersand replacement=\&]
		{A \cap B} \& {A \smallsetminus \{z\}} \& A \\
		B \& {P \smallsetminus \{z\}} \& P
		\arrow[from=1-1, to=1-2]
		\arrow[from=1-1, to=2-1]
		\arrow[from=1-2, to=1-3]
		\arrow[from=1-2, to=2-2]
		\arrow[from=1-3, to=2-3]
		\arrow[from=2-1, to=2-2]
		\arrow[from=2-2, to=2-3]
	\end{tikzcd}\]
	of strong sieve embeddings. By the induction hypothesis, the left square is mapped to a pullback of sets under $\Cell(-,\mathcal C)$. To see that the same holds for the right square, consider the diagram
	\[
	\begin{tikzcd}[ampersand replacement=\&]
		{\partial z} \& {A \smallsetminus \{z\}} \& {P \smallsetminus\{z\}} \\
		{z^{\le}} \& A \& P.
		\arrow[from=1-1, to=1-2]
		\arrow[from=1-1, to=2-1]
		\arrow[from=1-2, to=1-3]
		\arrow[from=1-2, to=2-2]
		\arrow[from=1-3, to=2-3]
		\arrow[from=2-1, to=2-2]
		\arrow[from=2-2, to=2-3]
	\end{tikzcd}\]
	By Lemma \ref{lem: pasting, first version}, both the left and the outer composite squares are mapped under $\Cell(-,\mathcal C)$ to pullbacks of sets, so the same holds for the right one, which corresponds to the right square in (\texttt{*}). It follows that $\Cell(-,P)$ maps the outer square in (\texttt{*}) to a pullback of sets, so the claim holds for $n$.
\end{proof}

\subsubsection{Cellular diagrams indexed by products}

\label{subsubsec: cellular diagrams indexed by products}

Suppose given $n \ge 1$ and finite posets $(P_1,\le_1)$, ..., $(P_n, \le_n)$. For $x_1 \in P_1$, ..., $x_n \in P_n$ and a subset $I \subset \{1, ..., n\}$, we let
$$
\partial_I(x_1, ..., x_n) = d_I(x_1) \times \cdots \times d_I(x_n)
$$
where $d_I(x_i)$ equals $\partial x_i$ if $i \in I$, and $x_i^{\le_i}$ otherwise.

Denote by $\mathscr P^*(A)$ the poset, via the inclusion order, of nonempty subsets of a set $A$.

Note that the subset
$$
\mathcal S(x_1, ..., x_n) \coloneqq \{\partial_I(x_1, ..., x_n) \mid I \in \mathscr P^*(\{1,...,n\})\}
$$
of $\Sieve(P_1 \times \cdots \times P_n)$ is closed under intersections and satisfies
$$
\bigcup \mathcal S(x_1, ..., x_n) = \partial (x_1, ..., x_n).
$$

Now, consider a linearization $\tle$ of $P_1 \times \cdots \times P_n$, a contextual category $\mathcal C$, and, denoting by $\le$ the product of the orders $\le_1$, ..., $\le_n$, a cellular diagram $F:(P_1 \times \cdots \times P_n,\le)^{\text{op}} \rightarrow \mathcal C$ of shape $(P_1 \times \cdots \times P_n,\le,\tle)$. This induces a length-$1$ display map
$$
F(x_1, ..., x_n) = F[(x_1, ..., x_n)^\le] \longrightarrow F[\partial(x_1, ..., x_n)]
$$
(as in Definition \ref{def: finite poset-shaped cellular diagram}(iii)) and, by Proposition \ref{prop: generalized limit preservation}, a limit cone
$$
(F[\partial(x_1, ..., x_n)] \longrightarrow F[\partial_I(x_1, ..., x_n)])_{I \in \mathscr P^*(\{1,...,n\})}.
$$

\begin{construction}
	\label{constr: natural transformation between diagrams of boundaries}
	In the above setting, suppose we are also given $y_1 \in P_1$, ..., $y_n \in P_n$ and a natural transformation
	$$
	\psi:F[\partial_\bullet(x_1, ..., x_n)] \Longrightarrow F[\partial_\bullet(y_1, ..., y_n)]
	$$
	between functors $\mathscr P(\{1, ..., n\})^{\text{op}} \rightarrow \mathcal C$. This induces a commutative square
	\[
	\tag{\texttt{*}}
	\dsqua{F(x_1, ..., x_n) = F[(x_1, ..., x_n)^\le]}{F(y_1, ..., y_n) = F[(y_1, ..., y_n)^\le]}{F[\partial(x_1, ..., x_n)]}{F[\partial(y_1, ..., y_n)].}{\psi_\varnothing}{\lim_{I \neq \varnothing}\psi_I}{}{}
	\]
	When we discuss multimorphisms of contextual categories, diagrams of this form will play a key role as they will allow us to relate distinguished squares in the domain contextual categories to those in the codomain one. The property of such a diagram that will be most useful to us is that, while it may not be a distinguished square, being distinguished is independent, in a sense explained below, of the choice of linearization of $P_1 \times \cdots \times P_n$.
	
	Suppose given another linearization $\tle'$ of $P_1 \times \cdots \times P_n$. Let $F':(P_1 \times \cdots \times P_n,\le)^{\text{op}} \rightarrow \mathcal C$ and $\varphi:F' \Rightarrow F$ be as in Proposition \ref{prop: transporting cellular diagrams}. Let $\psi'$ be the composite natural transformation
	$$
	F'[\partial_\bullet(x_1, ..., x_n)]_{\tle'} \overset{\varphi}{\Longrightarrow} F[\partial_\bullet(x_1, ..., x_n)]_\tle \overset{\psi}{\Longrightarrow} F[\partial_\bullet(y_1, ..., y_n)]_\tle \overset{\varphi^{-1}}{\Longrightarrow} F[\partial_\bullet(x_1, ..., x_n)]_{\tle'}
	$$
	between functors $\mathscr P(\{1,...,n\})^{\text{op}} \rightarrow \mathcal C$. Then (\texttt{*}) is a distinguished square if and only if
	\[
	\tag{\texttt{\texttt{**}}}
	\dsqua{F'(x_1, ..., x_n) = F'[(x_1, ..., x_n)^\le]_{\tle'}}{F'(y_1, ..., y_n) = F'[(y_1, ..., y_n)^\le]_{\tle'}}{F'[\partial(x_1, ..., x_n)]_{\tle'}}{F'[\partial(y_1, ..., y_n)]_{\tle'}.}{\psi'_\varnothing}{\lim_{I \neq \varnothing}\psi'_I}{}{}
	\]
	is distinguished.
	
	Indeed, if (\texttt{*}) is distinguished, then (\texttt{\texttt{**}}) being distinguished follows from commutativity of the diagram
	\[\begin{tikzcd}[ampersand replacement=\&]
		\& {F(x_1, ..., x_n)} \&\& {F(y_1, ..., y_n)} \\
		{F'(x_1, ..., x_n)} \&\& {F'(y_1, ..., y_n)} \\
		\& {F[\partial(x_1, ..., x_n)]_\tle} \&\& {F[\partial(y_1, ..., y_n)]_\tle} \\
		{F'[\partial(x_1, ..., x_n)]_{\tle'}} \&\& {F'[\partial(y_1, ..., y_n)]_{\tle'},}
		\arrow["{\varphi_{(x_1, ..., x_n)}}", from=2-1, to=1-2]
		\arrow["{\varphi_{(y_1, ..., y_n)}}"{pos=0.2}, from=2-3, to=1-4]
		\arrow["{\psi'_\varnothing}"{pos=0.8}, from=2-1, to=2-3]
		\arrow["{\lim_{I \neq \varnothing}\psi'_I}"', from=4-1, to=4-3]
		\arrow["{\lim_{I \neq \varnothing}\psi_I}"'{pos=0.2}, from=3-2, to=3-4]
		\arrow["{\lim_{s \in \partial(x_1, ..., x_n)}\varphi_s}"'{pos=0.8}, from=4-1, to=3-2]
		\arrow["{\lim_{s \in \partial(y_1, ..., y_n)}\varphi_s}"'{pos=0.8}, from=4-3, to=3-4]
		\arrow[two heads, from=2-1, to=4-1]
		\arrow[two heads, from=1-2, to=3-2]
		\arrow["{\psi_{\varnothing}}", from=1-2, to=1-4]
		\arrow[two heads, from=2-3, to=4-3]
		\arrow[two heads, from=1-4, to=3-4]
	\end{tikzcd}\]
	whose left and right faces are distinguished squares, and the pasting law for distinguished pullbacks. If (\texttt{\texttt{**}}) is distinguished, we obtain that so is (\texttt{*}) by using a similar argument with the cube where each arrow from the front face to the back face is replaced by its inverse (for example, $\varphi_{(x_1, ..., x_n)}$ is replaced by $\varphi^{-1}_{(x_1, ..., x_n)}$).
\end{construction}

The following will be crucial when relating exponentials to multimorphisms of contextual categories.

For $n \ge 0$, we let $\mathcal O_n^{\text{pre}}$ be the precontextual category with underlying category the poset $\{0 < \cdots < n\}^{\text{op}}$, length function $i \mapsto i$, all arrows as display maps, and no distinguished squares.

\begin{proposition}
	\label{prop: cellular diagram indexed by product by an ordinal}
	Let $\mathcal C$ be a contextual category. Given $(P,\le,\tle) \in \LFPos$ and $n \ge 0$, consider the function
	$$
	\alpha_{n,P}:\Cell(P,\mathcal C^{\mathcal O_n^{\text{pre}}}) \longrightarrow \Fun((\{1 < \cdots < n\} \times P)^{\text{op}}, C)
	$$
	given by sending $F:P^{\text{op}} \rightarrow \mathcal C^{\mathcal O_n^{\text{pre}}}$ to the functor adjunct to the composite
	$$
	P^{\text{op}} \overset{F}{\longrightarrow} \mathcal C^{\mathcal O_n^{\text{pre}}} \overset{(-)_\varheartsuit}{\longrightarrow} \Fun(\{0 < \cdots < n\}^{\text{op}},C) \overset{\text{restriction}}{\longrightarrow} \Fun(\{1 < \cdots < n\}^{\text{op}},\mathcal C).
	$$
	Then, viewing the poset product $\{1 < \cdots < n\} \times P$ as a linearized poset via the lexicographic product between $\{1 < \cdots < n\}$ and $(P,\tle)$, the map $\alpha_{n,P}$ defines a bijection between $\Cell(P,\mathcal C^{\mathcal O_n^{\text{pre}}})$ and $\Cell(\{1 < \cdots < n\} \times P, \mathcal C)$.
\end{proposition}

\begin{proof}
	The idea is to work inductively on both $P$ and $n$ using that, denoting by $\top$ the top element of $(P,\tle)$, we can express cellular diagrams of shape $\{1 , \ldots ,  n\} \times P$ by combining:
	\begin{enumerate}[label=(\roman*)]
		\item The explicit procedure for extending cellular diagrams along the sieve embedding
		$$
		\partial(\{1 , \ldots ,  n\} \times P) = (\{1 , \ldots ,  n\} \times P) \smallsetminus \{(n,\top)\} \hookrightarrow \{1 , \ldots ,  n\} \times P
		$$
		given by Proposition \ref{prop: properties cellular diagrams}(c).
		
		\item The expression of $\partial(\{1 , \ldots ,  n\} \times P)$ as a union of sieves as in the diagram
		\[
		\begin{tikzcd}
			\{1 , \ldots ,  n-1\} \times \partial P \arrow[hook]{r}{} \arrow[hook]{d}{}	& \{1 , \ldots ,  n\} \times \partial P \arrow[hook]{d}{}\\
			\{1 , \ldots ,  n-1\} \times P \arrow[hook]{r}{} & \partial(\{1 , \ldots ,  n\} \times P)
		\end{tikzcd}
		\]
		of the form required in Proposition \ref{prop: pasting, second version}.
	\end{enumerate}
	More precisely, we will prove by induction on $N$ that the claim holds for all $n \ge 0$ and $(P,\le,\tle) \in \LFPos$ such that $n \cdot \sharp P = N$.
	
	For $N = 0$, either $n = 0$ or $P = \varnothing$. If $P = \varnothing$, the claim holds trivially. If $n = 0$, then the claim holds as $\{1 , \ldots ,  n\} \times P = \varnothing$ and, on the other hand, $\mathcal C^{\mathcal O_n^{\text{pre}}}$ is the terminal contextual category $\mathcal Z$ (see Example \ref{ex: cellular diagram to N}), so $\Cell(P,\mathcal C^{\mathcal O_n^{\text{pre}}})$ is a singleton.

	Given $N \ge 1$, suppose that the claim holds for $0$, ..., $N-1$. Consider $n$ and $P$ such that $n \cdot \sharp P = N$. 
	
	Since $\alpha_{\bullet,\bullet}$ is a natural transformation between functors $\mathbb N^{\text{op}} \times \LFPos^{\text{op}} \rightarrow \Set$, we have a commutative diagram
	\[
	\tag{\texttt{*}}
	\begin{tikzcd}[ampersand replacement=\&]
		\& {} \\
		\& {\Cell(\{1 , \ldots ,  n\}\times P,\mathcal C)} \&\& {\Cell(\{1 , \ldots ,  n-1\}\times P,\mathcal C)} \\
		{\Cell(P,\mathcal C^{\mathcal O_n^{\text{pre}}})} \&\& {\Cell(P,\mathcal C^{\mathcal O_{n-1}^{\text{pre}}})} \\
		\& {\Cell(\{1 , \ldots ,  n\}\times \partial P,\mathcal C)} \&\& {\Cell(\{1 , \ldots ,  n-1\}\times \partial P,\mathcal C)} \\
		{\Cell(\partial P,\mathcal C^{\mathcal O_n^{\text{pre}}})} \&\& {\Cell(\partial P,\mathcal C^{\mathcal O_{n-1}^{\text{op}}})}
		\arrow[from=2-2, to=2-4]
		\arrow[from=2-2, to=4-2]
		\arrow[from=2-4, to=4-4]
		\arrow[from=3-1, to=3-3]
		\arrow[from=3-1, to=5-1]
		\arrow["\cong"{description}, from=3-3, to=2-4]
		\arrow[from=3-3, to=5-3]
		\arrow[from=4-2, to=4-4]
		\arrow["\cong"{description}, from=5-1, to=4-2]
		\arrow[from=5-1, to=5-3]
		\arrow["\cong"{description}, from=5-3, to=4-4]
	\end{tikzcd}\]
	where the isomorphisms between the front and back faces are given by the induction hypothesis. We will describe a bijection $\Cell(P,\mathcal C^{\textbf{n+1}^{\text{op}}}) \rightarrow \Cell(\{1,...,n\}\times P,\mathcal C)$ given by restriction of $\alpha_{n,P}$ and completing (\texttt{*}) into a commutative cube.
	
	Let $Q = \partial(\{1 , \ldots ,  n\} \times P) = (\{1 , \ldots ,  n\} \times P) \smallsetminus \{(n,\top)\}$. By Proposition \ref{prop: pasting, second version}, we have a bijection
	$$
	\Cell(Q,\mathcal C) \cong \Cell(\{1 , \ldots ,  n\}\times \partial P,\mathcal C) \times_{\Cell(\{1 , \ldots ,  n-1\}\times \partial P,\mathcal C)}\Cell(\{1 , \ldots ,  n-1\}\times P,\mathcal C)
	$$
	compatible with the respective canonical maps from $\Cell(\{1 , \ldots ,  n\} \times P,\mathcal C)$.
	
	Hence by taking pullbacks in the front and back faces of (\texttt{*}) we obtain a diagram
	\[
	\tag{\texttt{\texttt{**}}}
	\begin{tikzcd}[ampersand replacement=\&]
		{\Cell(P,\mathcal C^{\mathcal O_n^{\text{pre}}})} \&\& {\Cell(\{1,\ldots,n\}\times P,\mathcal C)} \\
		\\
		{\Cell(\partial P,\mathcal C^{\mathcal O_n^{\text{pre}}}) \times_{\Cell(\partial P,\mathcal C^{\mathcal O_{n-1}^{\text{pre}}})}\Cell(P,\mathcal C^{\mathcal O_{n-1}^{\text{pre}}})} \&\& {\Cell(Q,\mathcal C).}
		\arrow["\varphi"', from=1-1, to=3-1]
		\arrow["\psi", from=1-3, to=3-3]
		\arrow["\cong"{description}, from=3-1, to=3-3]
	\end{tikzcd}\]
	Denoting by $\alpha$ the inverse of the horizontal arrow in (\texttt{**}), we will construct for each $F \in \Cell(Q,\mathcal C)$ a bijection between $\psi^{-1}(F)$ and $\varphi^{-1}(\alpha(F))$. We will do this by describing the fibers of $\varphi$ and $\psi$:
	\begin{itemize}
		\item Let $F \in \Cell(Q,\mathcal C)$. An element of $\psi^{-1}(F)$ corresponds, by Proposition \ref{prop: properties cellular diagrams}(c), to a length-$1$ display map with codomain $F[\partial(n,\top)]$. On the other hand, the latter fits into a distinguished square
		\[
		\dsqua{F[\partial(n,\top)]}{F[\{1, \ldots, n\} \times \partial \top]}{F[\{1, \ldots, n-1\} \times \top^\le]}{F[\{1, \ldots, n-1\} \times \partial \top].}{}{}{}{}
		\]
		We then obtain (see Definition \ref{def: gap map, relative length-$1$ display map}) a bijection between $\psi^{-1}(F)$ and the set of relative length-$1$ display maps of the form
		\[
		\tag{A}
		\dsqua{c}{F[\{1, \ldots, n\} \times \partial \top]}{F[\{1, \ldots, n-1\} \times \top^\le]}{F[\{1, \ldots, n-1\} \times \partial \top].}{}{}{}{}
		\]
		
		\item Let $(G,H) \in \Cell(\partial P,\mathcal C^{\mathcal O_n^{\text{pre}}}) \times_{\Cell(\partial P,\mathcal C^{\mathcal O_{n-1}^{\text{pre}}})}\Cell(P,\mathcal C^{\mathcal O_{n-1}^{\text{pre}}})$. Explicitly, we have a commutative square
		\[
		\tag{\texttt{***}}
		\begin{tikzcd}
			\partial P^{\text{op}} \arrow[]{r}{G} \arrow[swap,hook]{d}{} & \mathcal C^{\mathcal O_n^{\text{pre}}} \arrow[]{d}{}\\
			P^{\text{op}} \arrow[swap]{r}{H} & \mathcal C^{\mathcal O_{n-1}^{\text{pre}}}
		\end{tikzcd}
		\]
		where $G$, $H$ are cellular diagrams and the right vertical arrow is the morphism of contextual categories induced by the inclusion $\iota:\mathcal O_{n-1}^{\text{pre}} \rightarrow \mathcal O_n^{\text{pre}}$ of precontextual categories. An element of $\varphi^{-1}(G,H)$ is then a cellular diagram $K:P^{\text{op}} \rightarrow \mathcal C^{\mathcal O_n^{\text{pre}}}$ filling (\texttt{***}) in the sense that the top and bottom triangles commute.
		
		By Proposition \ref{prop: properties cellular diagrams}(c), such a $K$ is characterized by a choice of length-$1$ display map
		$$
		\pi:S \twoheadrightarrow G[\partial \top]
		$$
		whose top level,
		\[\begin{tikzcd}[ampersand replacement=\&]
			{\mathcal O_n^{\text{pre}}} \&\& {\mathcal C,}
			\arrow[""{name=0, anchor=center, inner sep=0}, "{S_\varheartsuit}", curve={height=-18pt}, from=1-1, to=1-3]
			\arrow[""{name=1, anchor=center, inner sep=0}, "{G[\partial \top]_\varheartsuit}"', curve={height=18pt}, from=1-1, to=1-3]
			\arrow["\pi", shorten <=5pt, shorten >=5pt, Rightarrow, from=0, to=1]
		\end{tikzcd}\]
		has the following property: the restriction
		\[\begin{tikzcd}[ampersand replacement=\&]
			{\mathcal O_{n-1}^{\text{pre}}} \&\& {\mathcal C}
			\arrow[""{name=0, anchor=center, inner sep=0}, "{S_\varheartsuit \circ \iota}", curve={height=-18pt}, from=1-1, to=1-3]
			\arrow[""{name=1, anchor=center, inner sep=0}, "{G[\partial \top]_\varheartsuit \circ \iota = H[\partial \top]_\varheartsuit}"', curve={height=18pt}, from=1-1, to=1-3]
			\arrow["{\pi \iota}", shorten <=5pt, shorten >=5pt, Rightarrow, from=0, to=1]
		\end{tikzcd}\]
		is such that $S_\varheartsuit \circ \iota = H(\top)_\varheartsuit$ and $\pi \iota$ equals the length-$1$ display map $H(\top) \rightarrow H[\partial \top]$. From this and the construction of display maps in $\mathcal C^{\mathcal O_n^{\text{pre}}}$ it can be checked that the data of such a $K$ corresponds to a relative length-$1$ display map in $\mathcal C$ of the form
		\[
		\tag{B}
		\dsqua{c}{G[\partial \top]_\varheartsuit(n)}{H(\top)_\varheartsuit(n-1)}{G[\partial \top]_\varheartsuit(n-1) = H[\partial \top]_\varheartsuit(n-1).}{}{}{}{}
		\]
	\end{itemize}
	When $(G,H) = \alpha(F)$, diagram (B) becomes diagram (A), so we obtain a bijection $\beta:\Cell(P,\mathcal C^{\textbf{n+1}^{\text{op}}}) \rightarrow \Cell(\{1,...,n\} \times P, \mathcal C)$ completing (\texttt{**}) into a commutative square. It follows from the construction that $\beta$ is given by restriction of $\alpha_{n,P}$ $-$ details were omitted. This concludes the induction step.
\end{proof}

\subsubsection{A lemma on cellular diagrams in a contextual category associated with a category with attributes}

Let $\mathcal C$ be a category with attributes and $\textbf{P}=(P,\le,\tle)$ a linearized finite poset. We let $U:\cont(\mathcal C) \rightarrow \mathcal C$ be the forgetful functor.

\begin{lemma}
	\label{lem: cellular diagrams in a cont cat associated to a cwa}
	Suppose that $F$, $G:(P,\le)^{\text{op}} \rightarrow \cont(\mathcal C)$ are $\textbf{P}$-shaped cellular diagrams in $\mathcal C$ such that $UF = UG$. Then $F = G$.
\end{lemma}

\begin{proof}
	We will prove by induction on $N \ge 0$ that if $X$ is a sieve in $(P,\le)$ of cardinality at most $N$, then $F[X] = G[X]$ and, for any inclusion $X \subset Y$ of such sieves, the induced maps $F[Y] \rightarrow F[X]$ and $G[Y] \rightarrow G[X]$ are equal.
	
	For $N = 0$ the claim holds as $F[\varnothing]$ and $G[\varnothing]$ are both the distinguished terminal object of $\cont(\mathcal C)$. Let $N \ge 1$ and suppose that the claim holds for $0$, ..., $N-1$. Suppose that $X \subset P$ is a sieve with $N$ elements, and let $\top$ be the top element of $X$ with respect to $\tle$. Consider the distinguished squares
	\[\begin{tikzcd}[ampersand replacement=\&]
		{F[X]} \& {F(\top) = F[\top^\le]} \&\& {G[X]} \& {G(\top) = G[\top^\le]} \\
		{F[X \smallsetminus\{\top\}]} \& {F[\partial\top],} \&\& {G[X \smallsetminus \{\top\}]} \& {G[\partial \top].}
		\arrow[from=1-1, to=1-2]
		\arrow[two heads, from=1-1, to=2-1]
		\arrow[two heads, from=1-2, to=2-2]
		\arrow[from=1-4, to=1-5]
		\arrow[two heads, from=1-4, to=2-4]
		\arrow[two heads, from=1-5, to=2-5]
		\arrow[from=2-1, to=2-2]
		\arrow[from=2-4, to=2-5]
	\end{tikzcd}\]
	By assumption, the cones
	$$
	(UF(\top) \longrightarrow UF(x))_{x \in \partial \top},
	$$
	$$
	(UF(\top) \longrightarrow UF(x))_{x \in \partial \top}
	$$
	in $C$ are equal. Also, by the induction hypothesis the cones
	$$
	(U(F[\partial \top]) \longrightarrow U(F[x^\le]) = UF(x))_{x \in \partial\top},
	$$
	$$
	(U(G[\partial \top]) \longrightarrow U(G[x^\le]) = UG(x))_{x \in \partial\top}
	$$
	are equal. It follows that the images under $U$ of the arrows $F(\top) \rightarrow F[\partial \top]$ and $G(\top) \rightarrow G[\partial \top]$ in the above diagrams are equal. But $F(\top)$ is obtained by adjoining $UF(\top) \rightarrow U(F[\partial \top])$ to $F[\partial \top]$ to the chain of morphisms $F[\partial \top]$; similarly, $G(\top)$ is obtained by adjoining $UG(\top) \rightarrow U(G[\partial \top])$ to $G[\partial\top]$. Hence $F(\top) = G(\top)$. Since, by the induction hypothesis, the lower horizontal arrows in the above squares are also equal, we conclude that the two squares are equal.
	
	Moreover, suppose given a proper sieve $Y \subset X$. Let us prove that the arrows $F[X] \rightarrow F[Y]$ and $G[X] \rightarrow G[Y]$ are equal.
	
	If $X \neq \top^\le$, we note that it suffices to prove that the pairs of arrows
	$$
	(F[X] \rightarrow F[Y \cap \partial X],\;\; F[X] \rightarrow F[Y \cap \top^\le]),
	$$
	$$
	(G[X] \rightarrow G[Y \cap \partial X],\;\; G[X] \rightarrow G[Y \cap \top^\le])
	$$
	are equal. But, using the induction hypothesis and what we proved above, the composites $F[X] \rightarrow F[\partial X] \rightarrow F[Y \cap \partial X]$ and $G[X] \rightarrow G[\partial X] \rightarrow G[Y \cap \partial X]$ are equal. We conclude similarly (now using that $X \neq \top^\le$) that the composites $F[X] \rightarrow F[\top^\le] \rightarrow F[Y \cap \top^\le]$ and $G[X] \rightarrow G[\top^\le] \rightarrow G[Y \cap \top^\le]$ are equal.
	
	Now, suppose that $X = \top^\le$. Then $Y \subset \partial \top$, and from what was proved above and the induction hypothesis we obtain that the composites $F[\top^\le] \rightarrow F[\partial \top] \rightarrow F[Y]$ and $G[\top^\le] \rightarrow G[\partial \top] \rightarrow G[Y]$ are equal.
	
	This concludes the induction step.
	
	As for each $x \rightarrow y$ in $P$ the arrow $F(y) \rightarrow F(x)$ is equal to the morphism $F[y^\le] \rightarrow F[x^\le]$ induced by the inclusion of sieves $x^\le \subset y^\le$, and similarly for $G(y) \rightarrow G(x)$, we conclude that $F = G$.
\end{proof}

\subsection{Generalized cellular diagrams}

\begin{definition}
\label{def: locally finite poset, local linearization}
	A poset $(P,\le)$ is said to be \emph{locally finite} if $x^\le = \{y \in P \mid y \le x\}$ is finite for all $x \in P$.
	
	\vspace{0.5em}
	
	For a locally finite poset $P$, we write $\cub(P)$ for the set of all pairs $(x,y) \in P \times P$ such that $x$, $y$ have a common upper bound, that is, there exists $z \in P$ with $x$, $y \le z$.
	
	We define a \emph{local linearization} of $(P,\le)$ as a binary relation $\tle \; \subset \cub(P) \subset P \times P$ such that for each $x \in P$, the restriction of $\tle$ to $x^\le$ is a linearization of the latter (Definition \ref{def: linearization}).
	
	\vspace{0.5em}
	
	A triple $(P,\le,\tle)$ where $(P,\le)$ is a locally finite poset and $\tle$ is a local linearization of $(P,\le)$ will be referred to as a \emph{locally linearized poset}.
\end{definition}

\begin{remark}
	\leavevmode
	\begin{enumerate}[label=(\roman*)]
		\item Note that a local linearization is completely determined, as it is contained $\cub(P)$, by its restrictions to the sub-posets $x^\le$. Hence a local linearization can be equivalently described as a family $(\tle_x)_x$ consisting of a linearization of $x^\le$ for each $x \in P$ such that if $a$, $b$ are bounded above by both $x$ and $y$, then $a \tle_x b$ if and only if $a \tle_y b$.
		
		\item If $(P,\le)$ is a finite poset, then a local linearization of it is a linearization precisely when it is a linear order. On the other hand, a linearization of it will only be a local linearization when it is contained in $\cub(P)$, which happens precisely when there is a top element with respect to $\le$.
		
		It can be proved that any local linearization of $(P,\le)$ extends to a linearization, and that if $\tle$ is a linearization, then $\tle \cap \; \cub(P)$ is a local linearization.
	\end{enumerate}	
\end{remark}

\begin{definition}
	\label{def: cellular diagram general}
	Let $\textbf{P} = (P,\le,\tle)$ be a locally linearized poset. For a contextual category $\mathcal C$, a \emph{cellular diagram} of shape $\textbf{P}$ in $\mathcal C$ is a functor $F:(P,\le)^{\text{op}} \rightarrow \mathcal C$ such that for all $x \in P$, the restriction $F|_{x^\le}:x^{\le,op} \rightarrow \mathcal C$ is a cellular diagram with respect to the restriction of $\tle$ to $x^\le$.
	
	\vspace{0.5em}
	
	For a (finite) sieve $X \subset (P,\le)$ such that $X \subset x^\le$ for some $x \in P$, the restriction of $\tle$ to $X$ is a linearization of the latter. We write $F[X]_\tle$, or just $F[X]$ when $\tle$ is implicit, for $F[X]_{\overline{\tle}}$ where $\overline{\tle}$ is the restriction of $\tle$ to $X$ (equivalently, we can let $\overline{\tle}$ be the restriction of $\tle$ to $x^\le$ for any $x$ such that $X \subset x^\le$).
\end{definition}

\begin{remark}
\label{rem: cellular diagram from locally linearized finite poset}
	It can be proved by induction on the cardinality of $P$, using Lemma \ref{lem: pasting, first version}, that if $(P,\le,\tle)$ is a linearized finite poset, then $F:(P,\le)^{\text{op}} \rightarrow \mathcal C$ is cellular in the sense of Definition \ref{def: finite poset-shaped cellular diagram} precisely when it is cellular, as in Definition \ref{def: cellular diagram general}, with respect to the locally linearized poset $(P,\le,\tle \cap \; \cub(P))$.
\end{remark}

\begin{notation}
For a locally linearized poset $\textbf{P}$ and a contextual category $\mathcal C$, we write $\Cell(\textbf{P},\mathcal C)$ for the set of cellular diagrams of shape $\textbf{P}$ in $\mathcal C$. Note that we have a functor $\Cell(\textbf{P},-):\Cont \rightarrow \Set$.
\end{notation}

In principle, this clashes with Definition \ref{def: strong sieve embedding, set of cellular diagrams}, but Remark \ref{rem: cellular diagram from locally linearized finite poset}, if $(P,\le, \tle)$ shows that this abuse of notation is harmless.

\begin{proposition}
\label{prop: representability cellular diagrams}
For any locally linearized poset $\textbf{P}$, the functor $\Cell(\textbf{P},-):\Cont \rightarrow \Set$ is representable.
\end{proposition}

\begin{proof}[Sketch of proof]
An element $x$ of a locally finite poset $X$ has a \emph{height}, say $h(x)$, defined as the largest $k \ge 0$ such that there exists a chain $x_1 < x_2 < \cdots < x_k = x$. It can then be checked that $h(x) = \text{max}\{h(y) \mid y < x\} + 1$. In particular, if two distinct elements have the same height, they are not comparable.

Write $\textbf{P} = (P,\le,\tle)$ and, for each $n \in \bbN$, let $P_n \subset P$ consist of all elements of height at most $n$. We also regard $P_n$ as a locally linearized poset via the restrictions $\le|_{P_n}$ and $\tle \cap \cub(P_n)$.

\vspace{0.5em}

Let us prove by induction that $\Cell(P_n,-):\Cont \rightarrow \Set$ is representable for all $n \ge 0$.

Firstly, as $P_0$ is empty, $\Cell(P_0,-)$ is represented by the initial contextual category, $\mathcal O_0$, whose only object is the distinguished terminal one. Given $1 \le n < \omega$, assume that the claim holds for $0$, ..., $n-1$. Let $\mathcal A$ be a contextual category that represents $\Cell(P_{n-1},-)$, and fix a universal cellular diagram $I:P_{n-1} \rightarrow \mathcal A$.

Note that if $x \in P$ has height $n$, then $\partial x \subset P_{n-1}$ is linearly ordered with respect to $\tle$. By Remark \ref{rem: cellular diagram from locally linearized finite poset}, the restriction of a cellular diagram $F:P_{n-1}^{\text{op}} \rightarrow \mathcal C$ to $\partial x$ is cellular in the sense of Definition \ref{def: finite poset-shaped cellular diagram}. Hence this restriction has a distinguished limit $F[\partial x]$. Now, extending $F$ to a cellular diagram $P_n^{\text{op}} \rightarrow \mathcal C$ amounts to choosing for each height-$n$ element $x$ an object $c_x \in \mathcal C$ such that $\partial(c_x) = F[\partial x]$. On the other hand, in terms of the bijection $\Cell(P_{n-1},\mathcal C) \cong \Hom_\Cont(\mathcal A,\mathcal C)$, the morphism $F':\mathcal A \rightarrow \mathcal C$ corresponding to $F$ sends $I[\partial x]$ to $F[\partial x]$.

Define $\mathcal B$ as the precontextual category obtained from $\mathcal A$ by freely adjoining, for each height-$n$ element $x \in P$, an object $a_x$ such that $\partial(a_x) = I[\partial x]$. Let $\mathcal A' = L(\mathcal B)$ (recall that $L$ is a left adjoint of $\Cont \hookrightarrow \Precont$), and let $b_x$ be the image of $a_x$ under the canonical morphism $\mathcal B \rightarrow \mathcal A'$. Then we have $\Cell(P_n,-) \cong \Hom_\Cont(\mathcal A',-)$ with the corresponding universal cellular diagram $I':P_n^{\text{op}} \rightarrow \mathcal A'$ given on $P_{n-1}$ by $I$ and on height-$n$ elements by $I'(x) = a_x$.

This concludes the induction step.

\vspace{0.5em}

Finally, let us construct a representing object for $\Cell(\textbf{P},-)$. Consider a sequence $\mathcal A_1 \rightarrow \mathcal A_2 \rightarrow \cdots$ of contextual categories and contextual functors where $\mathcal A_n$ represents $\Cell(P_n,-)$ and $\mathcal A_n \rightarrow \mathcal A_{n+1}$ is induced, via the universal property of $\mathcal A_n$, by the composite $P_n^{\text{op}} \hookrightarrow P_{n+1}^{\text{op}} \rightarrow \mathcal A_{n+1}$.

Letting $\mathcal A_\omega$ be a colimit of $(\mathcal A_n)_{n \in \mathbb N}$, we have
$$
\Hom_\Cont(\mathcal A_\omega,-) \cong \lim_{n < \omega}\Hom_\Cont(\mathcal A_n,-) \cong \lim_{n < \omega}\Cell(P_n,-).
$$
Now, note that $\lim_{n < \omega}\Cell(P_n,\mathcal C)$ is in bijection, naturally in $\mathcal C$, with the set of all functors $P^{\text{op}} \rightarrow \mathcal C$ whose restriction to $P_n^{\text{op}}$ is cellular for all $n$. But the latter are precisely the cellular diagrams of shape $\textbf{P}$ in $\mathcal C$; indeed, $P^{\text{op}} \rightarrow \mathcal C$ being cellular only depends, by definition, on the restrictions $x^{\le,op} \rightarrow \mathcal C$ for $x \in P$.
\end{proof}

\begin{remark}
A similar construction can be performed for any locally finite direct category $D$; see \cite{Sub21}. The corresponding contextual category, $\mathcal C(D)$, is such that $|\mathcal C(D)|^{\text{op}}$ is equivalent to the category $\PSh_{fp}(D)$ of finitely presentable presheaves on $D$. In particular, $\mathcal A$ from the proof of Proposition \ref{prop: representability cellular diagrams} is such that $|\mathcal A|^{\text{op}} \simeq \PSh_{fp}(D)$.
\end{remark}

\section{Multimorphisms}

In this section, we define and study multimorphisms of contextual categories. We start by introducing, as a preliminary step, the concept of a \emph{pre-multimorphism} (Definition \ref{def: pre-multimorphism}). These will be, for contextual categories $\mathcal A_1$, ..., $\mathcal A_n$, $\mathcal C$,\footnote{More generally, $\mathcal A_1$, ..., $\mathcal A_n$ can be any precontextual categories} functors $\mathcal A_1 \times \cdots \times \mathcal A_n \rightarrow \mathcal C$ that are cellular with respect to the product of trees of display maps $T(\mathcal A_1) \times \cdots \times T(\mathcal A_n)$. Passing from pre-multimorphisms to multimorphisms (Definition \ref{def: multimorphism}) amounts to imposing a form of compatibility with distinguished squares in $\mathcal A_1$, ..., $\mathcal A_n$.

After that, we discuss the effect on (pre-)multimorphisms of changing the choice of local linearization of the poset $T(\mathcal A_1) \times \cdots \times T(\mathcal A_n)$; the main example to keep in mind is moving across the lexicographic local linearizations corresponding to different permutations of $\{1, ..., n\}$. This will lead to the concept of a \emph{permutative morphism}, a kind of transformation between multimorphisms that will be one of the key ingredients for studying the symmetry of multimorphisms (and, later, of the symmetry and associativity of the tensor product of contextual categories).

In \S\ref{subsec: comparison with bimorphisms as introduced previously}, we will see how specializing to the case $n = 2$ recovers the bimorphisms from \S\ref{sec: exponentiation and bimorphisms}. In \S\ref{subsec: multimorphisms and exponentials}, which is comparatively quite technical, we construct a natural isomorphism between categories of multimorphisms
$$
\iiHom(\mathcal A_1, ..., \mathcal A_1; \mathcal C) \cong \iiHom(\mathcal A_2, ..., \mathcal A_n;\mathcal C^{\mathcal A_1}).
$$
Finally, in \S\ref{subsec: multimorphisms and exponentials via precont}, we check that $\mathcal C^\mathcal A \cong \mathcal C^{L\mathcal A}$ and $\Hom(\mathcal A_1, ..., \mathcal A_n; \mathcal C) \cong \Hom(L\mathcal A_1, ..., L\mathcal A_n; \mathcal C)$ where $L:\Precont \rightarrow \Cont$ is the reflection functor.

\subsection{Cellular diagrams out of products of trees}

\begin{definition}
By a \emph{tree} we will mean a poset $(T,\le)$ such that for all $x \in T$, $x^\le$ is finite and linearly ordered by the restriction of $\le$.
\end{definition}

Note that $\le$ is itself a local linearization of $(T,\le)$.

\vspace{0.5em}

Our reason for considering trees is that they arise from contextual categories in the following way:

\begin{definition}
If $\mathcal A$ is a contextual category, its \emph{associated tree}, denoted by $T(\mathcal A)$, is the poset whose elements are the objects of $\mathcal A$ of length $\ge 1$, and where $x \le y$ if and only if there exists a display map $y \rightarrow x$.
\end{definition}

When defining multimorphisms of contextual categories, the key ingredient will be the following definition and proposition, which allow us to consider cellular diagrams out of a product of trees of the form $T(\mathcal A_1) \times \cdots \times T(\mathcal A_n)$. The latter, equipped with the local linearization which we describe next, will encode an important part of the structure of multimorphisms $(\mathcal A_1, ..., \mathcal A_n) \rightarrow \mathcal C$.

\begin{definition}
	\label{def: lexicographic product}
	Consider trees $T_1$, ..., $T_n$. We define their \emph{lexicographic product} as the binary relation $\btle$ on $T_1 \times \cdots \times T_n$ given by $(x_1, ..., x_n) \btle (y_1, ..., y_n)$ if and only if either
	\begin{itemize}
		\item $(x_1, ..., x_n) = (y_1, ..., y_n)$, or
		
		\item there exists $i \in \{1, ..., n+1\}$ such that (i) $x_j = y_j$ for $j < i$, (ii) $x_i < y_i$, and (iii) for $j > i$, $x_j$ and $y_j$ are comparable, i.e. $x_j \le y_j$ or $y_j \le x_j$.
	\end{itemize}
\end{definition}

\begin{proposition}
\label{prop: lexicographic product is local linearization}
	In the notation of Definition \ref{def: lexicographic product}, $\btle$ is a local linearization of $T_1 \times \cdots \times T_n$.
\end{proposition}

\begin{proof}
	Firstly, note that if $(x_1, ..., x_n) \btle (y_1, ..., y_n)$, then $x_i$ and $y_i$ are comparable for $1 \le i \le n$, so $(\max\{x_i,y_i\})_i$ is a common upper bound of $(x_i)_i$ and $(y_i)_i$. Hence $\btle \; \subset \cub(T_1 \times \cdots \times T_n)$.
	
	Also, given $x_i \in T_i$ for $1 \le i \le n$, as $x_i^\le$ is linearly ordered, the restriction of $\btle$ to $(x_1, ..., x_n)^\le$ is the usual lexicographic (linear) order on $x_1^\le \times \cdots \times x_n^\le$.
	
	Finally, if $(x_i)_i \le (y_i)_i$, then $(x_i)_i \btle (y_i)_i$: in the notation of Definition \ref{def: lexicographic product}, either $(x_i)_i = (y_i)_i$, or we can take $j$ as the smallest $i$ such that $x_i < y_i$.
\end{proof}

\begin{remark}
Any permutation of $\{1, ..., n\}$ defines a local linearization of $T_1 \times \cdots \times T_n$: for $\sigma \in S_n$, we can consider the local linearization $\btle$ of $T_{\sigma(1)} \times \cdots \times T_{\sigma(n)}$ given by Proposition \ref{prop: lexicographic product is local linearization}, and then pull it back along the poset isomorphism
$$
T_1 \times \cdots \times T_n \cong T_{\sigma(1)} \times \cdots \times T_{\sigma(n)}.
$$
\end{remark}

\subsection{Multimorphisms}

In what follows, we let $\mathcal A_1$, ..., $\mathcal A_n$ be precontextual categories and $\mathcal C$ a contextual category.

\begin{definition}
	\label{def: pre-multimorphism}
	Let $\tle$ be a local linearization of the (locally finite) poset $T(\mathcal A_1) \times \cdots \times T(\mathcal A_n)$. A \emph{$\tle$-shaped pre-multimorphism} from $(\mathcal A_1, ..., \mathcal A_n)$ to $\mathcal C$ is a functor $F:|\mathcal A_1| \times \cdots \times |\mathcal A_n| \rightarrow |\mathcal C|$ such that the restriction
	$$
	F|_{(T(\mathcal A_1) \times \cdots \times T(\mathcal A_n))^{\text{op}}}:(T(\mathcal A_1) \times \cdots \times T(\mathcal A_n))^{\text{op}} \longrightarrow \mathcal C
	$$
	is a cellular diagram with respect to $\tle$.
	
	The set of all such pre-multimorphisms will be denoted by $\Hom^?_\tle(\mathcal A_1, ..., \mathcal A_n;\mathcal C)$, and the corresponding full subcategory of $|\mathcal C|^{|\mathcal A_1| \times \cdots \times |\mathcal |A_n|}$ by $\iiHom^?_\tle(\mathcal A_1, ..., \mathcal A_n;\mathcal C)$.\footnote{The question mark is meant to indicate that a pre-multimorphism may or may not be a multimorphism in the sense of Definition \ref{def: multimorphism}.}
	
	\vspace{0.5em}
	
	Recalling that $\btle$ denotes the lexicographic local linearization of $T(\mathcal A_1) \times \cdots \times T(\mathcal A_n)$ from Definition \ref{def: lexicographic product}, we refer to a $\btle$-shaped pre-multimorphism simply as a \emph{pre-multimorphism}.
\end{definition}

For a subset $I \subset \{1, ..., n\}$ we let $\chi_I:\{1,...,n\} \rightarrow \{0,1\}$ be its characteristic function, i.e. $\chi_I(x)$ equals $1$ if $x \in I$, and $0$ otherwise. Given $x_1 \in \mathcal A_1$, ..., $x_n \in \mathcal A_n$, we write $\partial(x_1, ..., x_n)$ and $\partial_I(x_1, ..., x_n)$ for the respective constructions, as defined in \S\ref{subsubsec: cellular diagrams indexed by products}, in the case of the finite poset $x_1^\le \times \cdots \times x_n^\le = (x_1, ..., x_n)^{\le} \subset T(\mathcal A_1) \times \cdots \times T(\mathcal A_n)$.

\begin{construction}
\label{constr: product of dist squares}
In the above setting, suppose given for each $i = 1$, ..., $n$ a commutative square $Q_i$ in $\mathcal A_i$ of the form
\[
\dsqua{x^0_i}{y^0_i}{x^1_i}{y^1_i}{f^0_i}{f^1_i}{p_i}{q_i}
\]
where $p_i$, $q_i$ are length-$1$ display maps. 

Since $\partial x_i^0 = x_i^{1 \le}$ and $\partial y_i^0 = y_i^{1 \le}$ for each $I \subset \{1, ..., n\}$, in the notation of \ref{subsubsec: cellular diagrams indexed by products} we have
$$
F[\partial_I(x_1^0, ..., x_n^0)] = F(x^{\chi_I(1)}_1, ..., x^{\chi_I(n)}_n),
$$
$$
F[\partial_I(y_1^0, ..., y_n^0)] = F(y^{\chi_I(1)}_1, ..., y^{\chi_I(n)}_n).
$$

Now, the arrows $f^0_i$, $f^1_i$ define a natural transformation
$$
\psi:F[\partial_\bullet(x^0_1, ..., x^0_n)] \Longrightarrow F[\partial_\bullet(y^0_1, ..., y^0_n)]
$$
between functors $\mathscr P(\{1, ..., n\})^{\text{op}} \rightarrow C$, namely, whose $I$-component is $F(f^{\chi_I(1)}_1, ..., f^{\chi_I(n)}_n)$.

We then obtain, as in Construction \ref{constr: natural transformation between diagrams of boundaries}, a commutative square
\[
\tag{\texttt{*}}
\widedsqua{F(x^0_1, ..., x^0_n)}{F(y^0_1, ..., y^0_n)}{F[\partial(x^0_1, ..., x^0_n)]_\tle}{F[\partial(y^0_1, ..., y^0_n)]_\tle.}{\psi_\varnothing = F(f^0_1, ..., f^0_n)}{\lim_{I \neq \varnothing}\psi_I}{}{}
\]
We will denote (\texttt{*}) by $F\langle Q_1, ..., Q_n \rangle_\tle$.
\end{construction}

\begin{definition}
\label{def: multimorphism}
Let $\tle$ be a local linearization of $T(\mathcal A_1) \times \cdots \times T(\mathcal A_n)$. We define a \emph{$\tle$-shaped multimorphism} from $(\mathcal A_1, ..., \mathcal A_n)$ to $\mathcal C$ as a $\tle$-shaped pre-multimorphism $F:\mathcal A_1 \times \cdots \times \mathcal A_n \rightarrow \mathcal C$ that satisfies the following: given for each $i = 1$, ..., $n$ a length-$1$ distinguished square $Q_i$ in $\mathcal A_i$, the commutative square $F_\langle Q_1, ..., Q_n \rangle_\tle$ (see Construction \ref{constr: product of dist squares}) is distinguished.

The set of $\tle$-shaped multimorphisms from $(\mathcal A_1, ..., \mathcal A_n)$ to $\mathcal C$ will be denoted by $\Hom_\tle(\mathcal A_1, ..., \mathcal A_n;\mathcal C)$, and the corresponding full subcategory of $|\mathcal C|^{|\mathcal A_1| \times \cdots \times |\mathcal A_n|}$ by $\iiHom_\tle(\mathcal A_1, ..., \mathcal A_n;\mathcal C)$.

\vspace{0.5em}

Recalling that $\btle$ denotes the lexicographic local linearization of $T(\mathcal A_1) \times \cdots \times T(\mathcal A_n)$ from Definition \ref{def: lexicographic product}, we refer to a $\btle$-shaped multimorphism simply as a \emph{multimorphism}. If we want to emphasize $n$, we refer to a ($\tle$-shaped) multimorphism as a (\emph{$\tle$-shaped}) \emph{$n$-ary morphism}.
\end{definition}

The following criterion will be useful for checking whether a given pre-multimorphism is a multimorphism:

\begin{proposition}
	\label{prop: criterion multimorphism}
	In the notation of definitions \ref{def: pre-multimorphism} and \ref{def: multimorphism}, suppose that $F$ is a $\tle$-shaped pre-multimorphism from $(\mathcal A_1, ..., \mathcal A_n)$ to $\mathcal C$. For $F$ to be a multimorphism, it suffices that $F\langle Q_1, ..., Q_n \rangle_\tle$ be distinguished for all sequences $(Q_i)_i$ of length-$1$ distinguished squares such that for some $j \in \{1, ..., n\}$, $Q_i$ is of the form
	\[
	\dsqua{a}{a}{b}{b}{id_a}{id_b}{p}{p}
	\]
	for all $i \neq j$.
\end{proposition}

\begin{proof}
	This follows from functoriality of $F\langle - \rangle_\tle:\mathfrak D(\mathcal A_1) \times \cdots \times \mathfrak D(\mathcal A_n) \rightarrow \mathfrak D(\mathcal C)$ and the fact that $\mathfrak D(\mathcal A_1) \times \cdots \times \mathfrak D(\mathcal A_n)$ is spanned by those arrows corresponding to a sequence of squares $(Q_1, ..., Q_n)$ as in the statement.
\end{proof}

\subsection{Transporting (pre-)multimorphisms}
\label{subsec: transporting multimorphisms}

We now describe a construction, derived from Proposition \ref{prop: transporting cellular diagrams}, that will play a central role in our study of (pre-)multimorphisms. It will be used not only to express, in a straightforward way, symmetry of the tensor product of contextual categories, but also as a technical component of our proof of associativity.

Generally, if $F:\mathcal A_1 \times \cdots \times \mathcal A_n \rightarrow \mathcal C$ is a (pre-)multimorphism with respect to a given local linearization $\tle$ of $T(\mathcal A_1) \times \cdots T(\mathcal A_n)$, it will not remain a (pre-)multimorphism if we replace $\tle$ by another local linearization, say $\overline{\tle}$. Still, we can canonically modify, in a certain sense, $F$ to a $\overline{\tle}$-shaped (pre-)multimorphism $\overline{F}:\mathcal A_1 \times \cdots \times \mathcal A_n \rightarrow \mathcal C$; more precisely, we will have an isomorphism $F \cong \overline{F}$ characterized essentially as in Proposition \ref{prop: transporting cellular diagrams}. In Proposition \ref{prop: transporting multimorphism gives multimorphism} we prove that $\overline{F}$ is a multimorphism precisely when $F$ is a multimorphism.

\begin{construction}
	\label{constr: transporting pre-multimorphisms}

	Let $F:\mathcal A_1 \times \cdots \times \mathcal A_n \rightarrow \mathcal C$ be, as in Definition \ref{def: pre-multimorphism}, a $\tle$-shaped pre-multimorphism. Let $\overline{\tle}$ be another local linearization of $T(\mathcal A_1) \times \cdots \times T(\mathcal A_n)$.
	
	For each $x_1 \in \mathcal A_1$, ..., $x_n \in \mathcal A_n$ we can associate to the $\tle$-shaped\footnote{We also write $\tle$ for the restriction $\tle|_{(x_1,...,x_n)^\le}$.} cellular diagram
	$$
	F|_{(x_1, ..., x_n)^\le}:(x_1,...,x_n)^{\le,op} \longrightarrow \mathcal C,
	$$
	via Proposition \ref{prop: transporting cellular diagrams}, a natural isomorphism of the form
	$$
	\eta^{x_1, ..., x_n}: G^{x_1,...,x_n} \Longrightarrow F|_{(x_1, ..., x_n)^\le}
	$$
	where, in particular, $G^{x_1,...,x_n}$ is cellular with respect to $\overline{\tle}$.
	
	It can be verified that for $y_1 \in \mathcal A_1$, ..., $y_n \in \mathcal A_n$ such that $y_i \le x_i$ for each $i$, we have
	$$
	\eta^{x_1, ..., x_n}_{y_1, ..., y_n} = \eta^{y_1, ..., y_n}_{y_1, ..., y_n}.
	$$
	Thus for fixed $y_1$, ..., $y_n$ the arrow $\eta^{x_1, ..., x_n}_{y_1, ..., y_n}$ does not depend on $x_1$, ..., $x_n$ and, denoting $\eta^{y_1, ..., y_n}_{y_1, ..., y_n}$ by $\eta_{y_1, ..., y_n}$, the family
	$$
	(\eta_{y_1, ..., y_n})_{(y_1, ..., y_n) \in \mathcal A_1 \times \cdots \times \mathcal A_n}
	$$
	is a natural transformation from a uniquely determined $\overline{\tle}$-shaped pre-multimorphism $\overline{F}$ from $(\mathcal A_1, ..., \mathcal A_n)$ to $\mathcal C$. Explicitly, $\overline{F}$ maps an object $(x_1, ..., x_n)$ to the domain of $\eta_{x_1, ..., x_n}$, and an arrow $(f_1, ..., f_n):(x_1,...,x_n) \rightarrow (y_1, ..., y_n)$ in $\mathcal A_1 \times \cdots \times \mathcal A_n$ to the unique dashed arrow making the diagram
	\[
	\begin{tikzcd}
		\overline{F}(x_1, ..., x_n) \arrow[]{rr}{\eta_{x_1, ..., x_n}} \arrow[dashed]{d}{} && F(x_1, ..., x_n) \arrow[]{d}{F(f_1, ..., f_n)} \\
		\overline{F}(y_1, ..., y_n) \arrow[swap]{rr}{\eta_{y_1, ..., y_n}} && F(y_1, ..., y_n)
	\end{tikzcd}
	\]
	commute.
\end{construction}

\begin{proposition}
\label{prop: transporting multimorphism gives multimorphism}
	In the notation of Construction \ref{constr: transporting pre-multimorphisms}, if $F$ is a $\tle$-shaped multimorphism, then $\overline{F}$ is a $\overline{\tle}$-shaped multimorphism.
\end{proposition}

\begin{proof}
Suppose given for each $i = 1$, ..., $n$ a length-$1$ distinguished square $Q_i$ in $\mathcal A_i$. As remarked in Construction \ref{constr: natural transformation between diagrams of boundaries}, $\overline{F}\langle Q_1, ..., Q_n \rangle_{\overline{\tle}}$ is distinguished if and only if $F\langle Q_1, ..., Q_n \rangle_\tle$ is distinguished. By assumption, the latter is distinguished, hence so is the former, as required.
\end{proof}

The following result is an immediate application of the characterization from Proposition \ref{prop: transporting cellular diagrams} to the setting of \mbox{(pre-)}multimorphisms:

\begin{proposition}
\label{prop: transporting pre-multimorphisms (characterization)}
Let $\tle$ and $\overline{\tle}$ be local linearizations of $T(\mathcal A_1) \times \cdots \times T(\mathcal A_n)$. For a $\tle$-shaped (pre-)multimorphism $F:\mathcal A_1 \times \cdots \mathcal A_n \rightarrow \mathcal C$, there exists a unique pair $(\overline{F}, \eta)$ consisting of a $\overline{\tle}$-shaped (pre-)multimorphism $\overline{\tle}$-shaped multimorphism $\overline{F}:\mathcal A_1 \times \cdots \times \mathcal A_n \rightarrow \mathcal C$ and a natural isomorphism $\eta:\overline{F} \Rightarrow F$ such that for all $(a_1, ..., a_n) \in \mathcal A_1 \times \cdots \times \mathcal A_n$, the commutative square
\[
\widedsqua{\overline{F}(a_1, ..., a_n)}{F(a_1, ..., a_n)}{\overline{F}[\partial(a_1, ..., a_n)]_{\overline{\tle}}}{F[\partial(a_1, ..., a_n)]_\tle,}{\eta_{a_1, ..., a_n}}{}{}{}
\]
where the bottom square is induced by $\eta$ and functoriality of limits, is distinguished. Explicitly, $\overline{F}$ and $\eta$ are given by Construction \ref{constr: transporting pre-multimorphisms}. \qed
\end{proposition}

\begin{notation}
\label{not: transport pre-multimorphisms}
We will write $\mathfrak T_{\tle,\overline{\tle}}$ for both functions $F \mapsto \overline{F}$ from Proposition \ref{prop: transporting pre-multimorphisms (characterization)}, that is, $\Hom^?_\tle(\mathcal A_1, ..., \mathcal A_n; \mathcal C) \rightarrow \Hom^?_{\overline{\tle}}(\mathcal A_1, ..., \mathcal A_n; \mathcal C)$ and its co/restriction $\Hom_\tle(\mathcal A_1, ..., \mathcal A_n; \mathcal C) \rightarrow \Hom_{\overline{\tle}}(\mathcal A_1, ..., \mathcal A_n; \mathcal C)$.
\end{notation}

\begin{corollary}
\label{cor: functoriality of transport}
The following hold for any local linearizations $\tle$, $\overline{\tle}$, $\doverline{\tle}$ of $T(\mathcal A_1) \times \cdots \times T(\mathcal A_n)$:
\begin{itemize}
	\item $\mathfrak T_{\tle,\tle}$ is the identity map;
	
	\item $\mathfrak T_{\overline{\tle},\doverline{\tle}} \circ \mathfrak T_{\tle,\overline{\tle}} = \mathfrak T_{\tle,\doverline{\tle}}$.
\end{itemize}
In particular, $\mathfrak T_{\tle,\overline{\tle}}$ is bijective for all $\tle$, $\overline{\tle}$.
\end{corollary}

\subsubsection{Symmetry of multimorphisms}
\label{subsubsec: symmetry of multimorphisms}

Consider precontextual categories $\mathcal A_1$, ..., $\mathcal A_n$, and a contextual category $\mathcal C$.

\vspace{0.5em}

As usual, $\btle$ denotes the lexicographic local linearization of the cartesian product of trees $T(\mathcal A_1) \times \cdots \times T(\mathcal A_n)$ (see Definition \ref{def: lexicographic product}).

For a permutation $\sigma \in S_n$, we write $\btle_\sigma$ for the lexicographic local linearization of $T(\mathcal A_{\sigma 1}) \times \cdots \times T(\mathcal A_{\sigma n})$. Given a local linearization $\tle$ of $T(\mathcal A_{\sigma 1}) \times \cdots \times T(\mathcal A_{\sigma n})$ we let $\sigma^*(\tle)$ be the local linearization of $T(\mathcal A_1) \times \cdots \times T(\mathcal A_n)$ induced from $\tle$ via the cartesian symmetry isomorphism $\underline{\sigma}:T(\mathcal A_1) \times \cdots \times T(\mathcal A_n) \rightarrow T(\mathcal A_{\sigma 1}) \times \cdots \times T(\mathcal A_{\sigma n})$.

\vspace{0.5em}

Firstly, note that a functor $F:\mathcal A_{\sigma 1} \times \cdots \times \mathcal A_{\sigma n} \rightarrow \mathcal C$ is a $\tle$-shaped (pre-)multimorphism if and only if $F \circ \underline{\sigma}:\mathcal A_1 \times \cdots \times \mathcal A_n \rightarrow \mathcal C$ is a $\sigma^*(\tle)$-multimorphism. Hence we have bijections
$$
\Hom^?_\tle(\mathcal A_{\sigma 1}, ..., \mathcal A_{\sigma n}; \mathcal C) \overset{- \circ \underline{\sigma}}{\cong} \Hom^?_{\sigma^*(\tle)}(\mathcal A_1, ..., \mathcal A_n; \mathcal C), \qquad \Hom_\tle(\mathcal A_{\sigma 1}, ..., \mathcal A_{\sigma n}; \mathcal C) \overset{- \circ \underline{\sigma}}{\cong} \Hom_{\sigma^*(\tle)}(\mathcal A_1, ..., \mathcal A_n; \mathcal C). 
$$
Now, specializing to the case where $\tle$ is $\btle_\sigma$ and composing with $\mathfrak T_{\sigma^*(\btle_\sigma),\btle}$ (see Corollary \ref{cor: functoriality of transport}), we obtain a bijection
$$
\mathfrak T_{\sigma^*(\btle_\sigma),\btle}(- \circ \underline{\sigma}):\Hom^?_{\btle_\sigma}(\mathcal A_{\sigma 1}, ..., \mathcal A_{\sigma n}; \mathcal C) \longrightarrow \Hom^?_\btle(\mathcal A_1, ..., \mathcal A_n; \mathcal C),
$$
which in turn co/restricts to a bijection $\Hom_{\btle_\sigma}(\mathcal A_{\sigma 1}, ..., \mathcal A_{\sigma n}; \mathcal C) \cong \Hom_\btle(\mathcal A_1, ..., \mathcal A_n; \mathcal C)$. However, note that these maps do not encode an important related structure, namely, the natural isomorphism $\eta$ from Proposition \ref{prop: transporting pre-multimorphisms (characterization)}. Precisely, that proposition yields a diagram
\[\begin{tikzcd}[ampersand replacement=\&]
	\& {\mathcal C} \\
	\\
	{\mathcal A_1 \times \cdots \times \mathcal A_n} \&\& {\mathcal A_{\sigma 1} \times \cdots \times \mathcal A_{\sigma n}}
	\arrow[""{name=0, anchor=center, inner sep=0}, "{F \circ \underline{\sigma}}"{description}, from=3-1, to=1-2]
	\arrow[""{name=1, anchor=center, inner sep=0}, "{\overline{F}}", shift left=3, curve={height=-30pt}, from=3-1, to=1-2]
	\arrow["{\underline{\sigma}}"', from=3-1, to=3-3]
	\arrow[""{name=2, anchor=center, inner sep=0}, "F"', from=3-3, to=1-2]
	\arrow["{=}"{description}, shift right=4, draw=none, from=0, to=2]
	\arrow["\eta", Rightarrow,shorten=10pt,from=1, to=0]
\end{tikzcd}\]
where $\overline{F} = \mathfrak T_{\sigma^*(\btle_\sigma),\btle}(F \circ \underline{\sigma})$, and we have:

\begin{proposition}
\label{prop: shuffling diagram}
Let $F:\mathcal A_{\sigma 1} \times \cdots \mathcal A_{\sigma n} \rightarrow \mathcal C$ be a (pre-)multimorphism (that is, with respect to the lexicographic local linearization $\btle_\sigma$). Then there exists a unique pair $(F',\phi)$ consisting of a (pre-)multimorphism $F':\mathcal A_1 \times \cdots \times \mathcal A_n \rightarrow \mathcal C$ and a natural isomorphism $\phi$ as in
\[
\tag{\texttt{*}}
\begin{tikzcd}[ampersand replacement=\&]
	\& {\mathcal C} \\
	\\
	{\mathcal A_1 \times \cdots \times \mathcal A_n} \&\& {\mathcal A_{\sigma 1} \times \cdots \times \mathcal A_{\sigma n}}
	\arrow[""{name=0, anchor=center, inner sep=0}, "F'", from=3-1, to=1-2]
	\arrow["{\underline{\sigma}}"', from=3-1, to=3-3]
	\arrow[""{name=1, anchor=center, inner sep=0}, "F"', from=3-3, to=1-2]
	\arrow["\phi", Rightarrow,shift right=2,shorten=15pt, from=0, to=1]
\end{tikzcd}\]
such that for all $(a_1, ..., a_n) \in \mathcal A_1 \times \cdots \times \mathcal A_n$, the commutative square
\[
\tag{\texttt{**}}
\widedsqua{F'(a_1, ..., a_n)}{F(a_{\sigma 1}, ..., a_{\sigma n})}{F'[\partial(a_1, ..., a_n)]_{\btle}}{F[\partial(a_{\sigma 1}, ..., a_{\sigma n})]_{\btle_\sigma},}{\phi_{a_1, ..., a_n}}{}{}{}
\]
where the bottom square is induced by $\phi$ and functoriality of limits, is distinguished. In the notation of the above discussion, such $F'$, $\phi$ are $\overline{F}$, $\eta$, respectively. \qed
\end{proposition}

\begin{definition}
\label{def: shuffling diagram}
A triple $(F,F',\phi)$ as in Proposition \ref{prop: shuffling diagram} --- which we depict as in ($\texttt{*}$) --- will be referred to as a \emph{shuffling diagram}.
\end{definition}

\begin{remark}
\label{rem: shuffling square closure}
The fact that length-$1$ distinguished squares in $\mathcal C$ are closed under horizontal composition implies, due to the defining condition on a shuffling diagram encoded by ($\texttt{**}$) in Proposition \ref{prop: shuffling diagram}, that shuffling diagrams are closed under horizontal composition: suppose given $\sigma$, $\tau \in S_n$ and a diagram
\[\begin{tikzcd}[ampersand replacement=\&]
	\&\& {\mathcal C} \\
	\\
	{\mathcal A_1 \times \cdots \times \mathcal A_n} \&\& {\mathcal A_{\sigma 1} \times \cdots \times \mathcal A_{\sigma n}} \&\& {\mathcal A_{\tau 1} \times \cdots \times \mathcal A_{\tau n}}
	\arrow[""{name=0, anchor=center, inner sep=0}, "{F''}", from=3-1, to=1-3]
	\arrow["{\underline{\sigma}}"', from=3-1, to=3-3]
	\arrow[""{name=1, anchor=center, inner sep=0}, "{F'}", from=3-3, to=1-3]
	\arrow["{\underline{\sigma^{-1}\tau}}"', from=3-3, to=3-5]
	\arrow[""{name=2, anchor=center, inner sep=0}, "F"', from=3-5, to=1-3]
	\arrow["{\phi'}", Rightarrow,shorten=15pt, from=0, to=1]
	\arrow["\phi", Rightarrow,shorten=15pt, from=1, to=2]
\end{tikzcd}\]
where the left and right triangles are shuffling diagrams; then the composite outer triangle, i.e. $(F,F'',\; \phi \underline{\sigma} \circ \phi')$, is a shuffling diagram.

It also follows from Proposition \ref{prop: shuffling diagram} that shuffling diagrams are stable under pre- and post-composition with contextual functors in the following sense: if a triangle of the form ($\texttt{*})$ is a shuffling diagram, then for any contextual functors
$$
U_1:\mathcal A'_1 \rightarrow \mathcal A_1, \quad ..., \quad U_n:\mathcal A'_n \rightarrow \mathcal A_n, \quad V:\mathcal C \rightarrow \mathcal C',
$$
the triangle
\[
\begin{tikzcd}[ampersand replacement=\&]
	\& {\mathcal C'} \\
	\\
	{\mathcal A'_1 \times \cdots \times \mathcal A'_n} \&\& {\mathcal A'_{\sigma 1} \times \cdots \times \mathcal A'_{\sigma n}}
	\arrow[""{name=0, anchor=center, inner sep=0}, "V \circ F' \circ (U_1 \times \cdots \times U_n)", from=3-1, to=1-2]
	\arrow["{\underline{\sigma}}"', from=3-1, to=3-3]
	\arrow[""{name=1, anchor=center, inner sep=0}, "V \circ F \circ (U_{\sigma 1} \times \cdots \times U_{\sigma n})"', from=3-3, to=1-2]
	\arrow["V\phi(U_1 \times \cdots \times U_n)", Rightarrow,shift right=3,shorten=15pt, from=0, to=1]
\end{tikzcd}\]
is a shuffling diagram.
\end{remark}

We now have:

\begin{proposition}
\label{prop: symmetry of multimorphisms is natural}
Consider $n \ge 1$, precontextual categories $\mathcal A_1$, ..., $\mathcal A_n$, and $\sigma \in S_n$. Then we have a natural isomorphism
$$
\Hom(\mathcal A_{\sigma 1},...,\mathcal A_{\sigma n};-) \cong \Hom(\mathcal A_1,...,\mathcal A_n;-)
$$
between functors $\Cont \rightarrow \Set$ whose $\mathcal C$-component sends a multimorphism $F:\mathcal A_{\sigma 1} \times \cdots \times \mathcal A_{\sigma n} \rightarrow \mathcal C$ to the unique multimorphism $F':\mathcal A_1 \times \cdots \times \mathcal A_n \rightarrow \mathcal C$ that fits into a shuffling diagram $(F,F',\phi)$ (see ($\texttt{*}$) above).

Moreover, this isomorphism is natural in $\mathcal A_1$, ..., $\mathcal A_n$: for precontextual categories $\mathcal A'_1$, ..., $\mathcal A'_n$ and morphisms $U_1:\mathcal A'_1 \rightarrow \mathcal A_1$, ..., $U_n:\mathcal A'_n \rightarrow \mathcal A_n$, the diagram
\[\begin{tikzcd}[ampersand replacement=\&]
	{\Hom(\mathcal A_{\sigma 1}, ..., \mathcal A_{\sigma n};\mathcal C)} \& {\Hom(\mathcal A_1, ..., \mathcal A_n;\mathcal C)} \\
	{\Hom(\mathcal A'_{\sigma 1}, ..., \mathcal A'_{\sigma n};\mathcal C)} \& {\Hom(\mathcal A'_1, ..., \mathcal A'_n;\mathcal C)}
	\arrow["\cong", from=1-1, to=1-2]
	\arrow["{-\; \circ \; (U_{\sigma 1} \times \cdots \times U_{\sigma n})}"', from=1-1, to=2-1]
	\arrow["{-\; \circ \; (U_1 \times \cdots \times U_n)}", from=1-2, to=2-2]
	\arrow["\cong"', from=2-1, to=2-2]
\end{tikzcd}\]
commutes. \qed
\end{proposition}

\begin{remark}
\label{rem: shuffling is trivial on length-1 objects}
In the notation of Proposition \ref{prop: shuffling diagram}, if $a_1 \in \mathcal A_1$, ..., $a_n \in \mathcal A_n$ have length $1$, then the objects in the bottom row of ($\texttt{**}$) are both $1_\mathcal C$, which implies that $\phi_{a_1, ..., a_n}$ is an identity morphism.

In particular, if $\mathcal A_1$, ..., $\mathcal A_n$ only have objects of length at most $1$, then $\phi$ is the identity natural transformation. We will come back to this observation in \S\ref{subsec: embedding Cat into Cont} when studying the $\Cat$-enriched structure of $\Cont$.
\end{remark}

\subsubsection{Permutative morphisms}
\label{subsubsec: permutative morphisms}

We will now discuss an alternative way of encoding the content of shuffling diagrams (Definition \ref{def: shuffling diagram}). The idea is that while shuffling diagrams compare multimorphisms with a common codomain, the approach introduced below allows us to compare multimorphisms whose codomains are connected by suitable maps, the \emph{permutative morphisms}.

\begin{definition}
\label{def: permutative morphism}
Let $\mathcal A_1$, ..., $\mathcal A_n$ be contextual categories. We define $\Perm(\mathcal A_1, ..., \mathcal A_n)$ as the following category:
\begin{itemize}
	\item an object is a triple $(\mathcal B,\sigma,F)$ consisting of a contextual category $\mathcal B$, a permutation $\sigma \in S_n$, and a multimorphism $F \in \Hom(\mathcal A_{\sigma 1}, ..., \mathcal A_{\sigma n};\mathcal B)$.
	
	\item a morphism from $(\mathcal B,\sigma,F)$ to $(\mathcal C,\tau,G)$, which we call a \emph{permutative morphism}, is a contextual functor $P:\mathcal B \rightarrow \mathcal C$ such that there exists a (necessarily unique) shuffling diagram of the form
	\[
	\begin{tikzcd}[ampersand replacement=\&]
		\& {\mathcal C} \\
		\\
		{\mathcal A_{\sigma 1} \times \cdots \times \mathcal A_{\sigma n}} \&\& {\mathcal A_{\tau 1} \times \cdots \times \mathcal A_{\tau n};}
		\arrow[""{name=0, anchor=center, inner sep=0}, "P \circ F", from=3-1, to=1-2]
		\arrow["{\underline{\sigma^{-1}\tau}}"', from=3-1, to=3-3]
		\arrow[""{name=1, anchor=center, inner sep=0}, "G"', from=3-3, to=1-2]
		\arrow["\phi", Rightarrow,shift right=2,shorten=15pt, from=0, to=1]
	\end{tikzcd}\]
	
	\item permutative morphisms are composed in the usual way for contextual functors --- the validity of this definition, that is, a composite of permutative morphisms being permutative, follows from Remark \ref{rem: shuffling square closure}.
\end{itemize}
\end{definition}

It is useful to picture a permutative morphism as a square such as
\[\begin{tikzcd}[ampersand replacement=\&]
	{\mathcal B} \&\& {\mathcal C} \\
	{\mathcal A_{\sigma 1} \times \cdots \times \mathcal A_{\sigma n}} \&\& {\mathcal A_{\tau 1} \times \cdots \times \mathcal A_{\tau n}.}
	\arrow["P", from=1-1, to=1-3]
	\arrow[""{name=0, anchor=center, inner sep=0}, "F", from=2-1, to=1-1]
	\arrow["{\underline{\sigma^{-1}\tau}}"', from=2-1, to=2-3]
	\arrow[""{name=1, anchor=center, inner sep=0}, "G"', from=2-3, to=1-3]
	\arrow["\phi", Rightarrow,shorten=20pt, from=0, to=1]
\end{tikzcd}\]
Composition is then expressed by horizontally pasting the given natural transformations.

\vspace{0.5em}

The reason why we consider permutative morphisms is the following result:

\begin{proposition}
\label{prop: permutative morphisms main property}
Suppose given permutative morphisms
\[\begin{tikzcd}[ampersand replacement=\&, column sep=small]
	{\mathcal B} \&\& {\mathcal C} \&\& {\mathcal B} \&\& {\mathcal C} \\
	{\mathcal A_1 \times \cdots \times \mathcal A_n} \&\& {\mathcal A_{\sigma 1} \times \cdots \times \mathcal A_{\sigma n},} \&\& {\mathcal A_1 \times \cdots \times \mathcal A_n} \&\& {\mathcal A_{\sigma 1} \times \cdots \times \mathcal A_{\sigma n}.}
	\arrow["P", from=1-1, to=1-3]
	\arrow["Q", from=1-5, to=1-7]
	\arrow[""{name=0, anchor=center, inner sep=0}, "F", from=2-1, to=1-1]
	\arrow["{\underline{\sigma}}"', from=2-1, to=2-3]
	\arrow[""{name=1, anchor=center, inner sep=0}, "G"', from=2-3, to=1-3]
	\arrow[""{name=2, anchor=center, inner sep=0}, "F", from=2-5, to=1-5]
	\arrow["{\underline{\sigma}}"', from=2-5, to=2-7]
	\arrow[""{name=3, anchor=center, inner sep=0}, "G"', from=2-7, to=1-7]
	\arrow["\phi", Rightarrow,shorten=20pt, from=0, to=1]
	\arrow["\psi", Rightarrow,shorten=20pt, from=2, to=3]
\end{tikzcd}\]
Then $P \circ F = Q \circ F$.
\end{proposition}

\begin{proof}
By Proposition \ref{prop: shuffling diagram}, the shuffling diagrams
\[\begin{tikzcd}[ampersand replacement=\&,column sep=small]
	\& {\mathcal C} \&\&\&\& {\mathcal C} \\
	\\
	{\mathcal A_1 \times \cdots \times \mathcal A_n} \&\& {\mathcal A_{\sigma 1} \times \cdots \times \mathcal A_{\sigma n},} \&\& {\mathcal A_1 \times \cdots \times \mathcal A_n} \&\& {\mathcal A_{\sigma 1} \times \cdots \times \mathcal A_{\sigma n}}
	\arrow[""{name=0, anchor=center, inner sep=0}, "{P \circ F}", from=3-1, to=1-2]
	\arrow["{\underline{\sigma}}"', from=3-1, to=3-3]
	\arrow[""{name=1, anchor=center, inner sep=0}, "G"', from=3-3, to=1-2]
	\arrow[""{name=2, anchor=center, inner sep=0}, "{Q \circ F}", from=3-5, to=1-6]
	\arrow["{\underline{\sigma}}"', from=3-5, to=3-7]
	\arrow[""{name=3, anchor=center, inner sep=0}, "G"', from=3-7, to=1-6]
	\arrow["\phi", Rightarrow,shorten=15pt, from=0, to=1]
	\arrow["\psi", Rightarrow,shorten=15pt, from=2, to=3]
\end{tikzcd}\]
are equal.
\end{proof}

\begin{remark}
The above proposition is particularly useful if --- as will occur when we discuss $n$-ary tensor products of contextual categories --- $F$ is known to be a universal multimorphism out of $(\mathcal A_1, ..., \mathcal A_n)$, as this will imply $P = Q$.
\end{remark}

\subsection{Comparison with bimorphisms (as introduced previously)}
\label{subsec: comparison with bimorphisms as introduced previously}

We will now verify that $2$-ary morphisms of contextual categories, in the sense of Definition \ref{def: multimorphism}, are precisely the bimorphisms from Definition \ref{def: bimorphism}.

\begin{proposition}
\label{prop: comparison bimorphisms and 2-ary maps}
Given contextual categories $\mathcal A$, $\mathcal B$ and $\mathcal C$, a functor $H:\mathcal A \times \mathcal B \rightarrow \mathcal C$ is a bimorphism from $(\mathcal A,\mathcal B)$ to $\mathcal C$ in the sense of Definition \ref{def: bimorphism} if and only if it is a multimorphism (or $2$-ary morphism) in the sense of Definition \ref{def: multimorphism}.
\end{proposition}

\begin{proof}
\textbf{($\implies$)}
Assume that $H$ is a bimorphism. Firstly, let us show that $H$ is a $\btle$-shaped pre-multimorphism (where $\btle$ is the lexicographic local linearization on $T(\mathcal A) \times T(\mathcal B)$). For $a \in A$, $b \in B$, let
$$
a = a_m \twoheadrightarrow a_{m-1} \twoheadrightarrow \cdots \twoheadrightarrow a_1 \twoheadrightarrow a_0 = 1_\mathcal A,
$$
$$
b = b_n \twoheadrightarrow b_{n-1} \twoheadrightarrow \cdots \twoheadrightarrow b_1 \twoheadrightarrow b_0 = 1_\mathcal B
$$
be the respective towers of length-$1$ display maps. Let
$$
\underline{a}:\mathcal O_m^{\text{pre}} \longrightarrow \mathcal A
$$
$$
\underline{b}:\mathcal O_n^{\text{pre}} \longrightarrow \mathcal B
$$
be the morphisms of precontextual categories that classify $a$ and $b$, respectively. Restricting $H$ along $\underline{a} \times \underline{b}$ yields a bimorphism $H':(\mathcal O_m^{\text{pre}}, \mathcal O_n^{\text{pre}}) \rightarrow \mathcal C$, and by Proposition \ref{prop: characterization of maps to exponential, sym} we get a morphism $F:\mathcal O_n^{\text{pre}} \rightarrow \mathcal C^{\mathcal O_m^{\text{pre}}}$. As the latter restricts to a cellular diagram $\{1 < \cdots < n\}^{\text{op}} \rightarrow \mathcal C^{\mathcal O_m^{\text{pre}}}$, it follows from Proposition \ref{prop: cellular diagram indexed by product by an ordinal} that $H'$ restricts to a cellular diagram
$$
\{1,...,m\}^{\text{op}} \times \{1, ..., n\}^{\text{op}} \longrightarrow \mathcal C
$$
where the left-hand side is endowed with the lexicographic local linearization. It follows that $H$ is a $\btle$-shaped pre-multimorphism. To check that it is a multimorphism, we use the criterion from Proposition \ref{prop: criterion multimorphism}.

On the one hand, suppose that we have length-$1$ distinguished squares $Q_1$, $Q_2$ in $\mathcal A$, $\mathcal B$, resp., of the forms
\[\begin{tikzcd}[ampersand replacement=\&]
	{a^0} \& {a^{'0}} \&\& {b^0} \& {b^0} \\
	{a^1} \& {a^{'1},} \&\& {b^1} \& {b^1.}
	\arrow["f", from=1-1, to=1-2]
	\arrow["p"', two heads, from=1-1, to=2-1]
	\arrow["{p'}", two heads, from=1-2, to=2-2]
	\arrow["id", from=1-4, to=1-5]
	\arrow["q"', two heads, from=1-4, to=2-4]
	\arrow["q", two heads, from=1-5, to=2-5]
	\arrow["g"', from=2-1, to=2-2]
	\arrow["id"', from=2-4, to=2-5]
\end{tikzcd}\]
Then $H\langle Q_1,Q_2\rangle_\btle$ equals the comparison square between the gap map of
\[\begin{tikzcd}[ampersand replacement=\&]
	{H(a^0,b^0)} \&\& {H(a^0,b^1)} \\
	{H(a^1,b^0)} \&\& {H(a^1,b^1)}
	\arrow["{H(id,q)}", from=1-1, to=1-3]
	\arrow["{H(p,id)}"', two heads, from=1-1, to=2-1]
	\arrow["{H(p,id)}", two heads, from=1-3, to=2-3]
	\arrow["{H(id,q)}"', from=2-1, to=2-3]
\end{tikzcd}\]
and that of
\[\begin{tikzcd}[ampersand replacement=\&]
	{H(a^{'0},b^0)} \&\& {H(a^{'0},b^1)} \\
	{H(a^{'1},b^0)} \&\& {H(a^{'1},b^1).}
	\arrow["{H(id,q)}", from=1-1, to=1-3]
	\arrow["{H(p',id)}"', two heads, from=1-1, to=2-1]
	\arrow["{H(p',id)}", two heads, from=1-3, to=2-3]
	\arrow["{H(id,q)}"', from=2-1, to=2-3]
\end{tikzcd}\]
Since $H$ corresponds to a morphism $\mathcal B \rightarrow \mathcal C^\mathcal A$ (Proposition \ref{prop: characterization of maps to exponential, sym}), it follows from the construction of display maps in $\mathcal C^\mathcal A$ (given by indexed sorts, Definition \ref{def: indexed sort}) that $H\langle Q_1, Q_2 \rangle_\btle$ is a distinguished square.

On the other hand, consider length-$1$ distinguished squares $Q_1$, $Q_2$ in $\mathcal A$, $\mathcal B$, resp., of the forms
\[\begin{tikzcd}[ampersand replacement=\&]
	{a^0} \& {a^0} \&\& {b^0} \& {b^{'0}} \\
	{a^1} \& {a^1,} \&\& {b^1} \& {b^{'1}}
	\arrow["id", from=1-1, to=1-2]
	\arrow["p"', two heads, from=1-1, to=2-1]
	\arrow["p", two heads, from=1-2, to=2-2]
	\arrow["g", from=1-4, to=1-5]
	\arrow["q"', two heads, from=1-4, to=2-4]
	\arrow["{q'}", two heads, from=1-5, to=2-5]
	\arrow["id"', from=2-1, to=2-2]
	\arrow["f"', from=2-4, to=2-5]
\end{tikzcd}\]
Then $H_\tle(Q_1,Q_2)$ is the comparison square between the gap map of
\[\begin{tikzcd}[ampersand replacement=\&]
	{H(a^0,b^0)} \&\& {H(a^0,b^1)} \\
	{H(a^1,b^0)} \&\& {H(a^1,b^1)}
	\arrow["{H(id,q)}", from=1-1, to=1-3]
	\arrow["{H(p,id)}"', two heads, from=1-1, to=2-1]
	\arrow["{H(p,id)}", two heads, from=1-3, to=2-3]
	\arrow["{H(id,q)}"', from=2-1, to=2-3]
\end{tikzcd}\]
and that of
\[\begin{tikzcd}[ampersand replacement=\&]
	{H(a^0,b^{'0})} \&\& {H(a^0,b^{'1})} \\
	{H(a^1,b^{'0})} \&\& {H(a^1,b^{'1}).}
	\arrow["{H(id,q')}", from=1-1, to=1-3]
	\arrow["{H(p,id)}"', two heads, from=1-1, to=2-1]
	\arrow["{H(p,id)}", two heads, from=1-3, to=2-3]
	\arrow["{H(id,q')}"', from=2-1, to=2-3]
\end{tikzcd}\]
By condition (iii) from Proposition \ref{prop: characterization of maps to exponential, sym}, $H\langle Q_1, Q_2 \rangle_\btle$ is distinguished. This concludes the proof that $H$ is a $2$-ary morphism.

\vspace{0.5em}

\textbf{($\impliedby$)} Conversely, assume that $H$ is a $2$-ary morphism. Writing $A$, $C$ for the underlying categories of $\mathcal A$, $\mathcal C$, let $F:\mathcal B \rightarrow C^A$ be the adjunct of $H$. By the correspondence between bimorphisms $(\mathcal A,\mathcal B) \rightarrow \mathcal C$ and morphisms $\mathcal B \rightarrow \mathcal C^\mathcal A$ (Proposition \ref{prop: characterization of maps to exponential, sym}), it suffices to prove that $F$ factors through $D(\mathcal A,\mathcal C) \subset C^A$ and defines a morphism of categories with attributes $\att(\mathcal B) \rightarrow D_{\text{att}}(\mathcal A,\mathcal C)$. For the first part, note that given a display map $a' \rightarrow a$ in $\mathcal A$ and $b \in \mathcal B$, the arrow
$$
H(a',b) \overset{H(p,id)}{\longrightarrow} H(a,b)
$$
is a display map since $(a,b) \btle (a',b)$ and $H|_{(a',b)^{\le}}$ is cellular with respect to $\btle$. Now, let us prove that $F:\mathcal B \rightarrow D(\mathcal A,\mathcal C)$ defines a morphism of categories with attributes.
\begin{itemize}
	\item \textbf{$F$ preserves the distinguished terminal object.} $F(1_\mathcal B) = H(-,1_\mathcal B)$ is the constant functor on $1_\mathcal C$ as $H$ is $2$-ary.
	
	\item \textbf{$F$ sends length-$1$ display maps to indexed sorts.} Let $q:b^0 \rightarrow b^1$ be a length-$1$ display map in $\mathcal B$. If $p:a^0 \rightarrow a^1$ is a length-$1$ display map in $\mathcal A$, the square
	\[
	\squa{F(b^0)(a^0)}{F(b^1)(a^0)}{F(b^0)(a^1)}{F(b^1)(a^1)}{F(q)_{a^0}}{F(q)_{a^1}}{F(b^0)(p)}{F(b^1)(p)}
	\]
	equals
	\[
	\squa{H(a^0,b^0)}{H(a^0,b^1)}{H(a^1,b^0)}{H(a^1,b^1),}{H(id,q)}{H(id,q)}{H(p,id)}{H(p,id)}
	\]
	which is a relative length-$1$ display map. Moreover, if
	\[
	\dsqua{a^0}{a^{'0}}{a^1}{a^{'1}}{f}{g}{p}{p'}
	\]
	is a length-$1$ distinguished square in $\mathcal A$, then the induced comparison square between the gap map of
	\[\begin{tikzcd}[ampersand replacement=\&]
		{H(a^0,b^0)} \& {H(a^0,b^1)} \\
		{H(a^1,b^0)} \& {H(a^1,b^1)}
		\arrow["{H(id,q)}", from=1-1, to=1-2]
		\arrow["{H(p,id)}"', two heads, from=1-1, to=2-1]
		\arrow["{H(p,id)}", two heads, from=1-2, to=2-2]
		\arrow["{H(id,q)}"', from=2-1, to=2-2]
	\end{tikzcd}\]
	and that of
	\[\begin{tikzcd}[ampersand replacement=\&]
		{H(a^{'0},b^0)} \& {H(a^{'0},b^1)} \\
		{H(a^{'1},b^0)} \& {H(a^{'1},b^1)}
		\arrow["{H(id,q)}", from=1-1, to=1-2]
		\arrow["{H(p',id)}"', two heads, from=1-1, to=2-1]
		\arrow["{H(p',id)}", two heads, from=1-2, to=2-2]
		\arrow["{H(id,q)}"', from=2-1, to=2-2]
	\end{tikzcd}\]
	equals $H \langle Q_1,Q_2 \rangle_\btle$ where $Q_1$, $Q_2$ are, respectively,
	\[\begin{tikzcd}[ampersand replacement=\&]
		{a^0} \& {a^{'0}} \&\& {b^0} \& {b^0} \\
		{a^1} \& {a^{'1},} \&\& {b^1} \& {b^1.}
		\arrow["f", from=1-1, to=1-2]
		\arrow["p"', two heads, from=1-1, to=2-1]
		\arrow["{p'}", two heads, from=1-2, to=2-2]
		\arrow["id", from=1-4, to=1-5]
		\arrow["q"', two heads, from=1-4, to=2-4]
		\arrow["q", two heads, from=1-5, to=2-5]
		\arrow["g"', from=2-1, to=2-2]
		\arrow["id"', from=2-4, to=2-5]
	\end{tikzcd}\]
	As $H$ is $2$-ary, $H \langle Q_1,Q_2 \rangle_\btle$ is distinguished.
	
	\item \textbf{$F$ sends length-$1$ distinguished squares to indexed distinguished squares.} Suppose given a length-$1$ distinguished square $Q_1$, say
	\[
	\dsqua{b^0}{b^{'0}}{b^1}{b^{'1},}{g}{f}{q}{q'}
	\]
	in $\mathcal B$. By the previous item, $F(q)$ and $F(q')$ are indexed sorts, so $F(Q_1)$ satisfies condition IDS(i) from Remark \ref{rem: characterization indexed distinguished squares}. On the other hand, IDS(ii) follows from the fact that, letting $p:a^0 \rightarrow a_1$ be a length-$1$ display map in $\mathcal A$ and $Q_2$ be the square
	\[
	\dsqua{a^0}{a^0}{a^1}{a^1,}{id}{id}{p}{p}
	\]
	we have that $H\langle Q_1,Q_2 \rangle_\btle$ is distinguished.
\end{itemize}
This concludes the proof that $F$ defines a morphism $\att(\mathcal B) \rightarrow D_\att(\mathcal A,\mathcal C)$.
\end{proof}

\subsection{Relating multimorphisms and exponentials}
\label{subsec: multimorphisms and exponentials}

In what follows, we let $n \ge 2$ and consider precontextual categories $\mathcal A_1$, ..., $\mathcal A_n$ and a contextual category $\mathcal C$. To avoid ambiguity, we will denote their respective underlying categories by $A_1$, ..., $A_n$, $C$.

\begin{notation}
\label{not: set or category of multimorphisms for lexicographic local linearization}
We will write $\Hom(\cdots)$ for a set of multimorphisms $\Hom_\btle(\cdots)$, that is, of $\btle$-shaped multimorphisms where $\btle$ is the lexicographic local linearization. Similarly, we write $\iiHom(\cdots)$ for the category $\iiHom_\btle(\cdots)$.
\end{notation}

Our goal is to construct an isomorphism of categories
$$
\iiHom(\mathcal A_2, ..., \mathcal A_n; \mathcal C^{\mathcal A_1}) \cong \iiHom(\mathcal A_1, ..., \mathcal A_n; \mathcal C).
$$
We start by defining
$$
\widehat{(-)}: |\mathcal C^{\mathcal A_1}|^{A_2 \times \cdots \times A_n} \longrightarrow C^{A_1 \times \cdots \times A_n}
$$
as the composite functor
$$
|\mathcal C^{\mathcal A_1}|^{A_2 \times \cdots \times A_n} \longrightarrow (C^{A_1})^{A_2 \times \cdots \times A_n} \overset{\cong}{\longrightarrow} C^{A_1 \times \cdots \times A_n}
$$
where the left arrow is given by composition with the forgetful functor
\[
|\mathcal C^{\mathcal A_1}| \overset{(-)_\varheartsuit}{\longrightarrow} D(\mathcal A_1,\mathcal C) \hookrightarrow C^{A_1}.
\]

Note that $\widehat{(-)}$ is full-and-faithful. We will prove that its action on objects restricts to a bijection
$$
\Hom(\mathcal A_2, ..., \mathcal A_n; \mathcal C^{\mathcal A_1}) \rightarrow \Hom(\mathcal A_1, ..., \mathcal A_n; \mathcal C),
$$
thus inducing the desired isomorphism of categories. Firstly, we verify (Lemma \ref{lem: isomorphism of multihoms; pre-multimorphism}) that $\widehat{(-)}$ restricts to a map
$$
\eta:\Hom^?(\mathcal A_2, ..., \mathcal A_n; \mathcal C^{\mathcal A_1}) \longrightarrow \Hom^?(\mathcal A_1, ..., \mathcal A_n; \mathcal C).
$$
Then we prove in Lemma \ref{lem: isomorphism of multihoms; linear} that for $F \in \Hom^?(\mathcal A_2, ..., \mathcal A_n; \mathcal C^{\mathcal A_1})$, the pre-multimorphism $\eta(F)$ is $n$-ary precisely when $F$ is $(n-1)$-ary. In particular, $\eta$ restricts to a map
$$
\eta':\Hom(\mathcal A_2, ..., \mathcal A_n; \mathcal C^{\mathcal A_1}) \rightarrow \Hom(\mathcal A_1, ..., \mathcal A_n; \mathcal C).
$$
In Proposition \ref{prop: isomorphism of multihoms; section}, we construct a map
$$
\phi:\Hom(\mathcal A_1,...,\mathcal A_n;\mathcal C) \rightarrow \Hom^?(\mathcal A_2,...,\mathcal A_n;\mathcal C^{\mathcal A_1})
$$
such that $\eta\phi(H) = H$ for all $H$. Using the other direction of Lemma \ref{lem: isomorphism of multihoms; linear}, we have that $\phi$ restricts to a map
$$
\phi':\Hom(\mathcal A_1,...,\mathcal A_n;\mathcal C) \rightarrow \Hom(\mathcal A_2,...,\mathcal A_n;\mathcal C^{\mathcal A_1}),
$$
which is then a section of $\eta'$. In Theorem \ref{th: isomorphism of multihoms} we conclude the proof that $\eta'$ is bijective by verifying that it is injective.

\begin{notation}
For emphasis, we will now use distinct notations for the lexicographic local linearizations of distinct products of trees: we write $\btle$ for the one on $T(\mathcal A_1) \times \cdots \times T(\mathcal A_n)$, and $\overline{\btle}$ for the one on $T(\mathcal A_2) \times \cdots \times T(\mathcal A_n)$. 
\end{notation}

\begin{lemma}
\label{lem: isomorphism of multihoms; pre-multimorphism}
For each $F \in \Hom^?(\mathcal A_2,...,\mathcal A_n;\mathcal C^{\mathcal A_1})$ we have $\widehat{F} \in \Hom^?(\mathcal A_1,...,\mathcal A_n;\mathcal C)$.
\end{lemma}

\begin{proof}
Let $F \in \Hom^?(\mathcal A_2,...,\mathcal A_n;\mathcal C^{\mathcal A_1})$. For $x_1 \in A_1$, ..., $x_n \in A_n$, by assumption,
\[
\tag{\texttt{*}}
\widehat{F}|_{(x_2,...,x_n)^\le} :(x_2, ..., x_n)^{\le,op} \longrightarrow |\mathcal C^{\mathcal A_1}|
\]
is a cellular diagram in $\mathcal C^{\mathcal A_1}$ with respect to $\overline{\btle}$. Now, the inclusion of precontextual categories $x_1^{\le} \cup \{1_\mathcal A\} \hookrightarrow \mathcal A_1$ induces a morphism $\mathcal C^{\mathcal A_1} \rightarrow \mathcal C^{x_1^{\le,op}}$. Composing the latter with (\texttt{*}) we obtain a cellular diagram
\[
\tag{\texttt{**}}
(x_2, ..., x_n)^{\le,op} \longrightarrow |\mathcal C^{x_1^{\le,op}}|.
\]
in $\mathcal C^{x_1^{\le,op}}$. By Proposition \ref{prop: cellular diagram indexed by product by an ordinal}, the functor
$$
x_1^{\le,op} \times (x_2, ..., x_n)^{\le,op} \longrightarrow C
$$
corresponding to (\texttt{**}) is a cellular diagram in $\mathcal C$ with respect to the lexicographic product between $x_1^\le$ and $((x_2, ..., x_n)^\le,\overline{\btle})$. But this product is canonically isomorphic to the tree $((x_1, ..., x_n)^{\le},\btle)$, and it can be checked that the corresponding functor $(x_1, ..., x_n)^{\le,op} \longrightarrow |\mathcal C|$ coincides with the restriction of $\widehat{F}$. We conclude that $\widehat{F}$ is a pre-multimorphism.
\end{proof}

We let $\eta:\Hom^?(\mathcal A_2,...,\mathcal A_n;\mathcal C^{\mathcal A_1}) \rightarrow \Hom^?(\mathcal A_1,...,\mathcal A_n;\mathcal C)$ be the restriction of $\widehat{(-)}$ provided by Lemma \ref{lem: isomorphism of multihoms; pre-multimorphism}.

\begin{lemma}
\label{lem: isomorphism of multihoms; linear}
Let $F$ be an ($\overline{\btle}$-shaped) pre-multimorphism from $(\mathcal A_2, ..., \mathcal A_n)$ to $\mathcal C^{\mathcal A_1}$. Then $F$ is a multimorphism if and only if $\widehat{F}$ is a ($\btle$-shaped) multimorphism from $(\mathcal A_1, ..., \mathcal A_n)$ to $\mathcal C$.
\end{lemma}

\begin{proof}
Write $H$ for $\widehat{F}$. By Lemma \ref{lem: isomorphism of multihoms; pre-multimorphism}, $H$ is a pre-multimorphism. To study when it is a multimorphism, suppose given for each $i = 1$, ..., $n$ a length-$1$ distinguished square $Q_i$ in $\mathcal A_i$, say
\[
\tag{\texttt{*}}
\dsqua{a^0_i}{b^0_i}{a^1_i}{b^1_i.}{f^0_i}{f^1_i}{p_i}{q_i}
\]
We will characterize when $H \langle Q_1, ..., Q_n\rangle_\btle$ is a distinguished square. By functoriality of $H\langle Q_1, ..., Q_n \rangle_\btle$ in $Q_1$, ..., $Q_n$, it will be sufficient to do this in the following two cases:
\begin{enumerate}[label=(Case \arabic*)]
	\item For each $i = 2$, ..., $n$, the square $Q_i$ is of the form
	\[
	\dsqua{a^0_i}{a^0_i}{a^1_i}{a^1_i.}{id}{id}{p_i}{p_i}
	\]
	
	\item $Q_1$ is of the form
	\[
	\dsqua{a^0_1}{a^0_1}{a^1_1}{a^1_1.}{id}{id}{p_1}{p_1}
	\]
\end{enumerate}

\textbf{Proof that $H \langle Q_1, ..., Q_n \rangle_\btle$ is distinguished in case 1.}\\

In this case, $H\langle Q_1, ..., Q_n \rangle_\btle$ equals the comparison square
\[
\widedsqua{H(a_1^0,a_2^0...,a_n^0)}{H(b_1^0,a_2^0,...,a_n^0)}{\partial (H(a_1^0,a_2^0...,a_n^0))}{\partial (H(b_1^0,a_2^0,...,a_n^0))}{H(f_1^0,id, ..., id)}{\varphi}{p}{q}
\]
(where $p$, $q$ are length-$1$ display maps) between the gap maps of the left and right faces of the cube
\[\begin{tikzcd}
	& {} \\
	& {F[\partial(a_2^0,...,a_n^0)]_\varheartsuit(a_1^0)} && {F[\partial(a_2^0,...,a_n^0)]_\varheartsuit(b_1^0)} \\
	{F(a_2^0,...,a_n^0)_\varheartsuit(a_1^0)} && {F(a_2^0,...,a_n^0)_\varheartsuit(b_1^0)} \\
	& {F[\partial(a_2^0,...,a_n^0)]_\varheartsuit(a_1^1)} && {F[\partial(a_2^0,...,a_n^0)]_\varheartsuit(b_1^1)} \\
	{F(a_2^0,...,a_n^0)_\varheartsuit(a_1^1)} && {F(a_2^0,...,a_n^0)_\varheartsuit(b_1^1).}
	\arrow["{F[\partial(a_2^0,...,a_n^0)]_\varheartsuit(p_1)}", from=2-2, to=2-4]
	\arrow[two heads, from=2-2, to=4-2]
	\arrow[two heads, from=2-4, to=4-4]
	\arrow[from=3-1, to=2-2]
	\arrow["{F(a_2^0,...,a_n^0)_\varheartsuit(p_1)}", from=3-1, to=3-3]
	\arrow[two heads, from=3-1, to=5-1]
	\arrow[from=3-3, to=2-4]
	\arrow[two heads, from=3-3, to=5-3]
	\arrow["{F[\partial(a_2^0,...,a_n^0)]_\varheartsuit(q_1)}", from=4-2, to=4-4]
	\arrow[from=5-1, to=4-2]
	\arrow["{F(a_2^0,...,a_n^0)_\varheartsuit(q_1)}"', from=5-1, to=5-3]
	\arrow[from=5-3, to=4-4]
\end{tikzcd}\]

By applying condition IS(ii) from Definition \ref{def: indexed sort} to $Q_1$, we conclude that $H \langle Q_1, ..., Q_n \rangle_\btle$ is distinguished.\\

\textbf{Description of when $H\langle Q_1, ..., Q_n \rangle_\btle$ is distinguished in case 2.}\\

In this case, $H\langle Q_1, ..., Q_n \rangle_\btle$ equals the comparison square
\[
\widedsqua{H(a_1^0,a_2^0...,a_n^0)}{H(a_1^0,b_2^0,...,b_n^0)}{\partial (H(a_1^0,a_2^0...,a_n^0))}{\partial (H(a_1^0,b_2^0,...,b_n^0))}{H(id,f_2^0, ..., f_n^0)}{\varphi}{p}{q}
\]
(where $p$, $q$ are length-$1$ display maps) between the gap maps of the left and right faces of
\[\begin{tikzcd}
	& {} \\
	& {F[\partial(a_2^0,...,a_n^0)]_\varheartsuit(a_1^0)} && {F[\partial(b_2^0,...,b_n^0)]_\varheartsuit(a_1^0)} \\
	{F(a_2^0,...,a_n^0)_\varheartsuit(a_1^0)} && {F(b_2^0,...,b_n^0)_\varheartsuit(a_1^0)} \\
	& {F[\partial(a_2^0,...,a_n^0)]_\varheartsuit(a_1^1)} && {F[\partial(b_2^0,...,b_n^0)]_\varheartsuit(a_1^1)} \\
	{F(a_2^0,...,a_n^0)_\varheartsuit(a_1^1)} && {F(b_2^0,...,b_n^0)_\varheartsuit(a_1^1)}
	\arrow["{\psi_0}", from=2-2, to=2-4]
	\arrow[two heads, from=2-2, to=4-2]
	\arrow[two heads, from=2-4, to=4-4]
	\arrow[from=3-1, to=2-2]
	\arrow["{F(f_2^0, ..., f_n^0)_{a_1^0}}"{pos=0.7}, from=3-1, to=3-3]
	\arrow[two heads, from=3-1, to=5-1]
	\arrow[from=3-3, to=2-4]
	\arrow[two heads, from=3-3, to=5-3]
	\arrow["{\psi_1}"{pos=0.3}, from=4-2, to=4-4]
	\arrow[from=5-1, to=4-2]
	\arrow["{F(f_2^0, ..., f_n^0)_{a_1^1}}"', from=5-1, to=5-3]
	\arrow[from=5-3, to=4-4]
\end{tikzcd}\]
where we write $\psi_0$ and $\psi_1$, respectively, for the $a_1^0$- and $a_1^1$-components of the natural transformation
$$
\psi:F[\partial(a_2^0,...,a_n^0)]_\varheartsuit \Longrightarrow F\partial(b_2^0,...,b_n^0)]_\varheartsuit
$$
induced by $f_2$, ..., $f_n$.

Thus $H\langle Q_1, ..., Q_n \rangle_\btle$ is distinguished if and only if the diagram
\[
\tag{\texttt{*}}
\dsqua{F(a_2^0, ..., a_n^0)_\varheartsuit}{F(b_2^0, ..., b_n^0)_\varheartsuit}{F[\partial(a_2^0,...,a_n^0)]_\varheartsuit}{F[\partial(b_2^0,...,b_n^0)]_\varheartsuit}{F(f_2^0,...,f_n^0)}{\psi}{}{}
\]
in $D(\mathcal A_1, \mathcal C)$ satisfies condition IDS(ii) from Remark \ref{rem: characterization indexed distinguished squares} (on the characterization of indexed distinguished squares) with respect to the display map $p_1:a_1^0 \rightarrow a_1^1$. (Note that we already know that IDS(i) holds.)

In other words, $H\langle Q_1, ..., Q_n \rangle_\btle$ is distinguished for all $Q_1$, ..., $Q_n$ in case 1 if and only if every square of the form ($\texttt{*}$) above is an indexed distinguished square, hence if and only if $F$ is an $(n-1)$-ary morphism from $(\mathcal A_2, ..., \mathcal A_n)$ to $\mathcal C^{\mathcal A_1}$.

\vspace{0.5em}

From the two cases we conclude that $F$ is $(n-1)$-ary if and only if $H$ is $n$-ary.
\end{proof}

\begin{notation}
We let $\eta':\Hom(\mathcal A_2,...,\mathcal A_n;\mathcal C^{\mathcal A_1}) \rightarrow \Hom(\mathcal A_1,...,\mathcal A_n;\mathcal C)$ be the restriction of $\eta$ provided by Lemma \ref{lem: isomorphism of multihoms; linear}.
\end{notation}

\begin{proposition}
\label{prop: isomorphism of multihoms; section}
	Let $H \in \Hom(\mathcal A_1, ..., \mathcal A_n; \mathcal C)$. For $x_2 \in \mathcal A_2$, ..., $x_n \in \mathcal A_n$ and $0 \le i \le \ell(x_2) \cdots \ell(x_n)$, write $S_i(x_2, ..., x_n)$ for the $i$-th initial segment of $(x_2,...,x_n)^\le$ with respect to $\overline{\btle}$.

	There exists a unique $F \in \Hom^?(\mathcal A_2, ..., \mathcal A_n; \mathcal C^{\mathcal A_1})$ such that
	\begin{enumerate}[label=(\roman*)]
		\item $\eta(F) = H$.
		
		\item For each $x_2 \in A_2$, ..., $x_n \in A_n$, the chain of indexed sorts
		$$
		F(x_2, ..., x_n)_{\ell(x_2) \cdots \ell(x_n)} \longrightarrow \cdots \longrightarrow F(x_2, ..., x_n)_1 \longrightarrow F(x_2,...,x_n)_0
		$$
		in $C^{A_1}$ defining $F(x_2,...,x_n)$ equals
		$$
		H[(-)^\le \times S_{\ell(x_2) \cdots \ell(x_n)}(x_2,...,x_n)] \longrightarrow \cdots \longrightarrow H[(-)^\le \times S_1(x_2,...,x_n)] \longrightarrow H[(-)^\le \times S_0(x_2,...,x_n)].
		$$
	\end{enumerate}
\end{proposition}

\begin{proof}
	If such an $F$ exists, uniqueness follows from the fact that condition (ii) determines the action of $F$ on objects, while (i) determines its action on arrows.
	
	For existence, given $x_2 \in \mathcal A_2$, ..., $x_n \in \mathcal A_2$, let us prove that the chain
	$$
	H[(-)^\le \times S_{\ell(x_2) \cdots \ell(x_n)}(x_2,...,x_n)] \longrightarrow \cdots \longrightarrow H[(-)^\le \times S_1(x_2,...,x_n)] \longrightarrow H[(-)^\le \times S_0(x_2,...,x_n)]
	$$
	defines an object of $\mathcal C^{\mathcal A_1}$. Firstly, since $S_0(x_2, ..., x_n) = \varnothing$, we have that $H[(-)^\le \times S_0(x_2,...,x_n)] = H[\varnothing]$ is constant on $1_\mathcal C$. Now, let us verify that for $i = 0$, ..., $\ell(x_2) \cdots \ell(x_n) - 1$, the natural transformation
	$$
	H[(-)^\le \times S_{i+1}(x_2,...,x_n)] \longrightarrow H[(-)^\le \times S_i(x_2,...,x_n)],
	$$
	which for the moment we will denote by $\pi$, is an indexed sort. Let $p:a \rightarrow \partial a$ be a length-$1$ display map in $\mathcal A_1$. Then the diagram
	\[
	\tag{\texttt{*}}
	\squa{H[a^\le \times S_{i+1}(x_2,...,x_n)]}{H[a^\le \times S_i(x_2,...,x_n)]}{H[\partial a \times S_{i+1}(x_2,...,x_n)]}{H[\partial a \times S_i(x_2,...,x_n)]}{\pi_a}{\pi_{\partial a}}{}{}
	\]
	is such that the vertical arrows are display maps (as
	\begin{align*}
		\partial a \times S_i(x_2,...,x_n) & \hookrightarrow a^\le \times S_i(x_2,...,x_n), \\
		\partial a \times S_{i+1}(x_2,...,x_n) & \hookrightarrow a^\le \times S_{i+1}(x_2,...,x_n)
	\end{align*}
	are initial segment embeddings with respect to $\btle$). Moreover, writing $(y_2,...,y_n)$ for the unique element of $S_{i+1}(x_2,...,x_n) \smallsetminus S_i(x_2,...,x_n)$, note that
	$$
	(a^\le \times S_{i+1}(x_2,...,x_n)) \smallsetminus \{(a,y_2,...,y_n)\} = (\partial a \times S_{i+1}(x_2,...,x_n)) \cup (a^\le \times S_i(x_2,...,x_n)).
	$$
	It follows that the gap map of (\texttt{*}) is the arrow
	$$
	H[a^\le \times S_{i+1}(x_2,...,x_n)] \longrightarrow H[(a^\le \times S_{i+1}(x_2,...,x_n)) \smallsetminus \{(a,y_2,...,y_n)\}]
	$$
	induced by the inclusion $(a^\le \times S_{i+1}(x_2,...,x_n)) \smallsetminus \{(a,y_2,...,y_n)\} \subset a^\le \times S_{i+1}(x_2,...,x_n)$, hence a length-$1$ display map. This proves that $\pi$ satisfies condition IS(i) from Definition \ref{def: indexed sort}. For IS(ii), consider a length-$1$ distinguished square
	\[
	\tag{\texttt{**}}
	\dsqua{a}{b}{\partial a}{\partial b}{g}{f}{p}{q}
	\]
	in $\mathcal A_1$. Then the comparison square between the gap maps associated with $p$ and $q$ via the above construction,
	\[
	\dsqua{H[a^\le \times S_{i+1}(x_2,...,x_n)]}{H[b^\le \times S_{i+1}(x_2,...,x_n)]}{H[(a^\le \times S_{i+1}(x_2,...,x_n)) \smallsetminus \{(a,y_2,...,y_n)\}]}{H[(b^\le \times S_{i+1}(x_2,...,x_n)) \smallsetminus \{(b,y_2,...,y_n)\}],}{}{}{}{}
	\]
	equals $H\langle Q_1,...,Q_n \rangle_\btle$ where $Q_1$ is (\texttt{**}) and, for $2 \le i \le n$, $Q_i$ is
	\[
	\dsqua{x_i}{x_i}{\partial x_i}{\partial x_i.}{id}{id}{}{}
	\]
	Since $H$ is $n$-ary, $H\langle Q_1,...,Q_n \rangle$ is a distinguished square. Thus IS(ii) holds and $\pi$ is an indexed sort.
	
	Hence from the formula in condition (ii) we obtain a function $F_0:\Ob(A_2 \times \cdots \times A_n) \rightarrow \Ob(|\mathcal C^{\mathcal A_1}|)$. Note that
	\begin{align*}
		F_0(x_2,...,x_n)_{\ell(x_2) \cdots \ell(x_n)} & = H[(-)^\le \times S_{\ell(x_2) \cdots \ell(x_n)}(x_2,...,x_n)]\\
		& = H[(-)^\le \times (x_2,...,x_n)^\le]\\
		& = H[(-,x_2,...,x_n)^\le]\\
		& = H(-,x_2,...,x_n)	
	\end{align*}
	as $H$ is a pre-multimorphism.
	
	We extend $F_0$ to a functor $F:A_2 \times \cdots \times A_n \rightarrow |\mathcal C^{\mathcal A_1}|$ by sending each $(f_2,...,f_n):(x_2,...,x_n) \rightarrow (x'_2,...,x'_n)$ to the morphism
	$$
	F_0(x_2,...,x_n) \longrightarrow F_0(x'_2,...,x'_n)
	$$
	given by the natural transformation
	$$
	F_0(x_2,...,x_n)_{\ell(x_2) \cdots \ell(x_n)} = H(-,x_2,...,x_n) \Longrightarrow H(-,x'_2,...,x'_n) = F_0(x'_2,...,x'_n)_{\ell(x'_2) \cdots \ell(x'_n)}
	$$
	whose $a$-component is $H(id_a,f_2,...,f_n)$ for each $a \in \mathcal A_1$.
	
	\vspace{0.5em}
	
	To check that $F$ is a pre-multimorphism, let us prove that by induction on $N \ge 0$ that for every ``sub-principal" sieve $V \subset T_2 \times \cdots \times T_n$ --- by which we mean a sieve that is contained in some $(a_2,...,a_n)^\le$ --- of cardinality $N$, say whose elements are
	$$
	(x_2^1,...,x_n^1) \; \overline{\btle} \; (x_2^2,...,x_n^2) \;\overline{\btle}\; \cdots \; \overline{\btle} \; (x_2^N,...,x_n^N),
	$$
	the diagram
	$$
	F|_V:V^{\text{op}} \longrightarrow \mathcal C^{\mathcal A_1}
	$$
	is cellular (with respect to the restriction of $\overline{\btle}$ to $V$) with the following distinguished limit: $F|_V[V]$ is given by the chain of natural transformations
	$$
	H[(-)^\le \times V_k] \longrightarrow \cdots \longrightarrow H[(-)^\le \times V_1] \longrightarrow H[(-)^\le \times V_0]
	$$
	where $V_i = \{(x_2^1,...,x_n^1), ..., (x_2^i,...,x_n^i)\}$; for $1 \le i \le N$, the projection map
	$$
	F_V[V] \longrightarrow F(x_2^i,...,x_n^i)
	$$
	is the natural transformation
	$$
	H[(-)^\le \times V_k] \Longrightarrow H[(-)^\le \times (x_2^i,...,x_n^i)^\le]
	$$
	whose $a$-component is induced by the inclusion $(a,x_2^i,...,x_n^i)^\le \subset a^\le \times V_n$.

	For $N = 0$ the claim holds as the empty diagram $F|_\varnothing$ is cellular with distinguished limit the chain whose single entry is $\ct_{1_\mathcal C} = H[\varnothing]$. For $N \ge 1$, suppose that the claim holds for $0$, ..., $N-1$. Consider a sub-principal sieve $V \subset T_2 \times \cdots \times T_n$ with $N$ elements, and let $x_j^i$ be as above. Denote by $\top$ the top element $(x_2^N,...,x_n^N)$ of $V$ with respect to $\overline{\btle}$. By assumption, $F|_{V \smallsetminus \{\top\}}:(V \smallsetminus \{\top\})^{\text{op}} \longrightarrow |\mathcal C^{\mathcal A_1}|$ is cellular and $F|_{V \smallsetminus \{\top\}}[V \smallsetminus \{\top\}]$ is given by the chain described above.
	
	From the definition of $F$ and the induction hypothesis applied to the sieve $\partial \top$, the morphism
	$$
	F(\top) \longrightarrow F|_{V \smallsetminus \{\top\}}[\partial \top]
	$$
	is a length-$1$ display map. Hence by Proposition \ref{prop: properties cellular diagrams}(c), $F|_V$ is cellular. Moreover, $F|_V[V]$ is given by the distinguished square
	\[
	\dsqua{F[V]}{F(\top)}{F[V \smallsetminus \{\top\}]}{F[\partial \top].}{}{}{}{}
	\]
	It follows that:
	\begin{itemize}
		\item By the induction hypothesis, $\partial(F[V]) = F[V \smallsetminus \{\top\}]$ corresponds to the chain
		$$
		H[(-)^\le \times (V \smallsetminus \{\top\})] = H[(-)^\le \times V_{N-1}] \longrightarrow \cdots \longrightarrow H[(-)^\le \times V_1] \longrightarrow H[(-)^\le \times V_0].
		$$
		
		\item $F[V]$ is obtained by adjoining to the above chain the left vertical natural transformation in the indexed distinguished square
		\[
		\squa{J}{H(-,x_2^N,...,x_n^N)}{H[(-)^\le \times (V \smallsetminus \{\top\})]}{H[(-)^\le \times \partial \top].}{}{}{}{}
		\]
		We claim that this diagram equals
		\[
		\tag{\texttt{*}}
		\squa{H[(-)^\le \times V]}{H[(-)^\le \times (x_2^N,...,x_n^N)^\le] = H(-,x_2^N,...,x_n^N)}{H[(-)^\le \times (V \smallsetminus \{\top\})]}{H[(-)^\le \times \partial \top]}{}{}{\pi}{\pi'}
		\]
		where the four arrows are given by functoriality of distinguished limits of restrictions of $H$. For that, it suffices to check that ($\texttt{*}$) is an indexed distinguished square. We have already verified that $\pi$ is an indexed sort, so condition IDS(i) from Remark \ref{rem: characterization indexed distinguished squares} holds. For IDS(ii), suppose given a length-$1$ display map $p:a \rightarrow \partial a$ in $\mathcal A_1$. The gap map of
		\[
		\dsqua{H[a^\le \times V]}{H[a^\le \times (V \smallsetminus \{\top\})]}{H[\partial a \times V]}{H[\partial a \times (V \smallsetminus \{\top\})]}{\pi_a}{\pi_{\partial a}}{}{}
		\]
		is the length-$1$ display map $H[a^\le \times V] \twoheadrightarrow H[a^\le \times V \smallsetminus \{(a,x_2^N,...,x_n^N)\}]$, and that of
		\[
		\dsqua{H(a,x_2^N,...,x_n^N)}{H[a^\le \times \partial \top]}{H(\partial a,x_2^N,...,x_n^N)}{H[(\partial a)^\le \times \partial \top]}{\pi'_a}{\pi'_{\partial a}}{}{}
		\]
		is $H(a,x_2^N,...,x_n^N) \twoheadrightarrow H[(a,x_2^N,...,x_n^N)^\le \smallsetminus \{(a,x_2^N,...,x_n^N)\}]$. As $H$ is a pre-multimorphism, condition (iv) from Definition \ref{def: finite poset-shaped cellular diagram} implies that the induced comparison square,
		\[
		\dsqua{H[a^\le \times V]}{H(a,x_2^N,...,x_n^N)}{H[a^\le \times V \smallsetminus \{(a,x_2^N,...,x_n^N)\}]}{H[(a,x_2^N,...,x_n^N)^\le \smallsetminus \{(a,x_2^N,...,x_n^N)\}],}{}{}{}{}
		\]
		is distinguished. Hence IDS(ii) holds for ($\texttt{*}$).
		
		\item From the above items, we have that $F[V]$ is given by the chain
		$$
		H[(-)^\le \times V] = H[(-)^\le \times V_N] \longrightarrow \cdots \longrightarrow H[(-)^\le \times V_1] \longrightarrow H[(-)^\le \times V_0],
		$$
		as required. Also, it follows from the induction hypothesis (more precisely, the description of the cone of the distinguished limit of $F|_{V \smallsetminus \{\top\}}$) and the construction of diagram ($\texttt{*}$) that the associated limit cone from $F[V]$ to $F|_V$ has as its $(x_2^i,...,x_n^i)$-component the natural transformation $H[(-)^\le \times V] \rightarrow H[(\le)^\le \times (x_2^i,...,x_n^i)^\le]$ induced by the inclusion $(x_2^i,...,x_n^i)^\le \subset V$.
	\end{itemize}
	
	This concludes the induction step, and it follows that $F$ is a pre-multimorphism.
\end{proof}

\begin{notation}
Let $\phi:\Hom(\mathcal A_1,...,\mathcal A_n;\mathcal C) \rightarrow \Hom^?(\mathcal A_2,...,\mathcal A_n;\mathcal C^{\mathcal A_1})$ be the function sending $H$ to the pre-multimorphism $F$ obtained in Proposition \ref{prop: isomorphism of multihoms; section}. By Lemma \ref{lem: isomorphism of multihoms; linear}, $\phi$ restricts to a map
$$
\phi':\Hom(\mathcal A_1,...,\mathcal A_n;\mathcal C) \rightarrow \Hom(\mathcal A_2,...,\mathcal A_n;\mathcal C^{\mathcal A_1}).
$$
\end{notation}

\begin{theorem}
\label{th: isomorphism of multihoms}
The function $\eta':\Hom(\mathcal A_2,...,\mathcal A_n;\mathcal C^{\mathcal A_1}) \rightarrow \Hom(\mathcal A_1,\mathcal A_2,...,\mathcal A_n;\mathcal C)$ is bijective with inverse $\phi'$. We thus obtain an isomorphism of categories $\iiHom(\mathcal A_2,...,\mathcal A_n;\mathcal C^{\mathcal A_1}) \cong \iiHom(\mathcal A_1,\mathcal A_2,...,\mathcal A_n;\mathcal C)$.
\end{theorem}

\begin{proof}
By construction we have that $\eta' \circ \phi'$ is the identity map. Hence it suffices to prove that $\eta'$ is injective.

Suppose given $F$, $G \in \Hom(\mathcal A_2,...,\mathcal A_n;\mathcal C^{\mathcal A_1})$ such that $\eta(F) = \eta(G)$. Writing $U$ for the forgetful functor $(-)_\varheartsuit:|\mathcal C^{\mathcal A_1}| \rightarrow D(\mathcal A_1,\mathcal C)$, by definition of $\eta$ we have $UF = UG$. Then the desired equality $F = G$ follows if we prove that $F$ and $G$ have the same action on objects. For that, note that for each $x_2 \in \mathcal A_2$, ..., $x_n \in \mathcal A_n$, the restrictions $F|_{(x_2,...,x_n)^\le}$, $G|_{(x_2,...,x_n)^\le}:(x_2,...,x_n)^{\le,op} \rightarrow |\mathcal C^{\mathcal A_1}|$ are cellular diagrams and $UF|_{(x_2,...,x_n)^\le} = UG|_{(x_2,...,x_n)^\le}$. Thus by Lemma \ref{lem: cellular diagrams in a cont cat associated to a cwa} applied to the category with attributes $D_{\text{att}}(\mathcal A_1,\mathcal C)$ we have $F|_{(x_2,...,x_n)^\le}=G|_{(x_2,...,x_n)^\le}$. As $x_2$, ..., $x_n$ were arbitrary we conclude that $F_\Ob = G_\Ob$, as required.
\end{proof}

\subsection{Using precontextual categories to describe multimorphisms and exponentials}
\label{subsec: multimorphisms and exponentials via precont}

\begin{proposition}
	\label{prop: reflection induces iso between exponentials}
	Let $\mathcal A$ be a precontextual category and $\mathcal C$ a contextual category. Then the contextual functor $\mathcal C^{\iota_\mathcal A}:\mathcal C^{L(\mathcal A)} \rightarrow \mathcal C^\mathcal A$ induced by the reflection morphism $\iota_\mathcal A:\mathcal A \rightarrow L(\mathcal A)$ is an isomorphism.
\end{proposition}

\begin{proof}
	For $\mathcal B \in \Cont$, consider the diagram of functions
	\[\begin{tikzcd}
		{\Hom(\mathcal B;\mathcal C^{L(\mathcal A)})} && {\Hom(\mathcal B;\mathcal C^\mathcal A)} \\
		{\Hom(L(\mathcal A),\mathcal B;\mathcal C)} && {\Hom(\mathcal A,\mathcal B;\mathcal C)} \\
		{\Hom(\mathcal B,L(\mathcal A);\mathcal C)} && {\Hom(\mathcal B,\mathcal A;\mathcal C)} \\
		{\Hom(L(\mathcal A),\mathcal C^\mathcal B)} && {\Hom(\mathcal A,\mathcal C^\mathcal B)}
		\arrow["{\mathcal C^{\iota_\mathcal A} \; \circ \; -}", from=1-1, to=1-3]
		\arrow["\cong"{marking, allow upside down}, draw=none, from=1-1, to=2-1]
		\arrow["\cong"{marking, allow upside down}, draw=none, from=1-3, to=2-3]
		\arrow["{- \; \circ \; (\iota_\mathcal A \times \Id_\mathcal B)}", from=2-1, to=2-3]
		\arrow["\cong"{marking, allow upside down}, draw=none, from=2-1, to=3-1]
		\arrow["\cong"{marking, allow upside down}, draw=none, from=2-3, to=3-3]
		\arrow["{- \; \circ \; (\Id_\mathcal B \times \iota_\mathcal A)}", from=3-1, to=3-3]
		\arrow["\cong"{marking, allow upside down}, draw=none, from=3-1, to=4-1]
		\arrow["\cong"{marking, allow upside down}, draw=none, from=3-3, to=4-3]
		\arrow["{-\; \circ \; \iota_\mathcal A}"', from=4-1, to=4-3]
	\end{tikzcd}\]
	where the indicated isomorphisms are as follows:
	\begin{itemize}
		\item[-] the top left, top right, bottom left and bottom right ones are as in Theorem \ref{th: isomorphism of multihoms};
		
		\item[-] the middle left and middle right ones are as in Proposition \ref{prop: symmetry of multimorphisms is natural}.
	\end{itemize}
	Naturality of the isomorphism from Theorem \ref{th: isomorphism of multihoms} implies that the top and bottom squares commute, and by Proposition \ref{prop: symmetry of multimorphisms is natural} so does the middle one. Hence the outer composite square commutes. But by the universal property of $\iota_\mathcal A$, the bottom arrow is an isomorphism, from which we conclude that so is the top one, as required.
	
	As the isomorphism $\mathcal C^{\iota_\mathcal A} \circ -: \Hom(\mathcal B;\mathcal C^{L(\mathcal A)}) \rightarrow \Hom(\mathcal B;\mathcal C^\mathcal A)$ is natural in $\mathcal B$, it follows from the Yoneda lemma that $\mathcal C^{\iota_\mathcal A}$ is an isomorphism, as desired.
\end{proof}

\begin{corollary}
\label{cor: category of contextual functors, reflection morphism}
Let $\mathcal A$ be a precontextual category and $\mathcal C$ a contextual category. Then the functor
$$
\iiHom_\iiPrecont(L(\mathcal A),\mathcal C) \xrightarrow{- \circ \iota_\mathcal A} \iiHom_\iiCont(\mathcal A,\mathcal C)
$$
is an isomorphism of categories.
\end{corollary}

\begin{proof}
It follows from Proposition \ref{prop: reflection induces iso between exponentials} and Remark \ref{rem: category of contextual functors as subcategory of exponential}.
\end{proof}

\begin{proposition}
\label{prop: multimorphisms from contextual vs precontextual categories}
	For precontextual categories $\mathcal A_1$, ..., $\mathcal A_n$, precomposition with $\iota_{\mathcal A_1} \times \cdots \times \iota_{\mathcal A_n}$ defines an isomorphism of categories
	$$
	\iiHom(L(\mathcal A_1), ..., L(\mathcal A_n);\mathcal C) \cong \iiHom(\mathcal A_1, ..., \mathcal A_n;\mathcal C)
	$$
	natural in $\mathcal C \in \Cont$.
\end{proposition}

\begin{proof}
	Let us prove by induction on $n \ge 1$ that the claim holds for all $\mathcal A_1$, ..., $\mathcal A_n$, $\mathcal C$. The case $n = 1$ corresponds to Corollary \ref{cor: category of contextual functors, reflection morphism}. Given $n = 2$, suppose that the claim holds for $1$, ..., $n-1$. Given $\mathcal A_1$, ..., $\mathcal A_n \in \Precont$ and $\mathcal C \in \Cont$, we have a chain of isomorphisms
	\begin{align*}
		\iiHom(L(\mathcal A_1), L(\mathcal A_2), ..., L(\mathcal A_n);\mathcal C) & \cong \iiHom(L(\mathcal A_2), ..., L(\mathcal A_n); \mathcal C^{L(\mathcal A_1)}) \tag{Theorem \ref{th: isomorphism of multihoms}}\\
		& \cong \iiHom(\mathcal A_2, ..., \mathcal A_n; \mathcal C^{L(\mathcal A_1)}) \tag{induction hypothesis}\\
		& \cong \iiHom(\mathcal A_2, ..., \mathcal A_n; \mathcal C^{\mathcal A_1}) \tag{Proposition \ref{prop: reflection induces iso between exponentials}}\\
		& \cong \iiHom(\mathcal A_1, \mathcal A_2, ..., \mathcal A_n;\mathcal C). \tag{Theorem \ref{th: isomorphism of multihoms}}
	\end{align*}
	We leave it as an exercise to verify, using the definition of each of these maps, that the composite is indeed given by precomposition with $\iota_{\mathcal A_1} \times \cdots \times \iota_{\mathcal A_n}$.
\end{proof}

\section{The multicategory structure}
\label{sec: multicategory structure}

We will now verify that the sets of multimorphisms $\Hom(\mathcal A_1,...,\mathcal A_n;\mathcal B) := \Hom_\btle(\mathcal A_1, ..., \mathcal A_n;\mathcal B)$ introduced above can be used to define a multicategory structure that extends the category of contextual categories and contextual functors. We refer the reader to \cite{Lei04}, Ch. 2 and 3, or \cite{Her00} for an introduction to the theory of multicategories. Familiarity with their relationship to monoidal categories via the concept of a representable multicategory is recommended but not strictly required (see Remark \ref{rem: about the construction of monoidal structure}).

\vspace{0.5em}

Consider the multicategory $\mathscr{Cat}$ where\footnote{Succinctly, $\mathscr{Cat}$ is the multicategory represented by the cartesian monoidal structure on $\Cat$.}: objects are the small categories; the hom-set from $(A_1,...,A_n)$ to $B$ is the set of functors $\Fun(A_1 \times \cdots \times A_n,B)$ (for $n = 0$, we have the set of functors from the terminal category $\textbf{1} = \{*\}$ to $B$); and, given a category $C$, integers $k_1$, ..., $k_n \ge 0$ and, for $1 \le i \le n$, small categories $B_i$, $A_i^1$, ..., $A_i^{k_i}$, the composition function
\[
\begin{tikzcd}[row sep=small]
	\Fun(B_1 \times \cdots \times B_n,C) \times \Fun(A_1^1 \times \cdots \times A_1^{k_1},B_1) \times \cdots \times \Fun(A_n^1,...,A_n^{k_n},B_n) \arrow[]{d}{} \\
	\Fun(A_1^1 \times ... \times A_1^{k_1} \times ... \times A_n^1 \times ... \times A_n^{k_n},C)
\end{tikzcd}
\]
maps $(G,F_1,...,F_n)$ to the functor $G \circ (F_1,...,F_n)$ defined in the usual way from the cartesian structure on $\Cat$, i.e. it is given on objects and arrows by\footnote{Here we make an abuse of notation: if $k_i = 0$, then there are no variables of the form $x^j_i$ on the left-hand side, but on the right-hand side we evaluate $F_i$ at the unique object or arrow of the terminal category. If $k_i = 0$ for all $i$, then the functor maps $*$ (resp. $id_*$) to $G(F_1(*),...,F_n(*))$ (resp. $G(F_1(id_*),...,F_n(id_*))$).}
$$
(x_1^1,...,x_1^{k_1},...,x_n^1,...,x_n^{k_n}) \longmapsto G(F_1(x_1^1,...,x_1^{k_1}),...,F_n(x_n^1,...,x_n^{k_n})).
$$
Consider the map $|-|:\Ob(\Cont) \rightarrow \Ob(\Cat)$ that sends a contextual category to its underlying category. We have a multicategory structure $\mathscr{Cont}_{fun}$ on $\Ob(\Cont)$ where the hom-sets are
$$
\mathscr{Cont}_{fun}(\mathcal A_1,...,\mathcal A_n;\mathcal B) = \Fun(|\mathcal A_1| \times \cdots \times |\mathcal A_n|,|\mathcal B|)
$$
and composition is computed as in $\mathscr{Cat}$. In particular, for each $\mathcal B$ we have $\mathscr{Cont}_{fun}(\;;\mathcal B) = \Fun(\textbf{1},|\mathcal B|)$.

We will prove that restricting to the subsets
$$
\Hom(\mathcal A_1, ..., \mathcal A_n; \mathcal B) \subset \Fun(|\mathcal A_1| \times \cdots \times |\mathcal A_n|,|\mathcal B|)
$$
gives a sub-multicategory of $\mathscr{Cont}_{fun}$. However, that also requires a suitable definition of $0$-ary morphism to a contextual category.

\begin{definition}
For a contextual category $\mathcal C$, we let $\Hom(\;;\mathcal C)$ be the set of all functors $\textbf{1} \rightarrow |\mathcal C|$ that send the single object to a length-$1$ object in $\mathcal C$.

Writing $\Ob_1(\mathcal C)$ for the set of length-$1$ objects, we have $\Hom(\;;\mathcal C) \cong \Ob_1(\mathcal C) \cong \Hom_\Cont(\mathcal O_1,\mathcal C)$.
\end{definition}

To obtain the desired sub-multicategory, it now suffices to prove that multimorphisms are closed under multicomposition in $\mathscr{Cont}_{fun}$.

\begin{notation}
In what follows, when talking about (pre)contextual categories such as $\mathcal A_i$, $\mathcal B$, etc., we will write $A_i$, $B$, etc. for their underlying categories. 
\end{notation}

\begin{lemma}
\label{lem: multicategory structure 1}
Suppose given $n$, $k \ge 1$, precontextual categories $\mathcal A_1$, ..., $\mathcal A_n$, $\mathcal B_1$, ..., $\mathcal B_k$, $\mathcal C$ such that $\mathcal A_n$ and $\mathcal C$ are contextual categories, and multimorphisms
$$
F:(\mathcal A_1,...,\mathcal A_n) \longrightarrow \mathcal C, \qquad G:(\mathcal B_1,...,\mathcal B_k) \longrightarrow \mathcal A_n.
$$
Then
$$
F(-,...,-,G(-,...,-)):A_1 \times \cdots \times A_{n-1} \times B_1 \times \cdots \times B_k \longrightarrow C
$$
is a multimorphism from $(\mathcal A_1, ..., \mathcal A_{n-1},\mathcal B_1,...,\mathcal B_k)$ to $\mathcal C$.
\end{lemma}

\begin{proof}
Recursively applying Theorem \ref{th: isomorphism of multihoms}, $F$ defines a morphism $\overline{F}:\mathcal A_n \longrightarrow (\cdots(\mathcal C^{\mathcal A_1})^{\mathcal A_2}\cdots)^{\mathcal A_{n-1}}$. Thus
$$
\overline{F} \circ G:B_1 \times \cdots \times B_k \longrightarrow |(\cdots(\mathcal C^{\mathcal A_1})^{\mathcal A_2}\cdots)^{\mathcal A_{n-1}}|
$$
is a multimorphism from $(\mathcal B_1,...,\mathcal B_k)$ to $(\cdots(\mathcal C^{\mathcal A_1})^{\mathcal A_2}\cdots)^{\mathcal A_{n-1}}$. Another recursive application of Theorem \ref{th: isomorphism of multihoms} gives a multimorphism
$$
H:(\mathcal A_1,...,\mathcal A_{n-1},\mathcal B_1,...,\mathcal B_k) \longrightarrow \mathcal C.
$$
It can be checked from the construction of the isomorphism in Theorem \ref{th: isomorphism of multihoms} that $H = F(-,...,-,G(-,...,-))$.
\end{proof}

\begin{remark}
Our goal is to generalize Lemma \ref{lem: multicategory structure 1} by allowing $G$ to be, more generally, a multimorphism $(\mathcal B_1, ..., \mathcal B_k) \rightarrow \mathcal A_i$ for any $i \in \{1, ..., n\}$. However, the asymmetric form of Theorem \ref{th: isomorphism of multihoms} prevents us from directly arguing as in the above proof. Informally, the strategy will be to use the symmetry construction from \S\ref{subsec: transporting multimorphisms} to move $i$ to the last position, apply Lemma \ref{lem: multicategory structure 1} to obtain a multimorphism by precomposition with $G$, and use the symmetry construction again to move $i$ back to the original position. The main technical work lies in checking that the two applications of the symmetry construction are, in a certain sense, inverse to each other.
\end{remark}

\begin{proposition}
\label{prop: multicategory structure 2}
Suppose given $n$, $k \ge 1$, precontextual categories $\mathcal A_1$, ..., $\mathcal A_n$, $\mathcal B_1$, ..., $\mathcal B_k$, $\mathcal C$ such that $\mathcal A_1$ and $\mathcal C$ are contextual categories, and multimorphisms
$$
F:(\mathcal A_1,...,\mathcal A_n) \longrightarrow \mathcal C, \qquad G:(\mathcal B_1,...,\mathcal B_k) \longrightarrow \mathcal A_1.
$$
Then the functor
$$
F(G(-,...,-),-,...,-):B_1 \times \cdots \times B_k \times A_2 \times \cdots \times A_n \longrightarrow C
$$
is a multimorphism from $(\mathcal B_1,...,\mathcal B_k,\mathcal A_2, ..., \mathcal A_n)$ to $\mathcal C$. Moreover, consider the diagram
\[\begin{tikzcd}[ampersand replacement=\&]
	\& {C} \\
	\\
	{A_2 \times \cdots \times A_n \times A_1} \&\& {A_1 \times A_2 \times \cdots \times A_n} \\
	{A_2 \times \cdots \times A_n \times B_1 \times \cdots \times B_k} \&\& {B_1 \times \cdots \times B_k \times A_2 \times \cdots \times A_n}
	\arrow[""{name=0, anchor=center, inner sep=0}, "{F'}", from=3-1, to=1-2]
	\arrow["\cong"{description}, from=3-1, to=3-3]
	\arrow[""{name=1, anchor=center, inner sep=0}, "F"', from=3-3, to=1-2]
	\arrow["{\Id_{A_2 \times \cdots \times A_n} \times G}", from=4-1, to=3-1]
	\arrow["\cong"{description}, from=4-1, to=4-3]
	\arrow["{G \times \Id_{A_2 \times \cdots \times A_n}}"', from=4-3, to=3-3]
	\arrow["\phi"{description},shorten=15pt, Rightarrow, shift right=2,from=0, to=1]
\end{tikzcd}\]
where the horizontal arrows are symmetry isomorphisms (so the lower square commutes strictly) and the upper triangle is a shuffling diagram (Definition \ref{def: shuffling diagram}). Then the following pasted triangle is also a shuffling diagram:
\[
\begin{tikzcd}[ampersand replacement=\&]
	\& {C} \\
	\\
	{A_2 \times \cdots \times A_n \times B_1 \times \cdots \times B_k} \&\& {B_1 \times \cdots \times B_k \times A_2 \times \cdots \times A_n.}
	\arrow[""{name=0, anchor=center, inner sep=0}, "{F'(-,...,-,G(-,...,-))}", from=3-1, to=1-2]
	\arrow["\cong"{description}', from=3-1, to=3-3]
	\arrow[""{name=1, anchor=center, inner sep=0}, "{F(G(-,...,-),-,...,-)}"', from=3-3, to=1-2]
	\arrow["{\phi(\Id_{A_2 \times \cdots \times A_n} \times G)}"',shift right=2,Rightarrow,shorten=15pt, from=0, to=1]
\end{tikzcd}\]
\end{proposition}

\begin{proof}
Write $H$ for $F(G(-,...,-),-,...,-)$. We will prove by induction that the following holds for all $N \ge 0$: suppose given sieves $X \subset T(\mathcal A_2) \times \cdots \times T(\mathcal A_n)$ and $Y \subset T(\mathcal B_1) \times \cdots \times T(\mathcal B_k)$ that are linearly ordered by the respective lexicographic local linearizations (Definition \ref{def: lexicographic product}) and satisfy $\sharp(Y \times X) = N$; then
$$
H|_{Y \times X}:(Y \times X)^{\text{op}} \longrightarrow C
$$
is cellular\footnote{With respect to the lexicographic product of the given linear orders on $Y$ and $X$. It can be checked to coincide with the restriction to $Y \times X$ of the lexicographic local linearization of $T(\mathcal B_1) \times \cdots \times T(\mathcal B_k) \times T(\mathcal A_1) \times \cdots \times T(\mathcal A_n)$.} and its distinguished limit is given by
$$
H[Y \times X] = F[G[Y]^{\le} \times X]
$$
with associated cone having as components the canonical arrows
$$
F[G[Y]^{\le} \times X] \longrightarrow F[G[(b_1,...,b_k)^\le]^{\le} \times (a_2,...,a_n)^\le] = F(G(b_1,...,b_k),a_2,...,a_n)
$$
for $(b_1,...,b_k,a_2,...,a_n) \in Y \times X$.

For $N = 0$, the claim holds as $F[G[\varnothing]^\le \times \varnothing] = F[\varnothing] = 1_\mathcal C$.

Given $N \ge 1$, assume that the claim holds for $0$, ..., $N-1$, and consider $X$, $Y$ of the required form with $\sharp(Y \times X) = N$. Let $(x_2,...,x_n)$ and $(y_1,...,y_k)$ be the top elements of $X$ and $Y$ with respect to the corresponding linear orders. By the induction hypothesis, recalling that $\partial X := X \smallsetminus \{(x_2, ..., x_n)\}$ and $\partial Y := Y \smallsetminus \{(y_1, ..., y_k)\}$, we have that $H$ is cellular on $\partial Y \times X$, $Y \times \partial X$ and $\partial Y \times \partial X$; hence it is cellular on $\partial(Y \times X) = (\partial Y \times X) \cup (Y \times \partial X)$ and the following square is distinguished:
\[
\dsqua{H[\partial(Y \times X)]}{H[Y \times \partial X]}{H[\partial Y \times X]}{H[\partial Y \times \partial X].}{}{}{}{}
\]
By the induction hypothesis, the right and bottom arrows are equal, respectively, to those in the distinguished square
\[
\dsqua{F[\partial(G[Y]^\le \times X)]}{F[G(Y)^\le \times \partial X]}{F[G[\partial Y]^\le \times X] = F[\partial(G[Y]) \times X]}{F[G[\partial Y]^\le \times \partial X] = F[\partial(G[Y]) \times \partial X].}{}{}{}{}
\]
It follows that $H[\partial(Y \times X)] = F[\partial(G[Y]^\le \times X)]$ and the associated limit cones are equal. To verify that $H$ is cellular on $Y \times X$ and describe the corresponding distinguished limit cone, we consider two cases:

\begin{enumerate}[label=(\arabic*)]
	\item $X = (x_2,...,x_n)^\le$ and $Y = (y_1,...,y_k)^\le$.
	
	In this case, it suffices to prove that the canonical arrow $H(y_1,...,y_k,x_2,...,x_n) \rightarrow H[\partial(y_1,...,y_k,x_2,...,x_n)]$ is a length-$1$ display map. But note that it equals
	$$
	F(G(y_1,...,y_k),x_2,...,x_n) \longrightarrow F[\partial(G(y_1,...,y_k),x_2,...,x_n)],
	$$
	which is a length-$1$ display map as $F$ is a multimorphism.

	\item $X \neq (x_2,...,x_n)^\le$ or $Y \neq (y_1,...,y_k)^\le$.
	
	This is equivalent to $(y_1,...,y_k,x_2,...,x_n)^\le$ being a proper subset of $Y \times X$. By the induction hypothesis, $H$ is cellular on $(y_1,...,y_k,x_2,...,x_n)^\le$, hence on $\partial(X \times Y) \cup (y_1,...,y_k,x_2,...,x_n)^\le = X \times Y$, and we have a distinguished square
	\[
	\dsqua{H[Y \times X]}{H(y_1,...,y_k,x_2,...,x_n)}{H[\partial(Y \times X)]}{H[\partial(y_1,...,y_k,x_2,...,x_n)].}{}{}{}{}
	\]
	By comparing the right and bottom arrows, it also follows from the induction hypothesis that this diagram equals the distinguished square
	\[
	\dsqua{F[G[Y]^\le \times X]}{F(G(y_1,...,y_k),x_2,...,x_n)}{F[\partial(G[Y]^\le \times X)]}{F[\partial(G(y_1,...,y_k),x_2,...,x_n)].}{}{}{}{}
	\]
	This also yields the desired description of the associated cone of $H[Y \times X]$.
\end{enumerate}

This concludes the induction step. In particular, we have proved that $H \in \Hom^?(\mathcal B_1,...,\mathcal B_k,\mathcal A_2,...,\mathcal A_n;\mathcal C)$. Let us check that it is a multimorphism.

\vspace{0.5em}

Using the construction of shuffling diagrams from \S\ref{subsubsec: symmetry of multimorphisms}, we obtain from $F$ a multimorphism $F':(\mathcal A_2,...,\mathcal A_n,\mathcal A_1) \rightarrow \mathcal C$ and a natural isomorphism $\phi:F' \rightarrow F \circ \sigma$ where $\sigma:A_2 \times \cdots \times A_n \times A_1 \rightarrow A_1 \times A_2 \times \cdots \times A_n$ is the symmetry isomorphism, with the following property: for each $a_1 \in A_1$, ..., $a_n \in A_n$ the square
\[
\dsqua{F'(a_2,...,a_n,a_1)}{F(a_1,a_2,...,a_n)}{F'[\partial(a_2,...,a_n,a_1)]}{F[\partial(a_1,a_2,...,a_n)]}{\phi_{a_2, ..., a_n, a_1}}{}{}{}
\]
is distinguished of length $1$, where the bottom arrow is induced by $\phi$ and functoriality of limits.

Now, let $H'= F'(-,...,-,G(-,...,-))$. Pasting $\phi$ with $\Id_{A_2 \times \cdots \times A_n} \times G$ induces an isomorphism $\phi':H' \rightarrow H \circ \sigma'$ where
$$
\sigma':A_2 \times \cdots \times A_n \times B_1 \times \cdots \times B_k \longrightarrow B_1 \times \cdots \times B_k \times A_2 \times \cdots \times A_n
$$
is the symmetry isomorphism. By Lemma \ref{lem: multicategory structure 1}, $H'$ is a multimorphism from $(\mathcal A_2,...,\mathcal A_n,\mathcal B_1,...,\mathcal B_k)$ to $\mathcal C$. Moreover, for $a_2 \in A_2$, ..., $a_n \in A_n$, $b_1 \in B_1$, ..., $b_k \in B_k$ of length $\ge 1$, the length-$1$ distinguished square
\[
\widedsqua{F'(a_2,...,a_n,G(b_1,...,b_k))}{F(G(b_1,...,b_k),a_2,...,a_n)}{F'[\partial(a_2,...,a_n,G(b_1,...,b_k))]}{F[\partial(G(b_1,...,b_k),a_2,...,a_n)]}{\phi_{a_2, ..., a_n,G(b_1, ..., b_k)}}{}{}{}
\]
equals
\[
\widedsqua{H'(a_2,...,a_n,b_1,...,b_k)}{H(b_1,...,b_k,a_2,...,a_n)}{H'[\partial(a_2,...,a_n,b_1,...,b_k)]}{H[\partial(b_1,...,b_k,a_2,...,a_n)],}{\phi'_{a_2, ..., a_n, b_1, ..., b_k}}{}{}{}
\]
where the bottom arrow is induced by $\phi'$ and functoriality of limits.

We conclude from \S\ref{subsubsec: symmetry of multimorphisms} that $H$ is also a multimorphism.
\end{proof}

\begin{proposition}
	\label{prop: multicategory structure 3}
	Suppose given $n$, $k \ge 1$, $\; i \in \{1, ..., n\}$, precontextual categories $\mathcal A_1$, ..., $\mathcal A_n$, $\mathcal B_1$, ..., $\mathcal B_k$, $\mathcal C$ such that $\mathcal A_i$ and $\mathcal C$ are contextual categories, and multimorphisms
	$$
	F:(\mathcal A_1,...,\mathcal A_n) \longrightarrow \mathcal C, \qquad G:(\mathcal B_1,...,\mathcal B_k) \longrightarrow \mathcal A_i.
	$$
	Then the functor
	$$
	F(-,...,G(-,...,-),...,-):A_1 \times \cdots \times A_{i-1} \times B_1 \times \cdots \times B_k \times A_{i+1} \times \cdots \times A_n \longrightarrow C
	$$
	is a multimorphism from $(\mathcal A_1, ...,\mathcal A_{i-1},\mathcal B_1,...,\mathcal B_k,\mathcal A_{i+1},..., \mathcal A_n)$ to $\mathcal C$.
\end{proposition}

\begin{proof}
Using Theorem \ref{th: isomorphism of multihoms} we obtain a multimorphism
$$
\overline{F}:(\mathcal A_i,\mathcal A_{i+1},...,\mathcal A_n) \longrightarrow (\cdots(\mathcal C^{\mathcal A_1})^{\mathcal A_2}\cdots)^{\mathcal A_{i-1}}.
$$
By Proposition \ref{prop: multicategory structure 2},
$$
\overline{F}(G(-,...,-),-,...,-):B_1 \times \cdots \times B_k \times A_{i+1} \times \cdots \times A_n \longrightarrow |(\cdots(\mathcal C^{\mathcal A_1})^{\mathcal A_2}\cdots)^{\mathcal A_{i-1}}|
$$
is a multimorphism from $(\mathcal B_1,...,\mathcal B_k,\mathcal A_{i+1},...,\mathcal A_n)$ to $(\cdots(\mathcal C^{\mathcal A_1})^{\mathcal A_2}\cdots)^{\mathcal A_{i-1}}$. Again using Theorem \ref{th: isomorphism of multihoms}, we obtain a multimorphism
$$
H:(\mathcal A_1, ...,\mathcal A_{i-1},\mathcal B_1,...,\mathcal B_k,\mathcal A_{i+1},\cdots, \mathcal A_n) \longrightarrow \mathcal C.
$$
It follows from the construction of the isomorphism in Theorem \ref{th: isomorphism of multihoms} that $H = F(-,...,G(-,...,-),...,-)$.
\end{proof}

We now verify the analogue of Proposition \ref{prop: multicategory structure 3} for $k = 0$.

\begin{lemma}
\label{lem: multicategory structure 4}
Suppose given $n \ge i \ge 1$, contextual categories $\mathcal A_1$, ..., $\mathcal A_n$, $\mathcal C$, and a multimorphism
$$
F:(\mathcal A_1,...,\mathcal A_n) \longrightarrow \mathcal C.
$$
Then for each length-$1$ object $x$ of $\mathcal A_i$, the functor
$$
F(-,...,x,...,-):A_1 \times \cdots \times A_{i-1} \times A_{i+1} \times \cdots \times A_n \longrightarrow C
$$
is a multimorphism from $(\mathcal A_1, ...,\mathcal A_{i-1},\mathcal A_{i+1},\cdots,\mathcal A_n)$ to $\mathcal C$.
\end{lemma}

\begin{proof}
In what follows, we consider the precontextual category $\mathcal O_1^{\text{pre}}$ whose underlying category is the poset $o_1 \rightarrow o_0$ with $o_0$, $o_1$ of length $0$, $1$, resp. Let $\underline{x}:\mathcal O_1^{\text{pre}} \rightarrow \mathcal A_i$ be the unique morphism sending $o_1$ to $x$.

From Theorem \ref{th: isomorphism of multihoms} we obtain a multimorphism
$$
\overline{F}:(\mathcal A_i,\mathcal A_{i+1},\cdots,\mathcal A_n) \longrightarrow (\cdots(\mathcal C^{\mathcal A_1})^{\mathcal A_2}\cdots)^{\mathcal A_{i-1}},
$$
and Proposition \ref{prop: multicategory structure 2} now yields a multimorphism
$$
\overline{F}(\underline{x}(-),-,...,-):(\mathcal O_1^{\text{pre}},\mathcal A_{i+1},\cdots,\mathcal A_n) \longrightarrow (\cdots(\mathcal C^{\mathcal A_1})^{\mathcal A_2}\cdots)^{\mathcal A_{i-1}}.
$$
By Theorem \ref{th: isomorphism of multihoms} we have a multimorphism
$$
K:(\mathcal A_{i+1},\cdots,\mathcal A_n) \longrightarrow ((\cdots(\mathcal C^{\mathcal A_1})^{\mathcal A_2}\cdots)^{\mathcal A_{i-1}})^{\mathcal O_1^{\text{pre}}}.
$$
But the inclusion functor $\{1\} \rightarrow \mathcal O_1^{\text{pre}}$ induces, by Proposition \ref{prop: precontextual category from a category}, an isomorphism
$$
((\cdots(\mathcal C^{\mathcal A_1})^{\mathcal A_2}\cdots)^{\mathcal A_{i-1}})^{\mathcal O_1^{\text{pre}}} \cong ((\cdots(\mathcal C^{\mathcal A_1})^{\mathcal A_2}\cdots)^{\mathcal A_{i-1}})^{\{1\}} \cong (\cdots(\mathcal C^{\mathcal A_1})^{\mathcal A_2}\cdots)^{\mathcal A_{i-1}}.
$$
By composing $K$ with this isomorphism and once again using Theorem \ref{th: isomorphism of multihoms}, we get a multimorphism
$$
H:(\mathcal A_1,...,\mathcal A_{i-1},\mathcal A_{i+1},...,\mathcal A_n) \longrightarrow \mathcal C,
$$
which equals $F(-,...,x,...,-)$.
\end{proof}

Finally, Proposition \ref{prop: multicategory structure 3} and Lemma \ref{lem: multicategory structure 4} imply:

\begin{theorem}
\label{th: multicategory structure 5}
The subsets
$$
\Hom(\mathcal A_1,...,\mathcal A_n;\;\mathcal B) \subset \Fun(|\mathcal A_1| \times \cdots \times |\mathcal A_n|,\;|\mathcal B|) = \mathscr{Cont}_{fun}(\mathcal A_1,...,\mathcal A_n;\;\mathcal B),
$$
with $n \ge 0$ and $\mathcal A_1$, ..., $\mathcal A_n$, $\mathcal B$ ranging over all small contextual categories, are closed under multi-composition in $\mathscr{Cont}_{fun}$. In particular, they define a sub-multicategory of $\mathscr{Cont}_{fun}$.
\end{theorem}

\begin{definition}
\label{def: multicategory of contextual categories}
We define the \emph{multicategory of (small) contextual categories} as the multicategory obtained in Theorem \ref{th: multicategory structure 5}. It will be denoted by $\mathscr{Cont}$.
\end{definition}

\begin{remark}
\label{rem: obtaining symmetric multicategory}
It is not difficult to prove from the results presented so far (particularly \S\ref{subsec: transporting multimorphisms} and the second part of Proposition \ref{prop: multicategory structure 2}) that the isomorphisms from Proposition \ref{prop: symmetry of multimorphisms is natural} realize $\mathscr{Cont}$ as a symmetric multicategory (see \cite{Lei04}, Def. 2.2.21). However, this will follow (ultimately relying on the same preliminaries) from the symmetry of the monoidal structure from \S\ref{sec: monoidal structure}.
\end{remark}

\section{The monoidal structure}
\label{sec: monoidal structure}

In this section, we will prove that for contextual categories $\mathcal A$ and $\mathcal B$, the functor $\Hom(\mathcal A,\mathcal B;-):\Cont \rightarrow \Set$ is representable by a contextual category $\mathcal A \otimes \mathcal B$. Moreover, the assignment $(\mathcal A,\mathcal B) \mapsto \mathcal A \otimes \mathcal B$ extends into a symmetric monoidal structure on $\Cont$, and for $n \ge 1$ and small contextual categories $\mathcal A_1$, ..., $\mathcal A_n$, the functor
	$$
	\Hom(\mathcal A_1,...,\mathcal A_n;-):\Cont \longrightarrow \Set
	$$
is representable by $\mathcal A_1 \otimes (\mathcal A_2 \otimes (\cdots (\mathcal A_{n-1} \otimes \mathcal A_n) \cdots ))$ (or by the object obtained by associating the same factors in any other way). This will allow us to conclude indirectly that the multicategory $\mathscr{Cont}$ from \S\ref{sec: multicategory structure} is representable (see \cite{Her00}, Def. 8.1 and Th. 9.8) and symmetric (see \cite{Lei04}).

\begin{remark}
\label{rem: about the construction of monoidal structure}
Conversely, we could start by proving, using similar techniques, that $\mathscr{Cont}$ is a representable (closed) symmetric multicategory and use that to obtain a (closed) symmetric monoidal structure that $\mathscr{Cont}$; see the discussions in \cite[\S3]{Web13} and \cite[Def. 3.3.1 and Prop. 3.3.2]{GamGarVas25}. This path is conceptually natural and is within reach from the results of \S\ref{sec: multicategory structure} (see Remark \ref{rem: obtaining symmetric multicategory}), and carrying it out is an interesting exercise. In fact, our self-contained approach to finishing the construction of the monoidal structure can be seen as providing a supply of pre-universal multimorphisms (as in \cite{Lei04}, Def. 3.3.2) and coherences (compare Lemma \ref{lem: symmetry coherence preliminary lemma} with condition (2:7) from \cite{Lei04}, Def. 2.2.21) that is sufficient to uniquely determine a representable symmetric multicategory structure. The reader familiar with representability of symmetric multicategories will observe that the coherences from Lemmas \ref{lem: pentagon}, \ref{lem: triangle} and Propositions \ref{prop: symmetry coherence 1}, \ref{prop: symmetry coherence 2} are analogous to those in \cite[Def. 9.6]{Her00} extended to the symmetric case.
\end{remark}

\begin{definition}
\label{def: tensor product}
Given $n \ge 1$, a \emph{tensor product} of a finite sequence of small contextual categories $\mathcal A_1$, ..., $\mathcal A_n$ is a contextual category $\mathcal T$ endowed with an isomorphism
\begin{equation*}
\label{eq: tensor product definition}
\Hom_\Cont(\mathcal T,-) \cong \Hom(\mathcal A_1,...,\mathcal A_n;-)
\end{equation*}
of functors $\Cont \rightarrow \Set$. By the Yoneda lemma, this datum is equivalent to that of a multimorphism
$$
F:(\mathcal A_1,...,\mathcal A_n) \longrightarrow \mathcal T
$$
which is universal in the sense that for any multimorphism $G:(\mathcal A_1,...,\mathcal A_n) \rightarrow \mathcal C$, there exists a unique contextual functor ${G':\mathcal T \rightarrow \mathcal C}$ such that $G = G' \circ F$. In particular, $\mathcal T$ is unique up to unique isomorphism compatible with $(\mathcal A_1,...,\mathcal A_n) \rightarrow \mathcal T$.
\end{definition}

We start by showing that binary tensor products exist.

\begin{construction}
\label{constr: binary tensor products}
	Let $\mathcal A$ and $\mathcal B$ be contextual categories. We will obtain their tensor product in few steps using the existence of contextual categories that represent cellular diagrams (Proposition \ref{prop: representability cellular diagrams}) and the reflection functor $L:\Precont \rightarrow \Cont$.
	
	Following Proposition \ref{prop: representability cellular diagrams}, let $M:(T(\mathcal A) \times T(\mathcal B))^{\text{op}} \rightarrow \mathcal P$ be a universal bimorphism. Taking a pushout of categories
	\[\begin{tikzcd}[ampersand replacement=\&]
		{(T(\mathcal A) \times T(\mathcal B))^{\text{op}}} \& {\mathcal P} \\
		{|\mathcal A| \times |\mathcal B|} \& {\mathcal Q,}
		\arrow[from=1-1, to=1-2]
		\arrow[hook, from=1-1, to=2-1]
		\arrow[from=1-2, to=2-2]
		\arrow[from=2-1, to=2-2]
	\end{tikzcd}\]
	note that as the left vertical arrow is bijective on objects, so is the right one. This allows us to regard $\mathcal Q$ as a precontextual category with display maps and distinguished squares inherited from $\mathcal P$. We now have isomorphisms
	$$
	\Hom_\Cont(L(\mathcal Q),\mathcal C) \cong \Hom_\Precont(\mathcal Q,\mathcal C) \cong \Hom^?(\mathcal A,\mathcal B;\mathcal C)
	$$
	natural in $\mathcal C \in \Cont$. Under this correspondence, the functor $\mathcal A \times \mathcal B \rightarrow \mathcal C$ corresponding to a morphism $F:L(\mathcal Q) \rightarrow \mathcal C$ is a bimorphism precisely when for all distinguished squares
	\[\begin{tikzcd}[ampersand replacement=\&]
		a \& {a'} \& b \& {b'} \\
		{\partial a} \& {\partial a',} \& {\partial b} \& {\partial b'}
		\arrow["{f'}", from=1-1, to=1-2]
		\arrow[two heads, from=1-1, to=2-1]
		\arrow[two heads, from=1-2, to=2-2]
		\arrow["{g'}", from=1-3, to=1-4]
		\arrow[two heads, from=1-3, to=2-3]
		\arrow[two heads, from=1-4, to=2-4]
		\arrow["f"', from=2-1, to=2-2]
		\arrow["g"', from=2-3, to=2-4]
	\end{tikzcd}\]
	in $\mathcal A$, $\mathcal B$, respectively, the square
	\[\begin{tikzcd}[ampersand replacement=\&]
		{M(a,b)} \&\& {M(a',b')} \\
		{\partial(M(a,b))} \&\& {\partial(M(a',b'))} \\
		{M(\partial a,b) \times_{M(\partial a,\partial b)} M(a,\partial b)} \&\& {M(\partial a',b') \times_{M(\partial a',\partial b')} M(a',\partial b')}
		\arrow["{M(f',g')}", from=1-1, to=1-3]
		\arrow[two heads, from=1-1, to=2-1]
		\arrow[two heads, from=1-3, to=2-3]
		\arrow["{M(f,g')\times_{M(f,g)} M(f',g)}"', from=2-1, to=2-3]
		\arrow["{||}"{description}, draw=none, from=2-1, to=3-1]
		\arrow["{||}"{description}, draw=none, from=2-3, to=3-3]
	\end{tikzcd}\]
	is sent by $F$ to a distinguished square. Letting $\mathcal R$ be the precontextual category obtained from $L(\mathcal Q)$ by imposing that every square of the above form be distinguished, we have
	$$
	\Hom_\Cont(L(\mathcal R),-) \cong \Hom_\Precont(\mathcal R,-) \cong \Hom(\mathcal A,\mathcal B;-).
	$$
	We conclude that the composite
	$$
	|\mathcal A| \times |\mathcal B| \rightarrow \mathcal Q \rightarrow L(\mathcal Q) \rightarrow \mathcal R \rightarrow L(\mathcal R)
	$$
	is a bimorphism that realizes $L(\mathcal R)$ as a tensor product of $\mathcal A$ and $\mathcal B$.
\end{construction}

\begin{notation}
\label{not: binary tensor product}
We write $\otimes_{\mathcal A,\mathcal B}:\mathcal A \times \mathcal B \rightarrow \mathcal A \otimes \mathcal B$ for the universal bimorphism $\mathcal A \times \mathcal B \rightarrow L(\mathcal R)$.
\end{notation}

We now show that from binary tensor products we can construct $n$-ary tensor products for all $n \ge 1$:

\begin{proposition}
\label{prop: n-ary tensor product}
For all $n \ge 1$, there exist tensor products (Definition \ref{def: tensor product}) of all sequences $(\mathcal A_1,...,\mathcal A_n)$: the functor
\begin{align*}
	\mathcal A_1 \times \cdots \times \mathcal A_n & \longrightarrow \mathcal A_1 \otimes (\mathcal A_2 \otimes (\cdots (\mathcal A_{n-1} \otimes \mathcal A_n) \cdots ))\\
	(f_1, ..., f_n) & \longmapsto f_1 \otimes (f_2 \otimes (\cdots (f_{n-1} \otimes f_n) \cdots ))
\end{align*}
is a universal multimorphism.
\end{proposition}

\begin{proof}
Let us prove the claim by induction on $n$. A tensor product of a length-$1$ sequence $(\mathcal A_1)$ is given by the identity functor on $\mathcal A_1$, and binary tensor products have already been proved to exist. Let $n \ge 3$ and assume that the claim holds for $1$, ..., $n-1$. Consider contextual categories $\mathcal A_1$, ..., $\mathcal A_n$. By the induction hypothesis, there exists a tensor product $\mathcal T$ of $\mathcal A_2$, ..., $\mathcal A_n$, so we can use Theorem \ref{th: isomorphism of multihoms} to obtain the following natural chain of isomorphisms:
\begin{align*}
	\Hom(\mathcal A_1,\mathcal A_2...,\mathcal A_n;\mathcal C) & \cong \Hom(\mathcal A_2,...,\mathcal A_n;\mathcal C^{\mathcal A_1})\\
	 & \cong \Hom_\Cont(\mathcal T,\mathcal C^{\mathcal A_1})\\
	 & \cong \Hom(\mathcal A_1,\mathcal T;\mathcal C)\\
	 & \cong \Hom_\Cont(\mathcal A_1 \otimes \mathcal T, \mathcal C).
\end{align*}
Thus $\mathcal A_1 \otimes \mathcal T$ endowed with the isomorphism $\Hom(\mathcal A_1,...,\mathcal A_n;-) \cong \Hom_\Cont(\mathcal A_1 \otimes \mathcal T,-)$ is a tensor product of $\mathcal A_1$, ..., $\mathcal A_n$.
\end{proof}

\begin{notation}
\label{not: n-ary tensor product}
We write $\otimes_{\mathcal A_1, ..., \mathcal A_n}:\mathcal A_1 \times \cdots \times \mathcal A_n \rightarrow \mathcal A_1 \otimes \cdots \otimes \mathcal A_n$ for the universal multimorphism $\mathcal A_1 \times \cdots \times \mathcal A_n \rightarrow \mathcal A_1 \otimes (\mathcal A_2 \otimes (\cdots (\mathcal A_{n-1} \otimes \mathcal A_n) \cdots ))$ from Proposition \ref{prop: n-ary tensor product}.
\end{notation}

\begin{construction}[Symmetry isomorphisms]
\label{const: symmetry isomorphisms}
Let $\mathcal A_1$, ..., $\mathcal A_n$ be contextual categories. By Proposition \ref{prop: symmetry of multimorphisms is natural}, for each permutation $\sigma \in S_n$ we have a natural isomorphism $\Hom(\mathcal A_{\sigma 1},...,\mathcal A_{\sigma n};-) \cong \Hom(\mathcal A_1,...,\mathcal A_n;-)$; by the Yoneda lemma, the latter is induced by a unique isomorphism
$$
\beta_\sigma:\mathcal A_1 \otimes \cdots \otimes \mathcal A_n \longrightarrow \mathcal A_{\sigma 1} \otimes \cdots \otimes \mathcal A_{\sigma n}.
$$
The same proposition shows that although the diagram
\[\begin{tikzcd}[ampersand replacement=\&]
	{\mathcal A_1 \otimes \cdots \otimes \mathcal A_n} \&\& {\mathcal A_{\sigma 1} \otimes \cdots \otimes \mathcal A_{\sigma n}} \\
	{\mathcal A_1 \times \cdots \times \mathcal A_n} \&\& {\mathcal A_{\sigma 1} \times \cdots \times \mathcal A_{\sigma n}}
	\arrow["{\beta_\sigma}", from=1-1, to=1-3]
	\arrow["{\otimes_{\mathcal A_1, ..., \mathcal A_n}}", from=2-1, to=1-1]
	\arrow["{\underline{\sigma}}"', from=2-1, to=2-3]
	\arrow["{\otimes_{\mathcal A_{\sigma 1},...,\mathcal A_{\sigma n}}}"', from=2-3, to=1-3]
\end{tikzcd}\]
generally does not commute strictly, it does up to a canonical natural isomorphism ${\beta_\sigma \circ \otimes_{\mathcal A_1, ..., \mathcal A_n} \Longrightarrow \otimes_{\mathcal A_{\sigma 1}, ..., \mathcal A_{\sigma n}} \circ \underline{\sigma}}$, namely, the unique isomorphism realizing $\beta_\sigma$ as a permutative morphism\footnote{This is an abuse of terminology: when we say that $\beta_\sigma$ is a permutative morphism from $\mathcal A_1 \otimes \cdots \otimes \mathcal A_n$ to $\mathcal A_n \otimes \cdots \otimes \mathcal A_{\sigma n}$, we mean that it is an arrow
	$$
	(\mathcal A_1 \otimes \cdots \otimes \mathcal A_n,\; id, \; \otimes_{\mathcal A_1, ..., \mathcal A_n}) \longrightarrow (\mathcal A_n \otimes \cdots \otimes \mathcal A_{\sigma n},\; \sigma,\; \otimes_{\mathcal A_{\sigma 1}, ..., \mathcal A_{\sigma n}})
	$$
	in the category $\Perm(\mathcal A_1, ..., \mathcal A_n)$ from Definition \ref{def: permutative morphism}.} from $\mathcal A_1 \otimes \cdots \otimes \mathcal A_n$ to $\mathcal A_n \otimes \cdots \otimes \mathcal A_{\sigma n}$ (see Definition \ref{def: permutative morphism}).
	
\vspace{0.5em}

For $\mathcal A$, $\mathcal B \in \Cont$, we let $\beta_{\mathcal A,\mathcal B}:\mathcal A \otimes \mathcal B \rightarrow \mathcal B \otimes \mathcal A$ be $\beta_\sigma$ where $\sigma$ is the non-identity element of $S_2$.
\end{construction}

\subsection{Associativity}

We will now prove that the tensor product of contextual categories is associative up to coherent isomorphism.

\vspace{0.5em}

Let $\mathcal A$, $\mathcal B$ and $\mathcal C$ be contextual categories. To define a comparison morphism $\mathcal A \otimes (\mathcal B \otimes \mathcal C) \rightarrow (\mathcal A \otimes \mathcal B) \otimes \mathcal C$, we start by considering the bimorphism
$$
\otimes_{\mathcal A \otimes \mathcal B, \mathcal C}:(\mathcal A \otimes \mathcal B) \times \mathcal C \longrightarrow (\mathcal A \otimes \mathcal B) \otimes \mathcal C.
$$
By Theorem \ref{th: multicategory structure 5} (or, more directly, Proposition \ref{prop: multicategory structure 2}), the functor
$$
((- \otimes_{\mathcal A,\mathcal B} -) \otimes_{\mathcal A \otimes \mathcal B, \mathcal C} -): \mathcal A \times \mathcal B \times \mathcal C \longrightarrow (\mathcal A \otimes \mathcal B) \otimes \mathcal C
$$
is a multimorphism. On the other hand, by Proposition \ref{prop: n-ary tensor product}, the functor
$$
(- \otimes_{\mathcal A,\mathcal B \otimes \mathcal C} (- \otimes_{\mathcal B,\mathcal C} -)):\mathcal A \times \mathcal B \times \mathcal C \longrightarrow \mathcal A \otimes (\mathcal B \otimes \mathcal C)
$$
is a universal multimorphism out of $(\mathcal A,\mathcal B,\mathcal C)$. It now follows:

\begin{proposition}
There exists a unique contextual functor
$$
\mathcal A \otimes (\mathcal B \otimes \mathcal C) \longrightarrow (\mathcal A \otimes \mathcal B) \otimes \mathcal C
$$
that sends $f \otimes (g \otimes h)$ to $(f \otimes g) \otimes h$ for all morphisms $f$, $g$, $h$ in $\mathcal A$, $\mathcal B$, $\mathcal C$, respectively. We will denote it by $\alpha_{\mathcal A,\mathcal B,\mathcal C}$ and call it the \emph{associator}. \qed
\end{proposition}

Our next goal is to prove that $\alpha_{\mathcal A,\mathcal B,\mathcal C}$ is an isomorphism. Note that we have no direct access to an analogous functor in the opposite direction. This is due to how the factors are associated in Proposition \ref{prop: n-ary tensor product}, which in turn follows from the form of Theorem \ref{th: isomorphism of multihoms}.

The strategy will be to compare $\mathcal A \otimes (\mathcal B \otimes \mathcal C)$ and $(\mathcal A \otimes \mathcal B) \otimes \mathcal C$ by means of a third object, namely, $\mathcal C \otimes (\mathcal A \otimes \mathcal B)$. This is motivated by the following observation: Construction \ref{const: symmetry isomorphisms} for $n = 3$ gives an isomorphism $\mathcal A \otimes (\mathcal B \otimes \mathcal C) \cong \mathcal C \otimes (\mathcal A \otimes \mathcal B)$, and for $n = 2$ it gives an isomorphism $\mathcal C \otimes (\mathcal A \otimes \mathcal B) \cong (\mathcal A \otimes \mathcal B) \otimes \mathcal C$. However, we cannot immediately conclude that the composite $\mathcal A \otimes (\mathcal B \otimes \mathcal C) \cong \mathcal C \otimes (\mathcal A \otimes \mathcal B) \cong (\mathcal A \otimes \mathcal B) \otimes \mathcal C$ equals $\alpha_{\mathcal A, \mathcal B, \mathcal C}$ since symmetry isomorphisms generally do not preserve pure tensors; that will be done using the properties of shuffling diagrams and of permutative morphisms described in \S\ref{subsec: transporting multimorphisms}.

\begin{proposition}
\label{prop: associator is an isomorphism}
$\alpha_{\mathcal A,\mathcal B,\mathcal C}:\mathcal A \otimes (\mathcal B \otimes \mathcal C) \rightarrow (\mathcal A \otimes \mathcal B) \otimes \mathcal C$ is an isomorphism.
\end{proposition}

\begin{proof}
Firstly, consider the universal trimorphisms (see Proposition \ref{prop: n-ary tensor product})

\noindent
\begin{minipage}{0.45\textwidth}
	\begin{align*}
		\otimes_{\mathcal A,\mathcal B,\mathcal C}:\mathcal A \times \mathcal B \times \mathcal C & \longrightarrow \mathcal A \otimes (\mathcal B \otimes \mathcal C) \\
		(f,g,h) & \longmapsto f \otimes (g \otimes h)
	\end{align*}
\end{minipage}
\begin{minipage}{0.45\textwidth}
	\begin{align*}
		\otimes_{\mathcal C,\mathcal A,\mathcal B}:\mathcal C \times \mathcal A \times \mathcal B & \longrightarrow \mathcal C \otimes (\mathcal A \otimes \mathcal B) \\
		(h,f,g) & \longmapsto h \otimes (f \otimes g)
	\end{align*}
\end{minipage}
\vspace{0.5em}

Letting $I:\mathcal A \times \mathcal B \times \mathcal C \rightarrow \mathcal C \times \mathcal A \times \mathcal B$ be the cartesian symmetry isomorphism, we obtain a shuffling diagram (see Definition \ref{def: shuffling diagram})

\[\begin{tikzcd}[ampersand replacement=\&]
	\& {\mathcal C \otimes (\mathcal A \otimes \mathcal B)} \\
	\\
	{\mathcal A \times \mathcal B \times \mathcal C} \&\& {\mathcal C \times \mathcal A \times \mathcal B.}
	\arrow[""{name=0, anchor=center, inner sep=0}, "F", from=3-1, to=1-2]
	\arrow["{I}"', from=3-1, to=3-3]
	\arrow[""{name=1, anchor=center, inner sep=0}, "\otimes_{\mathcal C,\mathcal A,\mathcal B}"', from=3-3, to=1-2]
	\arrow["\lambda", Rightarrow,shift right=2,shorten=15pt, from=0, to=1]
\end{tikzcd}\]
By Construction \ref{const: symmetry isomorphisms}, we have a factorization
$$
\mathcal A \times \mathcal B \times \mathcal C \overset{\otimes_{\mathcal A,\mathcal B,\mathcal C}}{\longrightarrow} \mathcal A \otimes (\mathcal B \otimes \mathcal C) \overset{I'}{\longrightarrow} \mathcal C \otimes (\mathcal A \otimes \mathcal B)
$$
of $F$ where $I'$ is an isomorphism. By construction, $I'$ is permutative (Definition \ref{def: permutative morphism}) with respect to $\otimes_{\mathcal A,\mathcal B,\mathcal C}$ and $\otimes_{\mathcal C,\mathcal A,\mathcal B}$.

\vspace{0.5em}

On the other hand, consider the universal bimorphism $\otimes_{\mathcal A \otimes \mathcal B,\mathcal C}:(\mathcal A \otimes \mathcal B) \times \mathcal C \rightarrow (\mathcal A \otimes \mathcal B) \otimes \mathcal C$. Letting $J:\mathcal C \times (\mathcal A \otimes \mathcal B) \rightarrow (\mathcal A \otimes \mathcal B) \times \mathcal C$ be the cartesian symmetry isomorphism, we obtain a shuffling diagram
\[\begin{tikzcd}[ampersand replacement=\&]
	\& {(\mathcal A \otimes \mathcal B) \otimes \mathcal C} \\
	\\
	{\mathcal C \times (\mathcal A \otimes \mathcal B)} \&\& {(\mathcal A \otimes \mathcal B) \times \mathcal C}
	\arrow[""{name=0, anchor=center, inner sep=0}, "G", from=3-1, to=1-2]
	\arrow["{J}"', from=3-1, to=3-3]
	\arrow[""{name=1, anchor=center, inner sep=0}, "\otimes_{\mathcal A \otimes \mathcal B,\mathcal C}"', from=3-3, to=1-2]
	\arrow["\mu", Rightarrow,shift right=2,shorten=15pt, from=0, to=1]
\end{tikzcd}\]
and, by Construction \ref{const: symmetry isomorphisms}, a factorization
$$
\mathcal C \times (\mathcal A \otimes \mathcal B) \overset{\otimes_{\mathcal C,\mathcal A \otimes \mathcal B}}{\longrightarrow} \mathcal C \otimes (\mathcal A \otimes \mathcal B) \overset{J'}{\longrightarrow} (\mathcal A \otimes \mathcal B) \otimes \mathcal C
$$
of $G$ where $J'$ is an isomorphism. Also, by Proposition \ref{prop: multicategory structure 2}, the functor
\begin{align*}
	M:\mathcal A \times \mathcal B \times \mathcal C & \longrightarrow (\mathcal A \otimes \mathcal B) \otimes \mathcal C\\
	(f,g,h) & \longmapsto (f \otimes g) \otimes h
\end{align*}
is a multimorphism with the following property: the outer composite natural isomorphism in the diagram
\[\begin{tikzcd}[ampersand replacement=\&]
	\& {(\mathcal A \otimes \mathcal B) \otimes \mathcal C} \\
	{\mathcal C \times (\mathcal A \otimes \mathcal B)} \&\& {(\mathcal A \otimes \mathcal B) \times \mathcal C} \\
	{\mathcal C \times \mathcal A \times \mathcal B} \&\& {\mathcal A \times \mathcal B \times \mathcal C,}
	\arrow[""{name=0, anchor=center, inner sep=0}, "G", from=2-1, to=1-2]
	\arrow["J"', from=2-1, to=2-3]
	\arrow[""{name=1, anchor=center, inner sep=0}, "{\otimes_{\mathcal A \otimes \mathcal B, \mathcal C}}"', from=2-3, to=1-2]
	\arrow[""{name=2, anchor=center, inner sep=0}, "{\Id_\mathcal C \times \otimes_{\mathcal A,\mathcal B}}", from=3-1, to=2-1]
	\arrow["{I^{-1}}"', from=3-1, to=3-3]
	\arrow[""{name=3, anchor=center, inner sep=0}, "{\otimes_{\mathcal A,\mathcal B} \times \Id_\mathcal C}"', from=3-3, to=2-3]
	\arrow["\mu"', Rightarrow,shorten=20pt, from=0, to=1]
	\arrow["{=}"{description}, draw=none, from=2, to=3]
\end{tikzcd}\]
that is,
$$
\mu(\Id_\mathcal C \times \otimes_{\mathcal A,\mathcal B}):\;\; \overbrace{J' \circ \otimes_{\mathcal C,\mathcal A,\mathcal B}}^{G \circ (\Id_\mathcal C \times \otimes_{\mathcal A,\mathcal B})} \qquad =\joinrel=\joinrel=\joinrel=\joinrel\Longrightarrow \overbrace{M \circ I^{-1},}^{\otimes_{\mathcal A \otimes \mathcal B,\mathcal C} \circ (\otimes_{\mathcal A,\mathcal B} \times \Id_\mathcal C) \circ I^{-1}}
$$
defines a shuffling diagram $(M,\; J' \circ \otimes_{\mathcal C,\mathcal A,\mathcal B},\; \mu(\Id_\mathcal C \times \otimes_{\mathcal A,\mathcal B}))$. Hence $J'$ is permutative with respect to $\otimes_{\mathcal C,\mathcal A,\mathcal B}$ and $M$.

\vspace{0.5em}

It now follows that $J' \circ I':\mathcal A \otimes (\mathcal B \otimes \mathcal C) \rightarrow (\mathcal A \otimes \mathcal B) \otimes \mathcal C$ is permutative with respect to $\otimes_{\mathcal A,\mathcal B,\mathcal C}$ and $M$. As $\alpha_{\mathcal A,\mathcal B,\mathcal C}$ is also permutative, by Proposition \ref{prop: permutative morphisms main property} we have $J' \circ I' \circ \otimes_{\mathcal A,\mathcal B,\mathcal C} = \alpha_{\mathcal A,\mathcal B,\mathcal C} \circ \otimes_{\mathcal A,\mathcal B,\mathcal C}$. But from $\otimes_{\mathcal A,\mathcal B,\mathcal C}$ being a universal trimorphism we conclude that $J' \circ I' = \alpha_{\mathcal A,\mathcal B,\mathcal C}$ and, in particular, that $\alpha_{\mathcal A,\mathcal B,\mathcal C}$ is an isomorphism.
 \end{proof}

One of the requirements for $(\Cont,\otimes, \alpha, \ldots)$ to be a monoidal category is that the \emph{pentagon identity} hold:

\begin{lemma}
\label{lem: pentagon}
For all contextual categories $\mathcal A$, $\mathcal B$, $\mathcal C$, $\mathcal D$, the following diagram commutes:
\[\begin{tikzcd}
	{((\mathcal A \otimes \mathcal B) \otimes \mathcal C)) \otimes \mathcal D} &&&& {(\mathcal A \otimes \mathcal B) \otimes (\mathcal C \otimes \mathcal D)} \\
	{(\mathcal A \otimes (\mathcal B \otimes \mathcal C))\otimes \mathcal D} && {\mathcal A \otimes ((\mathcal B \otimes \mathcal C) \otimes \mathcal D)} && {\mathcal A \otimes (\mathcal B \otimes (\mathcal C \otimes \mathcal D))}
	\arrow["{\alpha_{\mathcal A \otimes \mathcal B,\mathcal C,\mathcal D}}", from=1-1, to=1-5]
	\arrow["{\alpha_{\mathcal A,\mathcal B,\mathcal C} \otimes \Id_\mathcal D}"', from=1-1, to=2-1]
	\arrow["{\alpha_{\mathcal A,\mathcal B,\mathcal C \otimes \mathcal D}}", from=1-5, to=2-5]
	\arrow["{\alpha_{\mathcal A,\mathcal B \otimes \mathcal C,\mathcal D}}"', from=2-1, to=2-3]
	\arrow["{\Id_\mathcal A \otimes \alpha_{\mathcal B,\mathcal C,\mathcal D}}"', from=2-3, to=2-5]
\end{tikzcd}\]
\end{lemma}

\begin{proof}
Note that the diagram consisting of the inverses of the above morphisms commutes when restricted to pure tensors $((f \otimes g) \otimes h) \otimes i$, where $f$, $g$, $h$, $i$ are arrows in $\mathcal A$, $\mathcal B$, $\mathcal C$, $\mathcal D$:
\[\begin{tikzcd}[ampersand replacement=\&]
	{(f \otimes g) \otimes h) \otimes i} \&\& {(f \otimes g) \otimes (h \otimes i)} \\
	{(f \otimes (g \otimes h)) \otimes i} \& {f \otimes ((g \otimes h) \otimes i)} \& {f \otimes (g \otimes (h \otimes i)).}
	\arrow[maps to, from=1-3, to=1-1]
	\arrow[maps to, from=2-1, to=1-1]
	\arrow[maps to, from=2-2, to=2-1]
	\arrow[maps to, from=2-3, to=1-3]
	\arrow[maps to, from=2-3, to=2-2]
\end{tikzcd}\]
But the functor $\mathcal A \times \mathcal B \times \mathcal C \times \mathcal D \rightarrow \mathcal A \otimes (\mathcal B \otimes (\mathcal C \otimes \mathcal D))$ given by $(f,g,h,i) \mapsto f \otimes (g \otimes (h \otimes i))$ is a universal multimorphism, so two contextual functors out of $\mathcal A \times \mathcal B \times \mathcal C \times \mathcal D \rightarrow \mathcal A \otimes (\mathcal B \otimes (\mathcal C \otimes \mathcal D))$ are equal provided that they agree on pure tensors. 
\end{proof}

\subsection{Unitality}

Recall that $\mathcal O_1$ is the contextual category associated with the precontextual category $\mathcal O_1^{\text{pre}}$ whose underlying category is the poset $1 \rightarrow 0$ with $0$, $1$ of length $0$, $1$, respectively. The universal property of $\mathcal O_1$ --- that is, being freely generated by a length-$1$ object --- and the equivalence $\Cont \simeq \GAT$ imply that $\mathcal O_1$ is isomorphic to the syntactic category of the \textsc{gat} generated by a single sort axiom $\vdash O \tp$. This syntactic category admits an explicit description: it has a unique length-$n$ object for each $n \ge 0$, namely, the equivalence class $[x_1:O, ..., x_n:O]$, and a morphism $[x_1:O,...,x_n] \rightarrow [x_1:O,...,x_m:O]$ corresponds to a function $\{x_1, ..., x_m\} \rightarrow \{x_1, ..., x_n\}$. In other words, the underlying category of $\mathcal O_1$ is opposite to a skeleton of the category of finite sets. We will denote the unique length-$1$ object of $\mathcal O_1$ by $o$; its unique length-$n$ object is then $o^n$. The canonical morphism $\mathcal O_1^{\text{pre}} \rightarrow \mathcal O_1$ sends $1$ to $o$.

\vspace{0.5em}

Note that we have isomorphisms
$$
\Hom(\mathcal A;\mathcal C) \cong \Ob_1(\mathcal C^\mathcal A) \cong \Hom(\mathcal O_1;\mathcal C^\mathcal A) \cong \Hom(\mathcal A,\mathcal O_1;\mathcal C) \cong \Hom(\mathcal A \otimes \mathcal O_1;\mathcal C)
$$
natural in $\mathcal A \in \Cont^{\text{op}}$ and $\mathcal C \in \Cont$. In particular, the Yoneda lemma implies that, setting $\mathcal C = \mathcal A$ and calculating the image of $\Id_\mathcal A$ along this chain of bijections, we have an isomorphism
$$
\rho_\mathcal A: \mathcal A \otimes \mathcal O_1 \longrightarrow \mathcal A
$$
given on pure tensors by sending $a \otimes o^n$ to $(\Id_\mathcal A^n)_\varheartsuit(a)$ where $\Id_\mathcal A$ is viewed as a length-$1$ object of $\mathcal A^\mathcal A$. It is called the \emph{left unitor}. In particular, $\rho_\mathcal A(a \otimes o) = a$.

Moreover, composing $\rho_\mathcal A$ with the symmetry isomorphism $\beta_{\mathcal O_1,\mathcal A}:\mathcal \mathcal O_1 \otimes \mathcal A \rightarrow \mathcal A \otimes \mathcal O_1$ yields an isomorphism, called the \emph{left unitor},
$$
\lambda_\mathcal A:\mathcal O_1 \otimes \mathcal A \longrightarrow \mathcal A.
$$
Expanding the construction of $\mathcal O_1 \otimes \mathcal A \cong \mathcal A \otimes \mathcal O_1$, it is straightforward that it sends $o \otimes f$ to $f \otimes o$ for any arrow $f$ in $\mathcal A$. Since $\lambda_\mathcal A(- \otimes -):\mathcal O_1 \times \mathcal A \rightarrow \mathcal A$, being a bimorphism, preserves distinguished powers in the first coordinate, we have
$$
\lambda_\mathcal A(o^n \otimes f) = \lambda_\mathcal A((o \otimes f)^n) = (\lambda_\mathcal A(o \otimes f))^n = (\rho_\mathcal A(f \otimes o))^n = f^n.
$$

We now check the final requirement for $(\Cont,\otimes,\alpha,\lambda,\rho)$ to be a monoidal category --- the \emph{triangle identity}:

\begin{lemma}
\label{lem: triangle}
For any contextual categories $\mathcal A$, $\mathcal B$, the following diagram commutes:
\[\begin{tikzcd}[ampersand replacement=\&]
	{\mathcal A \otimes (\mathcal O_1 \otimes \mathcal B)} \&\& {(\mathcal A \otimes \mathcal O_1) \otimes \mathcal B} \\
	\& {\mathcal A \otimes \mathcal B.}
	\arrow["{\alpha_{\mathcal A,\mathcal O_1,\mathcal B}}", from=1-1, to=1-3]
	\arrow["{\Id_\mathcal A \otimes \lambda_\mathcal B}"', from=1-1, to=2-2]
	\arrow["{\rho_\mathcal A \otimes \Id_\mathcal B}", from=1-3, to=2-2]
\end{tikzcd}\]
\end{lemma}

\begin{proof}
Since the three arrows are isomorphisms, we can equivalently verify
$$
\alpha_{\mathcal A, \mathcal O_1,\mathcal B} \circ (\Id_\mathcal A \otimes \lambda_\mathcal B)^{-1} = (\rho_\mathcal A \otimes \Id_\mathcal B)^{-1}.
$$
This equality, in turn, follows from the fact that it holds when restricted to pure tensors $f \otimes g$ where $f$, $g$ are arrows in $\mathcal A$, $\mathcal B$:
$$
\alpha_{\mathcal A, \mathcal O_1,\mathcal B}((\Id_\mathcal A \otimes \lambda_\mathcal B)^{-1}(f \otimes g)) = \alpha_{\mathcal A, \mathcal O_1,\mathcal B}(f \otimes (o \otimes g)) = (f \otimes o) \otimes g = (\rho_\mathcal A \otimes \Id_\mathcal B)^{-1}(f,g).
$$
\end{proof}

\begin{theorem}
The category $\Cont$, of contextual categories and contextual functors, equipped with the tensor product $\otimes:\Cont \times \Cont \rightarrow \Cont$, associator $\alpha$, unit $\mathcal O_1$, left unitor $\lambda$, and right unitor $\rho$ is a monoidal category. \qed
\end{theorem}

\subsection{Symmetry}

We will now prove that $(\Cont,\otimes,\ldots)$ equipped with the isomorphisms $\beta_{\mathcal A,\mathcal B}:\mathcal A \otimes \mathcal B \rightarrow \mathcal B \otimes \mathcal A$ is a symmetric monoidal category (Propositions \ref{prop: symmetry coherence 1} and \ref{prop: symmetry coherence 2}). As a consequence, the multicategory $\mathscr{Cont}$ (Definition \ref{th: multicategory structure 5}) is symmetric and representable (see \cite{Her00}, \cite{Lei04}).

\vspace{0.5em}

The key for verifying the required coherence identities will be Proposition \ref{prop: permutative morphisms main property} on permutative morphisms.

\begin{lemma}
\label{lem: symmetry coherence preliminary lemma}
Let $\mathcal A$, $\mathcal B$, $\mathcal C$ be contextual categories. Recall that, by Theorem \ref{th: multicategory structure 5}, the functors
\begin{align*}
\otimes_{\mathcal A,\mathcal B,\mathcal C} = (- \otimes_{\mathcal A,\mathcal B \otimes \mathcal C} (- \otimes_{\mathcal B,\mathcal C} -)) :\mathcal A \times \mathcal B \times \mathcal C &\longrightarrow \mathcal A \otimes (\mathcal B \otimes \mathcal C),\\
\otimes_{\mathcal A,\mathcal C,\mathcal B} = (- \otimes_{\mathcal A,\mathcal C \otimes \mathcal B} (- \otimes_{\mathcal C,\mathcal B} -)) :\mathcal A \times \mathcal C \times \mathcal B &\longrightarrow \mathcal A \otimes (\mathcal C \otimes \mathcal B),\\
((- \otimes_{\mathcal A,\mathcal B} -) \otimes_{\mathcal A \otimes \mathcal B, \mathcal C} -): \mathcal A \times \mathcal B \times \mathcal C &\longrightarrow (\mathcal A \otimes \mathcal B) \otimes \mathcal C,\\
((- \otimes_{\mathcal B,\mathcal A} -) \otimes_{\mathcal B \otimes \mathcal A, \mathcal C} -): \mathcal B \times \mathcal A \times \mathcal C &\longrightarrow (\mathcal B \otimes \mathcal A) \otimes \mathcal C
\end{align*}
are multimorphisms. The isomorphisms
\begin{align*}
	\Id_\mathcal A \otimes \beta_{\mathcal C,\mathcal B}:\mathcal A \otimes (\mathcal C \otimes \mathcal B) &\longrightarrow \mathcal A \otimes (\mathcal B \otimes \mathcal C),\\
	\beta_{\mathcal B,\mathcal A} \otimes \Id_\mathcal C: (\mathcal B \otimes \mathcal A) \otimes \mathcal C &\longrightarrow (\mathcal A \otimes \mathcal B) \otimes \mathcal C
\end{align*}
are permutative with respect to the above multimorphisms.
\end{lemma}

\begin{proof}
In the first case, we must check that there exists a natural isomorphism $\phi$ as in
\[\begin{tikzcd}[ampersand replacement=\&]
	{\mathcal A \otimes (\mathcal C \otimes \mathcal B)} \&\& {\mathcal A \otimes (\mathcal B \otimes \mathcal C)} \\
	{\mathcal A \times \mathcal C \times \mathcal B} \&\& {\mathcal A \times \mathcal B \times \mathcal C}
	\arrow["\Id_\mathcal A \otimes \beta_{\mathcal C,\mathcal B}", from=1-1, to=1-3]
	\arrow[""{name=0, anchor=center, inner sep=0}, "\otimes_{\mathcal A,\mathcal C,\mathcal B}", from=2-1, to=1-1]
	\arrow["\cong"', from=2-1, to=2-3]
	\arrow[""{name=1, anchor=center, inner sep=0}, "\otimes_{\mathcal A,\mathcal B,\mathcal C}"', from=2-3, to=1-3]
	\arrow["\phi", Rightarrow,shorten=20pt, from=0, to=1]
\end{tikzcd}\]
such that $(\otimes_{\mathcal A,\mathcal B,\mathcal C},\; (\Id_\mathcal A \otimes \beta_{\mathcal C,\mathcal B}) \circ \otimes_{\mathcal A,\mathcal C,\mathcal B}, \; \phi)$ is a shuffling diagram (Definition \ref{def: shuffling diagram}).

By construction, the isomorphism $\beta_{\mathcal C,\mathcal B}$ is permutative with respect to the bimorphisms ${\otimes_{\mathcal C,\mathcal B}:\mathcal C \times \mathcal B \rightarrow \mathcal C \otimes \mathcal B}$ and $\otimes_{\mathcal B,\mathcal C}:\mathcal B \times \mathcal C \rightarrow \mathcal B \otimes \mathcal C$. Writing $b \otimes^* c$ for the image of $c \otimes b$ under $\beta_{\mathcal C,\mathcal B}$, let
$$
\eta = (\eta_{c,b}:b \otimes^* c \rightarrow b \otimes c)_{(c,b) \in \mathcal C \times \mathcal B}
$$
be the natural isomorphism that realizes $\eta_{\mathcal C,\mathcal B}$ as a permutative morphism. We will verify that $\phi:(\Id_\mathcal A \otimes \beta_{\mathcal C,\mathcal B}) \circ \otimes_{\mathcal A,\mathcal C,\mathcal B} \Longrightarrow \otimes_{\mathcal A,\mathcal B,\mathcal C}$ defined as
$$
(id_a \otimes \eta_{c,b})_{(a,c,b) \in \mathcal A \times \mathcal C \times \mathcal B}
$$
has the desired property. Observe that for each $b$, $c$ of length $\ge 1$ in $\mathcal B$, $\mathcal C$, we have a length-$1$ distinguished square
\[
\dsqua{b \otimes^* c}{b \otimes c}{\partial(b \otimes^* c)}{\partial(b \otimes c).}{\eta_{c,b}}{h}{p_{b \otimes^* c}}{p_{b \otimes c}}
\]
where $h$ is induced by $\eta$ and functoriality of limits. As $\otimes_{\mathcal A,\mathcal B \otimes \mathcal C}:\mathcal A \times (\mathcal B \otimes \mathcal C) \rightarrow \mathcal A \otimes (\mathcal B \otimes \mathcal C)$ is a multimorphism, by Proposition \ref{prop: comparison bimorphisms and 2-ary maps} it is a bimorphism in the sense of Definition \ref{def: bimorphism}; it follows that for each $a \in \mathcal A$ with $\ell(a) \ge 1$, the commutative diagram
\[\begin{tikzcd}[ampersand replacement=\&]
	{a \otimes (b \otimes^* c)} \&\&\& {a \otimes (b \otimes c)} \\
	\&\& {a \otimes \partial(b \otimes^* c)} \&\&\& {a \otimes \partial(b \otimes c)} \\
	{\partial a \otimes (b \otimes^* c)} \&\&\& {\partial a \otimes (b \otimes c)} \\
	\&\& {\partial a \otimes \partial(b \otimes^* c)} \&\&\& {\partial a \otimes \partial(b \otimes c)}
	\arrow["{id \otimes \eta_{c,b}}", from=1-1, to=1-4]
	\arrow["{id \otimes p_{b \otimes^* c}}"{description}, from=1-1, to=2-3]
	\arrow[two heads, from=1-1, to=3-1]
	\arrow["{id \otimes p_{b \otimes c}}", from=1-4, to=2-6]
	\arrow[two heads, from=1-4, to=3-4]
	\arrow["{id \otimes h}"{description, pos=0.8}, from=2-3, to=2-6]
	\arrow[two heads, from=2-3, to=4-3]
	\arrow[two heads, from=2-6, to=4-6]
	\arrow["{id \otimes \eta_{c,b}}"{description, pos=0.3}, from=3-1, to=3-4]
	\arrow["{id \otimes p_{b \otimes^* c}}"', from=3-1, to=4-3]
	\arrow["{id \otimes p_{b \otimes c}}"{description}, from=3-4, to=4-6]
	\arrow["{id \otimes h}"', from=4-3, to=4-6]
\end{tikzcd}\]
is such that the left and right faces are --- by (i) from Proposition \ref{prop: characterization of maps to exponential, sym} --- relative length-$1$ display maps, and the induced comparison square between the respective gap maps,
\[
\dsqua{a \otimes (b \otimes^* c)}{a \otimes (b \otimes c)}{\partial(a \otimes (b \otimes^* c))}{\partial(a \otimes (b \otimes c)),}{id \otimes \eta_{c,b}}{}{p_{a \otimes (b \otimes^* c)}}{p_{a \otimes (b \otimes c)}}
\]
is --- by (iv) from Proposition \ref{prop: characterization of maps to exponential, sym} ---  distinguished.

\vspace{0.5em}

A proof in the case of $\beta_{\mathcal B,\mathcal A} \otimes \Id_\mathcal C$ can be given similarly, with (iii) from Proposition \ref{prop: characterization of maps to exponential, sym} used instead of (iv).
\end{proof}

We are now ready to verify that $(\Cont,\otimes,\ldots,\beta)$ is a symmetric monoidal category.

In what follows, we talk about permutative morphisms without explicitly referring to the multimorphisms that are part of their structure. This will not cause ambiguity as we will only deal with multimorphisms (provided by Theorem \ref{th: multicategory structure 5}) that, like $((- \otimes_{\mathcal A,\mathcal B} -) \otimes_{\mathcal A \otimes \mathcal B,\mathcal C} -):\mathcal A \times \mathcal C \times \mathcal B \rightarrow (\mathcal A \otimes \mathcal C) \otimes \mathcal B$, map in the canonical way a cartesian product to an iterated binary tensor product that uses the same factors in the given order, but associated arbitrarily.

\begin{proposition}
\label{prop: symmetry coherence 1}
For all $\mathcal A$, $\mathcal B \in \Cont$, the isomorphisms $\beta_{\mathcal A,\mathcal B}:\mathcal A \otimes \mathcal B \rightarrow \mathcal B \otimes \mathcal A$ and $\beta_{\mathcal B,\mathcal A}:\mathcal B \otimes \mathcal A \rightarrow \mathcal A \otimes \mathcal B$ are inverses of each other.
\end{proposition}

\begin{proof}
By construction, $\beta_{\mathcal A,\mathcal B}$ and $\beta_{\mathcal B,\mathcal A}$ are permutative morphisms. As permutative morphisms are closed under composition, $\beta_{\mathcal B,\mathcal A} \circ \beta_{\mathcal A,\mathcal B}:\mathcal A \otimes \mathcal B \rightarrow \mathcal A \otimes \mathcal B$ is permutative. Since $\Id_{\mathcal A \otimes \mathcal B}$ is also permutative, Proposition \ref{prop: permutative morphisms main property} yields an equality
$$
\beta_{\mathcal B,\mathcal A} \circ \beta_{\mathcal A,\mathcal B} \circ \otimes_{\mathcal A,\mathcal B} = \otimes_{\mathcal A,\mathcal B}
$$
between bimorphisms $\mathcal A \times \mathcal B \rightarrow \mathcal A \otimes \mathcal B$. As $\otimes_{\mathcal A,\mathcal B}$ is a universal bimorphism, we have $\beta_{\mathcal B,\mathcal A} \circ \beta_{\mathcal A,\mathcal B} = \Id_{\mathcal A \otimes \mathcal B}$.

Exchanging the roles of $\mathcal A$ and $\mathcal B$ in the above argument gives $\beta_{\mathcal A,\mathcal B} \circ \beta_{\mathcal B,\mathcal A} = \Id_{\mathcal B \otimes \mathcal A}$.
\end{proof}

\begin{proposition}
\label{prop: symmetry coherence 2}
For all $\mathcal A$, $\mathcal B$, $\mathcal C \in \Cont$, the diagram
\[
\tag{\texttt{*}}
\begin{tikzcd}[ampersand replacement=\&]
	{\mathcal A \otimes (\mathcal B \otimes \mathcal C)} \& {\mathcal A \otimes (\mathcal C \otimes \mathcal B)} \& {(\mathcal A \otimes \mathcal C) \otimes \mathcal B} \\
	{(\mathcal A \otimes \mathcal B) \otimes \mathcal C} \& {\mathcal C \otimes (\mathcal A \otimes \mathcal B)} \& {(\mathcal C \otimes \mathcal A) \otimes \mathcal B}
	\arrow["{\Id_\mathcal A \otimes \beta_{\mathcal B,\mathcal C}}", from=1-1, to=1-2]
	\arrow["{\alpha_{\mathcal A,\mathcal B,\mathcal C}}"', from=1-1, to=2-1]
	\arrow["{\alpha_{\mathcal A,\mathcal C,\mathcal B}}", from=1-2, to=1-3]
	\arrow["{\beta_{\mathcal A,\mathcal C} \otimes \Id_\mathcal B}", from=1-3, to=2-3]
	\arrow["{\beta_{\mathcal A \otimes \mathcal B,\mathcal C}}"', from=2-1, to=2-2]
	\arrow["{\alpha_{\mathcal C,\mathcal A,\mathcal B}}"', from=2-2, to=2-3]
\end{tikzcd}\]
commutes.
\end{proposition}

\begin{proof}
Every morphism in the diagram is permutative: the components of $\alpha$ and of $\beta$ by construction, and $\Id_\mathcal A \otimes \beta_{\mathcal B,\mathcal C}$ and $\beta_{\mathcal A,\mathcal C} \otimes \Id_\mathcal B$ by Lemma \ref{lem: symmetry coherence preliminary lemma}. It follows that the composites, say $F$ and $G$, of the two paths from $\mathcal A \otimes (\mathcal B \otimes \mathcal C)$ to $(\mathcal C \otimes \mathcal A) \otimes \mathcal B$ in the diagram are permutative.

By Proposition \ref{prop: permutative morphisms main property}, the trimorphisms $F \circ \otimes_{\mathcal A,\mathcal B,\mathcal C}$ and $G \circ \otimes_{\mathcal A,\mathcal B,\mathcal C}$ from $(\mathcal A,\mathcal B,\mathcal C)$ to $(\mathcal C \otimes \mathcal A) \otimes \mathcal B$ are equal. It now follows from $\otimes_{\mathcal A,\mathcal B,\mathcal C}$ being a universal trimorphism that $F = G$.
\end{proof}
\section{A description of certain pushout-tensor maps}
\label{sec: description of pushout-tensor}

Every contextual category $\mathcal C$ is such that $\mathcal O_0 \rightarrow \mathcal C$ (where $\mathcal O_0$ is the initial contextual category; see Example \ref{example: On pre}) can be decomposed as a transfinite composite of pushouts of morphisms of four particular kinds, which we describe below. Although this statement can be verified directly, it also follows from the equivalence $\Cont \simeq \GAT$ and the fact that such four kinds of morphisms correspond to the four kinds of axioms out of which every \textsc{gat} is constructed: sort introduction, term introduction, sort equality, and term equality.

The four special kinds of contextual functors are the following, for $n \ge 1$:
\begin{enumerate}[label=(A\arabic*)]
	\item Following Example \ref{example: On pre}, we have an inclusion morphism
	$$
	\iota_n^S:\mathcal O_{n-1}^{\text{pre}} \longrightarrow \mathcal O_n^{\text{pre}}.
	$$
	We will also write $\iota_n^S$ for its image $\mathcal O_{n-1} \rightarrow \mathcal O_n$ under $L:\Precont \rightarrow \Cont$. Syntactically, taking a pushout of a morphism $F:\mathcal O_{n-1} \rightarrow \mathcal C$ along $\iota_n^S$ corresponds to adding a sort in the context classified by $F$.
	
	\item Define $\mathcal O_n^{+ \; \text{pre}}$ as the precontextual category obtained by freely adding a section $s:o_{n-1} \rightarrow o_n$ of $o_n \twoheadrightarrow o_{n-1}$, and let $\mathcal O_n^+ = L(\mathcal O_n^{+ \; \text{pre}})$. We have an inclusion morphism
	$$
	\iota_n^T:\mathcal O_n^{\text{pre}} \longrightarrow \mathcal O_n^{+ \; \text{pre}},
	$$
	whose image under $L$ will also be denoted by $\iota_n^T$. Taking a pushout of $F:\mathcal O_n \rightarrow \mathcal C$ along $\iota_n^T$ corresponds to adding a term whose context and sort are classified, respectively, by $F(o_{n-1})$ and $F(o_n) \twoheadrightarrow F(o_{n-1})$.
	
	\item Let $\mathcal O_n^{\vee \; \text{pre}} = \mathcal O_n^{\text{pre}} \sqcup_{\mathcal O_{n-1}^{\text{pre}}} \mathcal O_n^{\text{pre}}$; explicitly, it is freely generated by a diagram of length-$1$ display maps
	\[\begin{tikzcd}[row sep=tiny]
		{o'_n} \\
		& {o_{n-1}} & \cdots & {o_1} & {o_0.} \\
		{o_n}
		\arrow["{p'_n}", from=1-1, to=2-2,two heads]
		\arrow["{p_{n-1}}", from=2-2, to=2-3,two heads]
		\arrow["{p_2}", from=2-3, to=2-4,two heads]
		\arrow["{p_1}", from=2-4, to=2-5,two heads]
		\arrow["{p_n}"', from=3-1, to=2-2,two heads]
	\end{tikzcd}\]
	We let
	$$
	\pi_n^S:\mathcal O_n^{\vee \; \text{pre}} \longrightarrow \mathcal O_n^{\text{pre}}
	$$
	be the codiagonal map, and its image under $L$ will also be denoted by $\pi_n^S$. We write $\mathcal O_n^\vee$ for $L(\mathcal O_n^{\vee \; \text{pre}})$. Note that pushouts of this morphism encode sort equality axioms.
	
	\item Let $\mathcal O_n^{++ \; \text{pre}} = \mathcal O_n^{+ \; \text{pre}} \sqcup_{\mathcal O_n^{\text{pre}}} \mathcal O_n^{+ \; \text{pre}}$; it is obtained from $\mathcal O_n^{\text{pre}}$ by freely adding two sections $s$, $s':o_{n-1} \rightarrow o_n$ of $o_n \twoheadrightarrow o_{n-1}$. We let
	$$
	\pi_n^T: \mathcal O_n^{++ \; \text{pre}} \longrightarrow \mathcal O_n^{+ \; \text{pre}}
	$$
	be the codiagonal map, and its image under $L$ will also be denoted by $\pi_n^T$. We write $\mathcal O_n^{++}$ for $L(\mathcal O_n^{++ \; \text{pre}})$. Pushouts of this morphism encode term equality axioms.
\end{enumerate}

It is natural to ask how the tensor product of contextual categories interacts with pushouts of these morphisms: given contextual categories $\mathcal A$, $\mathcal B$ where $\mathcal O_0 \rightarrow \mathcal A$ and $\mathcal O_0 \rightarrow \mathcal B$ have been decomposed as transfinite composites of pushouts of (A1)-(A4), do we automatically get such a decomposition for $\mathcal O_0 \rightarrow \mathcal A \otimes \mathcal B$?

Note that this was the case by construction for the tensor product of generalized algebraic theories from \cite{Alm25}. We will now describe, independently, an entirely analogous phenomenon for the tensor product of contextual categories as defined in the present text. More precisely, we will calculate the pushout-tensor maps $ F \widehat{\otimes} G$ (Definition \ref{def: pushout-tensor}) where $F$, $G$ are pushouts of maps among (A1)-(A4). In fact, in all cases, $ F \widehat{\otimes} G$ is itself a pushout of a map among (A1)-(A4), and this matches the assignment
$$
\{\text{axioms of } \bbA\} \times \{\text{axioms of } \bbB\} \longrightarrow \{\text{axioms of } \bbA \otimes \bbB\}
$$
from \cite{Alm25}. The equivalence between these two points of view will be the essential tool for proving, in \S\ref{sec: comparing syntactic and categorical}, that the syntactic tensor product from \cite{Alm25} matches the one from the present article.

\begin{definition}
\label{def: pushout-tensor}
Let $F:\mathcal A \rightarrow \mathcal A'$ and $G:\mathcal B \rightarrow \mathcal B'$ be morphisms of contextual categories. We define the \emph{pushout-tensor map}
$$
F \widehat{\otimes} G :(\mathcal A' \otimes \mathcal B) \sqcup_{\mathcal A \otimes \mathcal B} (\mathcal A \otimes \mathcal B') \longrightarrow \mathcal A' \otimes \mathcal B'
$$
as unique contextual functor making the following diagram commute:
\[\begin{tikzcd}
	{\mathcal A \otimes \mathcal B} & {\mathcal A \otimes \mathcal B'} \\
	{\mathcal A' \otimes \mathcal B} & {(\mathcal A' \otimes \mathcal B) \sqcup_{\mathcal A \otimes \mathcal B} (\mathcal A \otimes \mathcal B')} \\
	&& {\mathcal A' \otimes \mathcal B'.}
	\arrow["{id_\mathcal A \otimes G}", from=1-1, to=1-2]
	\arrow["{F \otimes id_\mathcal B}"', from=1-1, to=2-1]
	\arrow[from=1-2, to=2-2]
	\arrow["{F \otimes id_{\mathcal B'}}", curve={height=-18pt}, from=1-2, to=3-3]
	\arrow[from=2-1, to=2-2]
	\arrow["{id_{\mathcal A'} \otimes G}"', curve={height=18pt}, from=2-1, to=3-3]
	\arrow[dashed, from=2-2, to=3-3]
\end{tikzcd}\]
\end{definition}

We will now outline a description of $F \widehat{\otimes} G$ where $F$, $G$ are contextual functors among (A1)-(A4).

\begin{remark}
By Proposition \ref{prop: multimorphisms from contextual vs precontextual categories} and the definition of the tensor product, for $\mathcal A$, $\mathcal B \in \Precont$ and $\mathcal C \in \Cont$ we have a bijective correspondence, natural in each argument, between bimorphisms
$$
\mathcal A \times \mathcal B \longrightarrow \mathcal C
$$
and morphisms
$$
L(\mathcal A) \otimes L(\mathcal B) \longrightarrow \mathcal C.
$$
This will be used repeatedly in what follows.
\end{remark}

\subsubsection*{Description of $\iota_m^S \widehat{\otimes} \iota_n^S$}

Observe that for a contextual category $\mathcal C$, morphisms
$$
(\mathcal O_m \otimes \mathcal O_{n-1}) \sqcup_{\mathcal O_{m-1} \otimes \mathcal O_{n-1}} (\mathcal O_{m-1} \otimes \mathcal O_n) \longrightarrow \mathcal C
$$
correspond bijectively (and naturally in $\mathcal C$) to functors
$$
H:(\mathcal O^{\text{pre}}_m \times \mathcal O^{\text{pre}}_{n-1}) \sqcup_{\mathcal O^{\text{pre}}_{m-1} \times \mathcal O^{\text{pre}}_{n-1}} (\mathcal O^{\text{pre}}_{m-1} \otimes \mathcal O^{\text{pre}}_n) \longrightarrow \mathcal C
$$
whose restrictions to $\mathcal O^{\text{pre}}_m \times \mathcal O^{\text{pre}}_{n-1}$ and to $\mathcal O^{\text{pre}}_{m-1} \times \mathcal O^{\text{pre}}_n$ are bimorphisms. Extending such an $H$ to a bimorphism $H':\mathcal O^{\text{pre}}_m \times \mathcal O^{\text{pre}}_n \rightarrow \mathcal C$ amounts to choosing an object $a \in \mathcal C$ endowed with a display map $p:a \twoheadrightarrow H(o_{m-1},o_n)$ and a morphism $f:a \rightarrow H(o_m,o_{n-1})$ such that
\[
\dsqua{a}{H(o_{m-1},o_n)}{H(o_m,o_{n-1})}{H(o_{m-1},o_{n-1})}{f}{!}{p}{!}
\]
is a relative length-$1$ display map, where $!$ denotes, in each case, the morphism induced by $H$ (note that the domain of $H$ is a poset category). Such a choice is, in turn, uniquely determined by a choice of length $1$ display map $a \twoheadrightarrow H(o_m,o_{n-1}) \times_{H(o_{m-1},o_{n-1})} H(o_{m-1},o_n)$. This implies that we have a pushout diagram of contextual categories
\[
\squa{\mathcal O_{mn-1}}{(\mathcal O_m \otimes \mathcal O_{n-1}) \sqcup_{\mathcal O_{m-1} \otimes \mathcal O_{n-1}} (\mathcal O_{m-1} \otimes \mathcal O_n)}{\mathcal O_{mn}}{\mathcal O_m \otimes \mathcal O_n}{}{}{\iota_{mn}^S}{ \iota_m^S \widehat{\otimes} \iota_n^S}
\]
where the top arrow sends $o_{mn-1}$ to $(o_m \otimes o_{n-1}) \times_{o_{m-1} \otimes o_{n-1}} (o_{m-1} \otimes o_n)$ and the bottom one sends $o_{mn}$ to $o_m \otimes o_n$.

\subsubsection*{Description of $ \iota_m^S \widehat{\otimes} \iota_n^T$ and $ \iota_m^T \widehat{\otimes} \iota_n^S$}

Let us start by describing $ \iota_m^S \widehat{\otimes} \iota_n^T$. For that, note that morphisms
$$
(\mathcal O_m \otimes \mathcal O_n) \sqcup_{\mathcal O_{m-1}\otimes \mathcal O_n} (\mathcal O_{m-1} \otimes \mathcal O_n^+) \longrightarrow \mathcal C
$$
are in natural bijection with functors
$$
H:(\mathcal O^{\text{pre}}_m \times \mathcal O^{\text{pre}}_n) \sqcup_{\mathcal O^{\text{pre}}_{m-1} \times \mathcal O^{\text{pre}}_n} (\mathcal O^{\text{pre}}_{m-1} \otimes \mathcal O^{+\; \text{pre}}_n) \longrightarrow \mathcal C
$$
whose restrictions to $\mathcal O^{\text{pre}}_m \times \mathcal O^{\text{pre}}_n$ and $\mathcal O^{\text{pre}}_{m-1} \otimes \mathcal O^{+\; \text{pre}}_n$ are bimorphisms. An extension of such an $H$ to a bimorphism $H':\mathcal O_m^{\text{pre}} \times \mathcal O_n^{+\; \text{pre}} \rightarrow \mathcal C$ corresponds a choice of morphism $u:H(o_m,o_{n-1}) \rightarrow H(o_m,o_n)$ such that, denoting by $s$ the generating section of $p_n$ in $\mathcal O_n^{+\; \text{pre}}$,
\begin{itemize}
	\item $u$ is a section of $H(id,p_n):H(o_m,o_n) \rightarrow H(o_m,o_{n-1})$, and
	
	\item The diagram
	\[
	\dsqua{H(o_m,o_{n-1})}{H(o_m,o_n)}{H(o_{m-1},o_{n-1})}{H(o_{m-1},o_n)}{u}{H(id,s)}{H(p_m,id)}{H(p_m,id)}
	\]
	commutes.
\end{itemize}
Observe the we have a commutative diagram
\[\begin{tikzcd}
	{H(o_m,o_n)} \\
	a && {H(o_m,o_{n-1})} \\
	{H(o_{m-1},o_n)} && {H(o_{m-1},o_{n-1})}
	\arrow[two heads, from=1-1, to=2-1]
	\arrow["{H(id,p_n)}", from=1-1, to=2-3]
	\arrow["{H(p_m,id)}"', curve={height=18pt}, from=1-1, to=3-1]
	\arrow[from=2-1, to=2-3]
	\arrow[two heads, from=2-1, to=3-1]
	\arrow["v", shift left=3, hook, from=2-3, to=2-1]
	\arrow["{H(p_m,id)}", two heads, from=2-3, to=3-3]
	\arrow["{H(id,p_n)}"{description}, from=3-1, to=3-3]
	\arrow["{H(id,s)}", shift left=3, hook, from=3-3, to=3-1]
\end{tikzcd}\]
where $a = H(o_{m-1},o_n) \times_{H(o_{m-1},o_{n-1})} H(o_m,o_{n-1})$ and $v$ is obtained by base change along $H(o_m,o_{n-1}) \twoheadrightarrow H(o_{m-1},o_{n-1})$. A routine calculation shows that precomposition with $v$ defines a bijection
$$
\{\text{sections of } H(o_m,o_n) \twoheadrightarrow a\} \cong \{\text{arrows } u:H(o_m,o_{n-1}) \rightarrow H(o_m,o_n) \text{ of the desired form}\}.
$$

Therefore we have a pushout square of contextual categories
\[
\squa{\mathcal O_{mn}}{(\mathcal O_m \otimes \mathcal O_n) \sqcup_{\mathcal O_{m-1} \otimes \mathcal O_n} (\mathcal O_{m-1} \otimes \mathcal O_n^+)}{\mathcal O_{mn}^+}{\mathcal O_m \otimes \mathcal O_n^+}{}{}{\iota_{mn}^T}{ \iota_{m-1}^S \widehat{\otimes} \iota_n^T}
\]
where the top morphism sends $o_{mn}$ to $o_m \otimes o_n$, and the bottom one sends the generating section $s$ to the arrow $s':(o_{m-1} \otimes o_n) \times_{o_{m-1} \otimes o_{n-1}} (o_m \otimes o_{n-1}) \rightarrow o_m \otimes o_n$ defined as the image under the inverse of the above bijection, taken with respect to the bimorphism $\mathcal O_m^{\text{pre}} \times \mathcal O_n^{+\; \text{pre}} \rightarrow \mathcal O_m \otimes \mathcal O_n^+$, of $id \otimes s:o_m \otimes o_{n-1} \rightarrow o_m \otimes o_n$.

\vspace{0.5em}

An analogous argument allows us to describe $ \iota_m^T \widehat{\otimes} \iota_n^S$ via a pushout square
\[
\squa{\mathcal O_{mn}}{(\mathcal O_m \otimes \mathcal O_n) \sqcup_{\mathcal O_m^+ \otimes \mathcal O_{n-1}} (\mathcal O_m \otimes \mathcal O_{n-1})}{\mathcal O_{mn}^+}{\mathcal O_m^+ \otimes \mathcal O_n.}{}{}{\iota_{mn}^T}{ \iota_m^T \widehat{\otimes} \iota_{n-1}^S}
\]

\subsubsection*{Description of $ \iota_m^T \widehat{\otimes} \iota_n^T$}

We have a bijective correspondence, natural in $\mathcal C$, between contextual functors
$$
\mathcal P:= (\mathcal O_m^+ \otimes \mathcal O_n) \sqcup_{\mathcal O_m \otimes \mathcal O_n} (\mathcal O_m \otimes \mathcal O_n^+) \longrightarrow \mathcal C
$$
and functors
$$
H:(\mathcal O_m^{+\; \text{pre}} \times \mathcal O_n^{\text{pre}}) \sqcup_{\mathcal O_m^{\text{pre}} \times \mathcal O_n^{\text{pre}}} (\mathcal O_m^{\text{pre}} \times \mathcal O_n^{+\; \text{pre}}) \longrightarrow \mathcal C
$$
whose restrictions to $\mathcal O_m^{+\; \text{pre}} \times \mathcal O_n^{\text{pre}}$ and $\mathcal O_m^{\text{pre}} \times \mathcal O_n^{+\; \text{pre}}$ are bimorphisms. Since $\iota_m^T:\mathcal O_m^{\text{pre}} \rightarrow \mathcal O_m^{+\; \text{pre}}$ and $\iota_n^T:\mathcal O_n^{\text{pre}} \rightarrow \mathcal O_n^{+\; \text{pre}}$ modify the underlying category without adding display maps or distinguished squares, every extension of a functor $H$ as above along
\[
\tag{\texttt{*}}
(\mathcal O_m^{+\; \text{pre}} \times \mathcal O_n^{\text{pre}}) \sqcup_{\mathcal O_m^{\text{pre}} \times \mathcal O_n^{\text{pre}}} (\mathcal O_m^{\text{pre}} \times \mathcal O_n^{+\; \text{pre}}) \longrightarrow \mathcal O_m^{+\; \text{pre}} \times \mathcal O_n^{+\; \text{pre}}
\]
is a bimorphism from $(\mathcal O_m^{+\; \text{pre}}, \mathcal O_n^{+\; \text{pre}})$ to $\mathcal C$. But (\texttt{*}) realizes the category $\mathcal O_m^{+\; \text{pre}} \times \mathcal O_n^{+\; \text{pre}}$ as a quotient of $(\mathcal O_m^{+\; \text{pre}} \times \mathcal O_n^{\text{pre}}) \sqcup_{\mathcal O_m^{\text{pre}} \times \mathcal O_n^{\text{pre}}} (\mathcal O_m^{\text{pre}} \times \mathcal O_n^{+\; \text{pre}})$ by imposing commutativity of the diagram
\[
\tag{\texttt{**}}
\begin{tikzcd}
	{(o_m,o_n)} & {(o_m,o_{n-1})} \\
	{(o_{m-1},o_n)} & {(o_{m-1},o_{n-1})}
	\arrow["{(id,s')_2}"', from=1-2, to=1-1]
	\arrow["{(s,id)_1}", from=2-1, to=1-1]
	\arrow["{(s,id)_1}"', from=2-2, to=1-2]
	\arrow["{(id,s')_2}", from=2-2, to=2-1]
\end{tikzcd}\]
where $s$ (resp. $s'$) is the generating section of $o_m \twoheadrightarrow o_{m-1}$ (resp. of $o_n \twoheadrightarrow o_{n-1}$) in $\mathcal O_m^{+\; \text{pre}}$ (resp. in $\mathcal O_n^{+\; \text{pre}}$), and we use the subscript $1$ (resp. $2$) to indicate an arrow from $\mathcal O_m^{\text{pre}} \times \mathcal O_n^{+\; \text{pre}}$ (resp. from $\mathcal O_m^{+\; \text{pre}} \times \mathcal O_n^{\text{pre}}$). In other words, a functor $H$ of the above form extends to a bimorphism $\mathcal O_m^{+\; \text{pre}} \times \mathcal O_n^{+\; \text{pre}} \rightarrow \mathcal C$ precisely when it maps (\texttt{**}) to a commutative diagram.

In order to describe this commutativity condition in terms of the maps (A1)-(A4), consider the (non-commutative) diagram
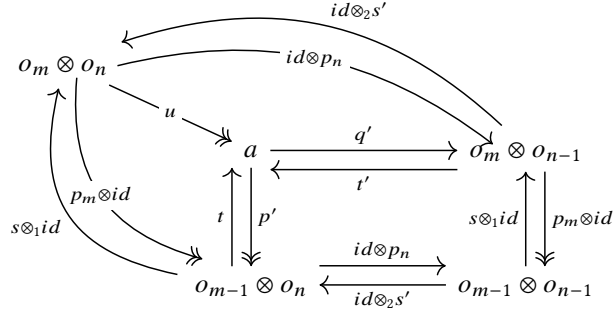
\[\begin{tikzcd}
	{o_m \otimes o_n} \\
	& a && {o_m \otimes o_{n-1}} \\
	\\
	& {o_{m-1} \otimes o_n} && {o_{m-1} \otimes o_{n-1}}
	\arrow["u"{description}, two heads, from=1-1, to=2-2]
	\arrow["{id \otimes p_n}"{description}, shift right=2, curve={height=-24pt}, from=1-1, to=2-4]
	\arrow["{p_m \otimes id}"{description}, shift left=2, curve={height=24pt}, two heads, from=1-1, to=4-2]
	\arrow["{q'}", from=2-2, to=2-4]
	\arrow["{p'}", two heads, from=2-2, to=4-2]
	\arrow["{id \otimes_2 s'}"', shift right, curve={height=30pt}, from=2-4, to=1-1]
	\arrow["{t'}", shift left=3, from=2-4, to=2-2]
	\arrow["{p_m \otimes id}", shift left=3, two heads, from=2-4, to=4-4]
	\arrow["{s \otimes_1 id}", shift left, curve={height=-30pt}, from=4-2, to=1-1]
	\arrow["t", shift left=3, from=4-2, to=2-2]
	\arrow["{id \otimes p_n}", shift left=3, from=4-2, to=4-4]
	\arrow["{s \otimes_1 id}", from=4-4, to=2-4]
	\arrow["{id \otimes_2 s'}", from=4-4, to=4-2]
\end{tikzcd}\]
in $\mathcal P$ where
\begin{itemize}
	\item[--] $\otimes_1$ (resp. $\otimes_2$) indicates morphisms in the image of $\mathcal O_m^+ \otimes \mathcal O_n \rightarrow \mathcal P$ (resp. of $\mathcal O_m \otimes \mathcal O_n^+ \rightarrow \mathcal P)$.
	
	\item[--] The square consisting of $p_m \otimes id$, $p'$, $id \otimes p_n$, $q'$ is distinguished.
	
	\item[--] The section $t$ of $p'$ is the base change of the section $s \otimes_1 id$ of $p_m \otimes id$ along $id \otimes p_n$. Similarly, the section $t'$ of $q'$ is the base change of the section $id \otimes_2 s'$ of $id \otimes p_n$ along $p_m \otimes id$.
	
	\item[--] $u$ is the gap map of the square consisting of $p_m \otimes id_{n-1}$, $p_m \otimes id_n$, $id_{m-1} \otimes p_n$, $id_m \otimes p_n$.
\end{itemize}

Although the equality $(s \otimes_1 id) \circ (id \otimes_2 s') = (id \otimes_2 s') \circ (s \otimes_1 id)$, which is equivalent to commutativity of (\texttt{**}), does not hold, the following weaker statement does:

\begin{lemma*}
We have $u \circ (s \otimes_1 id) \circ (id \otimes_2 s') = u \circ (id \otimes_2 s')$.
\end{lemma*}

\begin{proof}
	Since $p'$ and $q'$ realize $a$ as the fiber product $(o_{m-1} \otimes o_n) \times_{o_{m-1} \otimes o_{n-1}} (o_m \otimes o_{n-1})$, it suffices that the above equalities hold when composed with $p'$ and with $q'$. We check this as follows:
	\begin{align*}
		& p' \circ u \circ (s \otimes_1 id) \circ (id \otimes_2 s') = (p_m \otimes id) \circ (s \otimes_1 id) \circ (id \otimes_2 s) = id \otimes_2 s',\\
		& p' \circ u \circ (id \otimes_2 s') \circ (s \otimes_1 id) = (p_m \otimes id) \circ (id \otimes_2 s') \circ (s \otimes_1 id) = (id \otimes_2 s') \circ (p_m \circ id) \circ (s \otimes_1 id) = id \otimes_2 s',\\
		&\\
		&q' \circ u \circ (s \otimes_1 id) \circ (id \otimes_2 s') = (id \otimes p_n) \circ (s \otimes_1 id) \circ (id \otimes_2 s') = (s \otimes_1 id) \circ (id \otimes p_n) \circ (id \otimes_2 s') = s \otimes_1 id,\\
		& q' \circ u \circ (id \otimes_2 s') \circ (s \otimes_1 id) = (id \otimes p_n) \circ (id \otimes_2 s') \circ (s \otimes_1 id) = s \otimes_1 id.
	\end{align*}
\end{proof}

Denoting by $k$ the arrow in the statement of the lemma, consider the following distinguished square in $\mathcal P$:
\[
\dsqua{b}{o_m \otimes o_n}{o_{m-1} \otimes o_{n-1}}{a.}{k'}{k}{u'}{u}
\]
As $(s \otimes_1 id) \circ (id \otimes_2 s')$ is a lift of $k$ along $u$, there exists a unique section $w$ of $u'$ such that $k' \circ w = (s \otimes_1 id) \circ (id \otimes_2 s')$. Analogously, there exists a unique section $w'$ of $u'$ such that $k' \circ w' = (id \otimes_2 s') \circ (s \otimes_1 id)$.

Now, consider a contextual functor $F:\mathcal P \rightarrow \mathcal C$. Since, in particular, it sends distinguished squares to pullback squares, it factors through $ \iota_m^T \widehat{\otimes} \iota_n^T:\mathcal P \rightarrow \mathcal O_m^+ \otimes \mathcal O_n^+$ if and only if $F(w) = F(w')$.

As a consequence, we have a pushout square of contextual categories
\[
\squa{\mathcal O_{(m-1)(n-1) + 1}^{++}}{(\mathcal O_m^+ \otimes \mathcal O_n) \sqcup_{\mathcal O_m \otimes \mathcal O_n} (\mathcal O_m \otimes \mathcal O_n^+)}{\mathcal O_{(m-1)(n-1) + 1}^+}{\mathcal O_m^+ \otimes \mathcal O_n^+.}{}{}{\pi_{(m-1)(n-1) + 1}^T}{ \iota_m^T \widehat{\otimes}  \iota_n^T}
\]
where the top morphism sends $o_{(m-1)(n-1) + 1}$ to $b$ and the generating sections $\sigma$, $\sigma'$ of $o_{(m-1)(n-1) + 1} \twoheadrightarrow o_{(m-1)(n-1)}$ to $w$, $w'$, respectively.

\subsubsection*{Description of $ \iota_m^S \widehat{\otimes} \pi_n^S$ and $ \pi_m^S \widehat{\otimes} \iota_n^S$}

We will give a description of $ \iota_m^S \widehat{\otimes} \pi_n^S$; the map $ \pi_m^S \widehat{\otimes} \iota_n^S$ can be studied analogously. Note that we have a bijection, natural in $\mathcal C$, between contextual functors
$$
\mathcal P := (\mathcal O_m \otimes \mathcal O_n^{\vee}) \sqcup_{\mathcal O_{m-1} \otimes \mathcal O_n^{\vee}} (\mathcal O_{m-1} \otimes \mathcal O_n) \longrightarrow \mathcal C
$$
and functors
$$
H:(\mathcal O_m^{\text{pre}} \times \mathcal O_n^{\vee\; \text{pre}}) \sqcup_{\mathcal O_{m-1}^{\text{pre}} \times \mathcal O_n^{\vee\; \text{pre}}} (\mathcal O_{m-1}^{\text{pre}} \times \mathcal O_n^{\text{pre}}) \longrightarrow \mathcal C
$$
that restrict to bimorphisms from $\mathcal O_m^{\text{pre}} \times \mathcal O_n^{\vee\; \text{pre}}$ and from $\mathcal O_{m-1}^{\text{pre}} \times \mathcal O_n^{\text{pre}}$. Denote the tree of display maps defining $\mathcal O_n^{\vee\; \text{pre}}$ by
\[\begin{tikzcd}[row sep=tiny]
	{o'_n} \\
	& {o_{n-1}} & \cdots & {o_1} & {o_0.} \\
	{o_n}
	\arrow["{p'_n}", from=1-1, to=2-2]
	\arrow["{p_{n-1}}", from=2-2, to=2-3]
	\arrow["{p_2}", from=2-3, to=2-4]
	\arrow["{p_1}", from=2-4, to=2-5]
	\arrow["{p_n}"', from=3-1, to=2-2]
\end{tikzcd}\]
For a functor $H$ as above, say corresponding to a contextual functor $F$, we have that $H$ extends (and, if so, uniquely) to a bimorphism $H':\mathcal O_m^{\text{pre}} \times \mathcal O_n^{\text{pre}} \rightarrow \mathcal C$ if and only if the morphisms
$$
F(id \otimes p_n): F(o_m \otimes o_n) \rightarrow F(o_m \otimes o_{n-1}), \qquad F(id \otimes p'_n):F(o_m \otimes o'_n) \rightarrow F(o_m \otimes o_{n-1})
$$
are equal. But observe that, as $(o_{m-1},o_n)$ and $(o_{m-1},o'_n)$ are identified in $\mathcal P$, we have a commutative diagram
\[\begin{tikzcd}
	& {o_m \otimes o_n} \\
	{o_m \otimes o'_n} &&& {o_m \otimes o_{n-1}} \\
	&& {} \\
	& {o_{m-1} \otimes o_n} && {o_{m-1} \otimes o_{n-1}}
	\arrow[from=1-2, to=2-4]
	\arrow["{p_m \otimes id}"{description, pos=0.7}, two heads, from=1-2, to=4-2]
	\arrow[from=2-1, to=2-4]
	\arrow["{p_m \otimes id}"', two heads, from=2-1, to=4-2]
	\arrow["{p_m \otimes id}"{description}, two heads, from=2-4, to=4-4]
	\arrow[from=4-2, to=4-4]
\end{tikzcd}\]
where the squares with top-left vertices $o_m \otimes o_n$ and $o_m \otimes o'_n$ are relative length-$1$ display maps. It can be extended to a commutative diagram
\[\begin{tikzcd}
	& {o_m \otimes o_n} \\
	{o_m \otimes o'_n} &&& {o_m \otimes o_{n-1}} \\
	&& a \\
	& {o_{m-1} \otimes o_n} && {o_{m-1} \otimes o_{n-1}}
	\arrow[from=1-2, to=2-4]
	\arrow["u"{description, pos=0.8}, two heads, from=1-2, to=3-3]
	\arrow["{p_m \otimes id}"{description, pos=0.7}, two heads, from=1-2, to=4-2]
	\arrow[from=2-1, to=2-4]
	\arrow["{u'}"{description, pos=0.8}, two heads, from=2-1, to=3-3]
	\arrow["{p_m \otimes id}"', two heads, from=2-1, to=4-2]
	\arrow["{p_m \otimes id}"{description}, two heads, from=2-4, to=4-4]
	\arrow["f"', from=3-3, to=2-4]
	\arrow["q", two heads, from=3-3, to=4-2]
	\arrow["{\mathlarger{\mathlarger{\Diamond}}}"{description, pos=0.3}, draw=none, from=3-3, to=4-4]
	\arrow[from=4-2, to=4-4]
\end{tikzcd}\]
where the inner square $\Diamond$ is distinguished, and $u$, $u'$ are the respective gap maps of the relative length-$1$ display maps. Now, the assumption that $F(o_m \otimes o_n) \rightarrow F(o_m \otimes o_{n-1})$ equals $F(o_m\otimes o'_n)\rightarrow F(o_m \otimes o_{n-1})$ is equivalent to $F(u)=F(u')$. As a consequence, $\mathcal O_m \otimes \mathcal O_n$ is obtained from $\mathcal P$ by identifying $u$ and $u'$; equivalently, we have a pushout diagram of contextual categories
\[
\squa{\mathcal O_{mn}^{\vee}}{(\mathcal O_m \otimes \mathcal O_n^{\vee}) \sqcup_{\mathcal O_{m-1} \otimes \mathcal O_n^{\vee}} (\mathcal O_{m-1} \otimes \mathcal O_n)}{\mathcal O_{mn}}{\mathcal O_m \otimes \mathcal O_n}{}{}{\pi_{mn}^{S}}{ \iota_m^S \widehat{\otimes} \pi_n^S}
\]
where the top morphism sends $o_{mn}$ to $o_m \otimes o_n$ and $o'_{mn}$ to $o_m \otimes o'_n$.

\vspace{0.5em}

An analogous argument yields a pushout square
\[
\squa{\mathcal O_{mn}^{\vee}}{(\mathcal O_m \otimes \mathcal O_{n-1}) \sqcup_{\mathcal O_m^{\vee} \otimes \mathcal O_{n-1}} (\mathcal O_m^{\vee} \otimes \mathcal O_n)}{\mathcal O_{mn}}{\mathcal O_m \otimes \mathcal O_n.}{}{}{\pi_{mn}^{S}}{ \pi_m^S \widehat{\otimes} \iota_n^S}
\]

\subsubsection*{Description of $ \iota_m^S \widehat{\otimes} \pi_n^T$ and $ \pi_m^T \widehat{\otimes} \iota_n^S$}

Let us describe $ \iota_m^S \widehat{\otimes} \pi_n^T$; the map $ \pi_m^T \widehat{\otimes} \iota_n^S$ can be dealt with analogously. As in the previous cases, we have a natural bijection between contextual functors
$$
\mathcal P:= (\mathcal O_m \otimes \mathcal O_n^{++}) \sqcup_{\mathcal O_{m-1} \otimes \mathcal O_n^{++}} (\mathcal O_{m-1} \otimes \mathcal O_n^+) \longrightarrow \mathcal C
$$
and functors
$$
H:(\mathcal O_m^{\text{pre}} \times \mathcal O_n^{++\; \text{pre}}) \sqcup_{\mathcal O_{m-1}^{\text{pre}} \times \mathcal O_n^{++\; \text{pre}}} (\mathcal O_{m-1}^{\text{pre}} \times \mathcal O_n^{+\; \text{pre}}) \longrightarrow \mathcal C
$$
that restrict to bimorphisms from $\mathcal O_m^{\text{pre}} \times \mathcal O_n^{++\; \text{pre}}$ and from $\mathcal O_{m-1}^{\text{pre}} \times \mathcal O_n^{+\; \text{pre}}$.

Let $s$ and $s'$ be the sections of $o_n \twoheadrightarrow o_{n-1}$ in $\mathcal O_n^{++\; \text{pre}}$ (or $\mathcal O_n^{++}$). We also write $s$ for the section of $o_n \twoheadrightarrow o_{n-1}$ in $\mathcal O_n^+$. A functor $H$ as above, say corresponding to a contextual functor $F$, extends (and, if so, uniquely) to a bimorphism $H':\mathcal O_m^{\text{pre}} \times \mathcal O_n^{+\; \text{pre}} \rightarrow \mathcal C$ if and only if the morphisms
$$
F(id_m \otimes s):F(o_m \otimes o_{n-1}) \rightarrow F(o_m \otimes o_n), \qquad F(id_m \otimes s'):F(o_m \otimes o_{n-1}) \rightarrow F(o_m \otimes o_n)
$$
are equal.

Now, consider the (non-commutative) diagram
\[\begin{tikzcd}
	{o_m \otimes o_n} &&&& {o_m \otimes o_{n-1}} \\
	& a \\
	{o_{m-1} \otimes o_n} &&&& {o_{m-1} \otimes o_{n-1}}
	\arrow["{id \otimes p_n}"{description}, from=1-1, to=1-5]
	\arrow["u"{description}, shift right, two heads, from=1-1, to=2-2]
	\arrow["{p_m \otimes id}"', two heads, from=1-1, to=3-1]
	\arrow["{id \otimes s'}"{description}, shift left=4, from=1-5, to=1-1]
	\arrow["{id \otimes s}"{description}, shift right=4, from=1-5, to=1-1]
	\arrow["t"{description}, shift left=4, from=1-5, to=2-2]
	\arrow["{p_m \otimes id}", two heads, from=1-5, to=3-5]
	\arrow["f"{description}, shift right, from=2-2, to=1-5]
	\arrow["q"{description}, two heads, from=2-2, to=3-1]
	\arrow["{id \otimes p_n}"{description}, from=3-1, to=3-5]
	\arrow["{id \otimes s}"{description}, shift left=4, from=3-5, to=3-1]
\end{tikzcd}\]
in $\mathcal P$ where
\begin{itemize}
	\item[--] $q$, $f$ realize $a$ as the distinguished pullback of $o_m \otimes o_{n-1}$ along $id_{m-1} \otimes p_n$,
	\item[--] $u$ is characterized by $p_m \otimes id_n = q \circ u$ and $id_m \otimes p_n = f \circ u$, and
	\item[--] the section $t$ of $f$ is obtained by base change of $id_{m-1} \otimes s$ along $p_m \otimes id_{n-1}$.
\end{itemize}

Routine calculations show that $u \circ (id_m \otimes s) = t = u \circ (id_m \otimes s')$. This implies that, letting $b$, $u'$, $t'$ be as in the distinguished square
\[\begin{tikzcd}
	b & {o_m \otimes o_n} \\
	{o_m \otimes o_{n-1},} & a
	\arrow["{t'}", from=1-1, to=1-2]
	\arrow["{u'}"', two heads, from=1-1, to=2-1]
	\arrow["u", two heads, from=1-2, to=2-2]
	\arrow["t"', from=2-1, to=2-2]
\end{tikzcd}\]
there exist unique sections $w$, $w'$ of $u'$ such that $t' \circ w = id_m \otimes s$ and $t' \circ w' = id_m \otimes s'$. Hence $F:\mathcal P \rightarrow \mathcal C$ satisfies $F(id_m \otimes s) = F(id_m \otimes s')$ if and only if $F(w) = F(w')$. We conclude that $\mathcal O_m \otimes \mathcal O_n^+$ is obtained from $\mathcal P$ by identifying $w$ and $w'$; equivalently, we have a pushout diagram of contextual categories
\[
\squa{\mathcal O_{m(n-1)+1}^{++}}{(\mathcal O_m \otimes \mathcal O_n^{++}) \sqcup_{\mathcal O_{m-1} \otimes \mathcal O_n^{++}} (\mathcal O_{m-1} \otimes \mathcal O_n^+)}{\mathcal O_{m(n-1)+1}^ +}{\mathcal O_m \otimes \mathcal O_n^+}{}{}{\pi_{m(n-1)+1}^T}{ \iota_m^S \widehat{\otimes} \pi_n^T}
\]
where the top morphism sends $o_{m(n-1)+1}$ to $b$ and the generating sections $\sigma$, $\sigma'$ of $o_{m(n-1)+1} \twoheadrightarrow o_{m(n-1)}$ to $w$, $w'$, respectively.

\vspace{0.5em}

By performing an analogous construction it is possible to obtain a pushout square
\[
\squa{\mathcal O_{(m-1)n+1}^{++}}{(\mathcal O_m^+ \otimes \mathcal O_{n-1}) \sqcup_{\mathcal O_m^{++} \otimes \mathcal O_{n-1}} (\mathcal O_m^{++} \otimes \mathcal O_n)}{\mathcal O_{(m-1)n+1}^ +}{\mathcal O_m^+ \otimes \mathcal O_n.}{}{}{\pi_{(m-1)n+1}^T}{ \pi_m^T \widehat{\otimes} \iota_n^S}
\]

\subsubsection*{Description of $ \iota_m^T \widehat{\otimes} \pi_n^S$ and $ \pi_m^S \widehat{\otimes} \iota_n^T$}

We will describe $ \iota_m^T \widehat{\otimes} \pi_n^S$; the map $ \pi_m^S \widehat{\otimes} \iota_n^T$ can be studied analogously. As in the other cases, we will use the natural bijection between contextual functors
$$
\mathcal P := (\mathcal O_m^+ \otimes \mathcal O_n^{\vee}) \sqcup_{\mathcal O_m \otimes \mathcal O_n^{\vee}} (\mathcal O_m \otimes \mathcal O_n) \longrightarrow \mathcal C
$$
and functors
$$
H:(\mathcal O_m^{+\; \text{pre}} \times \mathcal O_n^{\vee\; \text{pre}}) \sqcup_{\mathcal O_m^{\text{pre}} \times \mathcal O_n^{\vee\; \text{pre}}} (\mathcal O_m^{\text{pre}} \times \mathcal O_n^{\text{pre}}) \longrightarrow \mathcal C
$$
that restrict to bimorphisms from $\mathcal O_m^{+\; \text{pre}} \times \mathcal O_n^{\vee\; \text{pre}}$ and from $\mathcal O_m^{\text{pre}} \times \mathcal O_n^{\text{pre}}$.

Denoting by $s$ the generating section of $o_m \twoheadrightarrow o_{m-1}$ in $\mathcal O_m^{+\; \text{pre}}$ and by
\[\begin{tikzcd}
	{o'_n} \\
	& {o_{n-1}} & \cdots & {o_1} & {o_0.} \\
	{o_n}
	\arrow["{p'_n}", from=1-1, to=2-2]
	\arrow["{p_{n-1}}", from=2-2, to=2-3]
	\arrow["{p_2}", from=2-3, to=2-4]
	\arrow["{p_1}", from=2-4, to=2-5]
	\arrow["{p_n}"', from=3-1, to=2-2]
\end{tikzcd}\]
the tree of display maps defining $\mathcal O_n^{\vee\; \text{pre}}$, a functor $H$ as above, say corresponding to a contextual functor $F$, extends (and, if so, uniquely) to a bimorphism $H':\mathcal O_m^{+\; \text{pre}} \times \mathcal O_n^{\text{pre}} \rightarrow \mathcal C$ if and only if the sections
$$
F(s \otimes id):F(o_{m-1} \otimes o_n) \rightarrow F(o_m \otimes o_n), \qquad F(s \otimes id'):F(o_{m-1} \otimes o'_n) \rightarrow F(o_m \otimes o'_n)
$$
of $F(o_m \otimes o_n) \twoheadrightarrow F(o_{m-1} \otimes o_n)$ are equal (note that in $\mathcal P$ we have $o_{m-1} \otimes o_n = o_{m-1} \otimes o'_n$ and $o_m \otimes o_n = o_m \otimes o'_n$).

Consider the (non-commutative) diagram
\[\begin{tikzcd}
	{o_m \otimes o_n} &&& {o_m \otimes o_{n-1}} \\
	& a \\
	\\
	{o_{m-1} \otimes o_n} &&& {o_{m-1} \otimes o_{n-1}}
	\arrow[from=1-1, to=1-4]
	\arrow["u"{description}, two heads, from=1-1, to=2-2]
	\arrow[shift right=2, two heads, from=1-1, to=4-1]
	\arrow[two heads, from=1-4, to=4-4]
	\arrow["f", from=2-2, to=1-4]
	\arrow["q"{description}, shift right, two heads, from=2-2, to=4-1]
	\arrow["{s \otimes id}", shift left=5, hook, from=4-1, to=1-1]
	\arrow["{s \otimes id'}"', shift right, hook, from=4-1, to=1-1]
	\arrow["t"', shift right=2, hook, from=4-1, to=2-2]
	\arrow["{id_{m-1} \otimes p_n}"', from=4-1, to=4-4]
	\arrow["{s \otimes id_{n-1}}"', shift right=3, hook, from=4-4, to=1-4]
\end{tikzcd}\]
in $\mathcal P$ where the square with top-left vertex $a$ is distinguished, $u$ is the evident gap map, and $t$ is the section of $q$ obtained by base change of $s \otimes id_{n-1}$ along $id_{m-1} \otimes p_n$. By arguing as in the previous cases we obtain $u \circ (s \otimes id) = u \circ (s \otimes id')$. Now, letting $b$, $u'$, $t'$ be as in the distinguished square
\[
\dsqua{b}{o_m \otimes o_n}{o_{m-1} \otimes o_n}{a,}{t'}{t}{u'}{u}
\]
there exist unique sections $w$, $w'$ of $u'$ such that $t' \circ w = s \otimes id$ and $t' \circ w' = s \otimes id'$. It follows that $F:\mathcal P \rightarrow \mathcal C$ satisfies $F(s \otimes id) = F(s \otimes id')$ if and only if $F(w) = F(w')$. We conclude that $\mathcal O_m^+ \otimes \mathcal O_n$ is obtained from $\mathcal P$ by identifying $w$ and $w'$, or, equivalently, we have a pushout diagram
\[
\squa{\mathcal O_{(m-1)n + 1}^{++}}{(\mathcal O_m^+ \otimes \mathcal O_n^{\vee}) \sqcup_{\mathcal O_m \otimes \mathcal O_n^{\vee}} (\mathcal O_m \otimes \mathcal O_n)}{\mathcal O_{(m-1)n + 1}^+}{\mathcal O_m^+ \otimes \mathcal O_n}{}{}{\pi_{(m-1)n + 1}^T}{ \iota_m^T \widehat{\otimes} \pi_n^S}
\]
where the top morphism sends $o_{(m-1)n + 1}$ to $b$ and the generating sections $\sigma$, $\sigma'$ of $o_{(m-1)n + 1} \twoheadrightarrow o_{(m-1)n}$ to $w$, $w'$, respectively.

\vspace{0.5em}

An analogous construction yields a pushout square
\[
\squa{\mathcal O_{m(n-1) + 1}^{++}}{(\mathcal O_m \otimes \mathcal O_n) \sqcup_{\mathcal O_m^{\vee} \otimes \mathcal O_n} (\mathcal O_m^{\vee} \otimes \mathcal O_n^+)}{\mathcal O_{m(n-1) + 1}^+}{\mathcal O_m \otimes \mathcal O_n^+.}{}{}{\pi_{m(n-1) + 1}^T}{ \pi_m^S \widehat{\otimes} \iota_n^T}
\]

\subsubsection*{Degenerate cases}

Arguing as above shows that for $m$, $n \ge 1$ the maps
$$
 \iota_m^T \widehat{\otimes} \pi_n^T,\;\;  \pi_m^T \widehat{\otimes} \iota_n^T, \;\;  \pi_m^S \widehat{\otimes} \pi_n^T, \;\;  \pi_m^T \widehat{\otimes} \pi_n^S, \;\;  \pi_m^S \widehat{\otimes} \pi_n^S, \;\;  \pi_m^T \widehat{\otimes} \pi_n^T
$$
are isomorphisms. For each of these, the proof, which is left as an exercise, involves checking that the extension problem analogous to the ones in the previous cases admits a unique solution, implying that the domain and the codomain of the pushout-tensor map have the same universal property.

\section{Comparing the syntactic and the categorical approaches}
\label{sec: comparing syntactic and categorical}

Let $\bbA$ and $\bbB$ be generalized algebraic theories. We will now sketch a proof that the comparison functor
$$
\mathcal C(\bbA) \times \mathcal C(\bbB) \longrightarrow \mathcal C (\bbA \otimes \bbB)
$$
constructed in \cite{Alm25}, which was verified in Proposition \ref{prop: comparison functor is bimorphism} to be a bimorphism from $(\mathcal C(\bbA), \mathcal C(\bbB))$ to $\mathcal C(\bbA \otimes \bbB)$, is in fact isomorphic to the universal bimorphism $\otimes_{\mathcal C(\bbA),\mathcal C(\bbB)}:\mathcal C(\bbA) \times \mathcal C(\bbB) \rightarrow \mathcal C(\bbA) \otimes \mathcal C(\bbB)$ from \S\ref{sec: monoidal structure}. Using that, we will conclude that the tensor product from \cite{Alm25} is part of a monoidal structure on $\GAT$ equivalent to the one on $\Cont$ obtained in \S\ref{sec: monoidal structure}.

\begin{proposition}
Let $\bbA$ and $\bbB$ be generalized algebraic theories. Then the functor $\otimes_{\bbA,\bbB}:\mathcal C(\bbA) \times \mathcal C(\bbB) \rightarrow \mathcal C(\bbA \otimes \bbB)$ from \cite{Alm25}, Construction 6.5 is an isomorphism.
\end{proposition}

\begin{proof}[Sketch of proof]
Choose an ordinal $\alpha$ and a sequence of theories $(\bbA_\mu)_{\mu \le \alpha}$ such that
\begin{itemize}
	\item $\bbA_\alpha = \bbA$;
	
	\item if $\mu < \alpha$, then $\bbA_{\mu+1}$ is obtained from $\bbA_\mu$ by adding a single axiom;
	
	\item if $\mu \le \alpha$ is a limit ordinal, then $\bbA_\mu = \bigcup_{\mu' < \mu}\bbA_{\mu'}$.
\end{itemize}
Analogously, we choose for $\bbB$ an ordinal $\beta$ and a sequence $(\bbB_\beta)_{\nu \le \beta}$. For $\mu \le \alpha$ and $\nu \le \beta$, write $\otimes_{\mu,\nu}$ for the bimorphism $\otimes_{\bbA_\mu,\bbB_\nu}:\mathcal C(\bbA_\mu) \times \mathcal C(\bbB_\nu) \rightarrow \mathcal C(\bbA_\mu \otimes \bbB_\nu)$, and let
$$
K_{\mu,\nu}:\mathcal C(\bbA_\mu) \otimes \mathcal C(\bbB_\nu) \longrightarrow \mathcal C(\bbA_\mu \otimes \bbB_\nu)
$$
be the induced contextual functor. Let $P$ be the subset of the (poset) product $(\alpha+1) \times (\beta+1)$ consisting of the pairs $(\mu,\nu)$ such that $K_{\mu,\nu}$ is an isomorphism. We will prove that $P = (\alpha+1) \times (\beta+1)$ by well-founded induction\footnote{The fact that $(\alpha+1) \times (\beta+1)$ is well-founded follows from ordinals and products of two well-founded sets being well-founded.}: it suffices to verify for each $(\mu,\nu) \in (\alpha+1) \times (\beta+1)$ that if $(\mu',\nu') \in P$ for all $(\mu',\nu') < (\mu,\nu)$, then $(\mu,\nu) \in P$. We have the following cases:

\begin{itemize}
	\item $\mu$ or $\nu$ is a limit ordinal (possibly $0$). As $K_{\mu,\nu}$ is a colimit of the functors $K_{\mu',\nu'}$ for $(\mu',\nu') < (\mu,\nu)$, if the latter are isomorphisms, then so is the former.
	
	\item Given successor ordinals $\mu+1 \le \alpha$ and $\nu+1 \le \beta$, assume that $(\mu,\nu+1)$, $(\mu+1,\nu) \in P$. We have a commutative diagram of contextual categories
	\[\begin{tikzcd}
		{\mathcal C(\bbA_{\mu}) \otimes \mathcal C(\bbB_{\nu})} && {\mathcal C(\bbA_{\mu} \otimes \bbB_{\nu})} \\
		& {\mathcal C(\bbA_{\mu}) \otimes \mathcal C(\bbB_{\nu+1})} && {\mathcal C(\bbA_{\mu} \otimes \bbB_{\nu+1})} \\
		\\
		{\mathcal C(\bbA_{\mu+1}) \otimes \mathcal C(\bbB_{\nu})} && {\mathcal C(\bbA_{\mu+1} \otimes \bbB_{\nu})} \\
		& {\mathcal C(\bbA_{\mu+1}) \otimes \mathcal C(\bbB_{\nu+1})} && {\mathcal C(\bbA_{\mu+1} \otimes \bbB_{\nu+1}).}
		\arrow["{K_{\mu,\nu}}"{description}, from=1-1, to=1-3]
		\arrow["{id \otimes \iota_\nu}"{description}, from=1-1, to=2-2]
		\arrow["{\iota_\mu \otimes id}"{description}, from=1-1, to=4-1]
		\arrow["{\mathcal C(I_{\mu,\nu}')}"{description}, from=1-3, to=2-4]
		\arrow["{\mathcal C(I_{\mu,\nu})}"{description}, from=1-3, to=4-3]
		\arrow["{K_{\mu,\nu+1}}"{description, pos=0.8}, from=2-2, to=2-4]
		\arrow["{\iota_m \otimes id}"{description}, from=2-2, to=5-2]
		\arrow["{\mathcal C(I_{\mu,\nu+1})}"{description}, from=2-4, to=5-4]
		\arrow["{K_{\mu+1,\nu}}"{description, pos=0.8}, from=4-1, to=4-3]
		\arrow["{id \otimes \iota_n}"{description}, from=4-1, to=5-2]
		\arrow["{\mathcal C(I'_{\mu+1,\nu})}"{description}, from=4-3, to=5-4]
		\arrow["{K_{\mu+1,\nu+1}}"{description}, from=5-2, to=5-4]
	\end{tikzcd}\]
	
	This yields a further commutative diagram
	\[\begin{tikzcd}
		{\mathcal C(\bbA_{\mu}) \otimes \mathcal C(\bbB_{\nu})} && {\mathcal C(\bbA_{\mu} \otimes \bbB_{\nu})} \\
		& {\mathcal C(\bbA_{\mu}) \otimes \mathcal C(\bbB_{\nu+1})} && {\mathcal C(\bbA_{\mu} \otimes \bbB_{\nu+1})} \\
		\\
		{\mathcal C(\bbA_{\mu+1}) \otimes \mathcal C(\bbB_{\nu})} && {\mathcal C(\bbA_{\mu+1} \otimes \bbB_{\nu})} \\
		& {\mathcal D} && {\mathcal C(\bbD)} \\
		\\
		& {\mathcal C(\bbA_{\mu+1}) \otimes \mathcal C(\bbB_{\nu+1})} && {\mathcal C(\bbA_{\mu+1} \otimes \bbB_{\nu+1})}
		\arrow[from=1-1, to=1-3]
		\arrow["{id \otimes \iota_\nu}"{description}, from=1-1, to=2-2]
		\arrow["{\iota_\mu \otimes id}"{description}, from=1-1, to=4-1]
		\arrow["{\mathcal C(I_{\mu,\nu}')}"{description}, from=1-3, to=2-4]
		\arrow["{\mathcal C(I_{\mu,\nu})}"{description}, from=1-3, to=4-3]
		\arrow[from=2-2, to=2-4]
		\arrow[dashed, from=2-2, to=5-2]
		\arrow["{\iota_m \otimes id}"{description}, curve={height=-24pt}, from=2-2, to=7-2]
		\arrow[dashed, from=2-4, to=5-4]
		\arrow["{\mathcal C(I_{\mu,\nu+1})}"{description}, curve={height=-24pt}, from=2-4, to=7-4]
		\arrow[from=4-1, to=4-3]
		\arrow[dashed, from=4-1, to=5-2]
		\arrow["{id \otimes \iota_n}"{description}, curve={height=24pt}, from=4-1, to=7-2]
		\arrow[dashed, from=4-3, to=5-4]
		\arrow["{\mathcal C(I'_{\mu+1,\nu})}"{description}, curve={height=24pt}, from=4-3, to=7-4]
		\arrow["{K'}"{description}, dashed, from=5-2, to=5-4]
		\arrow["F"{description}, dashed, from=5-2, to=7-2]
		\arrow["G"{description}, dashed, from=5-4, to=7-4]
		\arrow["{K_{\mu+1,\nu+1}}"{description}, from=7-2, to=7-4]
	\end{tikzcd}\]
	where: $\mathcal D$ is the pushout of $\iota_\mu \otimes id_{\mathcal C(\bbB_\nu)}$ and $id_{\mathcal C(\bbA_\mu)} \otimes \iota_\nu$; $\bbD$ is the pushout of $I_{\mu,\nu}$ and $I'_{\mu,\nu}$; $F$ and $G$ are the corresponding gap maps; and $K'$ is induced by the universal properties of $\mathcal D$ and of $\mathcal C(\bbD)$.
	
	Since $K_{\mu,\nu}$, $K_{\mu+1,\nu}$ and $K_{\mu,\nu+1}$ are isomorphisms, so is $K'$. Thus for $K_{\mu+1,\nu+1}$ to be an isomorphism it suffices that
	\[
	\widesqua{\mathcal D}{\mathcal C(\bbD)}{\mathcal C(\bbA_{\mu+1}) \otimes \mathcal C(\bbB_{\nu+1})}{\mathcal C(\bbA_{\mu+1} \otimes \bbB_{\nu+1})}{K'}{K_{\mu+1,\nu+1}}{F}{G}
	\]
	be a pushout square. Now, recall that we have given explicit descriptions of both $F$ and $G$, each of which is split into sixteen cases: we need to take into account which of the four kinds of axioms is added via $\bbA_\mu \rightarrow \bbA_{\mu+1}$, and similarly for $\bbB_\nu \rightarrow \bbB_{\nu+1}$.
	
	A description of $F$ was given in \S \ref{sec: description of pushout-tensor} using the universal property of the tensor product of contextual categories in terms of bimorphisms.
	
	On the other hand, a description of $G$ --- or, equivalently, of the morphism of theories $\bbD \rightarrow \bbA_{\mu+1} \otimes \bbB_{\nu+1}$ that presents it --- is given by the construction of the tensor product of theories; see \cite{Alm25}, \S 2. Precisely, in the notation of \cite{Alm25}, if a judgment $J$ (resp. $J'$) is the axiom added via $\bbA_\mu \subset \bbA_{\mu+1}$ (resp. via $\bbB_\nu \rightarrow \bbB_{\nu+1}$), then $\bbA_{\mu+1} \otimes \bbB_{\nu+1}$ is obtained from $\bbD$ by adding as an axiom the judgment $J \odot J'$. This allows us to express $G$ as a pushout of one of the maps (A1)-(A4) from \S\ref{sec: description of pushout-tensor}.

	To compare $F$ and $G$, we start by expressing $F$ as a pushout of a morphism $\varphi:\mathcal P \rightarrow \mathcal Q$ among (A1)-(A4). This yields a commutative diagram
	\[\begin{tikzcd}
		{\mathcal P} & {\mathcal D} && {\mathcal C(\bbD)} \\
		{\mathcal Q} & {\mathcal C(\bbA_{\mu+1}) \otimes \mathcal C(\bbB_{\nu+1})} && {\mathcal C(\bbA_{\mu+1} \otimes \bbB_{\nu+1}).}
		\arrow[from=1-1, to=1-2]
		\arrow["\varphi"', from=1-1, to=2-1]
		\arrow["{K'}", from=1-2, to=1-4]
		\arrow["F"', from=1-2, to=2-2]
		\arrow["G", from=1-4, to=2-4]
		\arrow[from=2-1, to=2-2]
		\arrow["{K_{\mu+1,\nu+1}}"', from=2-2, to=2-4]
	\end{tikzcd}\]
	By the pasting lemma for pushouts, to conclude that the right square is cocartesian it suffices to check that the outer composite one is cocartesian. A routine but lengthy calculation, whose details we have omitted, shows that, in each of the sixteen cases, $\varphi$ and the composite $\mathcal P \rightarrow \mathcal D \overset{K'}{\rightarrow} \mathcal C(\bbD)$ specify the axiom added via $\bbD \subset \bbA_{\mu+1} \otimes \bbB_{\nu+1}$; this, in turn, implies that the outer square is cocartesian.
	
	To illustrate this idea, suppose that $\bbA_\mu \subset \bbA_{\mu+1}$ adds a sort axiom $J = (\textbf{X} \vdash U \tp)$ and $\bbB_\nu \subset \bbB_{\nu+1}$ adds a sort axiom $J' = (\textbf{Y} \vdash V \tp)$. In terms of the associated contextual categories, the length $1$ display maps $[\textbf{X}'] \twoheadrightarrow [\textbf{X}]$ and $[\textbf{Y}'] \twoheadrightarrow [\textbf{Y}]$, where $\textbf{X}'$, $\textbf{Y}'$ are the respective extended contexts, are freely adjoined to $\mathcal C(\bbA_\mu)$, $\mathcal C(\bbB_\nu)$ to form $\mathcal C(\bbA_{\mu+1})$, $\mathcal C(\bbB_{\nu+1})$. Then, letting $m = \ell(\textbf{X})$ and $n = \ell(\textbf{Y})$, by \S\ref{sec: description of pushout-tensor} we have a pushout square
	\[
	\squa{\mathcal O_{(m+1)(n+1)-1}}{\mathcal D}{\mathcal O_{(m+1)(n+1)}}{\mathcal C(\bbA_{\mu+1}) \otimes \mathcal C(\bbB_{\nu+1})}{}{}{\iota_{(m+1)(n+1)}^S}{F}
	\]
	where the top morphism sends $o_{(m+1)(n+1) - 1}$ to $([\textbf{X}'] \otimes [\textbf{Y}]) \times_{[\textbf{X}] \otimes [\textbf{Y}]} ([\textbf{X}] \otimes [\textbf{Y}'])$, and the bottom one sends $o_{(m+1)(n+1)} \twoheadrightarrow o_{(m+1)(n+1) - 1}$ to $[\textbf{X}'] \otimes [\textbf{Y}'] \twoheadrightarrow ([\textbf{X}'] \otimes [\textbf{Y}]) \times_{[\textbf{X}] \otimes [\textbf{Y}]} ([\textbf{X}] \otimes [\textbf{Y}'])$. It follows that in the composite square
	\[
	\squa{\mathcal O_{(m+1)(n+1)-1}}{\mathcal C(\bbD)}{\mathcal O_{(m+1)(n+1)}}{\mathcal C(\bbA_{\mu+1} \otimes \bbB_{\nu+1}),}{}{}{\iota_{(m+1)(n+1)}^S}{G}
	\]
	the top morphism sends $o_{(m+1)(n+1) - 1}$ to $[\textbf{X}' \otimes \textbf{Y}] \times_{[\textbf{X} \otimes \textbf{Y}]} [\textbf{X} \otimes \textbf{Y}']$, and the bottom one sends $o_{(m+1)(n+1)} \twoheadrightarrow o_{(m+1)(n+1) - 1}$ to $[\textbf{X}' \otimes \textbf{Y}'] \twoheadrightarrow [\textbf{X}' \otimes \textbf{Y}] \times_{[\textbf{X} \otimes \textbf{Y}]} [\textbf{X} \otimes \textbf{Y}'] = [\partial(\textbf{X}' \otimes \textbf{Y}')]$. But since $\bbD \subset \bbA_{\mu+1} \otimes \bbB_{\nu+1}$ adds as an axiom the sort judgment
	$$
	J \odot J' = (\partial(\textbf{X}' \otimes \textbf{Y}') \vdash U \otimes V \tp),
	$$
	whose associated display map is $[\textbf{X}' \otimes \textbf{Y}'] \twoheadrightarrow [\partial(\textbf{X}' \otimes \textbf{Y}')]$, we conclude that the latter square is cocartesian, as desired.
\end{itemize}
\end{proof}

\begin{theorem}
The tensor product of generalized algebraic theories, as defined in \cite{Alm25}, extends to a functor $\otimes:\GAT \times \GAT \rightarrow \GAT$ whose action on arrows is as follows: for morphisms of theories $F:\bbA \rightarrow \bbA'$ and $G:\bbB \rightarrow \bbB'$,
$$
F \otimes G:\bbA \otimes \bbB \longrightarrow \bbA' \otimes \bbB'
$$
is the unique morphism whose associated contextual functor makes the following diagram commute:
\[\begin{tikzcd}
	{\mathcal C(\bbA \otimes \bbB)} && {\mathcal C(\bbA' \otimes \bbB')} \\
	{\mathcal C(\bbA) \times \mathcal C(\bbB)} && {\mathcal C(\bbA') \times \mathcal C(\bbB').}
	\arrow[dashed, from=1-1, to=1-3]
	\arrow["{\otimes_{\bbA,\bbB}}", from=2-1, to=1-1]
	\arrow["{\mathcal C(F) \times \mathcal C(G)}"', from=2-1, to=2-3]
	\arrow["{\otimes_{\bbA',\bbB'}}"', from=2-3, to=1-3]
\end{tikzcd}\]
Moreover, $\otimes$ is part of a closed symmetric monoidal structure on $\GAT$.
\end{theorem}

\section{Compatibility with the $\Cat$-enrichment}
\label{sec: compatibility with the Cat-enrichment}

Recall that we have a strict $2$-category -- or $\Cat$-enriched category -- $\iiCont$ whose underlying $1$-category is $\Cont$ and whose $2$-cells are the natural transformations between contextual functors. We will now show that this enrichment is compatible, in a certain sense, with the monoidal structure $\otimes$ on $\Cont$.

Following Notation \ref{not: set or category of multimorphisms for lexicographic local linearization}, for precontextual categories $\mathcal A_1$, ..., $\mathcal A_n$ and a contextual category $\mathcal B$ we write\linebreak$\iiHom(\mathcal A_1, ..., \mathcal A_n;\mathcal B)$ for the category of multimorphisms $(\mathcal A_1, ..., \mathcal A_n) \rightarrow \mathcal B$ and natural transformations between them. In \S\ref{subsec: cat-power}, we constructed for each $\mathcal B \in \Cont$ and $K \in \Cat$ a contextual category $\mathcal B^K$ equipped with an isomorphism
$$
\iiHom(\mathcal A;\mathcal B^K) \cong \iiHom(\mathcal A;\mathcal B)^K
$$
natural in each argument; we also saw that $\mathcal B^K \cong \mathcal B^{K_{\text{pre}}}$ where $K_{\text{pre}}$ is a precontextual category obtained by regarding each object of $K$ as a length-$1$ object (see Proposition \ref{prop: precontextual category from a category}).

\begin{lemma}
\label{lem: isomorphism category of bimorphisms}
For contextual categories $\mathcal A$, $\mathcal B$, $\mathcal C$, the functor
\[
\tag{\texttt{*}}
\iiHom(\mathcal A \otimes \mathcal B;\;\mathcal C) \longrightarrow \iiHom(\mathcal A,\mathcal B;\;\mathcal C)
\]
given by precomposition with $\otimes_{\mathcal A,\mathcal B}:\mathcal A \otimes \mathcal B \rightarrow \mathcal A \otimes \mathcal B$ is an isomorphism of categories.
\end{lemma}

\begin{proof}
By definition, ($\texttt{*}$) is bijective on objects. Also, we have a commutative diagram
\[\begin{tikzcd}[ampersand replacement=\&]
	{\Hom(\mathcal A \otimes \mathcal B,\mathcal C^{[1]})} \& {\Hom(\mathcal A,\mathcal B;\mathcal C^{[1]})} \\
	{\Fun([1],\iiHom(\mathcal A \otimes \mathcal B;\mathcal C))} \& {\Fun([1],\iiHom(\mathcal A,\mathcal B;\mathcal C))}
	\arrow["{- \; \circ \; \otimes_{\mathcal A,\mathcal B}}", from=1-1, to=1-2]
	\arrow["\cong"', from=1-1, to=2-1]
	\arrow["{(\texttt{**})}", from=1-2, to=2-2]
	\arrow[from=2-1, to=2-2]
\end{tikzcd}\]
where the bottom map is the action of ($\texttt{*}$) on arrows, the left one is the action on objects of $\iiHom(\mathcal A \otimes \mathcal B;\mathcal C)^{[1]} \cong \iiHom(\mathcal A \otimes \mathcal B,\mathcal C^{[1]})$, and ($\texttt{**}$) sends a bimorphism $F:\mathcal A \times \mathcal B \rightarrow \mathcal C^{[1]}$ to the natural transformation corresponding to the composite $\mathcal A \times \mathcal B \xrightarrow{F} \mathcal |C^{[1]}| \xrightarrow{U} |\mathcal C|^{[1]}$.

We claim that ($\texttt{**}$) is bijective, which will imply that ($\texttt{*}$) is bijective on arrows and thus an isomorphism of categories. Note that Theorem \ref{th: isomorphism of multihoms} yields a bijective correspondence between bimorphisms $(\mathcal A,\mathcal B) \rightarrow \mathcal C^{[1]} \cong \mathcal C^{[1]_{\text{pre}}}$ and trimorphisms $([1]_{\text{pre}},\mathcal A,\mathcal B) \rightarrow \mathcal C$. But since $[1]_{\text{pre}}$ has no distinguished squares or objects of length $> 1$, such trimorphisms are precisely the functors $F:[1]_{\text{pre}} \times \mathcal A \times \mathcal B \rightarrow \mathcal C$ such that, denoting the distinguished terminal object of $[1]_{\text{pre}}$ by $z$,
\begin{itemize}
	\item $F(z,-,-):\mathcal A \times \mathcal B \rightarrow \mathcal C$ is constant on $1_\mathcal C$;
	
	\item $F(0,-,-)$, $F(1,-,-):\mathcal A \times \mathcal B \rightarrow \mathcal C$ are bimorphisms.
\end{itemize}
As there is no constraint on the natural transformation $F(0,-,-) \Rightarrow F(1,-,-)$ induced by $0 \rightarrow 1$, the map
$$
\Hom([1]_{\text{pre}},\mathcal A,\mathcal B;\mathcal C) \rightarrow \Fun([1],\iiHom(\mathcal A,\mathcal B;\;\mathcal C))
$$
is an isomorphism.
\end{proof}

More generally, we have:

\begin{proposition}
\label{prop: isomorphism category of multimorphisms}
For $n \ge 1$ and contextual categories $\mathcal A_1$, ..., $\mathcal A_n$, $\mathcal C$, the functor
$$
\iiHom(\mathcal A_1 \otimes \cdots \otimes \mathcal A_n;\; \mathcal C) \longrightarrow \iiHom(\mathcal A_1,...,\mathcal A_n;\; \mathcal C)
$$
given by precomposition with $\otimes_{\mathcal A_1,...,\mathcal A_n}$ (see Notation \ref{not: n-ary tensor product}) is an isomorphism of categories.
\end{proposition}

\begin{proof}
We work by induction on $n$. The statement is trivial for $n = 1$. Given $n \ge 2$, assume that the claim holds for $1$, ..., $n-1$, and consider contextual categories $\mathcal A_1$, ..., $\mathcal A_n$, $\mathcal C$. Then we have isomorphisms
\begin{align*}
	\iiHom(\mathcal A_1 \otimes \mathcal A_2 \otimes \cdots \otimes \mathcal A_n;\; \mathcal C) & \cong \iiHom(\mathcal A_1,\; \mathcal A_2 \otimes \cdots \otimes \mathcal A_n;\; \mathcal C) \tag{Lemma \ref{lem: isomorphism category of bimorphisms}}\\
	& \cong \iiHom(\mathcal A_2 \otimes \cdots \otimes \mathcal A_n;\; \mathcal C^{\mathcal A_1}) \tag{Theorem \ref{th: isomorphism of multihoms}}\\
	& \cong \iiHom(\mathcal A_2,..., \mathcal A_n,\; \mathcal C^{\mathcal A_1}) \tag{induction hypothesis}\\
	& \cong \iiHom(\mathcal A_1, \mathcal A_2, ..., \mathcal A_n;\; \mathcal C). \tag{Theorem \ref{th: isomorphism of multihoms}}
\end{align*}
It can be verified in a straightforward way that the composite isomorphism is given by precomposition with $\otimes_{\mathcal A_1,...,\mathcal A_n}$ (we are essentially promoting the bijections from the proof of Proposition \ref{prop: n-ary tensor product}, which are used to define $\otimes_{\mathcal A_1,...,\mathcal A_n}$, to isomorphisms of categories).
\end{proof}

\begin{construction}
\label{constr: tensoring natural transformations}
Consider contextual categories $\mathcal A$, $\mathcal A'$, $\mathcal B$, $\mathcal B'$. Given contextual functors and natural transformations as in
\[\begin{tikzcd}[ampersand replacement=\&]
	{\mathcal A} \& {\mathcal B,} \& {\mathcal A'} \& {\mathcal B',}
	\arrow[""{name=0, anchor=center, inner sep=0}, "F", curve={height=-12pt}, from=1-1, to=1-2]
	\arrow[""{name=1, anchor=center, inner sep=0}, "G"', curve={height=12pt}, from=1-1, to=1-2]
	\arrow[""{name=2, anchor=center, inner sep=0}, "{F'}", curve={height=-12pt}, from=1-3, to=1-4]
	\arrow[""{name=3, anchor=center, inner sep=0}, "{G'}"', curve={height=12pt}, from=1-3, to=1-4]
	\arrow["\varphi",Rightarrow,shorten=2pt, from=0, to=1]
	\arrow["{\varphi'}",Rightarrow,shorten=2pt, from=2, to=3]
\end{tikzcd}\]
whiskering the diagram
\[\begin{tikzcd}[ampersand replacement=\&]
	{\mathcal A\times \mathcal A'} \&\& {\mathcal B \times \mathcal B'} \&\& {\mathcal B \otimes \mathcal B'.}
	\arrow[""{name=0, anchor=center, inner sep=0}, "{F \times F'}", curve={height=-12pt}, from=1-1, to=1-3]
	\arrow[""{name=1, anchor=center, inner sep=0}, "{G \times G'}"', curve={height=12pt}, from=1-1, to=1-3]
	\arrow["{\otimes_{\mathcal B,\mathcal B'}}", from=1-3, to=1-5]
	\arrow["{\varphi \times \varphi'}",shorten=2pt, Rightarrow, from=0, to=1]
\end{tikzcd}\]
yields a natural transformation $\otimes_{\mathcal B,\mathcal B'}(\varphi \times \varphi'): \otimes_{\mathcal B,\mathcal B'}(F \times F') \Rightarrow \otimes_{\mathcal B,\mathcal B'}(G \times G')$. As the source and target functors are bimorphisms $(\mathcal A,\mathcal A') \rightarrow \mathcal B \otimes \mathcal B'$, it follows from Proposition \ref{prop: isomorphism category of multimorphisms} that $\otimes_{\mathcal B,\mathcal B'}(\varphi \times \varphi')$ extends uniquely to a natural transformation $F \otimes F' \Rightarrow G \otimes G'$. We denote it by $\varphi \otimes \varphi'$. More explicitly, $\varphi \otimes \varphi'$ is completely characterized by the statement that
$$
(\varphi \otimes \varphi')_{a \otimes a'} = \varphi_a \otimes \varphi_{a'}
$$
for all $a \in \mathcal A$ and $a' \in \mathcal A'$ or, equivalently, that the diagram
\[\begin{tikzcd}[ampersand replacement=\&]
	{\mathcal A \otimes \mathcal A'} \&\& {\mathcal B \otimes \mathcal B'} \\
	\\
	{\mathcal A\times \mathcal A'} \&\& {\mathcal B \times \mathcal B'.}
	\arrow[""{name=0, anchor=center, inner sep=0}, "{F \otimes F'}", curve={height=-12pt}, from=1-1, to=1-3]
	\arrow[""{name=1, anchor=center, inner sep=0}, "{G \otimes G'}"', curve={height=12pt}, from=1-1, to=1-3]
	\arrow["{\otimes_{\mathcal A,\mathcal A'}}", from=3-1, to=1-1]
	\arrow[""{name=2, anchor=center, inner sep=0}, "{F \times F'}", curve={height=-12pt}, from=3-1, to=3-3]
	\arrow[""{name=3, anchor=center, inner sep=0}, "{G \times G'}"', curve={height=12pt}, from=3-1, to=3-3]
	\arrow["{\otimes_{\mathcal B,\mathcal B'}}"', from=3-3, to=1-3]
	\arrow["{\varphi \otimes \varphi'}",shorten=2pt, Rightarrow, from=0, to=1]
	\arrow["{\varphi \times \varphi'}",shorten=2pt, Rightarrow, from=2, to=3]
\end{tikzcd}\]
commutes. Pasting such diagrams yields equalities
\begin{align*}
	(\psi \otimes \psi') \circ (\varphi \otimes \varphi') & = (\psi \circ \varphi) \otimes (\psi' \circ \varphi'),\\
	(1_R \otimes 1_S) \circ (\varphi \otimes \varphi') & = \varphi \otimes \varphi',\\
	(\varphi \otimes \varphi') \circ (1_R \otimes 1_S) & = \varphi \otimes \varphi'
\end{align*}
whenever the arguments are composable as indicated, where $\circ$ denotes vertical composition of natural transformations and $1_R:R \Rightarrow R$ (resp. $1_S:S \Rightarrow S$) is the identity of a contextual functor $R$ (resp. $S$). Similarly, we have three equalities concerning horizontal composition.

These relations imply that $\otimes:\Cont \times \Cont \rightarrow \Cont$ extends, via the above construction, into a strict $2$-functor $\iiCont \times \iiCont \rightarrow \iiCont$. A direct calculation shows that the natural transformations $\alpha$, $\lambda$ and $\rho$ from \S\ref{sec: monoidal structure} are in fact (strict) $2$-natural transformations between the appropriate strict $2$-functors.
\end{construction}

\begin{remark}
We know from \S\ref{subsec: cat-power} that $\iiCont$ is $\Cat$-powered. Now, we can also check that it is $\Cat$-tensored: for $\mathcal A$, $\mathcal B \in \iiCont$ and $K \in \Cat$, we have
\begin{align*}
	\iiHom(L(K_{\text{pre}}) \otimes \mathcal A,\mathcal B) & \cong \iiHom(L(K_{\text{pre}}),\mathcal A;\; \mathcal B) \tag{Lemma \ref{lem: isomorphism category of bimorphisms}}\\
	 & \cong \iiHom(\mathcal A,\; \mathcal B^{L(K_{\text{pre}})}) \tag{Theorem \ref{th: isomorphism of multihoms}}\\
	 & \cong \iiHom(\mathcal A,\; \mathcal B^{K_{\text{pre}}}) \tag{Proposition \ref{prop: multimorphisms from contextual vs precontextual categories}}\\
	 & \cong \iiHom(\mathcal A,\; \mathcal B^K). \tag{Proposition \ref{prop: precontextual category from a category}}
\end{align*}
\end{remark}

\begin{remark}
Let $\mathcal A$, $\mathcal B$ be contextual categories. Working in a Grothendieck universe $\mathscr U^+ \ni \mathscr U = \Ob(\Set)$, Theorem \ref{th: isomorphism of multihoms} and Proposition \ref{prop: isomorphism category of multimorphisms} provide isomorphisms between the category of family-valued models $\iiHom(\mathcal A \otimes \mathcal B,\Fam)$ and
$$
\iiHom(\mathcal A;\; \Fam^\mathcal B), \quad \iiHom(\mathcal B;\; \Fam^\mathcal A), \quad \iiHom(\mathcal A,\mathcal B;\; \Fam).
$$
However, using this to derive a statement about $\Set$-models of the underlying clan of $\mathcal A \otimes \mathcal B$ is a subtle matter (see the discussion in \cite{BarHen25}, Remarks 2.17 and B.54) beyond the scope of the present article; we will study it in forthcoming work.
\end{remark}

\subsection{The embedding $\Cat \rightarrow \Cont$}
\label{subsec: embedding Cat into Cont}

It can be helpful to think about the above results in terms of the embedding of $\Cat$ into $\Cont$ given by composing
$$
\Cat \overset{(-)_{\text{pre}}}{\longrightarrow} \Precont \overset{L}{\longrightarrow} \Cont.
$$
Given $I$, $J \in \Cat$, as $I_{\text{pre}}$ and $J_{\text{pre}}$ have no distinguished squares or objects of length $> 1$, the composite
$$
I_{\text{pre}} \times J_{\text{pre}} \longrightarrow (I \times J)_{\text{pre}} \longrightarrow L((I \times J)_{\text{pre}}),
$$
say $\kappa$, is a universal bimorphism out of $(I_{\text{pre}},J_{\text{pre}})$. Now, Proposition \ref{prop: multimorphisms from contextual vs precontextual categories} implies that there exists a unique bimorphism $\kappa':L(I_{\text{pre}}) \times L(J_{\text{pre}}) \rightarrow L((I \times J)_{\text{pre}})$ such that $\kappa = \kappa' \circ (\iota_{I_{\text{pre}}} \times \iota_{J_{\text{pre}}})$ (where $\iota_\bullet:\bullet \rightarrow L(\bullet)$ is the reflection morphism), and, in fact, $\kappa'$ is also a universal bimorphism. This induces an isomorphism $\mu_{I,J}:L(I_{\text{pre}}) \otimes L(J_{\text{pre}}) \rightarrow L((I \times J)_{\text{pre}})$ which, by construction, is uniquely characterized as a contextual functor by the statement that, letting $\iota'_K$ be the composite $K \hookrightarrow K_{\text{pre}} \overset{\iota_K}{\rightarrow} L(K_{\text{pre}})$ for $K \in \Cat$,
\[
\begin{tikzcd}
	L(I_{\text{pre}}) \otimes L(J_{\text{pre}}) \arrow[]{rr}{\mu_{I,J}} & & L((I \times J)_{\text{pre}})\\
	& I \times J \arrow[]{ul}{\iota'_I(-) \otimes \iota'_J(-)} \arrow[swap]{ur}{\iota'_{I \times J}} & 
\end{tikzcd}
\]
strictly commutes. A syntactic version of this isomorphism was described in \cite{Alm25}, \S 3.2. Note that we also have an isomorphism $\eta:L(*_{\text{pre}}) \cong L(\mathcal O_1^{\text{pre}}) = \mathcal O_1$, where $*$ denotes the terminal category.

The isomorphisms $\mu_{I,J}$ ($I$, $J \in \Cat$) and $\eta$ realize $L(-_{\text{pre}}):\Cat \rightarrow \Cont$ as a strong monoidal functor --- see \cite{MLa98}, XI, 2 ---  with respect to $\times$ and $\otimes$; to verify the required coherence equalities, we use the fact that a morphism out of $L(I_{\text{pre}}) \otimes L(J_{\text{pre}})$ can be recovered from its action on pure tensors $f \otimes g$ for $f \in \Ar(I)$, $g \in \Ar(J)$. Moreover:

\begin{proposition}
\label{prop: Lpre is symmetric monoidal}
The strong monoidal functor $(L(-)_{\text{pre}},\mu,\eta): (\Cat,\times,\cdots) \rightarrow (\Cont,\otimes,\cdots)$ is symmetric: for $I$, $J \in \Cat$, the diagram
\[
\widesqua{L(I_{\text{pre}}) \otimes L(J_{\text{pre}})}{L(J_{\text{pre}}) \otimes L(I_{\text{pre}})}{L((I \times J)_{\text{pre}})}{L((J \times I)_{\text{pre}})}{\beta_{L(I_{\text{pre}}),L(J_{\text{pre}})}}{L((\beta_{I,J})_{\text{pre}})}{\mu_{I,J}}{\mu_{J,I}}
\]
commutes.
\end{proposition}

\begin{proof}
Below, we denote $\iota'_K(h)$ by $\underline{h}$ for $K \in \Cat$, $h \in \Ar(K)$.

Since $I_{\text{pre}} \times J_{\text{pre}} \rightarrow L(I_{\text{pre}}) \otimes L(J_{\text{pre}})$ is a universal bimorphism, it suffices to prove that the two composites coincide when applied to $\underline{f} \otimes \underline{g}$ for all $f \in \Ar(I)$, $g \in \Ar(J)$. The key point is that, as explained in Remark \ref{rem: shuffling is trivial on length-1 objects}, the shuffling construction from Proposition \ref{prop: shuffling diagram} acts trivially on tuples of length-$1$ objects (or of arrows between such objects): using that every object in the image of $\iota'_I$ or $\iota'_J$ has length $1$, we calculate
\begin{align*}
	\mu_{J,I}\big(\beta_{L(I_{\text{pre}}), L(J_{\text{pre}})}(\underline{f} \otimes \underline{g})\big) & = \mu_{J,I}\big(\underline{g} \otimes \underline{f}\big) \tag{Remark \ref{rem: shuffling is trivial on length-1 objects}}\\
	 & = \underline{(g,f)}\\
	 & = L((\beta_{I,J})_{\text{pre}})\underline{(f,g)}\\
	 & = L((\beta_{I,J})_{\text{pre}})\big(\mu_{I,J}(\underline{f} \otimes \underline{g})\big).
\end{align*}
\end{proof}

Note that $L(-_{\text{pre}}):\Cat \rightarrow \Cont$ is a left adjoint: letting $\iiOb_1:\Cont \rightarrow \Cat$ be the functor that sends a contextual category to its full subcategory spanned by the length-$1$ objects, we have isomorphisms
$$
\Hom_\Cat(I,\iiOb_1(\mathcal A)) \cong \Hom_\Precont(I_{\text{pre}},\mathcal A) \cong \Hom_\Cont(L(I_{\text{pre}}), \mathcal A)
$$
natural in $I \in \Cat^{\text{op}}$ and $\mathcal A \in \Cont$. The right adjoint $\iiOb_1$ can be equipped with a structure of (lax) monoidal functor. For $\mathcal A$, $\mathcal B \in \Cont$, the bimorphism $\otimes_{\mathcal A,\mathcal B}:\mathcal A \times \mathcal B \rightarrow \mathcal A \otimes \mathcal B$ sends pairs of length-$1$ objects to length-$1$ objects, so it restricts to a functor
$$
\otimes_{\mathcal A,\mathcal B}^1:\iiOb_1(\mathcal A) \times \iiOb_1(\mathcal B) \longrightarrow \iiOb_1(\mathcal A \otimes \mathcal B).
$$
We also have an isomorphism $\eta':* \rightarrow \iiOb_1(\mathcal O_1)$ where $*$ is a terminal category specified as the unit of $(\Cat,\times,\cdots)$. A routine calculation yields the coherence conditions required for $(\iiOb_1,\otimes^1_{\bullet,\bullet},'\eta)$ to be a monoidal functor (we use that the isomorphisms $\mathcal A \otimes (\mathcal B \otimes \mathcal C) \cong (\mathcal A \otimes \mathcal B) \otimes \mathcal C$ preserve pure tensors, and that $\mathcal O_1 \otimes \mathcal A \cong \mathcal A$ and $\mathcal A \otimes \mathcal O_1 \cong \mathcal A$ send $o_1 \otimes a$ and $a \otimes o_1$, respectively, to $a$). In fact, arguing as in the proof of Proposition \ref{prop: Lpre is symmetric monoidal} shows that $(\iiOb_1,\otimes^1_{\bullet,\bullet},\eta)$ is symmetric.

\begin{remark}
The pairs $(\mu,\eta)$ and $(\otimes^1_{\bullet,\bullet},\eta')$ realizing $L(-)_{\text{pre}}$ and $\iiOb_1$ as symmetric monoidal functors are actually determined from each other via Kelly's theory of doctrinal adjunctions applied to the $2$-monad on $\Cat$ whose algebras are the (possibly symmetric) monoidal categories. We refer to \cite{Kel74} for an explanation of this correspondence.
\end{remark}

This allows us to recover, by \cite{Rie14}, Th. 3.7.11, the fact that $\iiCont$ is tensored and powered over $\Cat$ (that is, we use the monoidal functor $\iiOb_1$ to transfer the self-enrichment of $\Cont$, given by its closed structure, to a $\Cat$ enrichment). Also, Cor. 3.7.12 shows that $L(-_{\text{pre}}) \dashv \iiOb_1$ then becomes a $\Cat$-enriched adjunction.

\vspace{0.5em}

More can be done using that $L(-_{\text{pre}})$ and $\iiOb_1$ are symmetric. Firstly, by \cite{JohYau24}, Prop. 3.1.11 and Th. 3.3.2, the closed symmetric monoidal structure on $\Cont$ gives rise, in a canonical way, to a symmetric monoidal structure on the latter in the $\Cont$-enriched setting --- see \cite{JohYau24}, Def. 1.4.2 and Def. 1.4.13. Now, by \cite{JohYau24}, Th. 2.4.10, the $\Cat$-enrichment of $\Cont$ and the symmetric monoidal structure $(\otimes,\cdots)$ can be canonically joined into a symmetric monoidal $\Cat$-enriched structure on $\iiCont$.

\vspace{0.5em}

To illustrate this approach, let us use it to obtain the natural transformations from Construction \ref{constr: tensoring natural transformations}. For contextual categories $\mathcal A$, $\mathcal A'$, $\mathcal B'$, $\mathcal B'$, we have an ``internal tensor product of arrows" morphism (appearing as $\text{Ten}$ in \cite{Kel82}, \S1.6)
$$
T:\mathcal B^{\mathcal A} \otimes \mathcal B^{'\mathcal A'} \longrightarrow (\mathcal B \otimes \mathcal B')^{\mathcal A \otimes \mathcal A'}
$$
obtained by transporting the composite $(\mathcal A \otimes \mathcal A') \otimes (\mathcal B^\mathcal A \otimes \mathcal B^{'\mathcal A'}) \cong (\mathcal A \otimes \mathcal B^\mathcal A) \otimes (\mathcal A' \otimes \mathcal B^{'\mathcal A'}) \xrightarrow{ev_{\mathcal A,\mathcal B} \otimes ev_{\mathcal A',\mathcal B'}} \mathcal B \otimes \mathcal B'$, where $ev_{\mathcal A,\mathcal B}$ and $ev_{\mathcal A',\mathcal B'}$ are the evaluation maps, along the bijection
$$
\Hom((\mathcal A \otimes \mathcal A') \otimes (\mathcal B^\mathcal A \otimes \mathcal B^{'\mathcal A'}); \mathcal B \otimes \mathcal B') \cong \Hom(\mathcal B^{\mathcal A} \otimes \mathcal B^{'\mathcal A'}; (\mathcal B \otimes \mathcal B')^{\mathcal A \otimes \mathcal A'}).
$$
Then we get a chain of functors
\[\begin{tikzcd}[ampersand replacement=\&,row sep=tiny]
	{\iiHom(\mathcal A;\mathcal B) \times \iiHom(\mathcal A';\mathcal B')} \&\&\&\& {\iiHom(\mathcal A \otimes \mathcal A';\mathcal B \otimes \mathcal B')} \\
	{\iiOb_1(\mathcal B^\mathcal A) \times \iiOb_1(\mathcal B^{'\mathcal A'})} \&\& {\iiOb_1(\mathcal B^\mathcal A \otimes \mathcal B^{'\mathcal A'})} \&\& {\iiOb_1((\mathcal B \otimes \mathcal B')^{\mathcal A \otimes \mathcal A'})}
	\arrow["\cong"{marking, allow upside down}, draw=none, from=1-1, to=2-1]
	\arrow["\cong"{marking, allow upside down}, draw=none, from=1-5, to=2-5]
	\arrow["{\otimes^1_{\mathcal B^\mathcal A, \mathcal B^{'\mathcal A'}}}"', from=2-1, to=2-3]
	\arrow["{\iiOb_1(T)}"', from=2-3, to=2-5]
\end{tikzcd}\]
Its composite is of the desired form, and checking that it matches Construction \ref{constr: tensoring natural transformations} can be done by tracking the action of $T$ on length-$1$ pure tensors.

\nocite{*}

\printbibliography

\end{document}